\documentclass[a4paper, 10pt, dvipsnames, svgnames]{article}

\usepackage[utf8]{inputenc}
\usepackage[top=1.0in, bottom=1.0in, left=1.1in, right=1.1in]{geometry}
\usepackage[all]{xy}
\usepackage{amsfonts}
\usepackage{amsmath}
\usepackage{amssymb}
\usepackage{amsthm}
\usepackage{calrsfs}
\usepackage{enumerate}
\usepackage{enumitem}
\usepackage{xparse, etoolbox}
\usepackage{ifthen}
\usepackage{mathtools, nccmath, textcomp}
\usepackage{relsize}
\usepackage{rsfso}
\usepackage{stmaryrd}
\usepackage{sectsty}
\usepackage{setspace}
\usepackage[nottoc]{tocbibind}
\usepackage{tensor}
\usepackage{tikz}
\usepackage{tikz-cd}
\usepackage[sf, bf]{titlesec}	
\usepackage[titles]{tocloft}
\usepackage{xcolor}
\usepackage[hyperref, backend=biber, style=alphabetic, sorting=nyt, doi=false, url=false, maxbibnames=99, maxalphanames=99]{biblatex}
\usepackage[bbgreekl]{mathbbol}

\usepackage[T1]{fontenc}
\usepackage{lmodern}

\definecolor{bondiblue}{rgb}{0.1, 0.58, 0.71}
\usepackage{hyperref}
\hypersetup{
	colorlinks=true,
	linkcolor=blue,
	citecolor=blue,
	urlcolor=black,
}

\usepackage{imakeidx}
\makeindex

\usepackage{fancyhdr}

\newcommand{\addperiod}[1]{#1.}
\titleformat{\section}[block]{\scshape\Large\filcenter}{\thesection.}{1em}{}
\titleformat{\subsection}[runin]{\normalfont\large\bfseries}{\thesubsection.}{1em}{\addperiod}
\titleformat{\subsubsection}[runin]{\normalfont\bfseries}{\thesubsubsection.}{1em}{\addperiod}

\numberwithin{equation}{section}

\DeclareSymbolFontAlphabet{\mathbbl}{bbold}
\newcommand{\prism}{{\mathlarger{\mathbbl{\Delta}}}}

\makeatletter
\def\keywords{\xdef\@thefnmark{}\@footnotetext}
\makeatother

\theoremstyle{plain}
\newtheorem{thm}{Theorem}[section]

\newtheorem{lem}[thm]{Lemma}
\newtheorem{prop}[thm]{Proposition}

\theoremstyle{definition}
\newtheorem{defi}[thm]{Definition}

\newtheorem{const}[thm]{Construction}
\newtheorem{assum}[thm]{Assumption}
\newtheorem{exam}[thm]{Example}
\newtheorem{rem}[thm]{Remark}

\AtEndEnvironment{const}{\qed}

\theoremstyle{remark}

\newtheorem*{nota}{Notation}

\newenvironment{enumarabicup}
{\begin{enumerate}[font=\upshape, labelindent=\parindent, label=(\arabic*)]}
{\end{enumerate}}

\DeclareMathOperator*{\colim}{\text{\scalebox{0.95}{$\textup{colim}$}}}

\DeclareMathAlphabet{\pazocal}{OMS}{zplm}{m}{n}
\DeclareMathAlphabet{\dutchcal}{U}{dutchcal}{m}{n}

\newcommand{\isomorphic}{\xrightarrow{\hspace{0.5mm} \sim \hspace{0.5mm}}}
\newcommand{\lisomorphic}{\xleftarrow{\hspace{0.5mm} \sim \hspace{0.5mm}}}

\newcommand{\pazc}{\pazocal{C}}

\newcommand{\paze}{\pazocal{E}}
\newcommand{\pazf}{\pazocal{F}}

\newcommand{\pazi}{\pazocal{I}}

\newcommand{\pazo}{\pazocal{O}}
\newcommand{\pazp}{\pazocal{P}}

\def\FF{\mathbb{F}}

\def\LL{\mathbb{L}}
\def\NN{\mathbb{N}}
\def\QQ{\mathbb{Q}}

\def\ZZ{\mathbb{Z}}

\newcommand{\mbfd}{\mathbf{D}}

\newcommand{\mbfn}{\mathbf{N}}

\newcommand{\smbfi}{\mathbf{i}}

\newcommand{\smbfk}{\mathbf{k}}

\newcommand{\frakS}{\mathfrak{S}}

\newcommand{\frakm}{\mathfrak{m}}

\newcommand{\algebra}{\textrm{-algebra}}
\newcommand{\algebras}{\textrm{-algebras}}
\newcommand{\subalgebra}{\textrm{-subalgebra}}
\newcommand{\an}{\textup{an}}

\newcommand{\Ainf}{A_{\inf}}

\newcommand{\Acrys}{A_{\textup{cris}}}
\newcommand{\OAcrys}{\pazo A_{\textup{cris}}}

\newcommand{\AF}{A_F}

\newcommand{\AR}{A_R}
\newcommand{\BR}{B_R}

\newcommand{\ARpd}{A^{\PD}}

\newcommand{\action}{\textrm{-action}}

\newcommand{\category}{\textrm{-category}}

\newcommand{\cont}{\textup{cont}}
\newcommand{\converge}{\textup{conv}}

\newcommand{\complete}{\textrm{-complete}}
\newcommand{\completely}{\textrm{-completely}}

\newcommand{\CR}{\textup{CR}}
\newcommand{\CRhat}{\textup{CR}^{\wedge}}
\newcommand{\CRhatphi}{\textup{CR}^{\wedge, \varphi}}

\newcommand{\crys}{\textup{cris}}

\newcommand{\CRYS}{\textup{CRIS}}
\newcommand{\crystal}{\textrm{-crystal}}
\newcommand{\crystals}{\textrm{-crystals}}

\newcommand{\dlog}{\hspace{1mm}d\textup{\hspace{0.5mm}log\hspace{0.3mm}}}

\newcommand{\equivariant}{\textrm{-equivariant}}
\newcommand{\eval}{\textup{ev}}
\newcommand{\etale}{\textup{\'et}}

\newcommand{\Fr}{\textup{Frac}}

\newcommand{\Gal}{\textup{Gal}}

\newcommand{\Hom}{\textup{Hom}}

\newcommand{\Abar}{\overline{A}}
\newcommand{\Bbar}{\overline{B}}
\newcommand{\Cbar}{\overline{C}}
\newcommand{\Dbar}{\overline{D}}
\newcommand{\Ebar}{\overline{E}}
\newcommand{\Czerobar}{\overline{C}_0}
\newcommand{\Cnbar}{\overline{C}_n}
\newcommand{\Dzerobar}{\overline{D}_0}
\newcommand{\Dnbar}{\overline{D}_n}
\newcommand{\Ezerobar}{\overline{E}_0}
\newcommand{\Yizerobar}{\overline{Y}_{i,0}}
\newcommand{\Yinbar}{\overline{Y}_{i,n}}
\newcommand{\Yimbar}{\overline{Y}_{i,m}}
\newcommand{\Yibar}{\overline{Y}_i}
\newcommand{\yibar}{\overline{y}_i}

\newcommand{\kert}{\textup{Ker }}

\newcommand{\LambdaFplus}{\Lambda_{F,+}}
\newcommand{\LambdaRplus}{\Lambda_{R,+}}
\newcommand{\LambdaRminus}{\Lambda_{R,-}}
\newcommand{\Lambdatilde}{\tilde{\Lambda}}
\newcommand{\LambdatildeR}{\tilde{\Lambda}_R}
\newcommand{\LambdatildeRo}{\tilde{\Lambda}_{R, 0}}
\newcommand{\LambdatildeRplus}{\tilde{\Lambda}_{R,+}}
\newcommand{\LambdatildeRminus}{\tilde{\Lambda}_{R,-}}

\newcommand{\Loc}{\textup{Loc}}
\newcommand{\linear}{\textrm{-linear}}

\newcommand{\module}{\textrm{-module}}
\newcommand{\modules}{\textrm{-modules}}
\newcommand{\mubar}{\overline{\mu}}

\newcommand{\MIC}{\textup{MIC}}

\newcommand{\Mod}{\textup{-Mod}}

\newcommand{\MLambda}{M_{\Lambda}}
\newcommand{\MLambdao}{M_{\Lambda, 0}}
\newcommand{\MLambdaplus}{M_{\Lambda, +}}
\newcommand{\MLambdaminus}{M_{\Lambda, -}}
\newcommand{\MLambdatilde}{M_{\tilde{\Lambda}}}
\newcommand{\MLambdatildeo}{M_{\tilde{\Lambda}, 0}}
\newcommand{\MLambdatildeplus}{M_{\tilde{\Lambda}, +}}
\newcommand{\MLambdatildeminus}{M_{\tilde{\Lambda}, -}}

\newcommand{\Nbar}{\overline{N}}

\newcommand{\wotimes}{\widehat{\otimes}}
\newcommand{\Finfty}{F_{\infty}}

\newcommand{\Aframe}{A_{\square}}

\newcommand{\Rframe}{R^{\square}}

\newcommand{\adic}{\textrm{-adic}}
\newcommand{\padic}{p\textrm{-adic}}
\newcommand{\adically}{\textrm{-adically}}
\newcommand{\pqheight}{[p]_q\textrm{-height}}

\newcommand{\PD}{\textup{PD}}

\newcommand{\ptilde}{\tilde{p}}
\newcommand{\stilde}{\tilde{s}}
\newcommand{\tautilde}{\tilde{\tau}}
\newcommand{\rank}{\textup{rk}}
\newcommand{\Rbar}{\overline{R}}
\newcommand{\Rep}{\textup{Rep}}

\newcommand{\Rinfty}{R_{\infty}}

\newcommand{\ring}{\textrm{-ring}}
\newcommand{\Rings}{\textup{Rings}}

\newcommand{\Shv}{\textup{Shv}}
\newcommand{\Spec}{\textup{Spec}\hspace{0.5mm}}
\newcommand{\Spf}{\textup{Spf }}

\newcommand{\qconnection}{q\textrm{-connection}}
\newcommand{\qconnections}{q\textrm{-connections}}

\newcommand{\Strat}{\textup{Strat}}

\newcommand{\textmod}{\textup{ mod }}

\newcommand{\torsion}{\textrm{-torsion}}
\newcommand{\torsionfree}{\textrm{-torsion-free}}

\newcommand{\Vect}{\textup{Vect}}

\title{\textsc{Prismatic $F\crystals$ and Wach modules}}
\author{\textsc{Abhinandan} \\ \footnotesize{IMJ-PRG, Sorbonne Universit\'e, 4 Place Jussieu, Paris, France} \\ \footnotesize{E-mail: \href{abhinandan@imj-prg.fr}{abhinandan@imj-prg.fr}}}

\newcommand{\Addresses}{{
  \footnotesize

  \rule{2cm}{0.4pt}\vspace{2mm}

  \textsc{Abhinandan}\par\nopagebreak
  \textsc{IMJ-PRG, Sorbonne Universit\'e, 4 Place Jussieu, Paris, France}\par\nopagebreak\vspace{-0.7mm}
  \textit{E-mail}: \footnotesize{\href{abhinandan@imj-prg.fr}{abhinandan@imj-prg.fr}}, \textit{Web}: \footnotesize{\href{https://abhinandan.perso.math.cnrs.fr/}{https://abhinandan.perso.math.cnrs.fr/}}
}}

\date{}

\begin{document}

\fontdimen2\font=0.3em
\pagenumbering{arabic}


\maketitle
{
	\vspace{-3mm}
	\textsc{Abstract.} We show that the category of analytic/completed prismatic $F\crystals$ on the absolute prismatic site of a small (unramified at $p$) base ring is naturally equivalent to the category of relative Wach modules from the theory of $(\varphi, \Gamma)\modules$.
	The result is obtained by showing that the data of the Galois action on a Wach module is equivalent to the data of a prismatic stratification on the underlying $\varphi\module$.
	Along the way, we obtain new descent results for relative Wach modules.
}

\sloppy


\section{Introduction}

Recent groundbreaking advances in integral $\padic$ Hodge theory have been brought on by the seminal works of Bhatt, Morrow and Scholze on $\Ainf\textrm{-cohomology}$ in \cite{bhatt-morrow-scholze-1, bhatt-morrow-scholze-2}, and of Bhatt and Scholze on prismatic cohomology in \cite{bhatt-scholze-prisms}.
In the latter theory, the study of prismatic $F\crystals$ has led to exciting applications towards the classification of $p\textrm{-divisible}$ groups over a quasisyntomic base \cite{anschutz-lebras}, and more generally, of all $p\adic$ crystalline local systems over smooth $\padic$ formal schemes \cite{bhatt-scholze-crystals, du-liu-moon-shimizu, guo-reinecke}.
However, similar to the crystals appearing in the theory of crystalline cohomology of Grothendieck and Berthelot, the prismatic $F\crystals$ are mysterious objects.

To unravel these objects, a common idea is to describe them in terms of certain equivalent and computable data, for example, in the crystalline cohomology theory, one usually replaces crystals with modules equipped with a flat connection as in \cite{berthelot-cohomologie-cristalline}.
In the prismatic theory, several successful attempts have been made to understand prismatic crystals in terms of more explicit data in various settings, for example, crystals on the relative prismatic site in terms of generalised representations, $q\textrm{-de Rham complexes}$ and $q\textrm{-Higgs fields}$ in \cite{morrow-tsuji, tsuji-qhiggs-field}, Hodge-Tate crystals in terms of Higgs fields in \cite{tian}, Laurent prismatic $F\crystals$ in terms of \'etale $(\varphi, \Gamma)\modules$ and related objects in \cite{wu, du-liu, min-wang}, and crystals on the relative/absolute prismatic site in terms of twisted/absolute differential calculus in \cite{gros-lestum-quiros-1, gros-lestum-quiros-2}.

The aim of this article is to describe analytic/completed prismatic $F\crystals$ of \cite{du-liu-moon-shimizu, guo-reinecke} in terms of a more explicit data.
More precisely, let $R$ denote a $p\textrm{-completely}$ \'etale algebra over a framed $\ZZ_p\textrm{-algebra}$ (see Section \ref{subsec:setup_nota} for notation), and let $R_{\prism}$ denote its absolute prismatic site (see Section \ref{subsec:prismatic_site_crystals} for definitions).
Then, we show the following (see Theorem \ref{intro_thm:fcrystals_wachmod} for a more precise statement):
\begin{thm}\label{intro_thm:analyticfcrys_wachmod_equiv}
	There exists a natural equivalence of categories:
	\begin{equation*}
		\{\textrm{Analytic prismatic } F\textrm{-crystals on } R_{\prism}\} \isomorphic \{\textrm{Wach modules for } R\}.
	\end{equation*}
\end{thm}
Here the category of Wach modules for $R$ denotes a certain full subcategory of \'etale $(\varphi, \Gamma)\modules$ for $R$, and was introduced in \cite{abhinandan-relative-wach-ii} (see Definition \ref{defi:wach_mods_relative}) to study $\ZZ_p\textrm{-lattices}$ inside $\padic$ crystalline representations of the \'etale fundamental group of $R[1/p]$ (see Section \ref{subsec:relation_previous_works}).
Theorem \ref{intro_thm:analyticfcrys_wachmod_equiv} is obtained by showing that the data of the Galois action (i.e.\ $\Gamma\action$) on a Wach module is \textit{equivalent} to the data of a prismatic stratification on the underlying $\varphi\module$ (see Theorem \ref{intro_thm:strat_wach_comp_relative}).
Obtaining such an equivalence is highly non-trivial and constitutes the heart of this article (see Remarks \ref{intro_rem:gammaact_prismstrat} and \ref{rem:3stepdescent_explain}).
Moreover, our approach is different and independent of all previous methods and results in \cite{abhinandan-relative-wach-ii}, \cite{bhatt-scholze-crystals, guo-reinecke, du-liu-moon-shimizu, du-liu, wu, min-wang}.

Let us note that in the case that $R$ is a complete discrete valuation ring with perfect residue field, Wach modules were studied in \cite{fontaine-phigamma}, \cite{wach-free, wach-torsion}, \cite{colmez-hauteur} and \cite{berger-limites}.
In this case, using our methods, we also show that the classical Wach modules from \cite{wach-free, berger-limites} descend to a smaller ring, beyond the Fontaine-Laffaille case treated in \cite{wach-torsion} (see Theorem \ref{intro_thm:wachmod_fpx_descent} and Remark \ref{intro_rem:wachmod_fpx_descent}).
Furthermore, let us note that the theory of Wach modules (in the context of $(\varphi, \Gamma)\modules$) and its relationship with crystalline representations in different settings was studied in \cite{abhinandan-relative-wach-i, abhinandan-imperfect-wach, abhinandan-relative-wach-ii}.

Besides being an explicit description of prismatic $F\crystals$, the usefulness of Wach modules stems from its applications towards the computation of $\padic$ vanishing cycles via syntomic complexes.
Indeed, in order to generalise the computation of $\padic$ vanishing cycles via crystalline syntomic complexes of \cite{colmez-niziol}, to the case of crystalline coefficients, in \cite{abhinandan-syntomic} we used the theory of Wach modules from \cite{abhinandan-relative-wach-i} as an important ingredient.
However, the results obtained in \cite{abhinandan-relative-wach-i} and \cite{abhinandan-syntomic} were restrictive.
Furthermore, note that in \cite{bhatt-morrow-scholze-2}, the authors defined a prismatic syntomic complex for smooth $\padic$ formal schemes, and compared it to the complex of $\padic$ vanishing cycles integrally.
Beyond the smooth case, similar comparison results have also been obtained in \cite{antieau-mathew-morrow-nikolaus} and \cite{bhatt-mathew}, where the latter uses the theory of prismatic cohomology.
The preceding results were obtained for the case of constant coefficients, and it is natural to ask the following: is it possible to generalise \cite[Theorem 10.1]{bhatt-morrow-scholze-2} to arbitrary crystalline coefficients, i.e.\ can one define a prismatic syntomic complex with coefficients in a prismatic $F\crystal$, and compare it to the complex of $\padic$ vanishing cycles?
In our approach to providing an answer to the preceding question, the prismatic interpretation of Wach modules from the current paper and the relationship between Wach modules and crystalline representations from \cite{abhinandan-relative-wach-ii} will serve as crucial inputs.

In the rest of this section, we will make Theorem \ref{intro_thm:analyticfcrys_wachmod_equiv} more precise (see Theorem \ref{intro_thm:fcrystals_wachmod}), and discuss the key ingredients required for its proof, i.e.\ Theorem \ref{intro_thm:strat_wach_comp_relative}, Theorem \ref{intro_thm:integral_comp_relative} and Theorem \ref{intro_thm:wachmod_fpx_descent}.

\subsection{The categorical equivalence}

Let $p$ be a fixed prime, $\kappa$ a perfect field of characteristic $p$ and $O_F \coloneq W(\kappa)$, the ring of $p\textrm{-typical}$ Witt vectors with coefficients in $\kappa$ and equipped with the natural Witt vector Frobenius endomorphism.
Let $R$ denote the $\padic$ completion of an \'etale algebra over the $p\adically$ completed Laurent polynomial ring $O_F\langle X_1^{\pm 1}, \ldots, X_d^{\pm 1}\rangle$ such that its special fibre $\Spec(R/pR)$ is connected.
We take $X \coloneq \Spf R$ to be an affine $\padic$ formal scheme, and consider its absolute prismatic ringed site $(X_{\prism}, \pazo_{\prism})$ in the sense of \cite{bhatt-scholze-prisms} (see Section \ref{subsec:prismatic_site_crystals}).
Let $\Vect^{\an, \varphi}(X_{\prism})$ denote the category of analytic prismatic $F\crystals$ from \cite{guo-reinecke} (note that one could also work with the equivalent notion of completed prismatic $F\crystals$ from \cite{du-liu-moon-shimizu}, see Section \ref{subsec:prismatic_fcrystals} for definitions and explanations).

Let $\Rinfty$ denote the $R\algebra$ obtained by adjoining to $R$ all $p\textrm{-power}$ roots of unity and $p\textrm{-power}$ roots of $X_i$, for all $1 \leqslant i \leqslant d$.
Set $\Gamma_R \coloneq \Gal(\Rinfty[1/p]/R[1/p]) \isomorphic \ZZ_p(1)^d \rtimes \ZZ_p^{\times}$.
Let $\Ainf(\Rinfty) \coloneq W(\Rinfty^{\flat})$, where $\Rinfty^{\flat}$ denotes the tilt of $\Rinfty$ (see Section \ref{subsec:setup_nota}).
Note that $\Ainf(\Rinfty)$ is equipped with a Witt vector Frobenius endomorphism $\varphi$ and a continuous action of $\Gamma_R$ (see Section \ref{subsec:important_rings}).
Moreover, we have a noetherian $O_F\textrm{-subalgebra}$ $\AR \subset \Ainf(\Rinfty)$, which is complete for the induced topology and stable under the action of $(\varphi, \Gamma_R)$ on $\Ainf(\Rinfty)$ (see Section \ref{subsubsec:ring_ar+}); we equip $\AR$ with the induced structures.
Let $\varepsilon \coloneq (1, \zeta_p, \ldots)$ denote a compatible system of $p\textrm{-power}$ roots of unity in $\Rinfty^{\flat}$ and let $q \coloneq [\varepsilon]$ denote its Teichm\"uller lift in $\Ainf(\Rinfty)$; set $\mu \coloneq q-1$ and $[p]_q \coloneq (q^p-1)/(q-1)$ in $\Ainf(\Rinfty)$.
We denote the category of Wach modules over $\AR$ (see Definition \ref{defi:wach_mods_relative}) by $(\varphi, \Gamma_R)\Mod_{\AR}^{[p]_q}$.

Now, note that the pair $(\AR, [p]_q)$ is a prism, an object of $X_{\prism}$ (see Lemma \ref{lem:ar+_pq_prism}) and a cover of the final object of the topos $\Shv(X_{\prism})$ (see Lemma \ref{lem:ar+_cover_fo}).
Moreover, the action of $\Gamma_R$ on $\AR$ induces automorphisms of $(\AR, [p]_q)$ in $X_{\prism}$ (see Lemma \ref{lem:gammar_act_prisaut}).
Evaluating an analytic prismatic $F\crystal$ on the prism $(\AR, [p]_q)$ gives a Wach module (see Proposition \ref{prop:afcrystal_eval_ar+}), and thus, a well-defined evaluation functor
\begin{equation}\label{intro_eq:afcrystal_eval_ar+}
	\begin{aligned}
		\eval_{\AR}^{\prism} \colon \Vect^{\an, \varphi}(X_{\prism}) &\longrightarrow (\varphi, \Gamma_R)\Mod_{\AR}^{[p]_q}\\
				\pazf &\longmapsto \pazf(\AR, [p]_q).
	\end{aligned}
\end{equation}
\begin{thm}[{Theorem \ref{thm:fcrystals_wachmod_relative}}]\label{intro_thm:fcrystals_wachmod}
	The evaluation functor in \eqref{intro_eq:afcrystal_eval_ar+} induces a natural equivalence of categories $\eval_{\AR}^{\prism} \colon \Vect^{\an, \varphi}(X_{\prism}) \isomorphic (\varphi, \Gamma_R)\Mod_{\AR}^{[p]_q}$.
\end{thm}

Our proof of Theorem \ref{intro_thm:fcrystals_wachmod} is direct, i.e.\ we do not take additional inputs, namely, the equivalence between crystalline $\ZZ_p\textrm{-representations}$ and Wach modules over $\AR$ from \cite{abhinandan-relative-wach-ii}, or the categorical equivalence results from \cite{bhatt-scholze-crystals, guo-reinecke, du-liu-moon-shimizu}.
As mentioned earlier, we show that the $\Gamma_R\action$ on a Wach module is equivalent to a prismatic stratification on the underlying $\varphi\module$.

\begin{rem}\label{intro_rem:gammaact_prismstrat}
	The equivalence between the $\Gamma_R\action$ and a prismatic stratification is not at all obvious: the quotient of $\AR$ by $\Gamma_R$ (in a stack-theoretic sense) is not isomorphic to the Cartier-Witt stack of $R$ (see \cite[Section 3.8]{bhatt-lurie-apc} for $R = \ZZ_p$).
	So, instead we employ an approach inspired by the work of \cite{morrow-tsuji}, where the authors studied coefficients for relative prismatic cohomology.
	However, the results obtained and the techniques employed in our proofs are vastly different from loc.\ cit.
	This stems largely from the fact that our respective settings are quite different: we work over the absolute prismatic site of $R$ and consider analytic prismatic $F\crystals$, whereas in loc.\ cit., the authors worked over the relative prismatic site and considered prismatic $F\crystals$ in vector bundles.
	Moreover, in this article the group $\Gamma_R$ is non-commutative and has an \textit{arithmetic} part (see $\Gamma_F$ in Section \ref{subsec:setup_nota}), in contrast with the commutative and \textit{geometric} $\Gamma$ considered in \cite{morrow-tsuji}.
	As a consequence, the techniques of loc.\ cit.\ cannot be adapted to the present setting (also see Remark \ref{rem:descent_to_strat_explain}).
\end{rem}

\subsection{Crystals as modules with stratification}

In order to prove Theorem \ref{intro_thm:fcrystals_wachmod}, our first course of action is to bring the source of the functor $\eval_{\AR}^{\prism}$ in \eqref{intro_eq:afcrystal_eval_ar+} on a somewhat equal footing with the target of that functor.
This is achieved by interpreting analytic prismatic $F\crystals$ as certain modules with stratification.
More precisely, let us note that $(\AR, [p]_q)$ is a cover of the final object of the topos $\Shv(X_{\prism})$ (see Lemma \ref{lem:ar+_cover_fo}).
We set $\AR(\bullet)$ to be the cosimplicial ring obtained by taking the prismatic \v{C}ech nerve $(\AR(\bullet), I(\bullet))$ of $(\AR, [p]_q)$ in $X_{\prism}$.
Then, it is possible to describe the $n^{\textrm{th}}\textrm{-term}$ of $\AR(\bullet)$ in the site $X_{\prism}$ (see Construction \ref{const:cech_nerve}), and in the case $n = 1, 2$, we may describe $\AR(n)$ very explicitly (see the discussion after Construction \ref{const:cech_nerve}).
Now, let $\Strat^{\an, \varphi}(\AR(\bullet))$ denote the category of pairs $(N, \varepsilon)$, with $N$ being an ``analytic'' and finitely generated $\varphi\module$ over $\AR$ equipped with a stratification $\varepsilon$ with respect to $\AR(\bullet)$ (see Definition \ref{defi:ar+mod_strat}).
Then, by the general theory of crystals, evaluation of analytic $F\crystals$ on the simplicial object $(\AR(\bullet), I(\bullet))$ induces the following natural equivalence of categories (see Construction \ref{const:afcrystal_eval_ar+strat} and Proposition \ref{prop:afcrystal_eval_ar+strat}):
\begin{equation}\label{intro_eq:afcrystal_eval_ar+strat}
	\eval_{\AR(\bullet)}^{\prism} \colon \Vect^{\an, \varphi}(X_{\prism}, \pazo_{\prism}) \isomorphic \Strat^{\an, \varphi}(\AR(\bullet)).
\end{equation}

\subsection{Prismatic stratifications and Galois action on Wach modules}

Our next course of action is to relate the category of $\varphi\modules$ over $\AR$ equipped with a stratification to the category of Wach modules over $\AR$.
So, let $(N, \varepsilon)$ denote an object of the category $\Strat^{\an, \varphi}(\AR(\bullet))$.
Moreover, let us note that the action of $\Gamma_R$ on $(\AR, [p]_q)$ induces a natural action of $\Gamma_R^{\times (n+1)}$ on $\AR(n)$, for each $n \in \NN$.
Then, by using the action of $\Gamma_R^2$ on $\AR(1)$ and the stratification $\varepsilon$, we equip $N$ with a continuous action of $\Gamma_R$ satisfying the properties of a Wach module over $\AR$ from Definition \ref{defi:wach_mods_relative} (see Construction \ref{const:strat_eval_ar+}).
In particular, we have a well-defined natural functor
\begin{equation}\label{intro_eq:strat_eval_ar+}
	\eval_{\AR}^{\Strat} \colon \Strat^{\an, \varphi}(\AR(\bullet)) \longrightarrow (\varphi, \Gamma_R)\Mod_{\AR}^{[p]_q},
\end{equation}
and we show that it agrees with the functors in \eqref{intro_eq:afcrystal_eval_ar+} and \eqref{intro_eq:afcrystal_eval_ar+strat}, i.e.\ the obvious diagram of functors naturally commutes (see Proposition \ref{prop:prisfcrys_strat_wach_relative}).
Consequently, in order to show that \eqref{intro_eq:afcrystal_eval_ar+} induces a categorical equivalence, it is enough to show that \eqref{intro_eq:strat_eval_ar+} induces a categorical equivalence.
Therefore, we show the following:
\begin{thm}[{Theorem \ref{thm:strat_wach_comp_relative}}]\label{intro_thm:strat_wach_comp_relative}
	The functor in \eqref{intro_eq:strat_eval_ar+} induces a natural equivalence of categories.
\end{thm}

In order to prove Theorem \ref{intro_thm:strat_wach_comp_relative}, we define a quasi-inverse to the functor in \eqref{intro_eq:strat_eval_ar+}.
More precisely, given a Wach module $N$ over $\AR$, we use the action of $\Gamma_R$ on $N$ to build a stratification on $N$ with respect to $\AR(\bullet)$, and obtain a natural quasi-inverse to the functor in \eqref{intro_eq:strat_eval_ar+}:
\begin{equation}\label{intro_eq:strat_from_wach_relative}
	\Strat_{\AR(\bullet)} \colon (\varphi, \Gamma_R)\Mod_{\AR}^{[p]_q} \longrightarrow \Strat^{\an, \varphi}(\AR(\bullet)).
\end{equation}
The construction of a stratification on a Wach module, using the action of $\Gamma_R$ on it, is the \textit{main technical heart} of the paper.
This is given via a \textit{3-step argument}.
The three steps correspond to the three subgroups of different nature making up the Galois group $\Gamma_R$ (see Remark \ref{rem:3stepdescent_explain}).
Also note that the case $p=2$ is different from the cases $p \geqslant 3$, and requires entirely different and quite technical arguments.
We refer the reader to Section \ref{subsec:wach_and_strat} for precise details on the construction of the functor in \eqref{intro_eq:strat_from_wach_relative}.

\begin{rem}\label{rem:3stepdescent_explain}
	Note that we have $\Gamma_R \isomorphic \ZZ_p(1)^d \rtimes \ZZ_p^{\times}$.
	Furthermore, $\ZZ_p^{\times}$ fits into the following exact sequence:
	\begin{equation*}
		1 \longrightarrow \Gamma_0 \longrightarrow \ZZ_p^{\times} \longrightarrow \Gamma_{\textrm{tor}} \longrightarrow 1,
	\end{equation*}
	where, for $p \geqslant 3$ we have $\Gamma_0 \isomorphic 1 + p\ZZ_p$, and for $p=2$ we have $\Gamma_0 \isomorphic 1 + 4\ZZ_2$.
	Moreover, for $p \geqslant 3$ we have $\Gamma_{\textrm{tor}} \isomorphic \FF_p^{\times}$, and for $p = 2$ we have $\Gamma_{\textrm{tor}} \isomorphic \{\pm 1\}$ as groups.
	Now, the first step in the proof of Theorem \ref{intro_thm:strat_wach_comp_relative} corresponds to the descent and Galois cohomology computations for the action of $\ZZ_p(1)^d$ (see Section \ref{subsubsec:geo_gamma_action}, \ref{subsec:geometric_descent} and \ref{subsubsec:geo_gamma_action_n1}).
	For $p \geqslant 3$, the second step corresponds to computations for the action of $\FF_p^{\times}$ (see Sections \ref{subsubsec:fpx_action}, \ref{subsubsec:fpx_descent} and \ref{subsubsec:fpx_action_n1}), and the third step corresponds to the action of $1 + p\ZZ_p$ (see Sections \ref{subsubsec:gamma0_action}, \ref{subsubsec:gamma0_descent} and \ref{subsubsec:gamma0_action_n1}).
	Additionally, for $p = 2$, due to the difference in the structure of $\ZZ_p^{\times}$, the computations in steps 2 and 3 mentioned above are completely different (see Sections \ref{subsubsec:p=2_gammaF_action}, \ref{subsubsec:a1_galact_p=2} and \ref{subsubsec:p=2_gammaF_action_n1}).

	Note that it is \textit{absolutely necessary} to carry out the computations in the order mentioned above because the more natural ordering, i.e.\ carrying out computations for the $\ZZ_p(1)^d\textrm{-action}$, followed by the $\Gamma_0\action$ and the $\Gamma_{\textrm{tor}}\action$ at the end, does not give the desired results.
	Furthermore, let us remark that while the proofs involving $\ZZ_p(1)^d$ and $\Gamma_0$ are ``prismatic/crystalline'' in nature and feature complexes similar to $q\textrm{-de Rham}$ complexes, the arguments involving $\Gamma_{\textrm{tor}}$ feature some techniques commonly used in Iwasawa theory (see Appendix \ref{app_subsec:module_zpx_action}).
\end{rem}

\subsection{Some descent results for Wach modules}

The most important input for the construction of the functor in \eqref{intro_eq:strat_from_wach_relative} is the following comparison result (notation are explained after the statement):

\begin{thm}[{Theorem \ref{thm:integral_comp_relative}}]\label{intro_thm:integral_comp_relative}
	Let $N$ be a Wach module over $\AR$, and let us consider the $R\module$ $M \coloneq (\AR(1)/p_1(\mu) \otimes_{p_2, \AR} N)^{\Gamma_R}$, equipped with the tensor product Frobenius.
	Then, $M$ is a finitely generated and $p\torsionfree$ $R\module$, and there exists a natural $(\varphi, \Gamma_R)\equivariant$ isomorphism
	\begin{equation*}
		\AR(1)/(p_1(\mu)) \otimes_{p_1, R} M \isomorphic \AR(1)/(p_1(\mu)) \otimes_{p_2, \AR} N.
	\end{equation*}
	Moreover, the multiplication map $\Delta \colon \AR(1) \rightarrow \AR$ induces a natural $\varphi\equivariant$ isomorphism of $R\modules$ $M \isomorphic N/\mu N$.
\end{thm}
In Theorem \ref{intro_thm:integral_comp_relative}, note that $(\AR(1), [p]_q)$ denotes the self product of $(\AR, [p]_q)$ in $X_{\prism}$, and $p_1, p_2 \colon \AR \rightarrow \AR(1)$ are the two projection maps to the two components (see Construction \ref{const:cech_nerve}).

Note that Theorem \ref{intro_thm:integral_comp_relative} is a descent statement for the action of $\Gamma_R$ and involves the ring $\AR(1)/p_1(\mu)$.
Therefore, in order to prove the statement, one requires an understanding of the action of $\Gamma_R$ on the underlying objects, and an explicit description of the ring $\AR(1)/p_1(\mu)$.
The latter is achieved in Proposition \ref{prop:a1modp1mu_pdring}, using some computations of Tsuji on prismatic envelopes and divided power rings (see Appendix \ref{app_sec:delta_rings_pd_alg}).
For the former, as mentioned in Remark \ref{rem:3stepdescent_explain}, our proof follows a \textit{3-step descent} argument corresponding to the 3 different subgroups making up $\Gamma_R$, i.e.\ we first show descent for the $\Gamma_R'\action$, followed by descent for the $\Gamma_{\textrm{tor}}\action$ and finally descent for the $\Gamma_0\action$.
Furthermore, note that the case $p=2$ is different from the cases $p \geqslant 3$, and requires entirely different arguments.
We refer the reader to Sections \ref{subsec:geometric_descent} and \ref{subsec:arithmetic_descent} for more details.

\begin{rem}\label{rem:descent_to_strat_explain}
	Theorem \ref{intro_thm:integral_comp_relative} plays a key role in obtaining a stratification on the Wach module $N$ (see Proposition \ref{prop:diagonal_m1_relative}).
	Let us note that Theorem \ref{intro_thm:integral_comp_relative} and Proposition \ref{prop:diagonal_m1_relative} are new, and do not follow from any previous works.
	More specifically, our statements are analogous to some statements in \cite[Sections 3.1 \& 3.2]{morrow-tsuji}, however, the results of loc.\ cit.\ do not adapt to our setting.
	This is due to the fact that in contrast with loc.\ cit.\ the group $\Gamma_R$ has an \textit{arithmetic} part, i.e.\ $\Gamma_F$, due to which entirely different computations are required for the descent isomorphism (compare loc.\ cit\ with Sections \ref{subsec:arithmetic_descent} and \ref{subsec:wach_and_strat}).
	Moreover, the computations of loc.\ cit.\ were carried out using the $q\textrm{-PD}$ formalism of \cite[Section 16]{bhatt-scholze-prisms} which made some of the proofs straightforward, however, the same does not apply to our current setting, and we resolve this via the \textit{3-step proofs} mentioned above.
\end{rem}

Now, let us assume that $p \geqslant 3$, and note that in the course of the proof of Theorem \ref{intro_thm:integral_comp_relative} (namely, the descent step for the action of $\Gamma_{\textrm{tor}} = \FF_p^{\times}$), we obtain a \textit{new} descent statement for (classical) Wach modules (see Remark \ref{intro_rem:wachmod_fpx_descent}).
In other words, Theorem \ref{intro_thm:wachmod_fpx_descent} below is an important input in the proof of Theorem \ref{intro_thm:integral_comp_relative}.
To state the result, note that we have $\AR \isomorphic R\llbracket\mu\rrbracket$ as rings (see Section \ref{subsubsec:ring_ar+}), and by transport of structure, we equip the target with a Frobenius endomorphism and an $R\linear$ action of $\Gamma_F$.
Let $N$ be a Wach module over $\AR$, i.e.\ a finitely generated module over $R\llbracket\mu\rrbracket$, equipped with an $R\llbracket\mu\rrbracket\linear$ Frobenius isomorphism $\varphi^*(N)[1/[p]_q] \isomorphic N[1/[p]_q]$, and an $R\linear$ and continuous action of $\Gamma_F$ commuting with Frobenius, and such that the action of $\Gamma_F$ is trivial on $N/\mu N$.
Let us set
\begin{equation*}
	\mu_0 \coloneq -p + \textstyle \sum_{a \in \FF_p} (1+\mu)^{[a]} \hspace{3mm} \textrm{ and } \hspace{3mm} \ptilde \coloneq \mu_0+p,
\end{equation*}
as elements of $R\llbracket\mu\rrbracket^{\FF_p^{\times}}$.
Then, from Lemma \ref{lem:rmu0_fpxinv}, we have a $(\varphi, \Gamma_0)\equivariant$ isomorphism of rings $R\llbracket\mu_0\rrbracket \isomorphic R\llbracket\mu\rrbracket^{\FF_p^{\times}}$, and we show the following:
\begin{thm}[{Proposition \ref{prop:wachmod_fpx_descent}}]\label{intro_thm:wachmod_fpx_descent}
	Let $N$ be a Wach module over $\AR$.
	Then, $N_0 \coloneq N^{\FF_p^{\times}}$ is a finitely generated $R\llbracket\mu_0\rrbracket\module$, equipped with a continuous and semilinear action of $\Gamma_0$, such that the action of $\Gamma_0$ is trivial on $N_0/\mu_0 N_0 \isomorphic N/\mu N$, and we have a natural $(\varphi, \Gamma_F)\equivariant$ isomorphism of $R\llbracket\mu\rrbracket\modules$ $R\llbracket\mu\rrbracket \otimes_{R\llbracket\mu_0\rrbracket} N_0 \isomorphic N$.
	Moreover, the sequences $\{p, \mu_0\}$ and $\{\mu_0, p\}$ are regular on $N_0$, and it is equipped with an $R\llbracket\mu_0\rrbracket\linear$ isomorphism $\varphi^*(N_0)[1/\ptilde] \isomorphic N_0[1/\ptilde]$, compatible with the respective natural actions of $\Gamma_0$.
\end{thm}

\begin{rem}\label{intro_rem:wachmod_fpx_descent}
	In Theorem \ref{intro_thm:wachmod_fpx_descent}, setting $R = O_F$ shows that classical Wach modules over $\AF = O_F\llbracket\mu\rrbracket$ canonically descend to Wach modules over $O_F\llbracket\mu_0\rrbracket \isomorphic \AF^{\FF_p^{\times}}$.
	Using different arguments, the claim for $R = O_F$ has also been proven in \cite[Theorem 1.6]{abhinandan-crystalline-galcoh}.
\end{rem}

\subsection{Relation to previous works}\label{subsec:relation_previous_works}

The theory of prismatic $F\crystals$ was introduced in \cite{bhatt-scholze-crystals}, and its generalisation to analytic/completed prismatic $F\crystals$ was studied in \cite{du-liu-moon-shimizu, guo-reinecke}.
The theory of classical Wach modules was introduced and developed in \cite{fontaine-phigamma, wach-free, colmez-hauteur, berger-limites}, and its generalisation to the relative case was introduced and developed in \cite{abhinandan-relative-wach-i, abhinandan-imperfect-wach, abhinandan-relative-wach-ii}.
Now, let us consider the following diagram (notation explained immediately afterwards):
\begin{equation}\label{eq:compatibility_diagram}
	\begin{tikzcd}
		\Vect^{\an, \varphi}(X_{\prism}) \arrow[rrr, hookrightarrow, "\paze \mapsto \paze{[}1/\pazi_{\prism}{]} \otimes_{\pazo{[}1/\pazi_{\prism}{]}} \pazo_{\prism}{[}1/\pazi_{\prism}{]}_p^{\wedge}"] \arrow[rd, "T_{\etale}^{\an}", "\sim"'] \arrow[dd, "\eval_{\AR}^{\prism} \eqref{intro_eq:afcrystal_eval_ar+}"', "\wr"] & & & \Vect^{\varphi}(X_{\prism}, \pazo_{\prism}[1/\pazi_{\prism}]_p^{\wedge}) \arrow[dd, "\eval_{\AR}^{\prism}", "\wr"'] \arrow[ld, "T_{\etale}"', "\sim"]\\
		& \Rep_{\ZZ_p}^{\crys}(G_R) \arrow[r, hookrightarrow] \arrow[ld, "\sim"', "\mbfn_R"] & \Rep_{\ZZ_p}(G_R) \arrow[rd, "\sim", "\mbfd_R"']\\
		(\varphi, \Gamma_R)\Mod_{\AR}^{[p]_q} \arrow[rrr, hookrightarrow, "N \mapsto N \otimes_{\AR} \BR"] & & & (\varphi, \Gamma_R)\Mod_{\BR}^{\etale}.
	\end{tikzcd}
\end{equation}
In the top row of the diagram \eqref{eq:compatibility_diagram}, the top right corner denotes the category of Laurent $F\crystals$ over $X_{\prism}$ (see \cite{bhatt-scholze-crystals} or Definition \ref{eq:compatibility_diagram}); in the middle row, $G_R$ denotes the \'etale fundamental group of $R[1/p]$ (see Section \ref{subsec:setup_nota}), the right-hand term denotes the category of $\ZZ_p\textrm{-representations}$ of $G_R$, and the left-hand term denotes the category of lattices inside $\padic$ crystalline representations of $G_R$ (see \cite{brinon-relatif}); in the bottom row, we have $B_R \coloneq \AR[1/\mu]_p^{\wedge}$ equipped with an induced action of $(\varphi, \Gamma_R)$, and the bottom right corner denotes the category of \'etale $(\varphi, \Gamma_R)\modules$ over $\BR$ (see \cite{andreatta-phigamma}).

The top horizontal arrow is fully faithful by \cite{guo-reinecke}, the middle horizontal arrow is easily seen to be a natural inclusion, and the bottom horizontal arrow is fully faithful by \cite[Proposition 3.12]{abhinandan-relative-wach-ii}.
From the first to the second row, the slanted arrow $T_{\etale}$ is the natural \'etale realisation functor and an equivalence by \cite{bhatt-scholze-crystals} (see Lemma \ref{lem:laurent_fcrystal_local_system}), and the slanted arrow $T_{\etale}^{\an}$ is an equivalence by \cite{guo-reinecke} (see Definition \ref{defi:etale_functor} and Remark \ref{rem:analytic_completed_crystalline_equiv}); the upper square commutes using loc.\ cit.
From the middle row to the bottom row, the slanted arrow $\mbfd_R$ is the natural \'etale $(\varphi, \Gamma)\module$ functor and an equivalence by \cite{andreatta-phigamma}, and the slanted arrow $\mbfn_R$ is the natural Wach module functor and an equivalence by \cite[Theorem 1.5]{abhinandan-relative-wach-ii}; commutativity of the bottom square follows from the compatibility between the results of \cite{abhinandan-relative-wach-ii} and \cite{andreatta-phigamma}.
Next, in the diagram \eqref{eq:compatibility_diagram}, the leftmost and the rightmost vertical arrows are evaluation functors, i.e.\ evaluation of an analytic (resp.\ Laurent) prismatic $F\crystal$ over the prism $(\AR, [p]_q)$; the left vertical arrow is an equivalence by Theorem \ref{intro_thm:fcrystals_wachmod}.
Note that the left-hand-side triangle commutes by comparing the explicit formulas for the slanted arrow $T_{\etale}^{\an}$ in Definition \ref{defi:etale_functor}, and the composition of the left vertical arrow $\eval_{\AR}^{\prism}$ in \eqref{intro_eq:afcrystal_eval_ar+} with the quasi-inverse of the slanted arrow $\mbfn_R$ described in \cite[Theorem 1.5]{abhinandan-relative-wach-ii}, and similarly for the right-hand-side triangle.
In particular, we obtain that the right vertical arrow is an equivalence (also see \cite{wu, min-wang}), and the outer square commutes by the definition of the arrows.
Hence, it follows that the results of this article are compatible with previous constructions.

As mentioned earlier, in the proof of Theorem \ref{intro_thm:fcrystals_wachmod} we do not assume the equivalence between crystalline $\ZZ_p\textrm{-representations}$ and Wach modules over $\AR$ from \cite{abhinandan-relative-wach-ii}, or the categorical equivalence results from \cite{bhatt-scholze-crystals, du-liu, du-liu-moon-shimizu, guo-reinecke} and \cite{wu, min-wang}.
Furthermore, our approach is inspired by the work of \cite{morrow-tsuji}, however, as mentioned earlier, the results obtained and the techniques employed in our proofs are different.

\subsection{Setup and notation}\label{subsec:setup_nota}

In this section, we will describe the setup for this paper and fix some notation which are essentially the same as in \cite[Section 1.4]{abhinandan-relative-wach-ii}.
We shall work under the convention that $0 \in \NN$, the set of natural numbers.

Let $p$ be a fixed prime number and $\kappa$ a perfect field of characteristic $p$, and set $O_F \coloneq W(\kappa)$ as the ring of $p\textrm{-typical}$ Witt vectors with coefficients in $\kappa$, and $F \coloneq O_F[1/p]$ the fraction field of $O_F$.
In particular, $F$ is an unramified extension of $\QQ_p$ with ring of integers $O_F$.
Let $\overline{F}$ denote a fixed algebraic closure of $F$ so that its residue field, denoted as $\overline{\kappa}$, is an algebraic closure of $\kappa$.
Set $G_F \coloneq \Gal(\overline{F}/F)$ as the absolute Galois group of $F$.

Next, set $\Finfty \coloneq F(\mu_{p^{\infty}})$, where $\mu_{p^{\infty}}$ denotes the set of all $p\textrm{-power}$ roots of unity in $\overline{F}$.
Then, we have Galois groups $H_F \coloneq \Gal(\overline{F}/F_{\infty})$ and $\Gamma_F \coloneq \Gal(\Finfty/F) \isomorphic \ZZ_p^{\times}$.
Note that the isomorphism $\chi \colon \Gamma_F \isomorphic \ZZ_p^{\times}$ is given via the $\padic$ cyclotomic character, and therefore, $\Gamma_F$ fits into the following exact sequence:
\begin{equation}\label{eq:gammaf_es}
	1 \longrightarrow \Gamma_0 \longrightarrow \Gamma_F \longrightarrow \Gamma_{\textrm{tor}} \longrightarrow 1,
\end{equation}
where, for $p \geqslant 3$, we have $\Gamma_0 \isomorphic 1 + p\ZZ_p$, and for $p=2$, we have $\Gamma_0 \isomorphic 1 + 4\ZZ_2$.
Moreover, for $p \geqslant 3$, we have $\Gamma_{\textrm{tor}} \isomorphic \FF_p^{\times}$, and the projection map in \eqref{eq:gammaf_es} admits a section $\Gamma_{\textrm{tor}} \isomorphic \FF_p^{\times} \rightarrow \ZZ_p^{\times} \lisomorphic \Gamma_F$, where the second map is given as $a \mapsto [a]$, the Teichm\"uller lift of $a$.
Furthermore, note that for $p = 2$, we have $\Gamma_{\textrm{tor}} \isomorphic \{\pm 1\}$ as groups.

\begin{nota}
	Let $s \geqslant 0$ be an integer and $\smbfk = (k_1, \ldots, k_s)$ in $\NN^s$ denote a multi-index, and we shall write $\smbfk \rightarrow +\infty$ to denote $\sum k_i \rightarrow +\infty$.
	For a topological algebra $S$ and indeterminate $Z_1, \ldots, Z_s$, define
	\begin{equation*}
		S\langle Z_1, \ldots, Z_s\rangle \coloneq \big\{\textstyle\sum_{\smbfk \in \NN^s} a_{\smbfk} Z_1^{k_1} \cdots Z_s^{k_s}, \hspace{1mm} \textrm{where} \hspace{1mm} a_{\smbfk} \in S \hspace{1mm} \textrm{and} \hspace{1mm} a_{\smbfk} \rightarrow 0 \hspace{1mm} \textrm{as} \hspace{1mm} \smbfk \rightarrow +\infty\big\}.
	\end{equation*}
\end{nota}

Fix an integer $d \geqslant 0$, and let $X_1, X_2, \ldots, X_d$ be some indeterminate.
Let $R$ denote the $\padic$ completion of an \'etale algebra over $\Rframe = O_F\langle X_1^{\pm 1}, \ldots, X_d^{\pm 1}\rangle$ with non-empty and connected special fibre; the map $\Rframe \rightarrow R$ is often referred to as a framing on $R$.

Let us fix an algebraic closure $\overline{\Fr(R)}$ of $\Fr(R)$ containing $\overline{F}$.
Let $\overline{R}$ denote the union of finite $R\textrm{-subalgebras}$ $R' \subset \overline{\Fr(R)}$, such that $R'[1/p]$ is \'etale over $R[1/p]$.
Let $\overline{\eta}$ denote the fixed geometric point of the generic fibre $\Spec R[1/p]$ (defined by $\overline{\Fr(R)}$), and let $G_R$ denote the \'etale fundamental group $\pi_1^{\etale}\big(\Spec R[1/p], \overline{\eta}\big)$.
We may write this étale fundamental group as a Galois group, i.e.\
\begin{equation*}
	G_R = \pi_1^{\etale}(\Spec(R[1/p]), \overline{\eta}) = \Gal(\overline{R}[1/p] / R[1/p]).
\end{equation*}

Next, let us recall some constructions that depend on the framing $\Rframe \rightarrow R$.
Set $\Rinfty \coloneq \cup_{i=1}^d R[\mu_{p^{\infty}}, X_i^{1/p^{\infty}}] \subset \Rbar$, and note that we have the following Galois groups (see \cite[Sections 2 \& 3]{abhinandan-relative-wach-i} and \cite[Section 2]{abhinandan-imperfect-wach}):
\begin{align*}
	& H_R \coloneq \Gal(\Rbar[1/p]/\Rinfty[1/p]),\\
	& \Gamma_R \coloneq G_R/H_R = \Gal(\Rinfty[1/p]/R[1/p]) \isomorphic \ZZ_p(1)^d \rtimes \ZZ_p^{\times},\\
	& \Gamma'_R \coloneq \Gal(R_{\infty}[1/p]/R(\mu_{p^{\infty}})[1/p]) \isomorphic \ZZ_p(1)^d,\\
	& \Gal\big(R(\mu_{p^{\infty}})[1/p]/R[1/p]) = \Gamma_R/\Gamma'_R \isomorphic \Gamma_F.
\end{align*}

Let $\varphi \colon \Rframe \rightarrow \Rframe$ denote a morphism extending the natural Witt vector Frobenius on $O_F$ by setting $\varphi(X_i) = X_i^p$, for all $1 \leqslant i \leqslant d$.
The endomorphism $\varphi$ of $\Rframe$ is flat by \cite[Lemma 7.1.5]{brinon-relatif}, and faithfully flat since $\varphi(\frakm) \subset \frakm$ for any maximal ideal $\frakm \subset \Rframe$.
Moreover, it is finite of degree $p^d$ using Nakayama Lemma and the fact that $\varphi$ modulo $p$ is evidently of degree $p^d$.
As the $O_F\algebra$ $R$ is given as the $\padic$ completion of an \'etale algebra $\Rframe$, therefore, the Frobenius endomorphism $\varphi$ on $\Rframe$ admits a unique extension $\varphi \colon R \rightarrow R$ such that the induced map $\varphi \colon R/p \rightarrow R/p$ is the absolute Frobenius $x \mapsto x^p$ (see \cite[Proposition 2.1]{colmez-niziol}).
Similar to above, we again note that the endomorphism $\varphi \colon R \rightarrow R$ is finite and faithfully flat of degree $p^d$.
Next, for $k \in \NN$, let $\Omega^k_R$ denote the $\padic$ completion of module of $k\textrm{-differentials}$ of $R$ relative to $\ZZ$.
Then, we have that $\Omega^1_R = \oplus_{i=1}^d R \dlog X_i$ and $\Omega^k_R = \wedge_R^k \Omega^1_R$.

Now, let $A  = \Rinfty$ or $\Rbar$.
Then, recall that the tilt of $A$ is given as $A^{\flat} = \lim_{\varphi} A/p$ (see \cite[Chapitre V, Section 1.4]{fontaine-pdivisibles} and \cite[Section 3]{bhatt-morrow-scholze-1}).
Finally, let $B$ be a $\ZZ_p\algebra$ equipped with a Frobenius endomorphism $\varphi$ lifting the absolute Frobenius on $B/p$, then for any $B\module$ $M$, we set $\varphi^*(M) \coloneq B \otimes_{\varphi, B} M$.

\subsection{Outline of the paper}

In Section \ref{sec:prisms} we recall some basic definitions and results from the prismatic theory.
In particular, in Section \ref{subsec:prismatic_site_crystals}, we recall the definition of prismatic site and prismatic crystals from \cite{bhatt-scholze-prisms, bhatt-scholze-crystals}, and in Section \ref{subsec:prismatic_fcrystals} we recall the notion of prismatic $F\crystals$ and its variations from \cite{bhatt-scholze-crystals, du-liu-moon-shimizu, guo-reinecke}, as well as, describe the \'etale realisation functors.
Section \ref{sec:theprism_ar+} is devoted to the study of the prism $(\AR, [p]_q)$ in detail.
In Section \ref{subsec:important_rings}, we describe the ring $\AR$, and study the structural properties of some of its subrings.
In Section \ref{subsec:theprism_ar+}, we show that $(\AR, [p]_q)$ is an object of $(\Spf R)_{\prism}$, explicitly compute the first few terms of its prismatic \v{C}ech nerve $(\AR(\bullet), I(\bullet))$ in $(\Spf R)_{\prism}$, and then study the action of $\Gamma_R^{\times (n+1)}$ on $\AR(n)$.
Then, in Section \ref{subsec:ar1p1mu}, we provide an explicit description of the ring $\AR(1)/p_1(\mu)$, where $p_1 \colon A \rightarrow A(1)$ is the first projection map, and study the action of $1 \times \Gamma_R \subset \Gamma_R^2$ on it.
Finally, in Section \ref{subsec:gamma_act_ar1}, we study the action of $\Gamma_R^2$ on $\AR(1)$, and carry out the \textit{3-step proof} for the ring $\AR$ -- described after Theorem \ref{intro_thm:strat_wach_comp_relative} -- in order to study the structure of several rings that will appear at various stages of the proof of Theorem \ref{intro_thm:strat_wach_comp_relative} in Section \ref{sec:prismatic_wach}.
The goal of Section \ref{sec:integral_comp} is to state and prove Theorem \ref{intro_thm:integral_comp_relative}.
We begin by recalling the definition of Wach modules and some of its properties from \cite{abhinandan-relative-wach-ii}, and in Section \ref{intro_thm:integral_comp_relative}, we describe the $\qconnection$ on a Wach module and its scalar extensions originating from the natural action of $\Gamma_R$.
Then, in Section \ref{subsec:geometric_descent}, we carry out the first step of the proof of Theorem \ref{intro_thm:integral_comp_relative}, i.e.\ the descent for the action of $\Gamma_R'$, and in Section \ref{subsec:arithmetic_descent}, we show the second and third steps, i.e.\ the descent for the action of $\Gamma_F \isomorphic \Gamma_0 \rtimes \Gamma_{\textrm{tor}}$.
Note that the Section \ref{subsec:arithmetic_descent} is divided into two parts: the first part deals with the case $p \geqslant 3$ where we also prove Theorem \ref{intro_thm:wachmod_fpx_descent}, and the second part deals with $p=2$ which requires different arguments.
Finally, in Section \ref{subsec:proof_integral_comp_relative}, we put everything together to prove Theorem \ref{intro_thm:integral_comp_relative}.
In Section \ref{sec:prismatic_wach}, we state and prove Theorem \ref{intro_thm:fcrystals_wachmod}.
We begin by describing the functor $\eval^{\prism}_{\AR}$, and then in Section \ref{subsec:modules_with_strat} we describe the relation between analytic/completed prismatic $F\crystals$ and modules with stratification, introduce the functor $\eval^{\Strat}_{\AR}$, and show that it is compatible with $\eval_{\prism}^{\AR}$.
Section \ref{subsec:wach_and_strat} is dedicated to the construction of stratification on Wach modules using the action of $\Gamma_R$ via a \textit{3-step argument}.
This completes the proof of Theorem \ref{intro_thm:fcrystals_wachmod}.

In Appendix \ref{app_subsec:basic_defi}, we recall some standard definitions that have been used throughout in the text.
Then, in Appendix \ref{app_subsec:module_zpx_action}, we describe the structure of modules admitting a continuous action of $\ZZ_p^{\times}$ by recalling some standard constructions of Iwasawa \cite{iwasawa}.
Appendix \ref{app_sec:delta_rings_pd_alg} has been adapted from some notes of Tsuji, and in that section we study the structure of certain $\delta\textrm{-rings}$ and their reduction modulo $\mu$, which is a crucial input for determining the structure of $\AR(1)/p_1(\mu)$ is Section \ref{subsec:ar1p1mu}.

\vspace{3mm}

\noindent \textbf{Acknowledgements.} 
I would like to sincerely thank Takeshi Tsuji for various discussions, and for generously sharing his ideas during the course of this project and some of his computations which have been included in Appendix \ref{app_sec:delta_rings_pd_alg}.
I would also like to thank Michel Gros for some helpful remarks.
Finally, many thanks to the referee for their careful reading of the previous version of the article, pointing out errors and providing many helpful remarks and suggestions for improvement.
This research was supported by JSPS KAKENHI Grant numbers 22F22711 and 22KF0094.

\section{Prismatic site and Prismatic \texorpdfstring{$F\textrm{-crystals}$}{-}}\label{sec:prisms}

In this section, we will recall some fundamental definitions and results on prismatic site and prismatic $F\crystals$ from \cite{bhatt-scholze-prisms, bhatt-scholze-crystals}, and analytic/completed prismatic $F\textrm{-crystals}$ from \cite{du-liu-moon-shimizu} and \cite{guo-reinecke}.
For some standard definitions used in this section, we refer the reader to Appendix \ref{app_subsec:basic_defi}.
We start with the following:
\begin{defi}\label{defi:delta_ring}
	A $\delta\ring$ is a pair $(A, \delta)$ where $A$ is a commutative ring, and $\delta \colon A \rightarrow A$ is a map of sets with $\delta(0) = \delta(1) = 0$ and satisfying:
	\begin{equation}\label{eq:delta_sum_prod}
		\begin{aligned}
			\delta(xy) &= x^p \delta(y) + y^p \delta(x) + p\delta(x)\delta(y), \\
			\delta(x+y) &= \delta(x) + \delta(y) + \frac{x^p + y^p - (x+y)^p}{p}.
		\end{aligned}
	\end{equation}
\end{defi}

\begin{rem}\label{rem:frob_delta_ring}
	Given a $\delta\ring$ $(A, \delta)$ define $\varphi \colon A \rightarrow A$ by the formula $\varphi(x) = x^p + p\delta(x)$, for any $x$ in $A$.
	This determines a lift of the absolute Frobenius on $A/pA$.
	Conversely, if $A$ is $p\textrm{-torsion-free}$ then any lift $\varphi \colon A \rightarrow A$ of the absolute Frobenius on $A/pA$ determines a unique $\delta\textrm{-structure}$ on $A$.
\end{rem}

\begin{defi}
	An element $d$ of a $\delta\ring$ $A$ is called distinguished if $\delta(d)$ is a unit.
\end{defi}

\subsection{Prismatic site and crystals}\label{subsec:prismatic_site_crystals}

Let $A$ be a $\delta\ring$, $I \subset A$ an ideal; we shall refer to $(A, I)$ as a $\delta\textrm{-pair}$.

\begin{defi}[Prism, {\cite[Definition 3.2]{bhatt-scholze-prisms}}]\label{defi:prism}
	A $\delta\textrm{-pair}$ $(A, I)$ is called a \textit{prism} if $I \subset A$ defines a Cartier divisor on $\Spec(A)$ such that $A$ is derived $(p, I)\complete$ and $p$ is in $I + \varphi(I) A$.
	The category of all prisms is the corresponding full subcategory of all $\delta\textrm{-pairs}$.
	The prism $(A, I)$ is called perfect if $A$ is a perfect $\delta\ring$, i.e.\ $\varphi \colon A \rightarrow A$ is bijective.
	Finally, $(A, I)$ is bounded if $A/I$ has bounded $p^{\infty}\torsion$.
	A map $(A, I) \rightarrow (B, J)$ of prisms is (faithfully) flat if the map $A \rightarrow B$ is $(p, I)\textrm{-completely}$ (faithfully) flat.
\end{defi}

\begin{rem}
	If $(A, I)$ is a bounded prism then $A$ is classically $(p, I)\complete$ (see \cite[Lemma 3.7]{bhatt-scholze-prisms}).
	Any morphism of prisms $(A, I) \rightarrow (B, J)$ induces an isomorphism $I \otimes_A B \isomorphic J$, in particular, we have that $IB = J$ (see \cite[Lemma 3.5]{bhatt-scholze-prisms}).
\end{rem}

\begin{defi}[Absolute prismatic site, {\cite[Corollary 3.12]{bhatt-scholze-prisms}}]
	Let $\Spf(\ZZ_p)_{\prism}$ denote the category opposite to that of the category of all bounded prisms $(A, I)$, and endow it with a topology for which covers are determined by faithfully flat maps of prisms.
	Then, $\Spf(\ZZ_p)_{\prism}$ forms a site.
	Moreover, the functor $\pazo_{\prism} \colon \Spf(\ZZ_p)_{\prism} \rightarrow \Rings$ (resp.\ $\pazi_{\prism} \colon \Spf(\ZZ_p)_{\prism} \rightarrow \textrm{Abelian Groups}$) defined via $(A, I) \mapsto A$ (resp.\ $(A, I) \mapsto I$) forms a sheaf for this topology with vanishing higher \v{C}ech cohomology.
\end{defi}

\begin{defi}[Absolute prismatic site of $X$, {\cite[Definition 2.3]{bhatt-scholze-crystals}}]\label{defi:abs_prismatic_site}
	Let $X$ be a $\padic$ formal scheme.
	Define the \textit{absolute prismatic site} of $X$, denoted as $X_{\prism}$, to be the category opposite to that of bounded prisms $(A, I)$ which are equipped with a map $\Spf(A/I) \rightarrow X$, and endow $X_{\prism}$ with the topology induced by the flat topology on prisms.
	We shall write $\pazo_{\prism}$ for the structure sheaf, and denote the ideal sheaf of the Hodge-Tate divisor by $\pazi_{\prism} \subset \pazo_{\prism}$.
	Denote by $\Shv(X_{\prism})$ the $\infty\category$ of sheaves on $X_{\prism}$.
\end{defi}

\begin{prop}[{\cite[Proposition 2.7]{bhatt-scholze-crystals}}]\label{prop:prism_crystal_vb}
	Let $X$ be a $\padic$ formal scheme, and let $\Vect(X_{\prism}, \pazo_{\prism})$ denote the category of vector bundles of $\pazo_{\prism}\modules$.
	Then, there is a natural equivalence
	\begin{equation*}
		\Vect(X_{\prism}, \pazo_{\prism}) \isomorphic \lim_{(A, I) \in X_{\prism}} \Vect(A).
	\end{equation*}
	Moreover, a similar statement holds after replacing $\pazo_{\prism}$ with $\pazo_{\prism}[1/\pazi_{\prism}]_p^{\wedge}$.
\end{prop}

Let $X$ be a bounded $\padic$ formal scheme, and let $X_{\eta}$ denote the generic fibre of $X$, regarded as a presheaf on perfectoid spaces over $\QQ_p$.

\begin{defi}[Laurent $F\crystals$, {\cite[Definition 3.2]{bhatt-scholze-crystals}}]\label{defi:laurent_fcrystal}
	Define the category of Laurent $F\crystals$ of vector bundles on $(X_{\prism}, \pazo_{\prism}[1/\pazi_{\prism}]_p^{\wedge})$ as
	\begin{equation*}
		\Vect(X_{\prism}, \pazo_{\prism}[1/\pazi_{\prism}]_p^{\wedge})^{\varphi=1} \coloneq \lim_{(A, I) \in X_{\prism}} \Vect(A[1/I]_p^{\wedge})^{\varphi=1},
	\end{equation*}
	i.e.\ any $\paze$ in $\Vect(X_{\prism}, \pazo_{\prism}[1/\pazi_{\prism}]_p^{\wedge})$ is equipped with an isomorphism $\varphi_{\paze} \colon \varphi^* \paze \isomorphic \paze$.
\end{defi}

\begin{lem}[{\cite[Corollary 3.8]{bhatt-scholze-crystals}}]\label{lem:laurent_fcrystal_local_system}
	Let $\Loc(X_{\eta}, \ZZ_p)$ denote the category of \'etale $\ZZ_p\textrm{-local}$ systems on $X_{\eta}$.
	Then, there is a natural equivalence
	\begin{equation*}
		T_{\etale} \colon \Vect(X_{\prism}, \pazo_{\prism}[1/\pazi_{\prism}]_p^{\wedge})^{\varphi=1} \isomorphic \Loc_{\ZZ_p}(X_{\eta}).
	\end{equation*}
	Passing to the associated isogeny categories, we obtain that
	\begin{equation*}
		\Vect(X_{\prism}, \pazo_{\prism}[1/\pazi_{\prism}]_p^{\wedge})^{\varphi=1} \otimes_{\ZZ_p} \QQ_p \isomorphic \Loc_{\ZZ_p}(X_{\eta}) \otimes_{\ZZ_p} \QQ_p.
	\end{equation*}
\end{lem}

\subsection{Prismatic \texorpdfstring{$F\crystals$}{-}}\label{subsec:prismatic_fcrystals}

In this section, we will recall the definition of prismatic $F\crystals$ and its analytic/completed variants. 

\subsubsection{Vector bundles}

Let $X$ be a $\padic$ formal scheme, and let $(A, I)$ denote a prism with $\varphi$ its Frobenius endomorphism.

\begin{defi}[Prismatic $F\crystals$, {\cite[Definition 4.1]{bhatt-scholze-crystals}}]\label{defi:prism_fcrystal}
	Define the category $\Vect^{\varphi}(A)$ of prismatic $F\crystals$ of vector bundles on $A$ as follows: 
	an object is a pair $(M, \varphi_M)$ with $M$ a finite projective $A\module$ equipped with an $A\linear$ isomorphism $\varphi_M \colon (\varphi^* M)[1/I] \isomorphic M[1/I]$.
	Morphisms between two objects are given as $A\linear$ maps compatible with $\varphi_M$.
	Say that $(M, \varphi_M)$ is effective if $\varphi_M$ carries $\varphi^* M$ into $M$.

	Define the category $\Vect^{\varphi}(X_{\prism}, \pazo_{\prism})$ of prismatic $F\crystals$ of vector bundles on $X_{\prism}$ as follows: an object is a pair $(\paze, \varphi_{\paze})$ with $\paze$ a vector bundle on $(X_{\prism}, \pazo_{\prism})$ equipped with an isomorphism $\varphi_{\paze} \colon (\varphi^* \paze)[1/\pazi_{\prism}] \isomorphic \paze[1/\pazi_{\prism}]$.
	Morphisms between two objects are maps of vector bundles compatible with $\varphi_{\paze}$.
	Say that $(\paze, \varphi_{\paze})$ is effective if $\varphi_{\paze}$ carries $\varphi^* \paze$ into $\paze$.
\end{defi}

\begin{rem}
	From Proposition \ref{prop:prism_crystal_vb} we have an equivalence $\Vect^{\varphi}(X_{\prism}, \pazo_{\prism}) \isomorphic \lim_{(A, I) \in X_{\prism}} \Vect^{\varphi}(A)$.
\end{rem}

\subsubsection{Completed \texorpdfstring{$F\crystals$}{-}}

Let $O_K$ be a finite ramified extension of $O_F$, and take $\pi$ in $O_K$ to be a uniformiser and $E \coloneq E(u)$ in $O_F[u]$ to be its minimal polynomial.

\begin{defi}[Completed prismatic crystals, {\cite[Definition 3.11]{du-liu-moon-shimizu}}]\label{defi:completed_crystal}
	Let $X$ be a smooth $\padic$ formal scheme over $O_K$.
	Define the category $\CRhat(X_{\prism})$ of \textit{finitely generated completed crystals of $\pazo_{\prism}\modules$} on $X_{\prism}$ as follows:
	an object is a sheaf $\paze$ of $\pazo_{\prism}\modules$ on $X_{\prism}$ satisfying the following:
	\begin{enumarabicup}
	\item For each $(A, I)$ in $X_{\prism}$, the evaluation of $\paze$ on $(A, I)$, i.e.\ $\paze(A) \coloneq \paze(A, I)$, is a finitely generated and $(p, I)\complete$ $A\module$.

	\item For every morphism $(A, I) \rightarrow (B, IB)$ in $X_{\prism}$, the natural map $B \widehat{\otimes}_A \paze(A) \rightarrow \paze(B)$ is an isomorphism, where we set $B \widehat{\otimes}_A \paze(A) \coloneq \lim_n (B \otimes_A \paze(A))/(p, I)^n$.
	\end{enumarabicup}

	A morphism $f \colon \paze \rightarrow \pazf$ between two completed crystals is a morphism of $\pazo_{\prism}\modules$ such that for any $(A, I)$ in $X_{\prism}$, the evaluation $f_A \colon \paze(A) \rightarrow \pazf(A)$ is a morphism of finitely generated $(p, I)\complete$ $A\textrm{-modules}$, and for any $(A, I) \rightarrow (B, IB)$ in $X_{\prism}$ the following diagram naturally commutes:
	\begin{center}
		\begin{tikzcd}
			B \widehat{\otimes}_A \paze(A) \arrow[r, "1 \otimes f_A"] \arrow[d, "\wr"] & B \widehat{\otimes}_A \pazf(A) \arrow[d, "\wr"]\\
			\paze(B) \arrow[r, "f_B"] & \pazf(B).
		\end{tikzcd}
	\end{center}
\end{defi}

\begin{rem}
	The category of completed crystals on $X_{\prism}$, denoted as $\CRhat(X_{\prism})$, satisfies descent for Zariski and \'etale topologies on $X$ (see \cite[Lemma 3.39, Remark 3.40]{du-liu-moon-shimizu}).
\end{rem}

A $p\textrm{-adically}$ completed $O_K\algebra$ $S$ is called \textit{small} if it is $p\adically$ completed \'etale over a $p\adically$ complete Laurent polynomial $O_K\algebra$ $O_K\langle X_1^{\pm 1}, \ldots, X_d^{\pm 1}\rangle$, for some $d \in \NN$.
For such an $S$ there exists a unique $O_F\algebra$ $R$ such that $O_K \otimes_{O_F} R \isomorphic S$, and $R$ admits a lifting of the absolute Frobenius modulo $p$, which we take to be the one extending the Witt vector Frobenius on $O_F$ and such that $\varphi(X_i) = X_i^p$; denote the lifting of Frobenius on $R$ as $\varphi \colon R \rightarrow R$.
Let $\frakS = R\llbracket u \rrbracket$ equipped with a Frobenius endomorphism $\varphi$ which extends the Frobenius on $R$ by setting $\varphi(u) = u^p$.
Then $(\frakS, E)$ is an object of $(\Spf S)_{\prism}$ and called the \textit{Breuil-Kisin} prism (see \cite[Example 3.4]{du-liu-moon-shimizu}).

\begin{defi}[Completed prismatic $F\crystals$, {\cite[Definition 3.16, Definition 3.42]{du-liu-moon-shimizu}}]\label{defi:completed_fcrystal}
	Let $X \coloneq \Spf S$ for $S$ a small $O_K\algebra$.
	Define the category $\CRhatphi(X_{\prism})$ of \textit{completed $F\crystals$ of $\pazo_{\prism}\modules$} on $X_{\prism}$ as follows:
	an object is a pair $(\paze, \varphi_{\paze})$ with $\paze$ a completed prismatic crystal on $X_{\prism}$ equipped with an $\pazo_{\prism}\linear$ isomorphism $\varphi_{\paze} \colon \varphi^*\paze[1/\pazi_{\prism}] \isomorphic \paze[1/\pazi_{\prism}]$.
	Moreover, for the Breuil-Kisin prism $(\frakS, E)$ in $X_{\prism}$, assume that the $\frakS\module$ $\paze(\frakS)$ is torsion-free, $\paze(\frakS)[1/p]$ (resp.\ $\paze(\frakS)[1/E]$) is finite projective over $\frakS[1/p]$ (resp.\ $\frakS[1/E]$) and $\paze(\frakS) = \paze(\frakS)[1/p] \cap \paze(\frakS)[1/E] \subset \paze(\frakS)[1/p, 1/E]$.
	Morphisms between two objects are given as maps of $\pazo_{\prism}\modules$ compatible with $\varphi_{\paze}$.
	Say that $(\paze, \varphi_{\paze})$ is effective if $\varphi_{\paze}$ carries $\varphi^*\paze$ into ${\paze}$.

	Let $X$ be a smooth $\padic$ formal scheme over $O_K$.
	Define the category $\CRhatphi(X_{\prism})$ of \textit{completed $F\crystals$ of $\pazo_{\prism}\modules$} on $X_{\prism}$ as follows:
	an object is a pair $(\paze, \varphi_{\paze})$ with $\paze$ a completed prismatic crystal on $X_{\prism}$ equipped with an $\pazo_{\prism}\linear$ isomorphism $\varphi_{\paze} \colon \varphi^*\paze[1/\pazi_{\prism}] \isomorphic \paze[1/\pazi_{\prism}]$.
	Moreover, there exists an affine open covering $X = \cup_{\lambda \in \Lambda} U_{\lambda}$ by affine $\padic$ formal schemes, where $U_{\lambda} = \Spf(S_{\lambda})$ is connected and small over $O_K$, and such that $(\paze, \varphi_{\paze})|_{U_{\lambda, \prism}}$ is an object of $\CRhatphi(U_{\lambda, \prism})$.
	Morphisms between two objects are given as maps of $\pazo_{\prism}\modules$ compatible with $\varphi_{\paze}$.
	Say that $(\paze, \varphi_{\paze})$ is effective if $\varphi_{\paze}$ carries $\varphi^*\paze$ into ${\paze}$.
\end{defi}

\begin{rem}
	Note that Definition \ref{defi:completed_fcrystal} is slightly more general than loc.\ cit.\ in the sense that we do not restrict ourselves to effective completed $F\crystals$.
\end{rem}

\subsubsection{Analytic \texorpdfstring{$F\crystals$}{-}}

Let $X$ be a $\padic$ formal scheme.

\begin{defi}[Analytic prismatic crystals, {\cite[Definition 3.1]{guo-reinecke}}]\label{defi:analytic_crystal}
	Define the category of \textit{analytic prismatic crystals of vector bundles} over $X$ as 
	\begin{equation*}
		\Vect^{\an}(X_{\prism}) \coloneq \lim_{(A, I) \in X_{\prism}} \Vect(\Spec(A) \setminus V(p, I)).
	\end{equation*}
\end{defi}

Let $(A, I)$ be a prism, and let $\varphi$ denote the Frobenius endomorphism on $A$.
Then, we have that $\varphi(I) \subset (p, I)$, since $\varphi(x) = x^p + p\delta(x)$ is in $(p, I)$, for each $x$ in $I$.
In particular, the Frobenius endomorphism $\varphi$ on $A$ preserves the subscheme $\Spec(A) \setminus V(p, I)$, and we denote the induced endomorphism on $\Spec(A) \setminus V(p, I)$ again by $\varphi$.

\begin{defi}[Analytic prismatic $F\crystals$, {\cite[Definition 3.6]{guo-reinecke}}]\label{defi:analytic_fcrystal}
	Define the category $\Vect^{\an, \varphi}(A)$ as follows: an object is a pair $(M, \varphi_M)$ with $M$ a vector bundle on $\Spec(A) \setminus V(p, I)$ equipped with an $A\linear$ isomorphism $\varphi_M \colon (\varphi^* M)[1/I] \isomorphic M[1/I]$.
	Morphisms between two objects are given as maps of vector bundles compatible with $\varphi_M$.
	Say that $(M, \varphi_M)$ is effective if $\varphi_M$ carries $\varphi^* M$ into $M$.

	Define the category of analytic prismatic $F\crystals$ on $X_{\prism}$ as follows:
	\begin{equation*}
		\Vect^{\an, \varphi}(X_{\prism}) \coloneq \lim_{(A, I) \in X_{\prism}} \Vect^{\varphi}(\Spec(A) \setminus V(p, I)).
	\end{equation*}
	Say that $(\paze, \varphi_{\paze})$ is effective if $\varphi_{\paze}$ carries $\varphi^* \paze$ into $\paze$.
\end{defi}

\begin{rem}
	Let $X = \Spf(O_F)$, then from \cite[Proposition 3.8]{guo-reinecke} note that the functor induced by restriction to open subsets induces an equivalence of categories $\Vect^{\varphi}(X_{\prism}, \pazo_{\prism}) \isomorphic \Vect^{\an, \varphi}(X_{\prism})$.
\end{rem}

\subsubsection{\'Etale realisation functors}

Let us now introduce the \'etale realisation functors from the category of (completed/analytic) prismatic $F\crystals$ using Lemma \ref{lem:laurent_fcrystal_local_system}.
Let $X$ be a bounded $\padic$ formal scheme over $\Spf(O_K)$, and $(\paze, \varphi_{\paze})$ an object of $\Vect^{\an, \varphi}(X_{\prism})$.
Consider the open embedding $\Spec(A[1/I]) \subset \Spec(A) \setminus V(p, I)$, for any prism $(A, I)$ in $X_{\prism}$, and let $\paze[1/\pazi_{\prism}]$ denote the pullback of $\paze$ along this embedding, i.e.\ it is an object of $\Vect^{\varphi}(X_{\prism}, \pazo_{\prism}[1/\pazi_{\prism}])$.
Using the equivalence $\Vect^{\varphi}(X_{\prism}, \pazo_{\prism}[1/\pazi_{\prism}]_p^{\wedge}) \isomorphic \Loc_{\ZZ_p}(X_{\eta})$ from Lemma \ref{lem:laurent_fcrystal_local_system}, we have the following:
\begin{defi}\label{defi:etale_functor}
	Define the \'etale realisation functor for completed $F\crystals$ as
	\begin{align*}
		T^{\wedge}_{\etale} \colon \CRhatphi(X_{\prism}) &\longrightarrow \Loc_{\ZZ_p}(X_{\eta})\\
		(\paze, \varphi_{\paze}) &\longmapsto \big(\paze[1/\pazi_{\prism}] \otimes_{\pazo[1/\pazi_{\prism}]} \pazo_{\prism}[1/\pazi_{\prism}]_p^{\wedge}\big)^{\varphi=1}.
	\end{align*}
	Similarly, define the \'etale realisation functor for analytic $F\crystals$ as
	\begin{align*}
		T^{\an}_{\etale} \colon \Vect^{\an, \varphi}(X_{\prism}) &\longrightarrow \Loc_{\ZZ_p}(X_{\eta})\\
		(\paze, \varphi_{\paze}) &\longmapsto \big(\paze[1/\pazi_{\prism}] \otimes_{\pazo[1/\pazi_{\prism}]} \pazo_{\prism}[1/\pazi_{\prism}]_p^{\wedge}\big)^{\varphi=1}.
	\end{align*}
\end{defi}

\begin{rem}
	The \'etale realisation functor from $\Vect^{\varphi}(X_{\prism}, \pazo_{\prism})$ in \cite[Construction 4.8]{bhatt-scholze-crystals} naturally factors through the functors $T^{\wedge}_{\etale}$ (see \cite[Proposition 3.43]{du-liu-moon-shimizu}) and $T^{\an}_{\etale}$ (see \cite[Construction 3.9]{guo-reinecke}).
\end{rem}

\begin{rem}\label{rem:analytic_completed_crystalline_equiv}
	Assume that $X$ is a smooth $\padic$ formal scheme over $\Spf(O_K)$.
	Then, the essential images of $T^{\wedge}_{\etale}$ and $T^{\an}_{\etale}$ coincide, and we have natural equivalences of categories (see \cite[Theorem 1.3]{du-liu-moon-shimizu} and \cite[Theorem A]{guo-reinecke}):
	\begin{equation*}
		\CRhatphi(X_{\prism}) \xrightarrow[\sim]{\hspace{2mm} T^{\wedge}_{\etale} \hspace{2mm}} \Loc_{\ZZ_p}^{\crys}(X_{\eta}) \xleftarrow[\sim]{\hspace{2mm} T^{\an}_{\etale} \hspace{2mm}} \Vect^{\an, \varphi}(X_{\prism}),
	\end{equation*}
	where $\Loc_{\ZZ_p}^{\crys}(X_{\eta})$ is the category of $\ZZ_p\textrm{-local systems}$ $\LL$ on $X_{\eta}$ such that $\LL \otimes_{\ZZ_p} \QQ_p$ is crystalline.
\end{rem}

\section{The prism \texorpdfstring{$(\AR, [p]_q)$}{-}}\label{sec:theprism_ar+}

In this section, we will define and study a prism $(\AR, [p]_q)$, which is of fundamental importance in stating and proving our main results in Sections \ref{sec:integral_comp} and \ref{sec:prismatic_wach}.
We will use the setup and notations from Section \ref{subsec:setup_nota}.

\subsection{Some important rings}\label{subsec:important_rings}
 
Let $\Rinfty^{\flat}$ and $\Rbar^{\flat}$ denote the tilt of $\Rinfty$ and $\Rbar$, respectively (see Section \ref{subsec:setup_nota}).
Set $\Ainf(\Rinfty) \coloneq W(\Rinfty^{\flat})$ and $\Ainf(\Rbar) \coloneq W(\Rbar^{\flat})$ to be the ring of $p\textrm{-typical}$ Witt vectors with coefficients in $\Rinfty^{\flat}$ and $\Rbar^{\flat}$, respectively, and equipped with the Frobenius endomorphism on Witt vectors and a continuous (for the weak topology) $\Gamma_R\action$ and $G_R\action$, respectively (see \cite[Proposition 7.2]{andreatta-phigamma}).
Moreover, from loc.\ cit., we have that $\Ainf(\Rinfty) = \Ainf(\Rbar)^{H_R}$.
We fix $\mubar \coloneq \varepsilon-1$, where $\varepsilon \coloneq (1, \zeta_p, \zeta_{p^2}, \ldots)$ is in $\Rinfty^{\flat}$, and let $q \coloneq [\varepsilon]$ in $\Ainf(\Rinfty)$ denote the Teichm\"uller lift of $\varepsilon$.
Let $\mu \coloneq q - 1$ and $\xi \coloneq \mu/\varphi^{-1}(\mu)$ be some fixed elements of $\Ainf(\Rinfty)$, and note that $\mu = \mubar \textrm{ mod } p$.
Then, from the description of the weak topology on $\Ainf(\Rinfty)$ and $\Ainf(\Rbar)$ in op.\ cit., it is easy to see that the weak topology on these rings is the same as the $(p, \mu)\adic$ topology.
Next, let us note that we have $\varphi(q) = q^p$, and for any $g$ in $G_R$ we have that $g(q) = q^{\chi(g)}$, where $\chi$ is the $\padic$ cyclotomic character, in particular, we see that $\varphi(\mu) = (1+\mu)^p - 1$ and $g(1+\mu) = (1+\mu)^{\chi(g)}$.
Furthermore, we have a $G_R\equivariant$ surjection $\theta \colon \Ainf(\Rbar) \rightarrow \widehat{\Rbar}$ and $\kert \theta = \xi \Ainf(\Rbar)$.
The map $\theta$ further induces a $\Gamma_R\equivariant$ surjection $\theta \colon \Ainf(\Rinfty) \rightarrow \widehat{R}_{\infty}$, with $\kert \theta = \xi \Ainf(\Rinfty)$.

For $1 \leqslant i \leqslant d$, we fix $X_i^{\flat} = (X_i, X_i^{1/p}, X_i^{1/p^2}, \ldots)$ in $\Rinfty^{\flat}$, and we fix topological generators of $\Gamma_R'$ as $\{\gamma_1, \ldots, \gamma_d\}$ and such that $\gamma_j(X_i^{\flat}) = \varepsilon X_i^{\flat}$, if $i=j$ and $0$ otherwise.
Moreover, we may view $\Gamma_F$ as a subgroup of $\Gamma_R$, and for any $g$ in $\Gamma_F$ we have $g \gamma_i g^{-1} = \gamma_i^{\chi(g)}$, for all $1 \leqslant i \leqslant d$.
Let us also fix Teichm\"uller lifts $[X_i^{\flat}]$ in $\Ainf(\Rinfty)$.

\subsubsection{The ring \texorpdfstring{$\AR$}{-}}\label{subsubsec:ring_ar+}

Let $\Aframe$ denote the $(p, \mu)\textrm{-adic}$ completion of the ring $O_F\llbracket \mu \rrbracket\big[[X_1^{\flat}]^{\pm 1}, \ldots, [X_d^{\flat}]^{\pm 1}\big]$.
Note that we have a natural injective homomorphism of rings $\Aframe \hookrightarrow \Ainf(\Rinfty)$ and $\Aframe$ is stable under the induced Frobenius endomorphism $\varphi$ and the action of $\Gamma_R$ (see \cite[Section 3]{abhinandan-relative-wach-i}); we equip $\Aframe$ with induced structures.
Moreover, we have an injective homomorphism $\iota \colon \Rframe \rightarrow \Aframe$, defined via the map $X_i \mapsto [X_i^{\flat}]$, for $1 \leqslant i \leqslant d$, and it extends to an isomorphism of rings $\Rframe\llbracket \mu \rrbracket \isomorphic \Aframe$.
Equip $\Rframe\llbracket \mu \rrbracket$ with a Frobenius endomorphism $\varphi$ extending the Frobenius on $\Rframe$ by setting $\varphi(\mu) = (1+\mu)^p-1$, and note that the map $\varphi$ is finite and faithfully flat of degree $p^{d+1}$.
Additionally, equip $\Rframe\llbracket \mu \rrbracket$ with an $\Rframe\linear$ continuous action of $\Gamma_F$ by setting $g(1+ \mu) = (1+\mu)^{\chi(g)}$, for any $g$ in $\Gamma_F$.
Then, the homomorphism $\iota$ and the isomorphism $\Rframe\llbracket \mu \rrbracket \isomorphic \Aframe^+$ are Frobenius and $\Gamma_F\equivariant$ (see \cite[Section 2.5]{abhinandan-relative-wach-ii}).

Let $\AR$ denote the $(p, \mu)\textrm{-adic}$ completion of the unique extension of the injective homomorphism $\Aframe^+ \hookrightarrow \Ainf(\Rinfty)$ along the $p\textrm{-adically}$ completed \'etale map $\Rframe \rightarrow R$ (see \cite[Section 3.3.2]{abhinandan-relative-wach-i}).
Then, we obtain a natural injective homomorphism $\AR \hookrightarrow \Ainf(\Rinfty)$ and note that $\AR$ is stable under the Frobenius and $\Gamma_R\action$ on the latter; we equip $\AR$ with induced structures.
Moreover, note that the homomorphism $\iota \colon \Rframe \hookrightarrow \Aframe^+ \hookrightarrow \AR$ and the isomorphism $\Rframe\llbracket \mu \rrbracket \isomorphic \Aframe^+ \hookrightarrow \AR$ extend respectively uniquely to an injective homomorphism $\iota \colon R \hookrightarrow \AR$ and an isomorphism of rings $R\llbracket \mu \rrbracket \isomorphic \AR$.
Equip $R\llbracket \mu \rrbracket$ with a Frobenius endomorphism $\varphi$ extending the Frobenius on $R$ by setting $\varphi(\mu) = (1+\mu)^p-1$ and note that the map $\varphi$ is finite and faithfully flat of degree $p^{d+1}$.
Additionally, equip $R\llbracket \mu \rrbracket$ with an $R\linear$ continuous action of $\Gamma_F$ by setting $g(1+ \mu) = (1+\mu)^{\chi(g)}$, for any $g$ in $\Gamma_F$.
Then, it is clear that the homomorphism $\iota$ and the isomorphism $R\llbracket \mu \rrbracket \isomorphic \AR$ are Frobenius and $\Gamma_F\equivariant$ (see \cite[Section 2.5]{abhinandan-relative-wach-ii}).
In particular, the Frobenius endomorphism $\varphi$ on $\AR$ is finite and faithfully flat of degree $p^{d+1}$.

\begin{rem}\label{rem:af+_nota}
	For $R = O_F$, we will denote the ring $\AR$ by $\AF$ which is equipped with a Frobenius endomorphism $\varphi$ and a continuous action of $\Gamma_F$.
\end{rem}

Next, let us fix the following elements inside $\AF$:
\begin{equation}\label{eq:pq_mu0_ptilde_defi}
	[p]_q \coloneq \tfrac{q^p-1}{q-1} = \tfrac{\varphi(\mu)}{\mu}, \hspace{3mm} \mu_0 \coloneq \textstyle\sum_{a \in \FF_p^{\times}} ((1+\mu)^{[a]}-1) = -p + \textstyle\sum_{a \in \FF_p} q^{[a]}, \hspace{3mm} \ptilde \coloneq \mu_0 + p = \textstyle\sum_{a \in \FF_p} q^{[a]}.
\end{equation}

\begin{rem}\label{rem:mu0_ptilde_nota}
	Note that for $p=2$, we have $\mu_0 = \mu$ and $[p]_q = \mu + 2 = \ptilde$.
	However, such equalities do not hold for $p \geqslant 3$.
	Let us also remark that we denote the sum $-p + \sum_{a \in \FF_p} (1+\mu)^{[a]}$ as $\mu_0$ following \cite[Section 5.2.5]{fontaine-corps-periodes}, where Fontaine denotes our $\mu$ by $\pi$ and $\mu_0$ by $\pi_0$.
	Moreover, the notation for $\ptilde = \mu_0 + p$ comes from \cite[Section 3.8]{bhatt-lurie-apc}.
\end{rem}

\begin{lem}\label{lem:pq_mu0p_unit}
	The element $\ptilde$ is the product of $[p]_q$ with a unit in $\AF$, in particular, both $\ptilde$ and $[p]_q$ generate the same ideal inside $\AF$.
\end{lem}
\begin{proof}
	Note that $\AF/([p]_q) = O_F[\zeta_p]$ and we have $[a] = a \textrm{ mod } p\ZZ_p$.
	Since $\zeta_p^p=1$ and $\ptilde = \sum_{a \in \FF_p} (1+\mu)^{[a]}$, we see that $\ptilde = \sum_{a \in \FF_p} \zeta_p^{[a]} = \sum_{a=0}^{p-1} \zeta_p^{a} = 0 \textrm{ mod } [p]_q$.
	In particular, $\ptilde = [p]_qx$, for some $x$ in $\AF$.
	Moreover, since $[p]_q = p \textrm{ mod } \mu$ and $\mu_0 = 0 \textrm{ mod } \mu$, we conclude that $x = 1 + \mu y$ for some $y$ in $\AF$, i.e.\ $x$ is a unit in $\AF$.
	This proves the claim.
\end{proof}

\subsubsection{The ring \texorpdfstring{$R\llbracket\mu_0\rrbracket$}{-}}\label{subsubsec:rmu0}

From Section \ref{subsubsec:ring_ar+}, recall that the ring $R\llbracket\mu\rrbracket$ is equipped with a continuous and $R\linear$ action of $\Gamma_F$.
Moreover, note that $\mu_0 = \sum_{a \in \FF_p} (1+\mu)^{[a]}$ is an element of $R\llbracket\mu\rrbracket$ and it is invariant under the action of the subgroup $\FF_p^{\times}$ of $\Gamma_F$ (see the discussion after \eqref{eq:gammaf_es} for notation).

Now, consider an injective morphism of $R\algebras$ $R\llbracket z \rrbracket \rightarrow R\llbracket \mu \rrbracket$ defined by sending $z \mapsto \mu_0$.
Denote the image inside $R\llbracket \mu \rrbracket$ by $R\llbracket\mu_0\rrbracket$ and we will view the latter as the ring of formal power series over $R$ in the variable $\mu_0$.
Note that the canonical injective homomorphism of $R\algebras$ $R\llbracket\mu_0\rrbracket \rightarrow R\llbracket\mu\rrbracket$ is continuous for the $(p, \mu_0)\adic$ topology on the former and $(p, \mu_0)\adic = (p, \mu)\adic$ topology on the latter.
Moreover, from the explicit description of $\mu_0$ in \eqref{eq:pq_mu0_ptilde_defi}, it follows that $R\llbracket\mu_0\rrbracket$ is stable under the Frobenius endomorphism and the continuous action of $\Gamma_0$ on $R\llbracket\mu\rrbracket$; we equip $R\llbracket\mu_0\rrbracket$ with the induced Frobenius endomorphism and continuous action of $\Gamma_0$.

\begin{lem}\label{lem:rmu0_fpxinv}
	Taking invariants of $R\llbracket\mu\rrbracket$ under the action of $\FF_p^{\times}$ induces a $(\varphi, \Gamma_0)\equivariant$ isomorphism of rings $R\llbracket\mu_0\rrbracket \isomorphic R\llbracket\mu\rrbracket^{\FF_p^{\times}}$.
	Similarly, we have a $(\varphi, \Gamma_0)\equivariant$ and $R\llbracket\mu_0\rrbracket\linear$ isomorphism $\mu_0 R\llbracket\mu_0\rrbracket \isomorphic (\mu R\llbracket\mu\rrbracket)^{\FF_p^{\times}}$.
\end{lem}
\begin{proof}
	Note that the map in the claim is $(\varphi, \Gamma_0)\equivariant$ by definition.
	Moreover, by Remark \ref{rem:mu0_ptilde_nota}, the claim is trivial for $p=2$.
	So, assume that $p \geqslant 3$ and consider the following $(\varphi, \Gamma_F)\equivariant$ commutative diagram:
	\begin{equation}\label{eq:rmu0_mod_ptilde}
		\begin{tikzcd}
			0 \arrow[r] & R\llbracket\mu_0\rrbracket \arrow[r, "\ptilde"] \arrow[d] & R\llbracket\mu_0\rrbracket \arrow[r] \arrow[d] & R \arrow[r] \arrow[d] & 0\\
			0 \arrow[r] & R\llbracket\mu\rrbracket \arrow[r, "\ptilde"] & R\llbracket\mu\rrbracket \arrow[r] & R[\zeta_p] \arrow[r] & 0,
		\end{tikzcd}
	\end{equation}
	where the vertical arrows are natural inclusions and the bottom row is exact since we have $R\llbracket\mu\rrbracket/([p]_q) \isomorphic R[\zeta_p]$ and $[p]_qR\llbracket\mu\rrbracket = \ptilde R\llbracket\mu\rrbracket$ using Lemma \ref{lem:pq_mu0p_unit}.
	Now, note that the top row is $\FF_p^{\times}\textrm{-invariant}$ and for the bottom row, consider the associated long exact sequence for the cohomology of $\FF_p^{\times}\action$ and observe that $H^1(\FF_p^{\times}, R\llbracket\mu\rrbracket) = 0$, since $p-1$ is invertible in $\ZZ_p$, and $R[\zeta_p]^{\FF_p^{\times}} = R$.
	In particular, from the long exact sequence we obtain that $R\llbracket\mu_0\rrbracket/(\ptilde) \isomorphic R\llbracket\mu\rrbracket^{\FF_p^{\times}}/(\ptilde)$.
	Since, $R\llbracket\mu_0\rrbracket$ and $R\llbracket\mu\rrbracket^{\FF_p^{\times}}$ are $\ptilde\textrm{-adically}$ complete and $\ptilde\torsionfree$, it follows that $R\llbracket\mu_0\rrbracket \isomorphic R\llbracket\mu\rrbracket^{\FF_p^{\times}}$.

	Next, for the second claim, we consider the following $(\varphi, \Gamma_F)\equivariant$ commutative diagram with exact rows:
	\begin{center}
		\begin{tikzcd}
			0 \arrow[r] & \mu_0 R\llbracket\mu_0\rrbracket \arrow[r] \arrow[d] & R\llbracket\mu_0\rrbracket \arrow[r] \arrow[d] & R \arrow[r] \arrow[d, equal] & 0\\
			0 \arrow[r] & \mu R\llbracket\mu\rrbracket \arrow[r] & R\llbracket\mu\rrbracket \arrow[r] & R \arrow[r] & 0,
		\end{tikzcd}
	\end{center}
	where the vertical arrows are natural inclusions.
	Note that the top row is $\FF_p^{\times}\textrm{-invariant}$ and for the bottom row, consider the associated long exact sequence for the cohomology of $\FF_p^{\times}\action$ and observe that $H^1(\FF_p^{\times}, \mu R\llbracket\mu\rrbracket) = 0$, since $p-1$ is invertible in $\ZZ_p$ and we have $R^{\FF_p^{\times}} = R$ and $R\llbracket\mu_0\rrbracket \isomorphic R\llbracket\mu\rrbracket^{\FF_p^{\times}}$ from the first part.
	Hence, it follows that the left vertical arrow induces a $(\varphi, \Gamma_0)\equivariant$ isomorphism $\mu_0 R\llbracket\mu_0\rrbracket \isomorphic (\mu R\llbracket\mu\rrbracket)^{\FF_p^{\times}}$.
\end{proof}

\begin{rem}\label{rem:rmu0_fpx_decomp}
	Note that from \eqref{eq:fpx_decomp}, we have an $\FF_p^{\times}\textrm{-decomposition}$ as $R\llbracket\mu\rrbracket = \bigoplus_{i=0}^{p-2} R\llbracket\mu\rrbracket_i$.
	Moreover, from Lemma \ref{lem:rmu0_fpxinv}, it follows that we have $R\llbracket\mu_0\rrbracket \isomorphic R\llbracket\mu\rrbracket^{\FF_p^{\times}} = R\llbracket\mu\rrbracket_0$ and the preceding decomposition is $R\llbracket\mu_0\rrbracket\textrm{-linear}$.
	In particular, for each $0 \leqslant i \leqslant p-2$, we have that $R\llbracket\mu\rrbracket_i$ is a $(p, \mu_0)\textrm{-adically}$ complete module over $R\llbracket\mu_0\rrbracket$, equipped with a continuous action of $\Gamma_0$.
\end{rem}

\begin{lem}\label{lem:mu0_gamma0_act}
	Let $g$ be any element of $\Gamma_0$, then $(g-1)\mu_0$ is an element of $\ptilde \mu_0 O_F\llbracket\mu_0\rrbracket$.
\end{lem}
\begin{proof}
	Recall that $O_F\llbracket\mu_0\rrbracket \subset \AF$ is stable under the action of $\Gamma_0$.
	So if $g$ is any element of $\Gamma_0$, then $(g-1)\mu_0$ is an element of $(\mu \AF)^{\FF_p^{\times}} \lisomorphic \mu_0 O_F\llbracket\mu_0\rrbracket$ (see Lemma \ref{lem:rmu0_fpxinv}).
	So let us write $(g-1)\mu_0 = \mu_0 x$.
	Now, let us consider the diagram in \eqref{eq:rmu0_mod_ptilde} with $R = O_F$, where the projection map $O_F\llbracket\mu_0\rrbracket \rightarrow O_F$ sends $\mu_0$ to $-p$.
	In particular, we see that $(g-1)\mu_0 = -p x \textrm{ mod } \ptilde O_F\llbracket\mu_0\rrbracket$.
	Moreover, from the $\Gamma_F\textrm{-equivariance}$ of the diagram, it follows that $(g-1)\mu_0 = 0 \textrm{ mod } \ptilde O_F\llbracket\mu_0\rrbracket$.
	As $O_F$ is $p\torsionfree$, therefore, we conclude that $x$ must be an element of $\ptilde O_F\llbracket\mu_0\rrbracket$.
	Hence, $(g-1)\mu_0$ is an element of $\ptilde \mu_0 O_F\llbracket\mu_0\rrbracket$.
\end{proof}

\begin{lem}\label{lem:mu0_mup1_unit}
	The element $\mu_0$ is the product of $\mu^{p-1}$ with a unit in $\AF$, in particular, both $\mu_0$ and $\mu^{p-1}$ generate the same ideal inside $\AF$.
\end{lem}
\begin{proof}
	The idea of the proof is motivated from \cite[Proposition 5.2.6]{fontaine-corps-periodes}.
	Let $\mu_1 = \prod_{a \in \FF_p^{\times}} ((1+\mu)^{[a]}-1)$ be an element of $\AF$.
	Then, by expanding the product, it is easy to see that $\mu_1 = \mu^{p-1} u$ for some unit $u$ in $\AF$.
	Moreover, $\mu_1$ is invariant under the action of $\FF_p^{\times}$ on $\AF$.
	Therefore, we get that $\mu_1$ is also an element of $O_F\llbracket\mu\rrbracket^{\FF_p^{\times}} \lisomorphic O_F\llbracket\mu_0\rrbracket$ (see Lemma \ref{lem:rmu0_fpxinv}).
	So, let us write $\mu_1 = \sum_{k \in \NN} a_k \mu_0^k$, where $a_k$ is in $O_F$ for each $k \in \NN$.
	But since $\mu_1 = 0 \textrm{ mod } \mu \AF$ and $\mu_0 = 0 \textrm{ mod } \mu \AF$, we conclude that $a_0$ must be $0$.
	In particular, inside $\AF$, we have 
	\begin{equation}\label{eq:mu0_mup1_unit}
		\mu^{p-1} u = \mu_1 = \mu_0(a_1 + \textstyle\sum_{k \geqslant 2} a_k \mu_0^{k-1}).
	\end{equation}
	Since $\AF$ is $\mu_0\adically$ complete, it is enough to show that in \eqref{eq:mu0_mup1_unit}, the element $a_1$ is a unit in $O_F$.
	Note that from the expression of $\mu_0$, it is easy to see that we have $\mu_0 = \mu^{p-1} \textrm{ mod } p \AF$.
	Therefore, reducing \eqref{eq:mu0_mup1_unit} modulo $p$, and using the preceding observation, we get that $\mu_0 u = \mu_0(a_1 + \sum_{k \geqslant 2} a_k \mu_0^{k-1}) \textrm{ mod } p \AF$.
	Now note that $\AF/p = \kappa\llbracket\mu\rrbracket$ is $\mu_0\torsionfree$.
	Therefore, from the preceding equation, we conclude that $u = a_1 + \sum_{k \geqslant 2} a_k \mu^{(p-1)(k-1)} \textrm{ mod } p \AF$.
	In particular, $a_1 = u - \sum_{k \geqslant 2} a_k \mu^{(p-1)(k-1)} \textrm{ mod } p \AF$, is a unit in $\AF/p = \kappa\llbracket\mu\rrbracket$.
	Since $a_1$ is an element of $O_F$, it follows that $a_1 \textrm{ mod } p$ is a unit in $\kappa$ and therefore $a_1$ is a unit in $O_F$.
	Hence, from \eqref{eq:mu0_mup1_unit} we obtain that $\mu^{p-1}$ is the product of $\mu_0$ with a unit in $\AF$.
\end{proof}

\begin{rem}\label{rem:mu0_mu_ff}
	From Lemma \ref{lem:mu0_mup1_unit} and its proof, it follows that $\AF$ is a free $O_F\llbracket\mu_0\rrbracket\module$ of rank $p-1$.
	In particular, the natural ring homomorphism $O_F\llbracket\mu_0\rrbracket \rightarrow \AF$ is faithfully flat and finite of degree $p-1$.
	Similarly, the natural ring homomorphism $R\llbracket\mu_0\rrbracket \rightarrow R\llbracket\mu\rrbracket$ is faithfully flat and finite of degree $p-1$.
\end{rem}

\begin{lem}\label{lem:mu_mu0_deltastable}
	Let us consider the ring $\AF$ equipped with Frobenius endomorphism $\varphi$ and its subring $O_F\llbracket\mu_0\rrbracket$ equipped with the induced Frobenius.
	Then, we have the following:
	\begin{enumarabicup}
	\item The ideal $\mu \AF \subset \AF$ is $\delta\textrm{-stable}$ in the sense of \cite[Example 2.10]{bhatt-scholze-prisms}.

	\item The ideal $\mu_0 O_F\llbracket\mu_0\rrbracket \subset O_F\llbracket\mu_0\rrbracket$ is $\delta\textrm{-stable}$.
	
	\item We may write $\varphi(\mu_0) = u \mu_0 \ptilde^{p-1}$, for some unit $u$ in $O_F\llbracket\mu_0\rrbracket$.
	\end{enumarabicup}
\end{lem}
\begin{proof}
	To prove (1), note that we have $\delta(\mu) = \tfrac{\varphi(\mu)-\mu^p}{p}$ belongs to $\mu \AF$.
	Then, using the product formula for $\delta\textrm{-structures}$ from \eqref{eq:delta_sum_prod}, it follows that for any $x$ in $\AF$, we have $\delta(\mu x) = \mu^p \delta(x) + x^p \delta(\mu) + p\delta(\mu) \delta(x)$, i.e.\ $\delta(\mu x)$ belongs to $\mu \AF$.
	Next, by induction on $m \in \NN$ and using the preceding observation, it easily follows that $\delta^m(\mu x)$ also belongs to $\mu \AF$.

	To show (2), let us first note that since the action of $\Gamma_F$ on $\AF$ commutes with the Frobenius, therefore, the action of $\Gamma_F$ also commutes with the $\delta\textrm{-map}$, i.e.\ for any $x$ in $\AF$ and $g$ in $\Gamma_F$, we have $g(\delta(x)) = \delta(g(x))$.
	Now, by using (1), we have that for any $y$ in $O_F\llbracket\mu_0\rrbracket$, the element $\mu_0 y$ belongs to $\mu \AF$.
	Therefore, for any $m \in \NN$, we get that $\delta^m(\mu_0 y)$ belongs to $\mu \AF$.
	But $\delta^m(\mu_0 y)$ also belongs to $O_F\llbracket\mu_0\rrbracket$, so we get that $\delta^m(\mu_0 y)$ is an element of $O_F\llbracket\mu_0\rrbracket \cap \mu \AF = (\mu \AF)^{\FF_p^{\times}} \lisomorphic \mu_0 O_F\llbracket\mu_0\rrbracket$ (see Lemma \ref{lem:rmu0_fpxinv}).
	In particular, the ideal $\mu_0 O_F\llbracket\mu_0\rrbracket \subset O_F\llbracket\mu_0\rrbracket$ is $\delta\textrm{-stable}$.

	The proof of (3) is similar to that of \cite[Proposition 5.2.6]{fontaine-corps-periodes}.
	Using Lemma \ref{lem:pq_mu0p_unit} and Lemma \ref{lem:mu0_mup1_unit}, let us write $[p]_q = v \ptilde$ and $\mu_0 = w \mu^{p-1}$, for some units $v, w \in \AF$.
	Then, we have that
	\begin{equation*}
		\varphi(\mu_0) = \varphi(w \mu^{p-1}) = \varphi(w) \varphi(\mu)^{p-1} = \varphi(w) \mu^{p-1} [p]_q^{p-1} = \varphi(w) v^{p-1} \mu^{p-1} \ptilde^{p-1} = u \mu_0 \ptilde^{p-1},
	\end{equation*}
	where $u = \varphi(w) w^{-1} v^{p-1}$ is a unit in $\AF$.
	Since, $\mu_0$ and $\ptilde$ belong to $O_F\llbracket\mu_0\rrbracket$, it follows that $u = \varphi(\mu_0)/(\mu_0 \ptilde^{p-1})$ is an element of $O_F\llbracket\mu\rrbracket^{\FF_p^{\times}} \lisomorphic O_F\llbracket\mu_0\rrbracket$ (see Lemma \ref{lem:rmu0_fpxinv}).
\end{proof}

\subsection{The prism \texorpdfstring{$(\AR, [p]_q)$}{-}}\label{subsec:theprism_ar+}

In this section, we will show that $(\AR, [p]_q)$ is an object of the site $(\Spf R)_{\prism}$ and a covering of the final object of the topos $\Shv((\Spf R)_{\prism})$.
Moreover, we will compute the first few terms of its prismatic \v{C}ech nerve and study the induced action of $\Gamma_R$ on those terms.

\subsubsection{An object of \texorpdfstring{$(\textbf{Spf } R)_{\prism}$}{-}}

We begin with the following important observation:

\begin{lem}\label{lem:ar+_pq_prism}
	The pair $(\AR, [p]_q)$ is a prism and an object of $(\Spf R)_{\prism}$.
\end{lem}
\begin{proof}
	Note that the ring $\AR$ is $p\torsionfree$ and equipped with a Frobenius endomorphism $\varphi$.
	So from Remark \ref{rem:frob_delta_ring}, we get that $\AR$ is a $\delta\ring$.
	Moreover, we have that $[p]_q = \mu^{p-1} + p\mu^{p-2} + \cdots + p$, so $\AR$ is $(p, [p]_q)\adically = (p, \mu)\adically$ complete.
	Now, since $\varphi(\mu) = [p]_q \mu$, therefore, we may write $p = \varphi([p]_q) - (\mu^{p-1}[p]_q^{p-1} + p\mu^{p-2}[p]_q^{p-2} + \cdots + p(p-1)\mu[p]_q/2)$, in particular, $p$ is an element of $[p]_q \AR + \varphi([p]_q)\AR$.

	Next, we have a $\Gamma_R\equivariant$ surjective map $\Aframe^+ \twoheadrightarrow \Aframe^+/([p]_q) \isomorphic \Rframe\big[\zeta_p, X_1^{1/p}, \ldots, X_d^{1/p}\big]$, where the isomorphism is obtained by sending $\mu \mapsto \zeta_p-1$ and $[X_i^{\flat}] \mapsto X_i^{1/p}$, for $1 \leqslant i \leqslant d$.
	The surjective map above extends uniquely along the $p\textrm{-adically}$ completed \'etale map $\Rframe \rightarrow R$, i.e.\ we have a $\Gamma_R\equivariant$ surjection
	\begin{equation}\label{eq:ar+_ralg}
		\AR \twoheadrightarrow \AR/([p]_q) \isomorphic R\big[\zeta_p, X_1^{1/p}, \ldots, X_d^{1/p}\big].
	\end{equation}
	Clearly, we have the structure map $R \rightarrow R\big[\zeta_p, X_1^{1/p}, \ldots, X_d^{1/p}\big]$ and $\AR/([p]_q)$ is $p\torsionfree$.
	Hence, $(\AR, [p]_q)$ is a prism and it satisfies the axioms of Definition \ref{defi:abs_prismatic_site}, therefore, it is an object of $(\Spf R)_{\prism}$.
\end{proof}

From Lemma \ref{lem:ar+_pq_prism} we see that $(\AR, [p]_q)$ may be regarded as an object of $(\Spf R)_{\prism}$ via $\AR \twoheadrightarrow \AR/([p]_q) \leftarrow R$.
Moreover, we have the following:
\begin{lem}\label{lem:ar+_cover_fo}
	The object $(\AR, [p]_q)$ is a cover of the final object of the topos $\Shv((\Spf R)_{\prism})$.
\end{lem}
\begin{proof}
	Note that $\widehat{R}_{\infty}$ is a perfectoid $R\algebra$ (see \cite[Definition 4.18]{bhatt-morrow-scholze-2}) and the $\delta\textrm{-pair}$ $(\Ainf(\Rinfty), [p]_q)$ is a prism and an object of $(\Spf R)_{\prism}$, since $\Ainf(\Rinfty)/([p]_q) \isomorphic \widehat{R}_{\infty}$.
	Moreover, the natural map $\AR \rightarrow \Ainf(\Rinfty)$ is compatible with the prism structure on both rings, in particular, we have a map $(\AR, [p]_q) \rightarrow (\Ainf(\Rinfty), [p]_q)$ of prisms in $(\Spf R)_{\prism}$.
	So, to show that $(\AR, [p]_q)$ is a cover of the final object of the topos $\Shv((\Spf R)_{\prism})$, it is enough to show that $(\Ainf(\Rinfty), [p]_q)$ is a cover of the final object of the topos $\Shv((\Spf R)_{\prism})$.
	Now, note that the natural map $R \rightarrow \widehat{R}_{\infty}$ is faithfully flat and the $p\complete$ Tor amplitude of the cotangent complex $L_{\widehat{R}_{\infty}/R}$ in $D(\widehat{R}_{\infty})$ (see Appendix \ref{app_subsec:basic_defi}), is in $[-1, 0]$ because $R$ is a $p\complete$ smooth $\ZZ_p\textrm{-algebra}$ and $\widehat{R}_{\infty}$ is a perfectoid $\ZZ_p\textrm{-algebra}$, in particular, $\widehat{R}_{\infty}$ is quasisyntomic and $L_{\widehat{R}_{\infty}/\ZZ_p}$ has $p\complete$ Tor amplitude concentrated in degree $-1$ (see \cite[Proposition 4.19]{bhatt-morrow-scholze-2}).
	Therefore, $R \rightarrow \widehat{R}_{\infty}$ is a quasisyntomic cover in the sense of \cite[Definition 4.10]{bhatt-morrow-scholze-2}, and from \cite[Proposition 7.11]{bhatt-scholze-prisms}, it follows that $(\Ainf(\Rinfty), [p]_q)$ is a cover of the final object of the topos $\Shv((\Spf R)_{\prism})$.
	This allows us to conclude.
\end{proof}

\begin{lem}\label{lem:gammar_act_prisaut}
	Let $g$ be an element of $\Gamma_R$.
	Then, the action of $g$ on $\AR$ induces an automorphism of $(\AR, [p]_q)$ in $(\Spf R)_{\prism}$.
	Moreover, we have that $(g-1)\AR \subset \mu \AR$.
\end{lem}
\begin{proof}
	As the action of $\Gamma_R$ on $\AR$ is continuous, it is enough to check the claim for the topological generators $\{\gamma_1, \ldots, \gamma_d\}$ of $\Gamma_R'$ and each $g$ in $\Gamma_F$.
	Note that for $1 \leqslant i \leqslant d$, using the description of $\AR$ from Section \ref{subsubsec:ring_ar+} and explicit computations, it easily follows that we have $(\gamma_i-1)\AR \subset \mu\AR$ and $\gamma_i([p]_q\AR) \subset [p]_q\AR$.
	For any $g$ in $\Gamma_F$, note that we have $(g-1)\mu = (1+\mu)^{\chi(g)}-1-\mu = \chi(g) \mu u - \mu$, for some unit $u$ in $\AR$ depending on $g$, and therefore, it follows that $(g-1)\AR \subset \mu\AR$.
	Similarly, note that $g([p]_q) = g(\varphi(\mu)/\mu) = \varphi(\mu u)/(\mu u)$, which belongs to $[p]_q\AR$.
	This proves both the claims.
\end{proof}

\subsubsection{Prismatic \v{C}ech nerve of \texorpdfstring{$(\AR, [p]_q)$}{-}}\label{subsec:cech_nerve_ar+}

In this section, we will provide a description of the simplicial object $(\AR(\bullet), I(\bullet))$ in $(\Spf R)_{\prism}$.
Such a description would help in describing any prismatic $F\crystal$ over $(\Spf R)_{\prism}$ in terms of an $\AR\module$ with stratification.

\begin{const}\label{const:cech_nerve}
	Let $\AR(\bullet)$ denote the cosimplicial ring obtained by taking the prismatic \v{C}ech nerve $(\AR(\bullet), I(\bullet))$ of $(\AR, [p]_q)$ in $(\Spf R)_{\prism}$.
	Clearly, we have that $\AR(0) = \AR$.
	As products in $(\Spf R)_{\prism}$ are computed via prismatic envelopes (which exist in our case), we may describe $\AR(n)$ a bit more explicitly (see Lemma \ref{lem:selfprod_prismenv} below explaining that the following construction of $\AR(n)$ represents the $(n+1)\textrm{-fold}$ self product of $(\AR, [p]_q)$ in $(\Spf R)_{\prism}$).
	For $n \in \NN$, let $(\AR)^{\otimes (n+1)}$ denote the $(n+1)\textrm{-fold}$ tensor product of $\AR$ over $O_F$ and we consider the $\AR\algebra$ structure on $(\AR)^{\otimes (n+1)}$ via the first component, i.e.\ $a \mapsto a \otimes 1^{\otimes n}$.
	Let $(\AR)^{\widehat{\otimes} (n+1)}$ denote the $(p, [p]_q)\adic$ completion of $(\AR)^{\otimes (n+1)}$, which may be written as the $(p, [p]_q)\adic$ completion of an \'etale algebra over a polynomial ring in finitely many variables defined over a power-series $\AR\algebra$ in finitely many variables, in particular, we see that $(\AR)^{\widehat{\otimes} (n+1)}$ is $(p, [p]_q)\textrm{-completely}$ flat over $(\AR, [p]_q)$.
	In other words, for each $m \in \NN_{\geqslant 1}$, we have that the natural map $\AR/(p, [p]_q)^m \rightarrow (\AR)^{\widehat{\otimes} (n+1)}/(p, [p]_q \otimes 1)^m$ is flat.
	Since $\AR$ is noetherian, from \cite[\href{https://stacks.math.columbia.edu/tag/0912}{Tag 0912}]{stacks-project}, it follows that the map $\AR \rightarrow (\AR)^{\widehat{\otimes} (n+1)}$ is flat.
	Now, for each $n \in \NN$, we have a surjection 
	\begin{equation}\label{eq:ker_j}
		(\AR)^{\widehat{\otimes} (n+1)} \twoheadrightarrow R\big[\zeta_p, X_1^{1/p}, \ldots, X_d^{1/p}\big]^{\otimes (n+1)},
	\end{equation}
	where in the right hand term, the tensor product is over $R$.

	For $1 \leqslant i \leqslant n+1$, let $n_i \colon \AR \rightarrow (\AR)^{\widehat{\otimes} (n+1)}$ denote the map, sending $a \mapsto 1 \otimes \cdots \otimes a \otimes \cdots \otimes 1$, in the $i^{\textrm{th}}$ position.
	Then, the kernel of \eqref{eq:ker_j} is given by the ideal 
	\begin{equation*}
		J(n) = (n_i([p]_q), n_j([X_s^{\flat}]^p) - n_k([X_s^{\flat}]^p) \textrm{ for } 1 \leqslant s \leqslant d, 1 \leqslant i \leqslant n+1, 1 \leqslant j \leqslant n, k = j+1) \subset (\AR)^{\widehat{\otimes} (n+1)}.
	\end{equation*}
	Note that from the proof of Lemma \ref{lem:ar+_pq_prism}, we have that $\AR/([p]_q) \isomorphic R\big[\zeta_p, X_1^{1/p}, \ldots, X_d^{1/p}\big]$ (see \eqref{eq:ar+_ralg}).
	Therefore, it easily follows that the sequence $\{p, n_i([p]_q), n_j([X_s^{\flat}]^p) - n_k([X_s^{\flat}]^p) \textrm{ for } 1 \leqslant s \leqslant d, 1 \leqslant i \leqslant n+1, 1 \leqslant j \leqslant n, k = j+1\}$ is regular on $(\AR)^{\widehat{\otimes} (n+1)}$.
	So, from \cite[Proposition 3.13]{bhatt-scholze-prisms}, we set $(\AR(n), n_1([p]_q))$ to be the prismatic envelope of $((\AR)^{\widehat{\otimes} (n+1)}, J(n))$ over the bounded prism $(\AR, [p]_q)$.
	By the universal property of prismatic envelopes, the maps $n_i$ from above extend uniquely to $n_i \colon \AR \rightarrow \AR(n)$, for $1 \leqslant i \leqslant n+1$.

	Let us denote the cosimplicial ring $\AR(\bullet)$ by the usual diagram
	\begin{center}
		\begin{tikzcd}[column sep=8em]
			\AR(\bullet) \colon \hspace{-6em} & \AR(0) \arrow[r, bend left=40, "p_2"] \arrow[r, bend right=40, "p_1"'] & \AR(1) \arrow[l, "\Delta"'] \arrow[r, bend left=40, "p_{23}"] \arrow[r, "p_{13}"] \arrow[r, bend right=40, "p_{12}"'] & \AR(2) \arrow[l, bend left=20] \arrow[l, bend right=20] \cdots,
		\end{tikzcd}
	\end{center}
	where explicitly we write $p_1 = 1_1 \colon \AR(0) \rightarrow \AR(1)$ for the map $a \mapsto a \otimes 1$ and $p_2 = 1_2 \colon \AR(0) \rightarrow \AR(1)$ for the map $a \mapsto 1 \otimes a$.
	Similarly, we define maps $p_{12}$, $p_{13}$ and $p_{23}$.
	To avoid confusion, we shall write $r_j = 2_j \colon \AR \rightarrow \AR(2)$ for the map sending $a$ to the $j^{\textrm{th}}$ position for $j = 1, 2, 3$.
\end{const}

For later calculations it will be convenient to have an explicit presentation of $\AR(1)$ and $\AR(2)$.
Recall that the ideal $J(1) = ([p]_q \otimes 1, 1 \otimes [p]_q, [X_i^{\flat}]^p \otimes 1 - 1 \otimes [X_i^{\flat}]^p \textrm{ for } 1 \leqslant i \leqslant d) \subset \AR \widehat{\otimes}_{O_F} \AR$, is the kernel of the map in \eqref{eq:ker_j} for $n=1$.
Consider the $\delta\algebra$ $(\AR \widehat{\otimes}_{O_F} \AR) \Big\{\frac{1 \otimes [p]_q}{[p]_q \otimes 1}, \frac{[X_1^{\flat}]^p \otimes 1 - 1 \otimes [X_1^{\flat}]^p}{[p]_q \otimes 1}, \ldots, \frac{[X_d^{\flat}]^p \otimes 1 - 1 \otimes [X_d^{\flat}]^p}{[p]_q \otimes 1}\Big\}_{\delta}$ obtained by freely adjoining $\tfrac{J(1)}{[p]_q \otimes 1}$ to $\AR \widehat{\otimes}_{O_F} \AR$.
Then, from \cite[Proposition 3.13]{bhatt-scholze-prisms} we have that
\begin{equation}\label{eq:ar1_explicit}
	\AR(1) = (\AR \widehat{\otimes}_{O_F} \AR) \Big\{\tfrac{1 \otimes [p]_q}{[p]_q \otimes 1}, \tfrac{[X_1^{\flat}]^p \otimes 1 - 1 \otimes [X_1^{\flat}]^p}{[p]_q \otimes 1}, \ldots, \tfrac{[X_d^{\flat}]^p \otimes 1 - 1 \otimes [X_d^{\flat}]^p}{[p]_q \otimes 1}\Big\}_{(p, [p]_q \otimes 1)}^{\wedge},
\end{equation}
i.e.\ $\AR(1)$ is the $(p, [p]_q \otimes 1)\textrm{-adic}$ completion of the $\delta\algebra$ $(\AR \widehat{\otimes}_{O_F} \AR) \Big\{\frac{J(1)}{[p]_q \otimes 1}\Big\}_{\delta}$.
Similarly, using the ideal $J(2) \subset (\AR)^{\widehat{\otimes} 3}$, from \cite[Proposition 3.13]{bhatt-scholze-prisms} we have that
\begin{equation*}
	\AR(2) = (\AR)^{\widehat{\otimes} 3} \Big\{\tfrac{J(2)}{[p]_q \otimes 1}\Big\}_{(p, [p]_q \otimes 1)}^{\wedge},
\end{equation*}
i.e.\ $\AR(2)$ is the $(p, [p]_q \otimes 1)\textrm{-adic}$ completion of the $\delta\algebra$ $(\AR)^{\widehat{\otimes} 3} \Big\{\frac{J(2)}{[p]_q \otimes 1}\Big\}_{\delta}$.

\begin{lem}\label{lem:pi_rj_ff}
	The $\AR\linear$ morphisms $n_i \colon \AR \rightarrow \AR(n)$, for $1 \leqslant i \leqslant n+1$, are faithfully flat.
	In particular, the morphisms $p_i$ for $i = 1, 2$ and $r_j$ for $j = 1, 2, 3$ are faithfully flat.
	Moreover, the element $\frac{n_i([p]_q)}{n_1([p]_q)}$ is a unit in $\AR(n)$, for $1 \leqslant i \leqslant n+1$.
	In particular, $\frac{1 \otimes [p]_q}{[p]_q \otimes 1}$ is a unit in $\AR(1)$.
\end{lem}
\begin{proof}
	Note that from \cite[Proposition 3.13]{bhatt-scholze-prisms}, the map $n_i \colon \AR \rightarrow \AR(n)$ is $(p, [p]_q)\completely$ flat.
	In particular, $n_i \colon \AR/(p, [p]_q)^m \rightarrow \AR(n)/(p, [p]_q \otimes 1)^m$, is flat for each $m \in \NN_{\geqslant 1}$.
	Since $\AR$ is noetherian, from \cite[\href{https://stacks.math.columbia.edu/tag/0912}{Tag 0912}]{stacks-project}, it follows that $n_i$ is flat.
	Moreover, as we have that $\Delta \circ n_i = id$, therefore, for an $\AR\module$ $M$, we see that $\AR(n) \otimes_{n_i, \AR} M = 0$ if and only if $M = 0$.
	Hence, the morphisms $n_i$ for $1 \leqslant i \leqslant n+1$ are faithfully flat.
	Finally, it is easy to see that $n_i([p]_q)$ is a distinguished element of $\AR(n)$ and writing $n_i([p]_q) = n_1([p]_q) \frac{n_i([p]_q)}{n_1([p]_q)}$, from \cite[Lemma 2.24]{bhatt-scholze-prisms}, we get that $\frac{n_i([p]_q)}{n_1([p]_q)}$ is a unit in $\AR(n)$.
\end{proof}

\begin{lem}\label{lem:selfprod_prismenv}
	The pair $(\AR(n), n_1([p]_q))$ is a prism and an object of $(\Spf R)_{\prism}$.
	Moreover, we have that $(\AR(n), n_1([p]_q))$ is the $(n+1)\textrm{-fold}$ self product of $(\AR, [p]_q)$ over the final object of $\Shv((\Spf R)_{\prism})$.
\end{lem}
\begin{proof}
	From \cite[Proposition 3.13]{bhatt-scholze-prisms} note that $(\AR(n), n_1([p]_q))$ is a prism and it is $(p, [p]_q)\completely$ flat over $(\AR, [p]_q)$, in particular, $(\AR(n), n_1([p]_q))$ is a bounded prism by \cite[Lemma 3.7]{bhatt-scholze-prisms}.
	Furthermore, note that we have a natural map $\Rframe \rightarrow \AR(n)/(n_1([p]_q))$, defined by sending $X_i \mapsto p_1([X_i^{\flat}]^p) \textrm{ mod } (n_1([p]_q))$, and by the $p\textrm{-complete}$ \'etaleness of $\Rframe \rightarrow R$, it extends to a natural map $R \rightarrow \AR(n)/(n_1([p]_q))$.
	So we have structure maps $\AR(n) \twoheadrightarrow \AR(n)/(n_1([p]_q)) \leftarrow R$, and it follows that $(\AR(n), n_1([p]_q))$ is an object of $(\Spf R)_{\prism}$.

	Next, we shall show that $(\AR(n), n_1([p]_q))$ is the $(n+1)\textrm{-fold}$ self product of $(\AR, [p]_q)$ over the final object of $\Shv((\Spf R)_{\prism})$.
	Let $\pazc$ denote the category opposite to that of $(\Spf R)_{\prism}$, and it suffices to show that for any prism $(B, K)$ in $(\Spf R)_{\prism}$, there exists a natural bijection of sets
	\begin{equation*}
		\Hom_{\pazc}(\AR(n), B) \isomorphic \textstyle\prod_{i=1}^n \Hom_{\pazc}(\AR, B).
	\end{equation*}
	From Construction \ref{const:cech_nerve}, recall that we have morphisms $n_i \colon \AR \rightarrow \AR(n)$, for each $1 \leqslant i \leqslant n+1$.
	So, given any $f \colon \AR(n) \rightarrow B$ in $\pazc$, precomposing it with $n_i$ yields morphisms $f \circ n_i \colon \AR \rightarrow B$, for each $1 \leqslant i \leqslant n+1$.
	In particular, this yields a natural homomorphism of sets
	\begin{align*}
		h \colon \Hom_{\pazc}(\AR(n), B) &\longrightarrow \textstyle\prod_{i=1}^{n+1} \Hom_{\pazc}(\AR, B)\\
			f &\longmapsto (f \circ n_i).
	\end{align*}

	From Lemma \ref{lem:pi_rj_ff}, recall that each $n_i \colon \AR \rightarrow \AR(n)$ is faithfully flat, therefore injective, and so it follows that $h$ is an injective map.
	To show the surjectvity of $h$, for each $1 \leqslant i \leqslant n+1$, let $f_i \colon \AR \rightarrow B$ denote a homomorphism in $\pazc$.
	Then, we have a natural Frobenius-equivariant homomorphism $\otimes_{i=1}^{n+1} f_i \colon \AR^{\otimes (n+1)} \rightarrow B$ defined via $(x_1, \ldots x_{n+1}) \mapsto \prod_{i=1}^{n+1} f_i(x_i)$, and since $B$ is $(p, K)\textrm{-adically} = (p, [p]_q)\textrm{-adically}$ complete and each $f_i$ is Frobenius-equivariant, we get that $\otimes_{i=1}^{n+1} f_i$ naturally extends to a homomorphism of $\delta\textrm{-algebras}$ $\widehat{\otimes}_{i=1}^{n+1} f_i \colon \AR^{\widehat{\otimes} (n+1)} \rightarrow B$.
	Moreover, note that the commutativity of the following diagram:
	\begin{center}
		\begin{tikzcd}
			\AR^{\widehat{\otimes} (n+1)} \arrow[rr, "\widehat{\otimes}_{i=1}^{n+1} f_i"] \arrow[d, "\eqref{eq:ker_j}", twoheadrightarrow] & & B \arrow[d, twoheadrightarrow]\\
			R\big[\zeta_p, X_1^{1/p}, \ldots, X_d^{1/p}\big]^{\otimes (n+1)} & R \arrow[l] \arrow[r] & B/K,
		\end{tikzcd}
	\end{center}
	shows that $(\widehat{\otimes}_{i=1}^{n+1} f_i)(J(n)) \subset K$.
	Consequently, from the universal property of prismatic envelopes \cite[Proposition 3.13]{bhatt-scholze-prisms}, it follows that the homomorphism $\widehat{\otimes}_{i=1}^{n+1} f_i$ naturally extends to a homomorphism $f \colon \AR(n) \rightarrow B$ in $\pazc$.
	In particular, we get that $h$ induces a bijection of sets, thus allowing us to conclude.
\end{proof}

\begin{rem}
	In the following, we will denote the prism $(\AR(1), [p]_q \otimes 1)$ simply as $(\AR(1), [p]_q)$.
	In the light of Lemma \ref{lem:pi_rj_ff}, this simplification should not cause any confusion to the reader.
\end{rem}

\subsubsection{Galois action on \texorpdfstring{$(\AR, [p]_q)$}{-}}

Note that for $n \in \NN$, the product $\Gamma_R^{\times (n+1)}$ of $n+1$ copies of $\Gamma_R$, naturally acts on the $\delta\ring$ $(\AR)^{\widehat{\otimes} (n+1)}$ and the ring homomorphism $(\AR)^{\widehat{\otimes} (n+1)} \twoheadrightarrow R\big[\zeta_p, X_1^{1/p}, \ldots, X_d^{1/p}\big]^{\otimes (n+1)}$ from \eqref{eq:ker_j} is easily checked to be $\Gamma_R^{\times (n+1)}\equivariant$.
So, by the universal property of prismatic envelopes, the continuous action of $\Gamma_R^{\times (n+1)}$ on $(\AR)^{\widehat{\otimes} (n+1)}$ naturally extends to a continuous action on $\AR(n)$.
Moreover, since the action of $\Gamma_R$ is the identity on $\AR/(\mu)$, we claim the following:
\begin{prop}\label{prop:arn_gammar_triv_modmu}
	The action of the $i^{\textrm{th}}$ component of $\Gamma_R^{\times (n+1)}$ is trivial on $\AR(n)/(n_i(\mu))$.
\end{prop}
\begin{proof}
	From the definition of the action of $\Gamma_R^{\times (n+1)}$ on $(\AR)^{\widehat{\otimes} (n+1)}$, it is easy to see that the $i^{\textrm{th}}$ component acts trivially on $(\AR)^{\widehat{\otimes} (n+1)}/(n_i(\mu))$.
	Then, from Construction \ref{const:cech_nerve}, we note that to get the claim, it is enough to show that for any $x$ in the set of generators $\{n_j([p]_q), n_k([X_s^{\flat}]^p) - n_l([X_s^{\flat}]^p)\}$, for $1 \leqslant s \leqslant d$, $1 \leqslant j, k, l \leqslant d$ and $k \neq l$, of the ideal $J(n) \subset (\AR)^{\widehat{\otimes} (n+1)}$, we have that
	\begin{equation}\label{eq:arn_gammar_triv_modmu}
		(g-1)\delta^m\big(\tfrac{x}{n_1([p]_q)}\big) \in n_i(\mu) \AR(n),
	\end{equation}
	for each $m \in \NN$ and any $g$ in the $i^{\textrm{th}}$ component of $\Gamma_R^{\times (n+1)}$.
	We may reduce this claim further as follows:
	using Lemma \ref{lem:mu_mu0_deltastable} (1), let us first note that $n_i(\mu)\AR(n)$ is a $\delta\textrm{-stable}$ ideal of $\AR(n)$ in the sense of \cite[Example 2.10]{bhatt-scholze-prisms}.
	Then, by using Lemma \ref{lem:gamma_deltam_mu} for $A = \AR(n)$ and $\alpha = n_i(\mu)$, we see that to prove our main claim, it is enough to show \eqref{eq:arn_gammar_triv_modmu} for $m = 0$, i.e.\ $(g-1)(\tfrac{x}{n_1([p]_q)})$ belongs to $ n_i(\mu) \AR(n)$.

	Let us first assume $1 < i \leqslant n+1$ and since the action of $\Gamma_R (= i^{\textrm{th}}\textrm{ component of } \Gamma_R^{\times (n+1)})$, on $\AR(n)$ is continuous, it is enough to check the preceding claim for the topological generators $\{\gamma_1, \ldots, \gamma_d\}$ of $\Gamma_R'$ and each $g$ in $\Gamma_F$.
	Now, take $x = n_k([X_s^{\flat}]^p) - n_l([X_s^{\flat}]^p)$ in $J(n)$.
	Then, the claim is obvious for $k \neq i$ and $l \neq i$.
	So, assume that $k = i$ and $l \neq i$ (the case $k \neq i$ and $l = i$ follows by a similar argument), and note that we have
	\begin{equation*}
		(\gamma_s-1)\big(\tfrac{x}{n_1([p]_q)}\big) = \tfrac{(1+n_i(\mu))^p n_i([X_s^{\flat}]^p) - n_i([X_s^{\flat}]^p)}{n_1([p]_q)} = \tfrac{n_i([p]_q)n_i(\mu)n_i([X_s^{\flat}]^p)}{n_1([p]_q)} \in n_i(\mu) \AR(n),
	\end{equation*}
	where we used the fact that $\tfrac{n_i([p]_q)}{n_1([p]_q)}$ is a unit in $\AR(n)$ (see Lemma \ref{lem:pi_rj_ff}).
	Next, let $x = n_i([p]_q)$, and note that for any $g$ in $\Gamma_F$, we have that
	\begin{equation*}
		(g-1)\big(\tfrac{x}{n_1([p]_q)}\big) = \tfrac{n_i([p]_q)n_i(\mu)y}{n_1([p]_q)} \in n_i(\mu) \AR(n),
	\end{equation*}
	where the first equality follows from Lemma \ref{lem:gamma_deltam_pq}.
	Therefore, by using Lemma \ref{lem:gamma_deltam_mu} for $A = \AR(n)$ and $\alpha = n_i(\mu)$, it follows that we have shown \eqref{eq:arn_gammar_triv_modmu} for all generators of $J(n)$ and $1 < i \leqslant n+1$.

	Let us now assume $i = 1$ and again note that since the action of $\Gamma_R = 1^{\textrm{st}}\textrm{-component of } \Gamma_R^{\times (n+1)}$, on $\AR(n)$ is continuous, therefore, it is enough to check the preceding claim for the topological generators $\{\gamma_1, \ldots, \gamma_d\}$ of $\Gamma_R'$ and each $g$ in $\Gamma_F$.
	An argument similar to above, shows that for all $x$ in $J(n)$ and $1 \leqslant s \leqslant d$, we have that $(\gamma_s-1)(\tfrac{x}{n_1([p]_q)})$ belongs to $n_1(\mu) \AR(n)$.
	Now, take any $g$ in $\Gamma_F$ and $x$ an element of $J(n)$.
	Then, we have that
	\begin{equation*}
		(g-1)\big(\tfrac{x}{n_1([p]_q)}\big) = x \big(\tfrac{1}{g(n_1([p]_q))} - \tfrac{1}{n_1([p]_q)}\big) \in n_i(\mu) \AR(n),
	\end{equation*}
	using Lemma \ref{lem:gamma_deltam_pq}.
	Therefore, by using Lemma \ref{lem:gamma_deltam_mu}, it follows that we have shown \eqref{eq:arn_gammar_triv_modmu} for all generators of $J(n)$ and $i = 1$.
	This allows us to conclude.
\end{proof}

The following observation was used above:
\begin{lem}\label{lem:gamma_deltam_pq}
	Let $g \in \Gamma_F$, then $(g-1)[p]_q$ belongs to $[p]_q\mu\AF$ and $(g-1)(\tfrac{1}{[p]_q})$ belongs to $\tfrac{\mu}{[p]_q} \AF$.
\end{lem}
\begin{proof}
	Let us first observe that $g(\mu) = \chi(g)u_g\mu$, where $\chi$ is the $\padic$ cyclotomic character, and $u_g = 1 + x_g\mu$, for some $x_g \in \AF$, is a unit in $\AF$.
	Then, we have that
	\begin{equation*}
		(g-1)[p]_q = (g-1)\tfrac{\varphi(\mu)}{\mu} = \tfrac{\varphi(g(\mu))}{g(\mu)} - \tfrac{\varphi(\mu)}{\mu} = \big(\tfrac{\varphi(u_g)}{u_g} - 1\big)[p]_q \in [p]_q\mu\AF.
	\end{equation*}
	Next, observe that
	\begin{equation*}
		(g-1)\big(\tfrac{1}{[p]_q}\big) = (g-1)\tfrac{\mu}{\varphi(\mu)} = \tfrac{g(\mu)}{\varphi(g(\mu))} - \tfrac{\mu}{\varphi(\mu)} = \big(\tfrac{u_g}{\varphi(u_g)} - 1\big)\tfrac{1}{[p]_q} \in \tfrac{\mu}{[p]_q}\AF.
	\end{equation*}
	This proves the lemma.
\end{proof}

The following is a useful general statement for $\delta\textrm{-rings}$ admitting an action of $\Gamma_R$ (see Section \ref{subsubsec:gamma0_action} for applications):
\begin{lem}\label{lem:gamma_deltam_mu}
	Let $A$ be a $p\adically$ complete $p\torsionfree$ $\delta\ring$ such that $A$ admits a continuous and $\varphi\equivariant$ action of $\Gamma_R$, and let $\alpha$ be an element of $A$ such that the ideal $(\alpha) \subset A$ is $\delta\textrm{-stable}$ in the sense of \cite[Example 2.10]{bhatt-scholze-prisms}.
	Assume that for some $x$ in $A$ and $g$ an element of $\Gamma_R^{\times(n+1)}$, we are given that $(g-1)x$ is in $(\alpha)$.
	Then, for each $m \in \NN$, we have that $(g-1)\delta^m(x)$ also belongs to $(\alpha)$.
\end{lem}
\begin{proof}
	Let us fix some $x$ in $A$ and $g$ an element of $\Gamma_R^{\times(n+1)}$ such that $(g-1)x$ lies in $(\alpha)$.
	We shall proceed by induction on $m \in \NN$.
	So, for some $m \in \NN$, let us write $(g-1)\delta^m x = \alpha y$.
	Then, observe that
	\begin{equation}\label{eqref:gamma_delatm_mu}
		\begin{aligned}
			(g-1)(\delta^{m+1}x) &= (g-1)\tfrac{\varphi(\delta^m x) - (\delta^m x)^p}{p}\\
			&= \tfrac{1}{p}(\varphi((g-1)\delta^m x) - (g-1)(\delta^m x)^p)\\
			&= \tfrac{1}{p}(\varphi(\alpha y) - ((g-1)\delta^m x + \delta^m x)^p + (\delta^m x)^p)\\
			&= \tfrac{1}{p}(\varphi(\alpha y) - \textstyle \sum_{k=1}^{p}\binom{p}{k} (\alpha y)^k(\delta^m x)^{p-k}) = \delta(\alpha y) - \tfrac{1}{p}\textstyle \sum_{k=1}^{p-1}\binom{p}{k} (\alpha y)^k(\delta^m x)^{p-k}.
		\end{aligned}
	\end{equation}
	Using the product formula for $\delta\textrm{-structures}$ from \eqref{eq:delta_sum_prod}, it follows that $\delta(\alpha y)$ belongs to $(\alpha)$.
	Hence, $(g-1)(\delta^{m+1}x)$ belongs to $(\alpha)$, thus proving the claim.
\end{proof}

\begin{rem}\label{rem:phigamma_act_ar+n}
	For $n \in \NN$, the $\delta\ring$ $(\AR)^{\widehat{\otimes} (n+1)}$ is equipped with a Frobenius endomorphism given as the tensor product of Frobenius on each component.
	By compatibility of $\delta\textrm{-structures}$ between $(\AR)^{\widehat{\otimes} (n+1)}$ and $\AR(n)$, we have that the Frobenius endomorphism on $\AR(n)$ is compatible with the one on $(\AR)^{\widehat{\otimes} (n+1)}$.
	Then, by the description of the Galois action on $\AR(1)$, it follows that the natural map $p_i \colon \AR \rightarrow \AR(1)$ is $(\varphi, \Gamma_R^2)\equivariant$ for $i = 1, 2$, where $\Gamma_R^2$ acts on the source via projection on to the $i^{\textrm{th}}$ coordinate.
	Similarly, the natural map $r_j \colon \AR \rightarrow \AR(2)$ is $(\varphi, \Gamma_R^3)\equivariant$ for $j = 1, 2, 3$, where $\Gamma_R^3$ acts on the source via projection on to the $j^{\textrm{th}}$ coordinate.
\end{rem}

\subsection{The ring \texorpdfstring{$\AR(1)/(p_1(\mu))$}{-}}\label{subsec:ar1p1mu}

From Remark \ref{rem:phigamma_act_ar+n} recall that we have a $(\varphi, \Gamma_R^2)\equivariant$ maps $p_i \colon \AR \rightarrow \AR(1)$ for $i = 1, 2$, where $\AR$ is equipped with a $\Gamma_R^2\action$ via projection on to the $i^{\textrm{th}}$ coordinate.
Moreover, we note that there is an induced action of $\Gamma_R^2$ on $\AR(1)/(p_1(\mu))$, where the action of the first component is trivial.
Additionally, from Section \ref{subsec:important_rings} recall that we have fixed topological generators $\{\gamma_1, \ldots, \gamma_d\}$ of $\Gamma_R'$.
The goal of this section is to give an explicit description of the ring $\AR(1)/(p_1(\mu))$ and explain some of its properties.
We start with some computations.

\subsubsection{Some divided power calculations}

Our first goal is to show the following claim:
\begin{prop}\label{prop:polynom_to_powerseries}
	The following natural map is a $\varphi\equivariant$ isomorphism of $p\torsionfree$ rings
	\begin{equation}\label{eq:polynom_to_powerseries}
		\beta \colon \Lambda_R \coloneq R[\mu, \{(\mu^{p-1}/p)^{[k]}\}_{k \in \NN}]_p^{\wedge} \isomorphic R\llbracket \mu \rrbracket[\{(\mu^{p-1}/p)^{[k]}\}_{k \in \NN}]_p^{\wedge}.
	\end{equation}
\end{prop}
\begin{nota}
	For any element $x$ and $k \in \NN$, we use $x^{[k]}$ to denote $x^k/(k!)$.
\end{nota}
\begin{proof}
	Let $B$ denote the divided power envelope of the $R[\mu]\subalgebra$ $R[\mu, \mu^{p-1}/p] \subset R[\mu][1/p]$, with respect to the ideal $(\mu^{p-1}/p) \subset R[\mu, \mu^{p-1}/p]$, i.e.\ $B = R[\mu, \{(\mu^{p-1}/p)^{[k]}\}_{k \in \NN}]$.
	Similarly, let $B'$ denote the divided power envelope of the $R\llbracket \mu \rrbracket\subalgebra$ $R\llbracket \mu \rrbracket[\mu^{p-1}/p] \subset R\llbracket \mu \rrbracket[1/p]$, with respect to the ideal $(\mu^{p-1}/p) \subset R\llbracket \mu \rrbracket[\mu^{p-1}/p]$, i.e.\ $B' = R\llbracket \mu \rrbracket[\{(\mu^{p-1}/p)^{[k]}\}_{k \in \NN}]$.
	Note that the source and the target of \eqref{eq:polynom_to_powerseries} are the respective $\padic$ completions of $B$ and $B'$.
	Moreover, we have a natural homomorphism of rings $B \rightarrow B'$ induced by the inclusion $R[\mu][1/p] \subset R\llbracket \mu \rrbracket[1/p]$.
	Evidently, the rings $B$ and $B'$ are $p\torsionfree$, the natural map $B \rightarrow B'$ is injective and to show that \eqref{eq:polynom_to_powerseries} is an isomorphism, it is enough to show that the natural map $B \rightarrow B'$ induces an isomorphism of rings $B/p \isomorphic B'/p$.
	Since, $\mu^{p-1} = p(\mu^{p-1}/p)$ in $B$ and $B'$, therefore $B/p = B/(\mu^{p-1}, p) = (B/\mu^{p-1})/p$ and $B'/p = B'/(\mu^{p-1}, p) = (B'/\mu^{p-1})/p$.

	Now, from Lemma \ref{lem:pdring_explicit}, using the explicit description of $B$ and $B'$, we have a natural commutative diagram
	\begin{center}
		\begin{tikzcd}
			\big((R[\mu]/\mu^{p-1})[Y_0, Y_1, \ldots]/(pY_0 - \mu^{p-1}, \{pY_{k+1} - Y_k^p\}_{k \in \NN}))/p \arrow[r, "\sim"] \arrow[d] & B/p \arrow[d]\\
			\big((R\llbracket \mu \rrbracket/\mu^{p-1})[Y_0, Y_1, \ldots]/(pY_0 - \mu^{p-1}, \{pY_{k+1} - Y_k^p\}_{k \in \NN}))/p \arrow[r, "\sim"] & B'/p,
		\end{tikzcd}
	\end{center}
	where the vertical arrows are natural maps, the top horizontal arrow is reduction modulo $p$ of the isomorphism in \eqref{eq:pdringb_explicit} and the bottom horizontal arrow is reduction modulo $p$ of the isomorphism in \eqref{eq:pdringbprime_explicit}.
	Since $R[\mu]/(\mu^{p-1}) \isomorphic R\llbracket \mu \rrbracket/(\mu^{p-1})$, it follows that the left vertical map is an isomorphism as well.
	So, we obtain that $B/p \isomorphic B'/p$ and hence \eqref{eq:polynom_to_powerseries} is an isomorphism of $p\torsionfree$ rings.
	Finally, note that by definition the map $\beta$ in \eqref{eq:polynom_to_powerseries} is $\varphi\equivariant$.
	This concludes our proof.
\end{proof}

The following descriptions of $B$ and $B'$ were used in the proof of Proposition \ref{prop:polynom_to_powerseries}:
\begin{lem}\label{lem:pdring_explicit}
	The following natural map is an isomorphism of rings
	\begin{equation}\label{eq:pdringb_explicit}
		R[\mu][Y_0, Y_1, \ldots]/(pY_0 - \mu^{p-1}, \{pY_{k+1} - Y_k^p\}_{k \in \NN}) \isomorphic B,
	\end{equation}
	where $Y_k \mapsto (\mu^{p-1}/p)^{[p^k]}$.
	Similarly, the following natural map is an isomorphism of rings
	\begin{equation}\label{eq:pdringbprime_explicit}
		R\llbracket \mu \rrbracket[Y_0, Y_1, \ldots]/(pY_0 - \mu^{p-1}, \{pY_{k+1} - Y_k^p\}_{k \in \NN}) \isomorphic B',
	\end{equation}
	where $Y_k \mapsto (\mu^{p-1}/p)^{[p^k]}$.
\end{lem}
\begin{proof}
	We shall only show that the map in \eqref{eq:pdringb_explicit} is bijective, the claim for \eqref{eq:pdringbprime_explicit} follows by a similar argument.
	Define an $R[\mu]\linear$ map 
	\begin{equation}\label{eq:poly_to_b}
		R[\mu][Y_0, Y_1, \ldots] \longrightarrow B, \hspace{3mm} Y_k \longmapsto (\tfrac{\mu^{p-1}}{p})^{[p^k]}.
	\end{equation}
	Let us first note that this map is surjective.
	Indeed, let $n \in \NN$ and write $n = \sum_{i=0}^r n_ip^i$ in base $p$, with $0 \leqslant n_i \leqslant p-1$.
	Then, we have that
	\begin{equation*}
		(\tfrac{\mu^{p-1}}{p})^{[n]} = \tfrac{1}{n!} (\tfrac{\mu^{p-1}}{p})^{\textstyle\sum_{i=0}^r n_ip^i} = \tfrac{1}{n!} \textstyle\prod_{i=0}^r (p^i!)^{n_i} \textstyle\prod_{i=0}^r  (\tfrac{\mu^{p-1}}{p})^{[p^i]}.
	\end{equation*}
	An easy computation shows that the $\padic$ valuation of $\prod_{i=0}^r (p^i!)^{n_i}$ equals the $\padic$ valuation of $n!$.
	So, it follows that \eqref{eq:poly_to_b} is surjective and it is clear that the kernel of \eqref{eq:poly_to_b} is given by the ideal $(pY_0 - \mu^{p-1}, \{pY_{k+1} - Y_k^p\}_{k \in \NN}) \subset R[\mu][Y_0, Y_1, \ldots]$.
	This proves the claim.
\end{proof}

\begin{rem}\label{rem:polynom_powerseries_comp}
	Let $\Lambda_F \coloneq O_F[\mu, \{(\mu^{p-1}/p)^{[k]}\}_{k \in \NN}]_p^{\wedge} \isomorphic O_F\llbracket \mu \rrbracket[\{(\mu^{p-1}/p)^{[k]}\}_{k \in \NN}]_p^{\wedge}$ and from Proposition \ref{prop:polynom_to_powerseries} note that it is a $p\torsionfree$ algebra over $\AF = O_F\llbracket \mu \rrbracket$.
	We equip $\Lambda_F$ with a continuous $(\varphi, \Gamma_F)\action$ by extending the $(\varphi, \Gamma_F)\action$ on $\AF$, and using the formulas $\varphi(\mu^{p-1}) = [p]_q^{p-1} \mu^{p-1}$ and $g(\mu^{p-1}) = \chi(g)^{p-1}\mu^{p-1} u$, for some unit $u$ in $\AF$ depending on $g$.
	Then, using the completed tensor product description, we get that the isomorphism \eqref{eq:polynom_to_powerseries} in Proposition \ref{prop:polynom_to_powerseries} is a $(\varphi, \Gamma_F)\equivariant$ isomorphism of rings
	\begin{equation}\label{eq:polynom_powerseries_comp}
		\beta \colon \Lambda_R \coloneq R \wotimes_{O_F} \Lambda_F \isomorphic R\llbracket \mu \rrbracket\wotimes_{O_F\llbracket \mu \rrbracket} \Lambda_F.
	\end{equation}
	Moreover, recall that from Section \ref{subsubsec:ring_ar+}, we have a $(\varphi, \Gamma_F)\equivariant$ isomorphism of rings $\iota \colon R\llbracket \mu \rrbracket \isomorphic \AR$.
	This extends to a $(\varphi, \Gamma_F)\equivariant$ isomorphism
	\begin{equation}\label{eq:pd_iota}
		\iota \colon R\llbracket \mu \rrbracket \wotimes_{O_F\llbracket \mu \rrbracket} \Lambda_F \isomorphic \AR \wotimes_{\AF} \Lambda_F,
	\end{equation}
	where we take the tensor product action of $\varphi$ and $\Gamma_F$ and completion is with respect to the $\padic$ topology.
\end{rem}

For any $k \in \NN$, let $k = (p-1) f(k) + r(k)$, with $r(k)$, $f(k)$ in $\NN$ and $0 \leqslant r(k) < p-1$.
Set $\mu^{\{k\}} = \frac{\mu^k}{f(k)!p^{f(k)}}$, and note that we have $\Lambda_F = O_F[\mu, \{(\mu^{p-1}/p)^{[k]}\}_{k \in \NN}]_p^{\wedge} = O_F[\mu, \{\mu^{\{k\}}\}_{k \in \NN}\big]_p^{\wedge}$.
In particular, for any $x$ in $\Lambda_F$ we have a unique presentation $x = \sum_{k \in \NN} a_k \mu^{\{k\}}$ in $\Lambda_F$, with $a_k$ in $O_F$ for all $k \in \NN$, and $p\textrm{-adically}$ $a_k \rightarrow 0$ as $k \rightarrow +\infty$.
We note an important observation which will be used throughout.
\begin{lem}\label{lem:tovermu_unit}
	The element $t = \log(1+\mu) = \sum_{k \in \NN} (-1)^k \frac{\mu^{k+1}}{k+1}$ converges in $\mu \Lambda_F$.
	Moreover, the elements $t/\mu$, $[p]_q/p$ and $\ptilde/p$ are units in $\Lambda_F$.
\end{lem}
\begin{proof}
	Formally, let us write $t/\mu = \sum_{k \in \NN} (-1)^k \frac{\mu^k}{k+1} = \sum_{k \in \NN} a_k \mu^k$ and $\mu/t = \frac{\mu}{\log(1+\mu)} = \sum_{k \in \NN} b_k \mu^k$, with $\upsilon_p(a_k), \upsilon_p(b_k) \geqslant -\frac{k}{p-1}$, for all $k \in \NN$.
	From the discussion before the lemma, recall that $\mu^{\{k\}} = \frac{\mu^k}{f(k)!p^{f(k)}}$.
	Then,
	\begin{equation*}
		\upsilon_p\big(\tfrac{a_k}{f(k)!p^{f(k)}}\big) \geqslant -\tfrac{k}{p-1} + \upsilon_p(f(k)!) + f(k) > \upsilon_p\big(\lfloor \tfrac{k}{p-1} \rfloor!\big) - 1,
	\end{equation*}
	where the last term goes to $+\infty$ as $k \rightarrow +\infty$.
	Similar claim holds with $a_k$ replaced by $b_k$ above.
	Hence, we conclude that $t/\mu$ and $\mu/t$ converge in $\Lambda_F$ as inverse to each other.
	Moreover, we have $[p]_q = \varphi(\mu)/\mu = p\varphi(\mu/t)t/\mu$ in $\Lambda_F$, so $[p]_q/p$ is a unit in $\Lambda_F$.
	Furthermore, using Lemma \ref{lem:pq_mu0p_unit}, we see that $\ptilde/p$ is also a unit in $\Lambda_F$.
\end{proof}

\subsubsection{Explicit description of \texorpdfstring{$\AR(1)/(p_1(\mu))$}{-}}

In this section, we consider the $\AR\algebra$ structure on $\AR(1)/(p_1(\mu))$ via the composition $\AR \xrightarrow{p_2} \AR(1) \rightarrow \AR(1)/(p_1(\mu))$ and the $\Gamma_R\action$ on $\AR(1)/(p_1(\mu))$ means the continuous (for the $\padic$ topology) action of $1 \times \Gamma_R$ on $\AR(1)/(p_1(\mu))$.
Let us consider the $p\textrm{-adically}$ completed divided power $\AR\algebra$ $\AR\big[\textstyle \prod_{\smbfk \in \NN^{d+1}} (\mu^{p-1}/p)^{[k_0]} T_1^{[k_1]} \cdots T_d^{[k_d]}\big]_p^{\wedge}$, where $T_1, \ldots T_d$ are variables.
We equip this ring with a continuous (for the $\padic$ topology) action of $\Gamma_R$ by extending the $\Gamma_R\action$ on $\AR$ by setting $\gamma_j(T_i) = \mu[X_i^{\flat}] + T_i$, for $j = i$, or $T_i$, otherwise.
Moreover, we have $\varphi(\mu^{p-1}) = [p]_q^{p-1} \mu^{p-1}$, so we may further equip this ring with a Frobenius endomorphism $\varphi$, extending the Frobenius on $\AR$ by setting $\varphi([X_i^{\flat}] - T_i) = ([X_i^{\flat}] - T_i)^p$.
Then, we claim the following:
\begin{prop}\label{prop:a1modp1mu_pdring}
	There exists a natural $(\varphi, \Gamma_R)\equivariant$ isomorphism of $p\adically$ complete and $p\torsionfree$ $\AR\algebras$ 
	\begin{equation}\label{eq:a1modp1mu_pdring}
		\alpha \colon \AR(1)/(p_1(\mu)) \isomorphic \AR\big[\textstyle \prod_{\smbfk \in \NN^{d+1}} (\mu^{p-1}/p)^{[k_0]} T_1^{[k_1]} \cdots T_d^{[k_d]}\big]_p^{\wedge},
	\end{equation}
	where $p_2(\mu) \mapsto \mu$, $p_2([X_i^{\flat}]) \mapsto [X_i^{\flat}]$ and $p_1([X_i^{\flat}]) \mapsto [X_i^{\flat}] - T_i$.
\end{prop}
\begin{proof}
	To construct the map in \eqref{eq:a1modp1mu_pdring} and to show that it is bijective, we shall use the results of Appendix \ref{app_sec:delta_rings_pd_alg}.
	Let $A \coloneq \AF \wotimes_{O_F} \AF$ denote the $(p, p_1(\mu))\adic$ completion of $\AF \otimes_{O_F} \AF$ equipped with the tensor product Frobenius, in particular, $A$ is a $p\torsionfree$ $\delta\ring$.
	Moreover, note that $\varphi^n(p_1(\mu))$ is a nonzerodivisor on $A$, for each $n \in \NN$, and $A/p_1(\mu)A \isomorphic \AF$ is $p\torsionfree$.
	Therefore, by setting $q$ to be $p_1(\mu) + 1$, we see that the ring $A$ satisfies Assumption \ref{app_assum:base_a}.
	Next, let $B \coloneq \AR \wotimes_{O_F} \AR$ denote the $(p, p_1(\mu))\adic$ completion of $\AR \otimes_{O_F} \AR$ equipped with the tensor product Frobenius, in particular, $B$ is a $p\torsionfree$ $\delta\algebra$ over $A$.
	Note that $\varphi^n(p_1(\mu))$ is a nonzerodivisor on $B$, for each $n \in \NN$, and $B/p_1(\mu)B \isomorphic R \wotimes_{O_F} \AR$, as the $\padic$ completion of $R \otimes_{O_F} \AR$, is $p\torsionfree$.
	Set $Y_0 \coloneq p_2([p]_q) - p_1([p]_q)$ and $Y_i \coloneq p_2([X_i^{\flat}]^p) - p_1([X_i^{\flat}]^p)$, for each $1 \leqslant i \leqslant d$, as elements in $B$.
	Then, it is clear that the sequence $\{Y_1, \ldots, Y_d\}$ is regular on $B$ and $B/pB$.
	Moreover, from the discussion in Construction \ref{const:cech_nerve}, recall that the sequence $\{p, p_1([p]_q), p_2([p]_q), Y_1, \ldots, Y_d\}$ is regular on $B$ and since $[p]_q = \mu^{p-1} \textrm{ mod } p$, therefore, it follows that the sequence $\{p, p_1(\mu), p_2([p]_q), Y_1, \ldots, Y_d\}$ is also regular on $B$.
	In particular, we see that $\{Y_0, \ldots, Y_d\}$ is regular on $B/p_1([p]_q)B$ and on $B/(p, p_1(\mu))B$.
	Therefore, we obtain that the $A\algebra$ $B$ satisfies Assumption \ref{app_assum:alg_b}.

	Now, let $I$ denote the set of natural numbers $\{0, \ldots, d\}$ and for each $i$ in $I$, let $y_i \coloneq Y_i/p_1([p]_q)$ be an element of $B[1/p_1([p]_q)]$.
	Set $E_0$ to be the $B\subalgebra$ of $B[1/p_1([p]_q)]$ generated by $\{y_i\}_{i \in I}$ (see the discussion before Lemma \ref{app_lem:e0_desc}), and $E$ to be the $\delta\subalgebra$ of $E_0[1/p, \{1/\varphi^n(p_1([p]_q))\}_{n \in \NN}]$ generated over $B$ by $\{y_i\}_{i \in I}$ (see the discussion before \eqref{app_eq:e_to_ebar}).
	From the discussion before \eqref{eq:ar1_explicit}, note that the ring $\AR(1)$ is given as the $(p, p_1([p]_q))\adic = (p, p_1(\mu))\adic$ completion of $E$, therefore, the ring $\AR(1)/(p_1(\mu))$ is given as the $\padic$ completion of $E/p_1(\mu)E$.

	Next, let $\Abar \coloneq A/\mu A$ and $\Bbar \coloneq B/\mu B$ equipped with the induced Frobenius, in particular, an associated $\delta\textrm{-structure}$, since $\Abar$ and $\Bbar$ are $p\torsionfree$ (see Remark \ref{rem:frob_delta_ring}).
	Let $\Ezerobar \coloneq E_0/\mu E_0$ and for each $i$ in $I$, denote by $\Yibar$ (resp.\ $\yibar$), the image of $Y_i$ (resp.\ $y_i$) in $\Ezerobar$.
	Define $\Ebar$ to be the $\delta\subalgebra$ of $\Ezerobar[1/p]$ generated by $\{\yibar\}_{i \in I}$ over $\Bbar$ (see the discussion before \eqref{app_eq:e_to_ebar}).
	Then, from \eqref{app_eq:e_to_ebar}, we have a natural surjective map $E \rightarrow \Ebar$ inducing an isomorphism of $p\torsionfree$ $\delta\textrm{-rings}$ $E/\mu E \isomorphic \Ebar$, by Proposition \ref{app_prop:dbar_ebar_iso}.

	Now, for each $1 \leqslant i \leqslant d$, set $\tau_i \coloneq p_2([X_i^{\flat}]) - p_1([X_i^{\flat}])$ as an element of $\Ezerobar$.
	Moreover, note that $p_2(\mu)^{p-1} = \overline{Y}_0 + p a$, for some $a$ in $p_2(\mu)\Ezerobar$, so we set $\tau_0 \coloneq p_2(\mu)^{p-1}/p = \overline{y}_0 + a$ as an element of $\Ezerobar$.
	Let us set $\Ebar' \coloneq \Ezerobar[\{\tau_i^{[n]}\}_{i \in I, n \in \NN}]$ as a divided power $\Ezerobar\subalgebra$ of $\Ezerobar[1/p]$.
	We claim that $\Ebar = \Ebar'$ in $\Ezerobar[1/p]$.
	Indeed, to show the inclusion $\Ebar' \subset \Ebar$, note that for each $1 \leqslant i \leqslant d$, we have $\tau_i^p = \Yibar - p\delta(\tau_i)$, where $\delta(\tau_i)$ is in $\Ezerobar$ by Lemma \ref{lem:delta_yibar}, in particular, $\tau_i^p/p = \yibar - \delta(\tau_i)$ is in $\Ezerobar \subset \Ebar$.
	Moreover, using the equalities above we have that $\tau_0^p = \overline{y}_0^p + pb$, for some $b$ in $\Ezerobar$ since $p_2(\mu)^p$ is in $p\Ezerobar$, and by using Lemma \ref{lem:delta_yibar} we have that $\overline{y}_0^p = \varphi(\overline{y}_0) - p\delta(\overline{y}_0)$ is in $p\Ebar$.
	Therefore, we get that $\tau_0^p/p = b + \overline{y}_0^p/p$ is in $\Ebar$.
	Hence, by using \cite[Lemma 2.35]{bhatt-scholze-prisms}, it follows that $\Ebar' \subset \Ebar$.

	Next, to show the inclusion $\Ebar \subset \Ebar'$, note that since we have $\yibar$ in $\Ebar'$, for each $i \in I$, therefore, it is enough to show that $\varphi$ on $\Ebar[1/p]$ preserves $\Ebar'$ and the resulting endomorphism of $\Ebar'$ gives a $\delta\textrm{-structure}$, or equivalently, that $\varphi$ restricts to a lift of Frobenius on $\Ebar'$.
	Note that $\varphi(\tau_0^{[n]}) = p_2([p]_q)^n\tau_0^{[n]}$ is in $p\Ebar'$, and for each $i \in I$, we may write
	\begin{equation*}
		\varphi(\tau_i^{[n]}) = \varphi(\tfrac{\tau_i^n}{n!}) = \tfrac{(\tau_i^p + p\delta(\tau_i))^n}{n!} = \textstyle \sum_{k=0}^n \tfrac{1}{n!} \binom{n}{k} \tau_i^{pk} p^{n-k} \delta(\tau_i)^{n-k} = \textstyle \sum_{k=0}^n \tfrac{(pk)! p^{n-k}}{n!} \binom{n}{k} \tau_i^{[pk]} \delta(\tau_i)^{n-k}.
	\end{equation*}
	An easy computation shows that the coefficient of $\tau_i^{[pk]} \delta(\tau_i)^{n-k}$, for $0 \leqslant k \leqslant n$, in the last term of the equation above, lie in $p\ZZ_p$.
	Therefore, it follows that $\varphi$ preserves $\Ebar'$ and $\varphi(\tau_i^{[n]}) = 0 \textrm{ mod } p\Ebar'$, for each $i \in I$ and $n \in \NN_{\geqslant 1}$.
	Moreover, note that $(\tau_i^{[n]})^p$ is also in $p\Ebar'$, for each $n \in \NN_{\geqslant 1}$.
	So, it follows that the endomorphism $\varphi$ on $\Ebar'$ reduces modulo $p$ to the absolute Frobenius on all generators, hence on all elements, of $\Ebar'$.
	Therefore, we obtain that $\Ebar' = \Ebar \subset \Ezerobar[1/p]$.

	Finally, let us consider $\Ebar$ as an $\AR\algebra$ via the map $p_2$ and define a ring homomorphism of $\AR\algebras$
	\begin{equation*}
		\Ebar = \Ebar' \longrightarrow \AR\big[\textstyle \prod_{\smbfk \in \NN^{d+1}} (\mu^{p-1}/p)^{[k_0]} T_1^{[k_1]} \cdots T_d^{[k_d]}\big],
	\end{equation*}
	by sending $p_2(\mu) \mapsto \mu$, $\tau_0^n/n! \mapsto (\mu^{p-1}/p)^n/n!$, $p_2([X_i^{\flat}]) \mapsto [X_i^{\flat}]$, $p_1([X_i^{\flat}]) \mapsto [X_i^{\flat}] - T_i$ and $\tau_i^n/n! \mapsto T_i^n/n!$.
	Evidently, the map above is bijective and passing to its $\padic$ completion gives the isomorphism in \eqref{eq:a1modp1mu_pdring}.
	It is straightforward to check that the map \eqref{eq:a1modp1mu_pdring} thus obtained is $(\varphi, \Gamma_R)\equivariant$.
	This finishes our proof.
\end{proof}

The following observation was used above:
\begin{lem}\label{lem:delta_yibar}
	For each $i \in I$, the element $\varphi(\yibar)$ is in $p\Ezerobar$, and for each $1 \leqslant i \leqslant d$, the element $\delta(\tau_i)$ is in $\Ezerobar$.
\end{lem}
\begin{proof}
	For $1 \leqslant i \leqslant d$, we have that
	\begin{align*}
		\varphi(\yibar) &= \tfrac{1}{p}(\varphi(p_2([X_i^{\flat}])^p - p_1([X_i^{\flat}])^p))\\
		&= \tfrac{1}{p}\big((p_2([X_i^{\flat}])^p - p_1([X_i^{\flat}])^p + p_1([X_i^{\flat}])^p)^p - p_1([X_i^{\flat}])^{p^2}\big)\\
		&= \tfrac{1}{p}\big((p_2([X_i^{\flat}])^p - p_1([X_i^{\flat}])^p)^p\big) + \textstyle \sum_{k=1}^{p-1} \tfrac{1}{p} \binom{p}{k} (p_2([X_i^{\flat}])^p - p_1([X_i^{\flat}])^p)^k p_1([X_i^{\flat}])^{p-k}\\
		&= p^{p-1}\yibar^p + \textstyle \sum_{k=1}^{p-1} p^{k-1} \binom{p}{k} \yibar^k p_1([X_i^{\flat}])^{p-k},
	\end{align*}
	is an element of $p\Ezerobar$.
	Moreover, for $i=0$, we have that
	\begin{align*}
		\varphi(\overline{y}_0) = \tfrac{p_2(q)^{p^2} - 1}{p(p_2(q)^p - 1)} - 1 &= \tfrac{(p_2(q)^p-1+1)^p - 1}{p(p_2(q)^p - 1)} - 1\\
		&= \tfrac{(p_2(q)^p - 1)^{p-1}}{p} + \textstyle\sum_{k=2}^{p-1} \tfrac{1}{p} \binom{p}{k} (p_2(q)^p-1)^{k-1}\\
		&= p^{p-2}(p_2(q)-1)^{p-1}(\overline{y}_0 + 1)^{p-1} + \textstyle\sum_{k=2}^{p-1} p^{k-1} \binom{p}{k} (p_2(q)-1)^{k-1}(\overline{y}_0 + 1)^{k-1}.
	\end{align*}
	For $p \geqslant 3$, it is clear that $\varphi(\overline{y}_0)$ is in $p\Ezerobar$, whereas for $p=2$, note that $p_2(q)-1 = p_2(q) + 1 - 2 = p_2([p]_q) - 2 = 2\overline{y}_0$, i.e.\ $\varphi(\overline{y}_0)$ is in $p\Ezerobar$.
	This shows the first claim.
	Next, for $1 \leqslant i \leqslant d$, note that
	\begin{align*}
		\delta(\tau_i) = \delta(p_2([X_i^{\flat}]) - p_1([X_i^{\flat}])) &= \tfrac{1}{p}\big(p_2([X_i^{\flat}])^p - p_1([X_i^{\flat}])^p - (p_2([X_i^{\flat}]) - p_1([X_i^{\flat}]))^p\big)\\
		&= \tfrac{-1}{p}\big(p_1([X_i^{\flat}])^p + (-p_1([X_i^{\flat}]))^p\big) - \textstyle\sum_{k=1}^{p-1} \tfrac{1}{p} \binom{p}{k} p_2([X_i^{\flat}])^k (-p_1([X_i^{\flat}]))^{p-k} ,
	\end{align*}
	is evidently in $\Ezerobar$.
	This proves the second claim.
\end{proof}

\begin{rem}\label{rem:a1modp1mu_lambda}
	In this remark, we will use $\wotimes$ to denote the $\padic$ completion of the usual tensor product.
	Let $\Lambda_F = O_F[\mu, \{(\mu^{p-1}/p)^{[k]}\}_{k \in \NN}]_p^{\wedge}$ and note that it is a $p\torsionfree$ algebra over $\AF = O_F\llbracket \mu \rrbracket$ (see Proposition \ref{prop:polynom_to_powerseries} for $R = O_F$).
	Note that the target ring in the isomorphism \eqref{eq:a1modp1mu_pdring} of Proposition \ref{prop:a1modp1mu_pdring} is $p\torsionfree$ and by definition the ring $\AR\big[\textstyle \prod_{\smbfk \in \NN^d} T_1^{[k_1]} \cdots T_d^{[k_d]}\big]$ is $p\torsionfree$ as well.
	Therefore, by checking modulo $p$, it is easy to see that we have a natural $(\varphi, \Gamma_R)\equivariant$ isomorphism of rings
	\begin{equation}\label{eq:pdring_lambda_relation}
		\AR\big[\textstyle \prod_{\smbfk \in \NN^d} T_1^{[k_1]} \cdots T_d^{[k_d]}\big] \wotimes_{\AF} \Lambda_F \isomorphic \AR\big[\textstyle \prod_{\smbfk \in \NN^{d+1}} (\mu^{p-1}/p)^{[k_0]} T_1^{[k_1]} \cdots T_d^{[k_d]}\big]_p^{\wedge},
	\end{equation}
	where we take the diagonal action of $\varphi$ and $\Gamma_R$ on the source.
	In particular, using \eqref{eq:a1modp1mu_pdring} and \eqref{eq:pdring_lambda_relation} it follows that $\AR(1)/(p_1(\mu))$ is the $\padic$ completion of a PD-polynomial algebra over $\AR \wotimes_{\AF} \Lambda_F$.

	Next, by an argument similar to \eqref{eq:pdring_lambda_relation}, we note that the isomorphism in \eqref{eq:pd_iota} readily extends to a $(\varphi, \Gamma_F)\equivariant$ isomorphism 
	\begin{equation*}
		\iota \colon R\llbracket \mu \rrbracket\big[\textstyle \prod_{\smbfk \in \NN^d} T_1^{[k_1]} \cdots T_d^{[k_d]}\big] \wotimes_{O_F\llbracket \mu \rrbracket} \Lambda_F \isomorphic \AR\big[\textstyle \prod_{\smbfk \in \NN^d} T_1^{[k_1]} \cdots T_d^{[k_d]}\big] \wotimes_{\AF} \Lambda_F.
	\end{equation*}
	Furthermore, again arguing similar to \eqref{eq:pdring_lambda_relation}, we note that the isomorphism in \eqref{eq:polynom_powerseries_comp} extends to a $(\varphi, \Gamma_F)\equivariant$ isomorphism
	\begin{equation*}
		\beta \colon R\big[\textstyle \prod_{\smbfk \in \NN^d} T_1^{[k_1]} \cdots T_d^{[k_d]}\big] \wotimes_{O_F} \Lambda_F \isomorphic R\llbracket \mu \rrbracket\big[\textstyle \prod_{\smbfk \in \NN^d} T_1^{[k_1]} \cdots T_d^{[k_d]}\big] \wotimes_{O_F\llbracket \mu \rrbracket} \Lambda_F.
	\end{equation*}
	Composing these with the isomorphism in Proposition \ref{prop:a1modp1mu_pdring}, we obtain a $(\varphi, \Gamma_F)\equivariant$ isomorphism
	\begin{equation}\label{eq:a1modp1mu_lambda}
		\beta^{-1} \circ \iota^{-1} \circ \alpha \colon \AR(1)/(p_1(\mu)) \isomorphic R\big[\textstyle \prod_{\smbfk \in \NN^d} T_1^{[k_1]} \cdots T_d^{[k_d]}\big] \wotimes_{O_F} \Lambda_F \eqcolon C_R.
	\end{equation}

	Let $P_R \coloneq R\big[\textstyle \prod_{\smbfk \in \NN^d} T_1^{[k_1]} \cdots T_d^{[k_d]}\big]_p^{\wedge}$ and note that the target ring of the isomorphism \eqref{eq:a1modp1mu_lambda} may also be written as $C_R = P_R \wotimes_R \Lambda_R = P_R[\mu, \{(\mu^{p-1}/p)^{[k]}\}_{k \in \NN}]_p^{\wedge} = \Lambda_R\big[\textstyle \prod_{\smbfk \in \NN^d} T_1^{[k_1]} \cdots T_d^{[k_d]}\big]_p^{\wedge}$.
	In particular, $C_R$ is the $\padic$ completion of a PD-polynomial algebra over $\Lambda_R$.
	Now for $k \in \NN$, let $k = (p-1) f(k) + r(k)$, with $r(k)$, $f(k)$ in $\NN$ and $0 \leqslant r(k) < p-1$.
	Set $\mu^{\{k\}} = \frac{\mu^k}{f(k)!p^{f(k)}}$, then we have $C_R = P_R[\mu, \{(\mu^{p-1}/p)^{[k]}\}_{k \in \NN}]_p^{\wedge} = P_R[\mu, \{\mu^{\{k\}}\}_{k \in \NN}\big]_p^{\wedge}$.
	In particular, from the isomorphism \eqref{eq:a1modp1mu_lambda}, it follows that for any $x$ in $\AR(1)/(p_1(\mu))$ we have a unique presentation $(\beta^{-1} \circ \iota^{-1} \circ \alpha)(x) = \sum_{k \in \NN} a_k \mu^{\{k\}}$ in $C_R$, with $a_k$ in $P_R$ for all $k \in \NN$, and since $\AR(1)/(p_1(\mu)) \isomorphic C_R$ are $p\adically$ separated, we have that $a_k \rightarrow 0$ as $k \rightarrow +\infty$.
\end{rem}

\begin{rem}\label{rem:a1modp1mu_oacrys_embed}
	The target ring of the isomorphism $\alpha$ in Proposition \ref{prop:a1modp1mu_pdring} admits a natural $\Gamma_R\equivariant$ injective homomorphism into the crystalline period ring $\OAcrys(\Rinfty)$ from \cite[Chapitre 6]{brinon-relatif}.
	Indeed, first of all note that the injective homomorphism $\AR \hookrightarrow \Ainf(\Rinfty)$ from Section \ref{subsec:important_rings} extends to a natural $(\varphi, \Gamma_R)\equivariant$ injective homomorphism $\AR \wotimes_{\AF} \Lambda_F \hookrightarrow \Ainf(\Rinfty) \wotimes_{\AF} \Lambda_F \isomorphic \Acrys(\Rinfty)$, where the latter isomorphism follows in a manner similar to the proof of \cite[Proposition 6.2.14]{brinon-relatif}.
	Then, from the isomorphism \eqref{eq:a1modp1mu_pdring} in Proposition \ref{prop:a1modp1mu_pdring} and from the description of $\OAcrys(\Rinfty)$ as the $\padic$ completion of a PD-polynomial algebra over $\Acrys(\Rinfty)$, i.e.\ $\OAcrys(\Rinfty) \isomorphic \Acrys(\Rinfty)\big[\textstyle \prod_{\smbfk \in \NN^d} T_1^{[k_1]} \cdots T_d^{[k_d]}\big]_p^{\wedge}$ (see \cite[Chapitre 6]{brinon-relatif} and \cite[Section 2.2]{abhinandan-relative-wach-i}), it follows that we have a $(\varphi, \Gamma_R)\equivariant$ injective homomorphism $\AR(1)/(p_1(\mu)) \hookrightarrow \OAcrys(\Rinfty)$.
\end{rem}

\begin{rem}\label{rem:a1modp1mu_ralg}
	Recall that there are two $\AR\algebra$ structures on $\AR(1)$ written as $p_i \colon \AR \rightarrow \AR(1)$ for $i = 1, 2$ (see Construction \ref{const:cech_nerve}).
	Composing it with the injective map $\iota \colon R \rightarrow \AR$ (see Section \ref{subsubsec:ring_ar+}), gives morphism of rings $R \xrightarrow{\hspace{1mm} \iota \hspace{1mm}} \AR \xrightarrow{\hspace{1mm} p_i \hspace{1mm}} \AR(1)$.
	Reducing $p_1$ modulo $\mu$ induces an $R\algebra$ structure on $\AR(1)/(p_1(\mu))$ via the composition $R \isomorphic \AR/(\mu) \xrightarrow{\hspace{1mm} p_1 \hspace{1mm}} \AR(1)/(p_1(\mu))$, which we again denote by $p_1$ and where the first isomorphism is induced by the map $\iota$.
	In particular, we have $(p_1 \circ \alpha)(X_i) = [X_i^{\flat}] - T_i$, for all $1 \leqslant i \leqslant d$.
	Moreover, we will denote the composition $R \xrightarrow{\hspace{1mm} \iota \hspace{1mm}} \AR \xrightarrow{\hspace{1mm} p_2 \hspace{1mm}} \AR(1)$ just by $p_2$.
	Note that both $p_1$ and $p_2$ described above are $\varphi\equivariant$.
\end{rem}

\begin{lem}\label{lem:a1modp1mu_p1p2_ff}
	The morphisms $p_i \colon R \rightarrow \AR(1)/(p_1(\mu))$ for $i = 1, 2$ are faithfully flat.
\end{lem}
\begin{proof}
	From Lemma \ref{lem:pi_rj_ff} the map $p_1 \colon \AR \rightarrow \AR(1)$ is faithfully flat, therefore the map $p_1 \colon R \isomorphic \AR/(\mu) \rightarrow \AR(1)/(p_1(\mu))$ is faithfully flat.
	Now recall that we have $\Aframe^+ \isomorphic \Rframe\llbracket \mu \rrbracket$ (see Section \ref{subsubsec:ring_ar+}) and consider the following automorphism:
	\begin{equation*}
		f \colon \Aframe^+\big[\textstyle \prod_{\smbfk \in \NN^{d+1}} (\mu^{p-1}/p)^{[k_0]} T_1^{[k_1]} \cdots T_d^{[k_d]}\big]_p^{\wedge} \isomorphic \Aframe^+\big[\textstyle \prod_{\smbfk \in \NN^{d+1}} (\mu^{p-1}/p)^{[k_0]} T_1^{[k_1]} \cdots T_d^{[k_d]}\big]_p^{\wedge},
	\end{equation*}
	sending $[X_i]^{\flat} \mapsto [X_i^{\flat}] + T_i$, $T_i \mapsto T_i$ and $\mu \mapsto \mu$.
	Extending this isomorphism along the $p\textrm{-adically}$ completed \'etale map $\Rframe \rightarrow R$ (see Section \ref{subsubsec:ring_ar+}) gives an automorphism $f$ of the $p\textrm{-adically}$ complete PD-algebra $\AR\big[\textstyle \prod_{\smbfk \in \NN^{d+1}} (\mu^{p-1}/p)^{[k_0]} T_1^{[k_1]} \cdots T_d^{[k_d]}\big]_p^{\wedge}$, and we consider the following commutative diagram:
	\begin{equation}\label{eq:a1modp1mu_p1p2_ff}
		\begin{tikzcd}
			R \arrow[r, "p_1"] \arrow[rd, "p_2"'] & \AR(1)/(p_1(\mu)) \arrow[d, "\wr"', "\alpha^{-1} \circ f \circ \alpha"]\\
			& \AR(1)/(p_1(\mu)).
		\end{tikzcd}
	\end{equation}
	Since $p_1$ is faithfully flat, therefore, from the diagram \eqref{eq:a1modp1mu_p1p2_ff} it follows that $p_2$ is faithfully flat as well.
\end{proof}

\subsubsection{\texorpdfstring{$\qconnection$ on $\AR(1)/(p_1(\mu))$}{-}}\label{subsubsec:a1modp1mu_qconnection}

In this section, we will interpret the action of the geometric part of $\Gamma_R$, i.e.\ $\Gamma_R'$ as a $\qconnection$ on $\AR(1)/(p_1(\mu))$.
We start by recalling the definition of a $\qconnection$ from \cite[Definition 2.1]{morrow-tsuji} and \cite[Section 5.1]{abhinandan-relative-wach-ii}.

Let $D$ be commutative ring and consider a $D\algebra$ $A$ equipped with $d$ commuting $D\algebra$ automorphisms $\gamma_1 \ldots, \gamma_d$, i.e.\ an action of $\ZZ^d$.
Moreover, fix an element $q$ in $D$ such that $q-1$ is a nonzerodivisor of $D$ and $\gamma_i = 1 \textmod (q-1)A$, for all $1 \leqslant i \leqslant d$.

\begin{defi}[$q\textrm{-de Rham complex}$]\label{defi:qdeRham_complex}
	Let $q\Omega^{\bullet}_{A/D} \coloneq \bigoplus_{k=0}^d q\Omega^k_{A/D}$ be a differential graded $D\algebra$ defined as follows:
	\begin{itemize}
		\item $q\Omega^0_{A/D} = A$ and $q\Omega^1_{A/D}$ is a free left $A\module$ on formal basis elements $\dlog(U_i)$.

		\item The right $A\module$ structure on $q\Omega^1_{A/D}$ is twisted by the rule $\dlog(U_i) \cdot f = \gamma_i(f) \dlog(U_i)$.

		\item $\dlog(U_i) \dlog(U_j) = - \dlog(U_j) \dlog(U_i)$, if $i \neq j$, and $0$, if $i=j$.

		\item The following map, where $I_k = \{\smbfi = (i_1, \ldots, i_k) \in \NN^k \textrm{ such that } 1 \leqslant i_1 < \cdots < i_k \leqslant d\}$, is an isomorphism of $A\modules$
			\begin{align*}
				\textstyle \bigoplus_{\smbfi \in I_k} A &\isomorphic q\Omega^k_{A/D}\\
				(f_{\smbfi}) &\longmapsto \textstyle\sum_{\smbfi \in I_k} f_{\smbfi} \dlog(U_{i_1}) \cdots \dlog(U_{i_k}).
			\end{align*}

		\item The $0^{\textrm{th}}$ differential $d_q \colon A \rightarrow q\Omega^1_{A/D}$ is given as $f \mapsto \sum_{i=1}^d \frac{\gamma_i(f)-f}{q-1} \dlog(U_i)$.

		\item The elements $\dlog(U_i) \in q\Omega^1_{A/D}$ are cocycles for all $1 \leqslant i \leqslant d$.
	\end{itemize}
	The data $d_q \colon A \rightarrow q\Omega^1_{A/D}$ forms a differential ring over $D$, i.e.\ $q\Omega^1_{A/D}$ is a $D\textrm{-bimodule}$, and $d_q$ is $D\linear$ satisfying the Leibniz rule $d_q(fg) = d_q(f)g + fd_q(g)$.
\end{defi}

\begin{rem}
	In all cases that we consider in this article, we will fix units $U_1, \ldots, U_d \in A^{\times}$ such that $\gamma_i(U_j) = q U_j$, if $i=j$, or $U_j$, if $i \neq j$.
\end{rem}

\begin{rem}\label{rem:nablaqi_nota}
	In Definition \ref{defi:qdeRham_complex}, let us denote the operator $\frac{\gamma_i-1}{q-1} \colon A \rightarrow A$ by $\nabla_{q, i}^{\log}$, and we shall refer to these as the $q\textrm{-differential}$ operators in logarithmic coordinates.
	For $U_i$ as in the preceding remark, the operators $\nabla_{q, i} \coloneq U_i \nabla_{q, i}^{\log}$, shall be referred to as the $q\textrm{-differential}$ operators in non-logarithmic coordinates.
\end{rem}

\begin{exam}\label{exam:ar+_qconnection}
	Take $D$ to be $\AF =  O_F\llbracket \mu \rrbracket$, and $A$ to be $\AR$ equipped with the $\AF\linear$ action of $\Gamma_R$, and let $\{\gamma_1, \ldots, \gamma_d\}$ be the topological generators of $\Gamma_R'$ (see Section \ref{subsec:important_rings}).
	Then, by setting $q = 1+\mu$ and $U_i = [X_i^{\flat}]$, for $1 \leqslant i \leqslant d$, we have that $\gamma_i = 1 \textmod \mu \AR$ for all $1 \leqslant i \leqslant d$.
	In particular, $\AR$ satisfies the hypotheses of Definition \ref{defi:qdeRham_complex}.
	Moreover, in this case, $q\Omega^1_{\AR/\AF}$ identifies with $\Omega^1_{\AR/\AF}$, i.e.\ the $(p, \mu)\textrm{-adic}$ completion of the module of K\"ahler differentials of $\AR$ with respect to $\AF$.
	Furthermore, note that we have an isomorphism of rings $\AR/(\mu) \isomorphic R$, so from \cite[Remarks 2.4 and 2.10]{morrow-tsuji}, reduction modulo $q-1$ of the differential ring $d_q \colon \AR \rightarrow \Omega^1_{\AR/\AF}$, is the usual continuous de Rham differential $d \colon R \rightarrow \Omega^1_R$.
\end{exam}

\begin{exam}\label{exam:arpd_qconnection}
	Take $D$ to be $\Lambda_F = O_F[\mu, \{(\mu^{p-1}/p)^{[k]}\}_{k \in \NN}]_p^{\wedge}$ equipped with an action of $(\varphi, \Gamma_F)$ as in Remark \ref{rem:polynom_powerseries_comp} and take $A$ to be $\ARpd \coloneq \AR[\{(\mu^{p-1}/p)^{[k]}\}_{k \in \NN}]_p^{\wedge} = \AR \wotimes_{\AF} \Lambda_F$, equipped with the natural and continuous tensor product $(\varphi, \Gamma_R)\action$.
	From Remark \ref{rem:polynom_powerseries_comp}, note that the structure map $\Lambda_F \rightarrow \ARpd$ is $(\varphi, \Gamma_F)\equivariant$ and from \eqref{eq:pd_iota} note that we also have a $(\varphi, \Gamma_F)\equivariant$ isomorphism of rings $\iota \colon \Lambda_R \isomorphic \ARpd$.
	Moreover, the $\Lambda_F\algebra$ $\ARpd$ is equipped with a $\Lambda_F\linear$ continuous (for the $\padic$ topology) action of $\Gamma_R' \subset \Gamma_R$ and we have $\{\gamma_1, \ldots, \gamma_d\}$ as topological generators of $\Gamma_R'$ (see Section \ref{subsec:important_rings}).
	Then, setting $q = 1+\mu$ and $U_i = [X_i^{\flat}]$, for $1 \leqslant i \leqslant d$, it follows that $\gamma_i = 1 \textmod \mu \ARpd$, for all $1 \leqslant i \leqslant d$.
	In particular, $\ARpd$ satisfies the hypotheses of Definition \ref{defi:qdeRham_complex}.
	Furthermore, in this case, $q\Omega^1_{\ARpd/\Lambda_F}$ identifies with $\Omega^1_{\ARpd/\Lambda_F}$, i.e.\ the $(p, \mu)\adic = (p, [p]_q)\adic = \padic$ (see Lemma \ref{lem:tovermu_unit}) completion of the module of K\"ahler differentials of $\ARpd$ with respect to $\Lambda_F$.
	From Definition \ref{defi:qdeRham_complex} the $\qconnection$ on $\ARpd$, denoted $\nabla_q \colon \ARpd \rightarrow q\Omega^1_{\ARpd/\Lambda_F}$, is given as $f \mapsto \sum_{i=1}^d \frac{\gamma_i(f)-f}{p_2(\mu)} \dlog([X_i^{\flat}])$.
	Moreover, the $\qconnection$ $\nabla_q$ on $\ARpd$ is $p\adically$ quasi-nilpotent because we have $\tfrac{\gamma_i-1}{p_2(\mu)[X_i^{\flat}]}([X_i^{\flat}]) = 1$, and it is flat because $\gamma_i$ commute with each other.
	Furthermore, using the $\qconnection$, $\ARpd$ may be equipped with a quasi-nilpotent flat connection $\nabla$ as in Proposition \ref{prop:a1modp1mu_connection}, which coincides with the universal $\Lambda_F\linear$ continuous de Rham differential $d \colon \ARpd \rightarrow \Omega^1_{\ARpd/\Lambda_F}$.
	Then, by using Proposition \ref{prop:a1modp1mu_connection}, it follows that we have $(\ARpd)^{\Gamma_R'} = (\ARpd)^{\nabla_q=0} \isomorphic (\ARpd)^{\nabla = 0} = (\ARpd)^{d=0}$, where the first equality follows because the action of $\Gamma_R$ on $\ARpd$ is continuous for the $\padic$ topology (see Lemma \ref{lem:cont_coh_disc}).
	Moreover, recall that from Section \ref{subsec:important_rings}, for any $g$ in $\Gamma_F$, we have that $g \gamma_i g^{-1} = \gamma_i^{\chi(f)}$, for all $1 \leqslant i \leqslant d$.
	Hence, from Remark \ref{rem:lambda_connection}, the isomorphism $\Lambda_F \isomorphic (\ARpd)^{\nabla_q=0} = (\ARpd)^{\Gamma_R'}$, induced by the structure map, is $(\varphi, \Gamma_F)\equivariant$.
\end{exam}

\begin{exam}\label{exam:a1modp1mu_qconnection}
	Take $D$ to be $\Lambda_R = R[\mu, \{(\mu^{p-1}/p)^{[k]}\}_{k \in \NN}]_p^{\wedge}$ equipped with an action of $(\varphi, \Gamma_F)$ as in Remark \ref{rem:polynom_powerseries_comp} and take $A$ to be $\Abar(1) \coloneq \AR(1)/(p_1(\mu))$ as a $\Lambda_R\algebra$ via the morphism of rings $\iota_{\Lambda} \colon \Lambda_R \rightarrow \Abar(1)$ defined by extending the $R\algebra$ structure $p_1 \colon R \rightarrow \Abar(1)$ in Remark \ref{rem:a1modp1mu_ralg} via $\mu \mapsto p_2(\mu)$ and $(\mu^{p-1}/p)^{[k]} \mapsto (p_2(\mu)^{p-1}/p)^{[k]}$.
	Using Lemma \ref{lem:a1modp1mu_p1p2_ff} and the explicit description of $\Abar(1)$ in Proposition \ref{prop:a1modp1mu_pdring} and Remark \ref{rem:a1modp1mu_lambda}, it follows that the map $\iota_{\Lambda}$ is injective and $(\varphi, \Gamma_F)\equivariant$.
	Moreover, the $\Lambda_R\algebra$ $\Abar(1)$ is equipped with a $\Lambda_R\linear$ (via $\iota_{\Lambda}$) continuous (for the $\padic$ topology) action of $\Gamma_R' \subset 1 \times \Gamma_R$ and we have $\{\gamma_1, \ldots, \gamma_d\}$ as topological generators of $\Gamma_R'$ (see Section \ref{subsec:important_rings}).
	Then, by setting $q = 1+\mu$ and $U_i = [X_i^{\flat}]$, for $1 \leqslant i \leqslant d$, and from the explicit description of $\Abar(1)$ in Proposition \ref{prop:a1modp1mu_pdring}, we have that $\gamma_i = 1 \textmod p_2(\mu) \Abar(1)$, for all $1 \leqslant i \leqslant d$.
	In particular, $\Abar(1)$ satisfies the hypotheses of Definition \ref{defi:qdeRham_complex}.
	Furthermore, in this case, $q\Omega^1_{\Abar(1)/\Lambda_R}$ identifies with $\Omega^1_{\Abar(1)/\Lambda_R}$, i.e.\ the $(p, \mu)\adic = (p, [p]_q)\adic = \padic$ (see Lemma \ref{lem:tovermu_unit}) completion of the module of K\"ahler differentials of $\Abar(1)$ with respect to $\Lambda_R$.
	From Definition \ref{defi:qdeRham_complex} the $\qconnection$ on $\Abar(1)$, denoted $\nabla_q \colon \Abar(1) \rightarrow q\Omega^1_{\Abar(1)/\Lambda_R}$, is given as $f \mapsto \sum_{i=1}^d \frac{\gamma_i(f)-f}{p_2(\mu)} \dlog([X_i^{\flat}])$.
	Moreover, the $\qconnection$ $\nabla_q$ on $\Abar(1)$ is $p\adically$ quasi-nilpotent because we have $\tfrac{\gamma_i-1}{p_2(\mu)[X_i^{\flat}]}([X_i^{\flat}]) = 1$, and it is flat because $\gamma_i$ commute with each other.
	Furthermore, using the $\qconnection$, $\Abar(1)$ may be equipped with a $p\adically$ quasi-nilpotent flat connection $\nabla$ as in Proposition \ref{prop:a1modp1mu_connection} which coincides with the universal $\Lambda_R\linear$ continuous de Rham differential $d \colon \Abar(1) \rightarrow \Omega^1_{\Abar(1)/\Lambda_R}$.
	Then, by using Proposition \ref{prop:a1modp1mu_connection}, it follows that we have $\Abar(1)^{\Gamma_R'} = \Abar(1)^{\nabla_q=0} \isomorphic \Abar(1)^{\nabla = 0} = \Abar(1)^{d=0}$, where the first equality follows because the action of $\Gamma_R$ on $\Abar(1)$ is continuous for the $\padic$ topology (see Lemma \ref{lem:cont_coh_disc}).
	Moreover, from Section \ref{subsec:important_rings}, recall that for any $g$ in $\Gamma_F$, we have that $g \gamma_i g^{-1} = \gamma_i^{\chi(f)}$, for all $1 \leqslant i \leqslant d$.
	Hence, the isomorphism $\Lambda_R \isomorphic \Abar(1)^{\Gamma_R'} = \Abar(1)^{\nabla_q=0} \isomorphic \Abar(1)^{\nabla=0}$, induced by $\iota_{\Lambda}$ in Remark \ref{rem:lambda_connection}, is $(\varphi, \Gamma_F)\equivariant$.
	In particular, it follows that $R \isomorphic \Lambda_R^{\Gamma_F} \isomorphic (\Abar(1)^{\nabla_q=0})^{\Gamma_F} \isomorphic \Abar(1)^{\Gamma_R}$, induced by $p_1 \colon R \rightarrow \Abar(1)$.
\end{exam}

From Lemma \ref{lem:tovermu_unit}, we have that $t = \log(1+\mu)$ converges in $\mu \Lambda_F \subset \mu \Lambda_R$ and $t/\mu$ is a unit.
\begin{prop}\label{prop:a1modp1mu_connection}
	Let $D = \Lambda_R$ (resp.\ $\Lambda_F$) and $A = \AR(1)/(p_1(\mu))$ (resp.\ $\ARpd$).
	Then, for $1 \leqslant i \leqslant d$, the series of operators $\nabla_i^{\log} = \frac{\log \gamma_i}{t} = \frac{1}{t}\sum_{k \in \NN} (-1)^k \frac{(\gamma_i-1)^k}{k+1}$ converge $p\adically$ on $A$.
	This defines a $D\linear$ $p\adically$ quasi-nilpotent flat connection on $A$, denoted $\nabla \colon A \rightarrow \Omega^1_{A/D}$ and given as $f \mapsto \sum_{i=1}^d \nabla_i^{\log}(f) \dlog([X_i^{\flat}])$.
	The connection $\nabla$ coincides with the universal $D\linear$ continuous de Rham differential operator $d \colon A \rightarrow \Omega^1_{A/D}$.
	Moreover, the data of the connection $\nabla$ on $A$ is equivalent to the data of the $\qconnection$ $\nabla_q$ described in Example \ref{exam:a1modp1mu_qconnection} (resp.\ Example \ref{exam:arpd_qconnection}), i.e.\ either may be recovered from the other.
	Furthermore, the $q\textrm{-de Rham}$ complex $q\Omega^{\bullet}_{A/D}$ is naturally quasi-isomorphic to the de Rham complex $\Omega^{\bullet}_{A/D}$.
	In particular, $A^{\nabla_q=0} \isomorphic A^{\nabla=0}$.
\end{prop}
\begin{proof}
	The first two claims, i.e.\ the convergence of the operators $\nabla_i^{\log}$ on $A$ and the fact that the operator $\nabla$ defines a $D\linear$ $p\adically$ quasi-nilpotent flat connection on $A$, and the claim of equivalence between the connection $\nabla$ and quasi-isomorphism of complexes, will be proven in Proposition \ref{prop:n1modp1mu_connection} for finite $A\modules$ admitting a continuous action of $\Gamma_R$ (trivial modulo $\mu$), in particular, $A$ itself.
	Moreover, in the proof of Proposition \ref{prop:n1modp1mu_connection}, we shall also show that the connection $\nabla$ coincides with the universal $D\linear$ continuous de Rham differential operator $d \colon A \rightarrow \Omega^1_{A/D}$.
	Furthermore, note that the proof of Proposition \ref{prop:n1modp1mu_connection} is independent of the subsequent claims proved in this section.
	Please refer to Proposition \ref{prop:n1modp1mu_connection} for more details.
\end{proof}

\begin{rem}\label{rem:lambda_connection}
	Let $\Abar(1) = \AR(1)/(p_1(\mu))$ as in Proposition \ref{prop:a1modp1mu_connection}.
	From the isomorphism \eqref{eq:a1modp1mu_lambda} in Remark \ref{rem:a1modp1mu_lambda} recall that we have a ring $C_R \lisomorphic \Abar(1)$.
	Moreover, from Remark \ref{rem:a1modp1mu_lambda}, we have that $C_R$ is the $\padic$ completion of a PD-polynomial algebra over $\Lambda_R$ and the structure map $\Lambda_R \rightarrow C_R$ coincides with the composition $\Lambda_R \xrightarrow{\hspace{1mm} \iota_{\Lambda} \hspace{1mm}} \Abar(1) \xrightarrow{\hspace{1mm} \eqref{eq:a1modp1mu_lambda} \hspace{1mm}} C_R$.
	Now note that $C_R$ is equipped with the universal $\Lambda_R\linear$ continuous de Rham differential $d \colon C_R \rightarrow \Omega^1_{C_R/\Lambda_R}$, which may be given as $f \mapsto \sum_{i=1}^d \partial_i^{\log}(f) \dlog(X_i)$, where $\partial_i^{\log} = X_i \partial_i$ and $\partial_i \colon C_R \rightarrow C_R$ is the unique $\Lambda_R\linear$ continuous differential operator such that $\partial_i(X_j) = 1$ if $i = j$ and $0$ otherwise.
	So, the isomorphism in \eqref{eq:a1modp1mu_lambda} is $\Lambda_R\linear$, and it induces an isomorphism between complexes of $\Abar(1)\modules$ $\Omega^{\bullet}_{\Abar(1)/\Lambda_R} \isomorphic \Omega^{\bullet}_{C_R/\Lambda_R} = C_R \otimes_{P_R} \Omega^1_{P_R/R}$, where the right hand side is an $\Abar(1)\module$ via the isomorphism \eqref{eq:a1modp1mu_lambda}.
	Now, let $f$ be any element of $C_R$, then from Remark \ref{rem:a1modp1mu_lambda} we have a unique presentation $f = \sum_{k \in \NN} a_k \mu^{\{k\}}$, with $a_k$ in $P_R$ for all $k \in \NN$ and $p\textrm{-adically}$ $a_k \rightarrow 0$ as $k \rightarrow +\infty$.
	Then, it follows that $f \in C_R^{d=0}$ if and only if $a_k$ in $R$, i.e.\ $f$ is in $\Lambda_R$.
	In particular, we get that $\Lambda_R \isomorphic C_R^{d = 0}$, so $\Lambda_R \isomorphic \Abar(1)^{d = 0}$ via $\iota_{\Lambda}$.
	Using the same argument as above by replacing $\Lambda_R$ by  $\Lambda_F$, $\Abar(1)$ by $\ARpd$, $C_R$ by $\Lambda_R$ and $P_R$ by $O_F$, we obtain that $\Lambda_F \isomorphic (\ARpd)^{d=0}$.
\end{rem}

\begin{rem}\label{rem:a1modp1mu_oacris_connection}
	Let us consider the ring $\ARpd$ from Example \ref{exam:arpd_qconnection}, equipped with a continuous action of $(\varphi, \Gamma_R)$.
	From \eqref{eq:pd_iota} in Remark \ref{rem:polynom_powerseries_comp}, note that we have a $(\varphi, \Gamma_F)\equivariant$ isomorphism of rings $\iota \colon \Lambda_R \isomorphic \ARpd$.
	Moreover, we have a $(\varphi, \Gamma_R)\equivariant$ embedding $\ARpd \rightarrow \Acrys(\Rinfty)$ from Remark \ref{rem:a1modp1mu_oacrys_embed}.
	Furthermore, the injective structure map $p_2 \colon \AR \rightarrow \AR(1)/(p_1(\mu)) = \Abar(1)$ from Remark \ref{rem:a1modp1mu_ralg} extends to an injective ring homomorphism $p_2 \colon \ARpd \rightarrow \Abar(1)$ and the target is the $\padic$ completion of a PD-polynomial algebra over the source (see Remark \ref{rem:a1modp1mu_lambda}).
	Now, consider the universal $\ARpd\linear$ continuous de Rham differential $d \colon \Abar(1) \rightarrow \Omega^1_{\Abar(1)/\ARpd}$, which may be given as $f \mapsto \sum_{i=1}^d \partial_i^{\log}(f) d Y_i$, where $Y_i = p_1(X_i)$ (see Remark \ref{rem:a1modp1mu_ralg}), $\partial_i^{\log} = Y_i \partial_i$ and $\partial_i \colon \Abar(1) \rightarrow \Abar(1)$ is the unique $\ARpd\linear$ continuous differential operator such that $\partial_i(Y_j) = 1$, if $i = j$, and $0$, otherwise.
	Then, it easily follows that the $(\varphi, \Gamma_R)\equivariant$ injective map $\Abar(1) \rightarrow \OAcrys(\Rinfty)$ from Remark \ref{rem:a1modp1mu_oacrys_embed} is $\ARpd\linear$ and compatible with the aforementioned $\ARpd\linear$ differential operator on $\Abar(1)$ and the unique $\Acrys(\Rinfty)\linear$ connection on $\OAcrys(\Rinfty)$ from \cite[Section 6.2]{brinon-relatif}.
\end{rem}

\begin{rem}\label{rem:arpd_abarone_connection_comp}
	Let $\Abar(1) = \AR(1)/(p_1(\mu))$ and $\ARpd$ as in Example \ref{exam:arpd_qconnection}.
	Note that we have a natural injective $(\varphi, \Gamma_R)\equivariant$ homomorphism of rings $p_2 \colon \ARpd \rightarrow \Abar(1)$ (see Remark \ref{rem:a1modp1mu_oacris_connection}).
	Then, by using the Leibniz rule for the connection on $\Abar(1)$ (see the proof of Proposition \ref{prop:n1modp1mu_connection}), it follows that the respective connections on $\Abar(1)$ and $\ARpd$ are compatible and we have $\Omega^1_{\Abar(1)/\Lambda_R} = \Abar(1) \otimes_{\ARpd} \Omega^1_{\ARpd/\Lambda_F}$.
\end{rem}

\subsection{Galois action on \texorpdfstring{$\AR(1)$}{-}}\label{subsec:gamma_act_ar1}

Note that we have the $\varphi\equivariant$ multiplication map $\Delta \colon \AR(1) \rightarrow \AR$.
Moreover, recall that in Remark \ref{rem:phigamma_act_ar+n}, we described the action of $\Gamma_R^2$ on $\AR(1)$.
In this section, our goal is to prove the following claim:
\begin{prop}\label{prop:delta_modmun}
	The $\varphi\equivariant$ homomorphism $\Delta \colon \AR(1) \rightarrow \AR$ restricts to a $\varphi\equivariant$ isomorphism $\AR(1)^{1 \times \Gamma_R} \isomorphic \AR$, which is further compatible with the action of $\Gamma_R$ on each side.
\end{prop}

In order to prove Proposition \ref{prop:delta_modmun}, we will study the action of $1 \times \Gamma_R$ on $\AR(1)$ in three steps, namely, the geometric part, i.e.\ $\Gamma_R'$, the torsion part, i.e.\ $\Gamma_{\textrm{tor}}$, and the (free) arithmetic part, i.e.\ $\Gamma_0$ (see Section \ref{subsec:setup_nota} for notations and explanations).  Let us note that the results for the action of the geometric part of $\Gamma_R$ in Section \ref{subsubsec:geo_gamma_action}, are applicable for all primes $p$.
However, since $\Gamma_F$ has different shape for $p \geqslant 3$, compared to $p = 2$, therefore, in Sections \ref{subsubsec:fpx_action} and \ref{subsubsec:gamma0_action} we will assume that $p \geqslant 3$.
For $p =2$, the arithmetic action of $\Gamma_R$ will be handled in Section \ref{subsubsec:p=2_gammaF_action}.

\begin{rem}
	Note that it is possible to prove Proposition \ref{prop:delta_modmun}, by studying the action of $1 \times \Gamma_R$ at once and using Galois cohomology arguments, instead of the \textit{3-step} argument presented here.
	However, the methods used in the \textit{3-step} proof are applicable to the case of Wach modules as well (see Section \ref{subsec:wach_and_strat}), whereas the Galois cohomology arguments do not seem to generalise.
\end{rem}

\subsubsection{The action of \texorpdfstring{$\Gamma_R'$}{-}}\label{subsubsec:geo_gamma_action}

In this section, our first goal is to show the following claim:
\begin{lem}\label{lem:a1_geometric_exact}
	For each $n \geqslant 1$, the following natural $(\varphi, \Gamma_R \times \Gamma_F)\equivariant$ sequence is exact:
	\begin{equation}\label{eq:a1_geometric_exact}
		0 \longrightarrow (\AR(1)/(p_1(\mu)))^{1 \times \Gamma_R'} \xrightarrow{\hspace{1mm} p_1(\mu)^n \hspace{1mm}} (\AR(1)/(p_1(\mu))^{n+1})^{1 \times \Gamma_R'} \longrightarrow (\AR(1)/(p_1(\mu))^{n})^{1 \times \Gamma_R'} \longrightarrow 0.
	\end{equation}
\end{lem}

Before proving Lemma \ref{lem:a1_geometric_exact}, we provide a more explicit description of the ring $\AR(1)^{1 \times \Gamma_R'}$ equipped with the induced $(\varphi, \Gamma_R \times \Gamma_F)\action$.
We start with the following:

\begin{const}\label{const:lambdatilder}
	From Construction \ref{const:cech_nerve}, let us consider the $(\varphi, \Gamma_R^2)\equivariant$ surjective map in \eqref{eq:ker_j}, for $n=1$.
	Taking invariants of that map under the action of $1 \times \Gamma_R'$ and an explicit calculations yields a surjective map
	\begin{equation}\label{eq:lambdatilder_surj}
		\AR \widehat{\otimes}_{O_F} \AF \twoheadrightarrow R\big[\zeta_p, X_1^{1/p}, \ldots, X_d^{1/p}\big] \otimes_{O_F} O_F[\zeta_p].
	\end{equation}
	Using the description of the kernel of the map in \eqref{eq:ker_j}, it follows that the kernel of the surjective map \eqref{eq:lambdatilder_surj}, is generated by the ideal $J = (1 \otimes [p]_q, [p]_q \otimes 1) \subset \AR \widehat{\otimes}_{O_F} \AF$.
	Clearly, the sequence $\{p, 1 \otimes [p]_q, [p]_q \otimes 1\}$ is regular on $\AR \widehat{\otimes}_{O_F} \AF$, so from \cite[Proposition 3.13]{bhatt-scholze-prisms}, we set $(\LambdatildeR, [p]_q \otimes 1)$ to be the prismatic envelope of $(\AR \widehat{\otimes}_{O_F} \AF, J)$ over the bounded prism $(\AR, [p]_q)$.
	More explicitly, similar to the description of $\AR(1)$ before Lemma \ref{lem:pi_rj_ff}, let us consider the $\delta\algebra$ $(\AR \widehat{\otimes}_{O_F} \AF) \Big\{\frac{1 \otimes [p]_q}{[p]_q \otimes 1}\Big\}_{\delta}$ obtained by freely adjoining $\tfrac{J}{[p]_q \otimes 1}$ to $\AR \widehat{\otimes}_{O_F} \AF$.
	Then, from \cite[Proposition 3.13]{bhatt-scholze-prisms}, we have that
	\begin{equation}\label{eq:lambdatilde_explicit}
		\LambdatildeR = (\AR \widehat{\otimes}_{O_F} \AF) \Big\{\tfrac{1 \otimes [p]_q}{[p]_q \otimes 1}\Big\}_{(p, [p]_q \otimes 1)}^{\wedge} = (\AR \widehat{\otimes}_{O_F} \AF) \Big\{\tfrac{1 \otimes \ptilde}{\ptilde \otimes 1}\Big\}_{(p, \ptilde \otimes 1)}^{\wedge},
	\end{equation}
	i.e.\ $\LambdatildeR$ is the $(p, [p]_q \otimes 1)\textrm{-adic}$ completion of the $\delta\algebra$ $(\AR \widehat{\otimes}_{O_F} \AF) \Big\{\frac{J}{[p]_q \otimes 1}\Big\}_{\delta}$, and in the second equality of \eqref{eq:lambdatilde_explicit}, we used that $\ptilde$ is the product of $[p]_q$ with a unit in $\AF$ (see Lemma \ref{lem:pq_mu0p_unit}).
	Using notations similar to Lemma \ref{lem:pi_rj_ff}, we denote the two projection maps as $p_1 \colon \AR \rightarrow \LambdatildeR$ and $p_2 \colon \AF \rightarrow \LambdatildeR$.
	Then, similar to Lemma \ref{lem:pi_rj_ff}, it is easy to see that both the preceding maps are faithfully flat, in particular, $\LambdatildeR$ is $p_1(\mu)\torsionfree$.
	Furthermore, note that we have a natural $(\varphi, \Gamma_R)\equivariant$ injective homomorphism $\AF \rightarrow \AR$, where the action of $\Gamma_R$ factors through $\Gamma_F \subset \Gamma_R$.
	So, by the universal property of prismatic envelopes, the continuous action of $\Gamma_R \times \Gamma_F$ on $\AR \wotimes_{O_F} \AF$ naturally extends to a continuous action on $\LambdatildeR$ and since the homomorphism in \eqref{eq:lambdatilder_surj} is compatible with the homomorphism in \eqref{eq:ker_j}, therefore, we obtain the following natural $(\varphi, \Gamma_R \times \Gamma_R)\equivariant$ homomorphism
	\begin{equation}\label{eq:lambdatilder_in_ar1}
		\iota_{\Lambdatilde} \colon \LambdatildeR \longrightarrow \AR(1),
	\end{equation}
	where the action of $\Gamma_R \times \Gamma_R$ on the source of $\iota_{\Lambdatilde}$ factors through $\Gamma_R \times \Gamma_F \subset \Gamma_R \times \Gamma_R$.
\end{const}

\begin{rem}\label{rem:lambdatilder_mod_p1mu}
	Using an argument similar to the proof of Proposition \ref{prop:a1modp1mu_pdring}, it may be shown that we have a natural $(\varphi, \Gamma_F)\equivariant$ isomorphism of rings
	\begin{equation}\label{eq:lambdatilder_mod_p1mu}
		\LambdatildeR/(p_1(\mu)) \isomorphic \Lambda_R = R[\mu, \{(\mu^{p-1}/p)^{[k]}\}_{k \in \NN}]_p^{\wedge},
	\end{equation}
	where $p_2(\mu) \mapsto \mu$ and $p_1([X_i]^{\flat}) \mapsto X_i$, for $1 \leqslant i \leqslant d$ (see Proposition \ref{prop:polynom_to_powerseries} for other descriptions of $\Lambda_R$).
\end{rem}

\begin{rem}\label{rem:qdeform_arpd}
	In Construction \ref{const:lambdatilder}, by switching the two components in \eqref{eq:lambdatilder_surj} and carrying out essentially the same steps, we may define another ring $E \coloneq (\AF \wotimes_{O_F} \AR) \Big\{\tfrac{1 \otimes [p]_q}{[p]_q \otimes 1}\Big\}_{(p, [p]_q \otimes 1)}^{\wedge}$, as the prismatic envelope of $(\AF \wotimes_{O_F} \AR, J)$ over the bounded prism $(\AF, [p]_q)$, where $J = (1 \otimes [p]_q, [p]_q \otimes 1) \subset \AF \wotimes_{O_F} \AR$, i.e.\ $E$ is the $(p, [p]_q \otimes 1)\adic$ completion of the $\delta\algebra$ obtained by freely adjoining $\tfrac{J}{[p]_q \otimes 1}$ to $\AF \wotimes_{O_F} \AR$, and denoted as $(\AR \wotimes_{O_F} \AF)\Big\{\frac{J}{[p]_q \otimes 1}\Big\}_{\delta}$.
	Now, if we use notations similar to Lemma \ref{lem:pi_rj_ff} to denote the two projection maps as $p_1 \colon \AF \rightarrow E$ and $p_2 \colon \AR \rightarrow E$.
	Then, similar to Lemma \ref{lem:pi_rj_ff}, it is easy to see that both the preceding maps are faithfully flat.
	Moreover, by the universal property of prismatic envelopes, we see that the continuous action of $\Gamma_F \times \Gamma_R$ on $\AF \wotimes_{O_F} \AR$ naturally extends to a continuous action on $E$.
	Furthermore, using an argument similar to the proof of Proposition \ref{prop:a1modp1mu_pdring}, it may be shown that we have a natural $(\varphi, \Gamma_R)\equivariant$ isomorphism of rings $E/(p_1(\mu)) \isomorphic \ARpd = \AR[\{(\mu^{p-1}/p)^{[k]}\}_{k \in \NN}]_p^{\wedge}$, where $p_2(\mu) \mapsto \mu$ and $p_2([X_i]^{\flat}) \mapsto [X_i^{\flat}]$, for $1 \leqslant i \leqslant d$ (see Example \ref{exam:arpd_qconnection} for the definition of $\ARpd$).
\end{rem}

\begin{lem}\label{lem:iota_lambdatilde_iso}
	For each $n \geqslant 1$, the map in \eqref{eq:lambdatilder_in_ar1} induces a $(\varphi, \Gamma_R \times \Gamma_F)\equivariant$ isomorphism
	\begin{equation}\label{eq:iota_lambdatilde}
		\iota_{\Lambdatilde} \colon \LambdatildeR/(p_1(\mu))^n \isomorphic (\AR(1)/(p_1(\mu))^n)^{1 \times \Gamma_R'}.
	\end{equation}
	In particular, \eqref{eq:lambdatilder_in_ar1} induces a $(\varphi, \Gamma_R \times \Gamma_F)\equivariant$ isomorphism $\iota_{\Lambdatilde} \colon \LambdatildeR \isomorphic \AR(1)^{1 \times \Gamma_R'}$.
\end{lem}
\begin{proof}
	As the action of $\Gamma_R \times \Gamma_R$ on $\LambdatildeR$ factors through $\Gamma_R \times \Gamma_F$, we see that reducing \eqref{eq:lambdatilder_in_ar1} modulo $p_1(\mu)^n$ and taking $(1 \times \Gamma_R')\textrm{-invariants}$, we obtain the $(\varphi, \Gamma_R \times \Gamma_F)\equivariant$ map in \eqref{eq:iota_lambdatilde}.
	It remains to show that it is an isomorphism.
	For $n=1$, we consider the following $(\varphi, \Gamma_F)\equivariant$ diagram
	\begin{center}
		\begin{tikzcd}[row sep=large]
			\LambdatildeR/(p_1(\mu)) \arrow[r, "\eqref{eq:iota_lambdatilde}"] \arrow[d, "\eqref{eq:lambdatilder_mod_p1mu}", "\wr"'] & (\AR(1)/(p_1(\mu)))^{1 \times \Gamma_R'}\\
			\Lambda_R \arrow[ru, "\iota_{\Lambda}", "\sim"'],
		\end{tikzcd}
	\end{center}
	where the isomorphism $\iota_{\Lambda}$ is from Example \ref{exam:a1modp1mu_qconnection}.
	It is easy to see that the diagram above commutes. 
	Therefore, the top horizontal map of the diagram is an isomorphism, i.e.\ $\LambdatildeR/(p_1(\mu)) \isomorphic (\AR(1)/(p_1(\mu)))^{1 \times \Gamma_R'}$.
	Now, as both $\AR(1)$ and $\LambdatildeR$ are $p_1(\mu)\torsionfree$ (see Lemma \ref{lem:pi_rj_ff} and Construction \ref{const:lambdatilder}), we consider the following diagram with exact rows
	\begin{center}
		\begin{tikzcd}
			0 \arrow[r] & \LambdatildeR/(p_1(\mu)) \arrow[r, "p_1(\mu)^n"] \arrow[d, "\eqref{eq:iota_lambdatilde}", "\wr"'] & \LambdatildeR/(p_1(\mu))^{n+1} \arrow[r] \arrow[d, "\eqref{eq:iota_lambdatilde}"] & \LambdatildeR/(p_1(\mu))^n \arrow[r] \arrow[d, "\eqref{eq:iota_lambdatilde}"] & 0\\
			0 \arrow[r] & (\AR(1)/(p_1(\mu)))^{1 \times \Gamma_R'} \arrow[r, "p_1(\mu)^n"] & (\AR(1)/(p_1(\mu))^{n+1})^{1 \times \Gamma_R'} \arrow[r] & (\AR(1)/(p_1(\mu))^n)^{1 \times \Gamma_R'} \arrow[r] & 0,
		\end{tikzcd}
	\end{center}
	where the second exact sequence is from \eqref{eq:a1_geometric_exact}.
	Using the diagram, an easy induction on $n \geqslant 1$ gives a natural $(\varphi, \Gamma_R \times \Gamma_F)\equivariant$ isomorphism induced from the map \eqref{eq:lambdatilder_in_ar1} as $\iota_{\Lambdatilde} \colon \LambdatildeR/(p_1(\mu))^{n+1} \isomorphic (\AR(1)/(p_1(\mu))^{n+1})^{1 \times \Gamma_R'}$, i.e.\ the isomorphism in \eqref{eq:iota_lambdatilde}.
	Finally, since both $\LambdatildeR$ and $\AR(1)$ are $p_1(\mu)\adically$ complete, therefore, by taking the limit over $n \geqslant 1$, and noting that inverse limit commutes with taking $(1 \times \Gamma_R')\textrm{-invariants}$, we obtain the $(\varphi, \Gamma_R \times \Gamma_F)\equivariant$ isomorphism $\iota_{\Lambdatilde} \colon \LambdatildeR \isomorphic \AR(1)^{1 \times \Gamma_R'}$.
	This allows us to conclude.
\end{proof}

Let us note an important observation for the action of $\Gamma_R \times \Gamma_F$ on $\LambdatildeR$:
\begin{lem}\label{lem:lambdatildeR_triv_modmu}
	The action of $1 \times \Gamma_F$ is trivial on $\LambdatildeR/(p_2(\mu))$ and the action of $\Gamma_R \times 1$ is trivial on $\LambdatildeR/(p_1(\mu))$.
\end{lem}
\begin{proof}
	The proof of the claim is similar to that of Proposition \ref{prop:arn_gammar_triv_modmu}.
	For the first part, let us note that the action of $1 \times \Gamma_F$ on $\AR \wotimes_{O_F} O_F\llbracket\mu\rrbracket$ is trivial modulo $p_2(\mu)$.
	Then, from the explicit description of $\LambdatildeR$ in Construction \ref{const:lambdatilder}, it is enough to show that for any $m \in \NN$ and $g$ in $1 \times \Gamma_F$, we have that
	\begin{equation*}
		(g-1)\delta^m\big(\tfrac{1 \otimes [p]_q}{[p]_q \otimes 1}) \in p_2(\mu) \LambdatildeR.
	\end{equation*}
	Using Lemma \ref{lem:mu_mu0_deltastable} (2), let us note that $p_2(\mu)\LambdatildeR$ is a $\delta\textrm{-stable}$ ideal of $\LambdatildeR$, in the sense of \cite[Example 2.10]{bhatt-scholze-prisms}.
	Then, using Lemma \ref{lem:gamma_deltam_mu} for $A = \LambdatildeR$ and $\alpha = p_2(\mu)$, we see that to prove our claim, it is enough to show that $(g-1)(\tfrac{1 \otimes [p]_q}{[p]_q \otimes 1})$ belongs to $p_2(\mu) \LambdatildeR$, which follows from Lemma \ref{lem:gamma_deltam_pq}.
	In particular, we have shown that the action of $1 \times \Gamma_F$ is trivial on $\LambdatildeR/(p_2(\mu))$.

	For the second claim, note that the action of $\Gamma_R \times 1$ on $\AR \wotimes_{O_F} O_F\llbracket\mu\rrbracket$ is trivial modulo $p_1(\mu)$.
	Now, from the explicit description of $\LambdatildeR$ in Construction \ref{const:lambdatilder}, it is enough to show that for any $m \in \NN$ and $g$ in $\Gamma_R \times 1$, we have that
	\begin{equation*}
		(g-1)\delta^m\big(\tfrac{1 \otimes [p]_q}{[p]_q \otimes 1}) \in p_1(\mu) \LambdatildeR.
	\end{equation*}
	Similar to the first part, using Lemma \ref{lem:mu_mu0_deltastable} (1), we note that $p_1(\mu)\LambdatildeR$ is a $\delta\textrm{-stable}$ ideal of $\LambdatildeR$.
	Then, using Lemma \ref{lem:gamma_deltam_mu} for $A = \LambdatildeR$ and $\alpha = p_1(\mu)$, we see that to prove our claim, it is enough to show that $(g-1)(\tfrac{1 \otimes [p]_q}{[p]_q \otimes 1})$ belongs to $p_1(\mu) \LambdatildeR$, which again follows from Lemma \ref{lem:gamma_deltam_pq}.
	Hence, the action of $\Gamma_R \times 1$ is trivial on $\LambdatildeR/(p_1(\mu))$.
\end{proof}

\begin{rem}\label{rem:gammaF_act_lambdar}
	From Lemma \ref{lem:lambdatildeR_triv_modmu}, note that the action of $1 \times \Gamma_F$ is trivial on $\LambdatildeR/(p_2(\mu))$, and multiplication by $p_1(\mu)$ on $\LambdatildeR$ is equivariant for this action.
	Therefore, it follows that for any $g$ in $1 \times \Gamma_F$ and any $x$ in $\LambdatildeR/(p_1(\mu))^n$ we have that $(g-1)x$ is an element of $p_2(\mu)\LambdatildeRo/(p_1(\mu))^n$.
	In particular, for $n=1$, using the isomorphism $\LambdatildeR/(p_1(\mu)) \isomorphic \Lambda_R$ from \eqref{eq:lambdatilder_mod_p1mu}, we get that for $g$ in $\Gamma_F$ and $x$ in $\Lambda_R$ the element $(g-1)x$ belongs to $\mu\Lambda_R$.
\end{rem}

We are now ready to prove the main claim:
\begin{proof}[Proof of Lemma \ref{lem:a1_geometric_exact}]
	To lighten notations, let us set $A(1) \coloneq \AR(1)$ and $\Abar(1) \coloneq \AR(1)/(p_1(\mu))$.
	Instead of working with the action of $1 \times \Gamma_R'$, we shall work with the $\qconnection$ arising from this action.
	More precisely, in the notation of Definition \ref{defi:qdeRham_complex}, take $D$ to be $\LambdatildeR \isomorphic A(1)^{1 \times \Gamma_R'}$ (see Lemma \ref{lem:iota_lambdatilde_iso}), and $A$ to be $A(1)$ equipped with a $\LambdatildeR\linear$ action of $1 \times \Gamma_R'$ and let $\{\gamma_1, \ldots, \gamma_d\}$ be the topological generators of $\Gamma_R'$ (see Section \ref{subsec:important_rings}).
	Then, setting $q = 1+p_2(\mu)$ and $U_i = p_2([X_i^{\flat}])$, for $1 \leqslant i \leqslant d$, we have that $\gamma_i = 1 \textmod p_2(\mu) \AR(1)$, for all $1 \leqslant i \leqslant d$ (see Proposition \ref{prop:arn_gammar_triv_modmu}).
	In particular, $A(1)$ satisfies the hypotheses of Definition \ref{defi:qdeRham_complex}.
	Moreover, in this case, $q\Omega^1_{A(1)/\LambdatildeR}$ identifies with $\Omega^1_{A(1)/\LambdatildeR}$, and given as the $(p, \mu)\textrm{-adic}$ completion of the module of K\"ahler differentials of $A(1)$ with respect to $\LambdatildeR$.
	From Definition \ref{defi:qdeRham_complex}, the $\qconnection$ on $A(1)$, denoted $\nabla_q \colon A(1) \rightarrow q\Omega^1_{A(1)/\LambdatildeR}$, is given as $f \mapsto \sum_{i=1}^d \frac{\gamma_i(f)-f}{p_2(\mu)} \dlog(p_2([X_i^{\flat}]))$.
	Additionally, the $\qconnection$ $\nabla_q$ on $A(1)$ is $(p, p_1(\mu))\adically$ quasi-nilpotent because we have $\tfrac{\gamma_i-1}{p_2(\mu)p_2([X_i^{\flat}])}(p_2([X_i^{\flat}])) = 1$, and it is flat because $\gamma_i$ commute with each other.
	Let us also note that $\LambdatildeR/(p_1(\mu)) \isomorphic \Lambda_R \isomorphic (A(1)/(p_1(\mu)))^{1 \times \Gamma_R'}$ (see \eqref{eq:lambdatilder_mod_p1mu}, Example \ref{exam:a1modp1mu_qconnection} and the proof of Lemma \ref{lem:iota_lambdatilde_iso}).
	Now, consider the following exact sequence of $q\textrm{-de Rham}$ complexes:
	\begin{equation*}
		0 \longrightarrow q\Omega^{\bullet}_{\Abar(1)/\Lambda_R} \xrightarrow{\hspace{1mm} p_1(\mu)^n \hspace{1mm}} A(1)/(p_1(\mu))^{n+1} \otimes_{A(1)} q\Omega^{\bullet}_{A(1)/\LambdatildeR} \longrightarrow A(1)/(p_1(\mu))^n \otimes_{A(1)} q\Omega^{\bullet}_{A(1)/\LambdatildeR} \longrightarrow 0.
	\end{equation*}
	Since the action of $1 \times \Gamma_R'$ is continuous for the $(p, p_1(\mu))\adic$ topology on $A(1)$, therefore from Lemma \ref{lem:cont_coh_disc}, we get that $(A(1)/(p_1(\mu))^n)^{1 \times \Gamma_R'} = (A(1)/(p_1(\mu))^n)^{\nabla_q=0}$.
	In particular, showing that \eqref{eq:a1_geometric_exact} is exact, is equivalent to showing that $H^1\big(q\Omega^{\bullet}_{\Abar(1)/\Lambda_R}\big) = 0$.
	Now, from Proposition \ref{prop:a1modp1mu_connection}, recall that the $q\textrm{-de Rham}$ complex $q\Omega^{\bullet}_{\Abar(1)/\Lambda_R}$ is quasi-isomorphic to the de Rham complex $\Omega^{\bullet}_{\Abar(1)/\Lambda_R}$ and from the explicit description of $\Abar(1)$ in Proposition \ref{prop:a1modp1mu_pdring}, it is clear that the de Rham complex $\Omega^{\bullet}_{\Abar(1)/\Lambda_R}$ is acyclic in positive degrees, in particular, $H^1\big(q\Omega^{\bullet}_{\Abar(1)/\Lambda_R}\big) = H^1\big(\Omega^{\bullet}_{\Abar(1)/\Lambda_R}\big) = 0$.
	Hence, it follows that \eqref{eq:a1_geometric_exact} is exact.
\end{proof}

For computations to be carried out in Section \ref{subsec:wach_and_strat}, let us consider a ``diagonal map for the geometric variables'' from $\AR(1)$ to $\LambdatildeR$.
More precisely, consider the following surjective $\AR\linear$ (via $p_1$) and $(\varphi, \Gamma_R \times \Gamma_F)\equivariant$ map:
\begin{equation*}
	\Delta' \colon \AR(1) \longrightarrow \LambdatildeR,
\end{equation*}
where $p_2(\mu) \mapsto p_2(\mu)$ and $p_2([X_i^{\flat}]) \mapsto p_1([X_i^{\flat}])$, for $1 \leqslant i \leqslant d$.
Then, it is easy to verify that the composition $\Delta' \circ \iota_{\Lambdatilde}$ (see \eqref{eq:lambdatilder_in_ar1}) is the identity on $\LambdatildeR$.
As $p_1(\mu)$ is invariant under the action of $1 \times \Gamma_R'$, we see that the reduction of $\Delta'$ modulo $p_1(\mu)^n$ is $(\varphi, \Gamma_R \times \Gamma_F)\equivariant$, for each $n \geqslant 1$.
Now, taking invariants of the source of $\Delta'$ under the action of $1 \times \Gamma_R'$, we obtain the following $(\varphi, \Gamma_F)\equivariant$ morphism:
\begin{equation}\label{eq:delta_prime}
	\Delta' \colon (\AR(1)/(p_1(\mu))^n)^{1 \times \Gamma_R'} \longrightarrow \LambdatildeR/(p_1(\mu))^n,
\end{equation}
and we claim the following:

\begin{lem}\label{lem:delta_prime_iso}
	The $(\varphi, \Gamma_F)\equivariant$ map in \eqref{eq:delta_prime} is an isomorphism and it induces a $(\varphi, \Gamma_F)\equivariant$ isomorphism $\Delta' \colon \AR(1)^{1 \times \Gamma_R'} \isomorphic \LambdatildeR$
\end{lem}
\begin{proof}
	Let us first consider the following $(\varphi, \Gamma_F)\equivariant$ diagram:
	\begin{equation*}
		\Lambda_R \xrightarrow[\sim]{\hspace{1mm} \iota_{\Lambda} \hspace{1mm}} (\AR(1)/(p_1(\mu)))^{1 \times \Gamma_R'} \xrightarrow{\hspace{1mm} \eqref{eq:delta_prime} \hspace{1mm}} \LambdatildeR/(p_1(\mu)) \xrightarrow[\sim]{\hspace{1mm} \eqref{eq:lambdatilder_mod_p1mu} \hspace{1mm}} \Lambda_R.
	\end{equation*}
	Using the isomorphism \eqref{eq:iota_lambdatilde} in Lemma \ref{lem:iota_lambdatilde_iso}, it is easy to verify that the composition above is the identity on $\Lambda_R$.
	Therefore, it follows that \eqref{eq:delta_prime} is an isomorphism, for $n=1$.
	Now, as both $\AR(1)$ and $\LambdatildeR$ are $p_1(\mu)\torsionfree$ (see Lemma \ref{lem:pi_rj_ff} and Construction \ref{const:lambdatilder}), we consider the following commutative diagram with exact rows:
	\begin{center}
		\begin{tikzcd}
			0 \arrow[r] & (\AR(1)/(p_1(\mu)))^{1 \times \Gamma_R'} \arrow[r, "p_1(\mu)^n"] \arrow[d, "\eqref{eq:delta_prime}", "\wr"'] & (\AR(1)/(p_1(\mu))^{n+1})^{1 \times \Gamma_R'} \arrow[r] \arrow[d, "\eqref{eq:delta_prime}"] & (\AR(1)/(p_1(\mu))^n)^{1 \times \Gamma_R'} \arrow[r] \arrow[d, "\eqref{eq:delta_prime}"] & 0\\
			0 \arrow[r] & \LambdatildeR/(p_1(\mu)) \arrow[r, "p_1(\mu)^n"] & \LambdatildeR/(p_1(\mu))^{n+1} \arrow[r] & \LambdatildeR/(p_1(\mu))^n \arrow[r] & 0,
		\end{tikzcd}
	\end{center}
	where the first exact sequence is from \eqref{eq:a1_geometric_exact}.
	Using the diagram, an easy induction on $n \geqslant 1$ gives the natural $(\varphi, \Gamma_R \times \Gamma_F)\equivariant$ isomorphism $\Delta' \colon (\AR(1)/(p_1(\mu))^{n+1})^{1 \times \Gamma_R'} \isomorphic \LambdatildeR/(p_1(\mu))^{n+1}$ in \eqref{eq:delta_prime}.
	Finally, since both $\AR(1)$ and $\LambdatildeR$ are $p_1(\mu)\adically$ complete, therefore, taking the limit over $n \geqslant 1$, and noting that inverse limit commutes with taking $(1 \times \Gamma_R')\textrm{-invariants}$, we obtain the $(\varphi, \Gamma_F)\equivariant$ isomorphism $\Delta' \colon \AR(1)^{1 \times \Gamma_R'} \isomorphic \LambdatildeR$.
\end{proof}

\subsubsection{The action of \texorpdfstring{$\FF_p^{\times}$}{-}}\label{subsubsec:fpx_action}

In this section, we will assume that $p \geqslant 3$ and consider the invariants of the exact sequence \eqref{eq:a1_geometric_exact} (more precisely, its image under the inverse of $\iota_{\Lambdatilde}$ in \eqref{eq:iota_lambdatilde}), for the action of $1 \times \FF_p^{\times}$.
\begin{lem}\label{lem:a1_fpx_exact}
	For each $n \geqslant 1$, the following natural $(\varphi, \Gamma_R \times \Gamma_0)\equivariant$ sequence is exact:
	\begin{equation}\label{eq:a1_fpx_exact}
		0 \longrightarrow (\LambdatildeR/(p_1(\mu)))^{1 \times \FF_p^{\times}} \xrightarrow{\hspace{1mm} p_1(\mu)^n \hspace{1mm}} (\LambdatildeR/(p_1(\mu))^{n+1})^{1 \times \FF_p^{\times}} \longrightarrow (\LambdatildeR/(p_1(\mu))^n)^{1 \times \FF_p^{\times}} \longrightarrow 0.
	\end{equation}
\end{lem}
\begin{proof}
	Using the discussion in Construction \ref{const:lambdatilder} and the isomorphism \eqref{eq:iota_lambdatilde} in Lemma \ref{lem:iota_lambdatilde_iso}, the exact sequence \eqref{eq:a1_geometric_exact} in Lemma \ref{lem:a1_geometric_exact} may be written as the following $(\varphi, \Gamma_R \times \Gamma_F)\equivariant$ exact sequence:
	\begin{equation*}
		0 \longrightarrow \LambdatildeR/(p_1(\mu)) \xrightarrow{\hspace{1mm} p_1(\mu)^n \hspace{1mm}} \LambdatildeR/(p_1(\mu))^{n+1} \longrightarrow \LambdatildeR/(p_1(\mu))^{n} \longrightarrow 0.
	\end{equation*}
	By considering the associated long exact sequence for the cohomology of $(1 \times \FF_p^{\times})\action$ and noting that $H^1(1 \times \FF_p^{\times}, \Lambda_R) = 0$, since $p-1$ is invertible in $\ZZ_p$, we obtain that the sequence in \eqref{eq:a1_fpx_exact} is exact.
\end{proof}

Next, let us describe the rings $(\LambdatildeR/(p_1(\mu))^n)^{1 \times \FF_p^{\times}}$, more explicitly.

\begin{const}\label{const:lambdatilder0}
	Note that from \eqref{eq:fpx_decomp}, the continuous action of $1 \times \FF_p^{\times}$ induces a decomposition $\LambdatildeR = \bigoplus_{i=0}^{p-2} (\LambdatildeR)_i$.
	Let us set $\LambdatildeRo \coloneq (\LambdatildeR)_0 = (\LambdatildeR)^{1 \times \FF_p^{\times}}$, complete for the $(p, p_1(\mu))\textrm{-adic}$ topology and equipped with the induced Frobenius endomorphism and continuous action of $\Gamma_R \times \Gamma_0$, from the corresponding structures on $\LambdatildeR$.
	Using the explicit presentation of $\LambdatildeR$ from Construction \ref{const:lambdatilder} (see \eqref{eq:lambdatilde_explicit}), we obtain an explicit presentation of $\LambdatildeRo$ as follows:
	\begin{equation}\label{eq:lambdatildeo_explicit}
		\LambdatildeRo = (\AR \widehat{\otimes}_{O_F} O_F\llbracket\mu_0\rrbracket) \Big\{\tfrac{1 \otimes \ptilde}{\ptilde \otimes 1}\Big\}_{(p, \ptilde \otimes 1)}^{\wedge}.
	\end{equation}
	Using notations from Construction \ref{const:lambdatilder}, we shall denote the two projection maps as $p_1 \colon \AR \rightarrow \LambdatildeRo$ and $p_2 \colon O_F\llbracket\mu_0\rrbracket \rightarrow \LambdatildeRo$.
	Then, similar to Lemma \ref{lem:pi_rj_ff}, it is easy to see that both the preceding maps are faithfully flat, in particular, $\LambdatildeRo$ is $p_1(\mu)\torsionfree$.
	Furthermore, note that from \eqref{eq:lambdatilder_mod_p1mu}, we have a natural $(\varphi, \Gamma_F)\equivariant$ isomorphism of rings $\LambdatildeR/(p_1(\mu)) \isomorphic \Lambda_R$.
	Now, using \eqref{eq:fpx_decomp}, the action of $1 \times \FF_p^{\times}$ induces a decomposition $\Lambda_R = \bigoplus_{i=0}^{p-2} (\Lambda_R)_i$.
	Note that $p_1(\mu)$ in $\LambdatildeR$ is invariant under the action of $1 \times \FF_p^{\times}$ and $p-1$ is invertible in $\ZZ_p$, in particular, $H^1(1 \times \FF_p^{\times}, \LambdatildeR) = 0$.
	Therefore, it follows that for each $n \geqslant 1$, we have a $(\varphi, \Gamma_R \times \Gamma_0)\equivariant$ isomorphism of rings
	\begin{equation}\label{eq:lambdatildero_mod_p1mun}
		\LambdatildeRo/(p_1(\mu))^n \isomorphic (\LambdatildeR/(p_1(\mu))^n)^{1 \times \FF_p^{\times}}.
	\end{equation}
	In particular, for $n =1$, we have a $(\varphi, \Gamma_R \times \Gamma_0)\equivariant$ isomorphism of rings $\LambdatildeRo/(p_1(\mu)) \isomorphic (\LambdatildeR/(p_1(\mu)))^{1 \times \FF_p^{\times}}$.
	Additionally, the isomorphism in \eqref{eq:lambdatilder_mod_p1mu} induces a $(\varphi, \Gamma_0)\equivariant$ isomorphism of rings 
	\begin{equation}\label{eq:lambdatildero_mod_p1mu}
		\LambdatildeRo/(p_1(\mu)) \isomorphic (\LambdatildeR/(p_1(\mu)))^{1 \times \FF_p^{\times}} \isomorphic \Lambda_R^{\FF_p^{\times}} =:f\Lambda_{R, 0}.
	\end{equation}

	Let us describe $\Lambda_{R, 0}$ more explicitly.
	From Proposition \ref{prop:polynom_to_powerseries}, recall that we have $\Lambda_R = R[\mu, \{(\mu^{p-1}/p)^{[k]}\}_{k \in \NN}]_p^{\wedge}$.
	Moreover, recall that $t = \log(1+\mu)$ is an element of $\Lambda_F$, and since $t/\mu$ is a unit in $\Lambda_F$ (see Lemma \ref{lem:tovermu_unit}), therefore, we may write $\Lambda_R = R[t, \{(t^{p-1}/p)^{[k]}\}_{k \in \NN}]_p^{\wedge}$.
	Now, since the action of $\FF_p^{\times}$ is trivial on $t^{p-1}$, so it follows that $\Lambda_{R, 0} = \Lambda_R^{\FF_p^{\times}} = R[\{(t^{p-1}/p)^{[k]}\}_{k \in \NN}]_p^{\wedge}$.
	Furthermore, recall that $\mu_0$ is the product of $\mu^{p-1}$ with a unit in $\AF$ (see Lemma \ref{lem:mu0_mup1_unit}), and since $t/\mu$ is a unit in $\Lambda_F$, therefore, we may write $\mu_0 = \upsilon t^{p-1}$, where $\upsilon$ is a unit in $\Lambda_F$.
	As the action of $\FF_p^{\times}$ is trivial on $\mu_0$ and $t^{p-1}$, it follows that $\upsilon$ belongs to $\Lambda_{F, 0} = \Lambda_F^{\FF_p^{\times}}$.
	In particular, $\mu_0$ is a product of $t^{p-1}$ with a unit in $\Lambda_{F, 0}$ and we have $(\varphi, \Gamma_0)\equivariant$ identifications
	\begin{equation}\label{eq:lambdar0_explicit}
		\Lambda_{R, 0} = \Lambda_R^{\FF_p^{\times}} = R[\{(t^{p-1}/p)^{[k]}\}_{k \in \NN}]_p^{\wedge} = R[\{(\mu_0/p)^{[k]}\}_{k \in \NN}]_p^{\wedge}.
	\end{equation}
	Note that $\mu^{p-1}$ is not an element of $\Lambda_{R, 0}$ as the action of $\FF_p^{\times}$ is non-trivial on $\mu^{p-1}$.
\end{const}

\subsubsection{The action of \texorpdfstring{$1 + p\ZZ_p$}{-}}\label{subsubsec:gamma0_action}

In this section, we will assume that $p \geqslant 3$ and consider the invariants of the exact sequence \eqref{eq:a1_fpx_exact}, for the action of $1 \times \Gamma_0 \isomorphic 1 \times (1 + p\ZZ_p)$, and show the following:
\begin{lem}\label{lem:a1_gamma0_exact}
	For each $n \geqslant 1$, the following natural $(\varphi, \Gamma_R \times 1)\equivariant$ sequence is exact:
	\begin{equation}\label{eq:a1_gamma0_exact}
		0 \longrightarrow (\LambdatildeR/(p_1(\mu)))^{1 \times \Gamma_F} \xrightarrow{\hspace{1mm} p_1(\mu)^n \hspace{1mm}} (\LambdatildeR/(p_1(\mu))^{n+1})^{1 \times \Gamma_F} \longrightarrow (\LambdatildeR/(p_1(\mu))^n)^{1 \times \Gamma_F} \longrightarrow 0.
	\end{equation}
\end{lem}

\begin{rem}\label{rem:a1_gammaF_exact}
	Via the natural $(\varphi, \Gamma_R \times \Gamma_F)\equivariant$ isomorphism $\iota_{\Lambdatilde}$ in Lemma \ref{lem:iota_lambdatilde_iso} (see \eqref{eq:iota_lambdatilde}), we see that the exact sequence in \eqref{eq:a1_fpx_exact} is the invariants, under the natural action of $1 \times \FF_p^{\times}$, of the exact sequence \eqref{eq:a1_geometric_exact} in Lemma \ref{lem:a1_geometric_exact}.
	Then, it follows that the exact sequence in \eqref{eq:a1_gamma0_exact} is the invariants, under the natural action of $1 \times \Gamma_F$, of the exact sequence \eqref{eq:a1_geometric_exact} in Lemma \ref{lem:a1_geometric_exact}.
\end{rem}

Note that by using the $(\varphi, \Gamma_R \times \Gamma_0)\equivariant$ isomorphism in \eqref{eq:lambdatildero_mod_p1mun} and \eqref{eq:lambdatildero_mod_p1mu}, the sequence in \eqref{eq:a1_gamma0_exact} may be rewritten as the following $(\varphi, \Gamma_R \times \Gamma_0)\equivariant$ sequence:
\begin{equation}\label{eq:lambdatilder0_gamma0_exact}
	0 \longrightarrow \Lambda_{R, 0}^{\Gamma_0} \xrightarrow{\hspace{1mm} p_1(\mu)^n \hspace{1mm}} (\LambdatildeRo/(p_1(\mu))^{n+1})^{1 \times \Gamma_0} \longrightarrow (\LambdatildeRo/(p_1(\mu))^{n})^{1 \times \Gamma_0} \longrightarrow 0.
\end{equation}
In order to prove that \eqref{eq:lambdatilder0_gamma0_exact} is exact, we will now look at the action of $\Gamma_R \times \Gamma_0$ on the rings introduced in Construction \ref{const:lambdatilder0}.
We start with the following observation:
\begin{lem}\label{lem:lambdatilde0_triv_modmu0}
	The action of $1 \times \Gamma_0$ is trivial on $\LambdatildeRo/(p_2(\mu_0))$ and the action of $\Gamma_R \times 1$ is trivial on $\LambdatildeRo/(p_1(\mu))$.
\end{lem}
\begin{proof}
	The proof of the claim is similar to that of Lemma \ref{lem:lambdatildeR_triv_modmu}.
	Let us first note that if $g$ is any element of $\Gamma_0$, then $(g-1)\mu_0$ is an element of $\mu_0 O_F\llbracket\mu_0\rrbracket$ (see Lemma \ref{lem:mu0_gamma0_act}).
	Now, for the first part, let us note that the action of $1 \times \Gamma_0$ on $\AR \wotimes_{O_F} O_F\llbracket\mu_0\rrbracket$ is trivial modulo $p_2(\mu_0)$.
	Then, from the explicit description of $\LambdatildeRo$ in \eqref{eq:lambdatildeo_explicit}, it is enough to show that for any $m \in \NN$ and $g$ in $1 \times \Gamma_0$, we have that
	\begin{equation*}
		(g-1)\delta^m\big(\tfrac{1 \otimes \ptilde}{\ptilde \otimes 1}) \in p_2(\mu_0) \LambdatildeRo.
	\end{equation*}
	Using Lemma \ref{lem:mu_mu0_deltastable} (2), we first note that $p_2(\mu_0)\LambdatildeRo$ is a $\delta\textrm{-stable}$ ideal of $\LambdatildeRo$, in the sense of \cite[Example 2.10]{bhatt-scholze-prisms}.
	Then, using Lemma \ref{lem:gamma_deltam_mu} for $A = \LambdatildeRo$ and $\alpha = p_2(\mu_0)$, we see that to prove our claim, it is enough to show that $(g-1)(\tfrac{1 \otimes \ptilde}{\ptilde \otimes 1})$ belongs to $p_2(\mu_0) \LambdatildeRo$.
	The assertion now follows from the earlier observation that $(g-1)\mu_0$ is an element of $\mu_0 O_F\llbracket\mu_0\rrbracket$.
	In particular, we have shown that the action of $1 \times \Gamma_0$ is trivial on $\LambdatildeRo/(p_2(\mu_0))$.
	The triviality of the action of $\Gamma_R \times 1$ on $\LambdatildeRo/(p_1(\mu))$ easily follows from Lemma \ref{lem:lambdatildeR_triv_modmu}.
\end{proof}

\begin{rem}\label{rem:gamma0_act_lambdar0}
	From Lemma \ref{lem:lambdatilde0_triv_modmu0}, note that the action of $1 \times \Gamma_0$ is trivial on $\LambdatildeRo/(p_2(\mu_0))$ and multiplication by $p_1(\mu)$ on $\LambdatildeRo$ is equivariant for this action.
	Therefore, it follows that for any $g$ in $1 \times \Gamma_0$ and any $x$ in $\LambdatildeRo/(p_1(\mu))^n$, we have that $(g-1)x$ is an element of $p_2(\mu_0)\LambdatildeRo/(p_1(\mu))^n$.
	In particular, for $n=1$, using the isomorphism $\LambdatildeRo/(p_1(\mu)) \isomorphic \Lambda_{R,0}$ from \eqref{eq:lambdatildero_mod_p1mu}, we get that for any $g$ in $\Gamma_0$ and any $x$ in $\Lambda_{R,0}$, the element $(g-1)x$ belongs to $\mu_0\Lambda_{R,0}$.
\end{rem}

Using the action of $1 \times \Gamma_0$ on $\LambdatildeRo$, let us define a $q\textrm{-de Rham}$ complex (see Definition \ref{defi:qdeRham_complex}), by considering the following element in $\LambdatildeRo$ as a parameter:
\begin{equation}\label{eq:elem_stilde}
	\stilde \coloneq \tfrac{1 \otimes \ptilde - \ptilde \otimes 1}{\ptilde \otimes 1} = \tfrac{p_2(\ptilde) - p_1(\ptilde)}{p_1(\ptilde)}.
\end{equation}

\begin{lem}\label{lem:gamma0_act_stilde}
	Let $\gamma_0$ be any element of $1 \times \Gamma_0$.
	Then, we have $(\gamma_0-1)\stilde = u \hspace{0.5mm} p_2(\mu_0)$, for some unit $u$ in $\LambdatildeRo$ depending on $\gamma_0$.
\end{lem}
\begin{proof}
	Note that it is enough to show the claim for a topological generator $\gamma_0$ of $1 \times \Gamma_0$ such that $\chi(\gamma_0) = 1 + pa$, for a unit $a$ in $\ZZ_p$.
	Now, recall that we have $\ptilde = \mu_0 + p$.
	Moreover, from Lemma \ref{lem:mu0_gamma0_act}, recall that we may write $(\gamma_0-1)\mu_0 = \ptilde \mu_0 x$, for some $x$ in $O_F\llbracket\mu_0\rrbracket$.
	So, using \eqref{eq:elem_stilde}, we have that 
	\begin{equation}\label{eq:gamma0_act_stilde}
		(\gamma_0-1)\stilde = \tfrac{(\gamma_0-1)p_2(\ptilde)}{p_1(\ptilde)} = \tfrac{(\gamma_0-1)p_2(\mu_0)}{p_1(\ptilde)} = \tfrac{p_2(\ptilde \mu_0 x)}{p_1(\ptilde)} = \tfrac{p_2(\ptilde)}{p_1(\ptilde)} p_2(x) p_2(\mu_0).
	\end{equation}
	From Lemma \ref{lem:pq_mu0p_unit} and Lemma \ref{lem:pi_rj_ff}, it follows that $\tfrac{1 \otimes \ptilde}{\ptilde \otimes 1}$ is a unit in $\AR(1)$.
	Moreover, from the description of $\LambdatildeRo$ as the $\FF_p^{\times}\textrm{-invariants}$ of $\LambdatildeR$, in Construction \ref{const:lambdatilder0}, it follows that $\tfrac{1 \otimes \ptilde}{\ptilde \otimes 1}$ is also a unit in $\LambdatildeRo$.
	Therefore, to show the claim, it is enough to show that $p_2(x)$ is a unit in $\LambdatildeRo$.
	Again, note that from Construction \ref{const:lambdatilder0}, the ring $\LambdatildeRo$ is $p_1(\mu)\adically$ complete and we have a $(\varphi, \Gamma_0)\equivariant$ isomorphism $\LambdatildeRo/(p_1(\mu)) \isomorphic \Lambda_{R, 0}$ (see \eqref{eq:lambdatildero_mod_p1mu}), therefore, we are reduced to showing that $p_2(\overline{x})$, the image of $p_2(x)$ under the preceding isomorphism, is a unit.
	Also, note that under the preceding isomorphism, $p_2(\mu_0)$ and $p_2(\ptilde)$ map respectively to $\mu_0$ and $\ptilde$ in $\Lambda_{R, 0}$.
	Now, reducing the equalities in \eqref{eq:gamma0_act_stilde} modulo $p_1(\mu)$, we obtain the following expression in $\Lambda_{R, 0}$:
	\begin{equation*}
		\tfrac{(\gamma_0-1)\mu_0}{p} = p_2(\overline{x}) \mu_0 \tfrac{\ptilde}{p}.
	\end{equation*}
	Note that $\ptilde/p$ is a unit in $\Lambda_{R,0}$, because it is the image of the unit $p_2(\ptilde)/p_1(\ptilde)$ in $\LambdatildeRo$ (also see Lemma \ref{lem:tovermu_unit}).
	Then, from Lemma \ref{lem:gamma0_act_s}, it follows that $p_2(\overline{x})$ is a unit in $\Lambda_{R, 0}$, and therefore, $p_2(x)$ is a unit in $\LambdatildeRo$.
	Hence, we have shown that $(\gamma_0-1)\stilde = u \hspace{0.5mm} p_2(\mu_0)$, for $u = (p_2(\ptilde)/p_1(\ptilde))x$ a unit in $\LambdatildeRo$.
\end{proof}

The following observation was used above:
\begin{lem}\label{lem:gamma0_act_s}
	Let $\gamma_0$ be a topological generator of $\Gamma_0$ such that $\chi(\gamma_0) = 1+pa$, for some unit $a$ in $\ZZ_p$.
	Then, we have $(\gamma_0-1)\tfrac{\mu_0}{p} = v\mu_0$, for some unit $v$ in $\Lambda_{F,0} = \Lambda_F^{\FF_p^{\times}}$ depending on $\gamma_0$.
\end{lem}
\begin{proof}
	From Construction \ref{const:lambdatilder0}, recall that $\mu_0/p$ is the product of $t^{p-1}/p$ with a unit in $\Lambda_{F, 0}$ (see \eqref{eq:lambdar0_explicit} and the discussion preceding it).
	So, let us write $\mu_0/p = u t^{p-1}/p$, for some unit $u$ in $\Lambda_{F, 0}$.
	Note that from Remark \ref{rem:gamma0_act_lambdar0}, we have that $(\gamma_0-1) u = \mu_0 x$, for some $x$ in $\Lambda_{F, 0}$.
	Moreover, we see that $(1+pa)^{p-1}-1 = pb$, where $b$ is a unit in $\ZZ_p$.
	Therefore, we may write
	\begin{align*}
		(\gamma_0-1)\tfrac{\mu_0}{p} = (\gamma_0-1)\tfrac{ut^{p-1}}{p} &= \tfrac{t^{p-1}}{p}(\gamma_0-1)u  + \gamma_0(u) (\gamma_0-1)\tfrac{t^{p-1}}{p}\\
				&= \tfrac{t^{p-1}}{p} \mu_0 x  + \gamma_0(u) (\chi(\gamma_0)^{p-1}-1)\tfrac{t^{p-1}}{p}\\
				&= \mu_0 \big(\tfrac{t^{p-1}}{p}x + \gamma_0(u)u^{-1}b\big) = \mu_0 v,
	\end{align*}
	where $v = \big(\tfrac{t^{p-1}}{p}x + \gamma_0(u)u^{-1}b\big)$ is a unit in $\Lambda_{F, 0}$ because $\gamma_0(u)u^{-1}b$ is a unit and $t^{p-1}/p$ is $p\adically$ nilpotent in $\Lambda_{F, 0}$.
	Hence, the lemma is proved.
\end{proof}

In the rest of the Section \ref{subsubsec:gamma0_action}, we will fix the choice of a topological generator $\gamma_0$ of $1 \times \Gamma_0$ such that $\chi(\gamma_0) = 1+pa$, for a unit $a$ in $\ZZ_p$.
Let us now consider the following operator on $\LambdatildeRo$:
\begin{equation}\label{eq:nablaq_stilde}
	\begin{aligned}
		\nabla_{q, \stilde} \colon \LambdatildeRo &\longrightarrow \LambdatildeRo\\
					x &\longmapsto \tfrac{(\gamma_0-1)x}{(\gamma_0-1)\stilde}.
	\end{aligned}
\end{equation}
From the triviality of the action of $1 \times \Gamma_0$ on $\LambdatildeRo/p_2(\mu_0)$ (see Lemma \ref{lem:lambdatilde0_triv_modmu0}) and from Lemma \ref{lem:gamma0_act_stilde}, it follows that the operator $\nabla_{q, \stilde}$ is well defined.
For each $n \geqslant 1$, using Remark \ref{rem:gamma0_act_lambdar0}, the operator in \eqref{eq:nablaq_stilde}, induces well-defined operators $\nabla_{q, \stilde} \colon \LambdatildeRo/(p_1(\mu))^n \longrightarrow \LambdatildeRo/(p_1(\mu))^n$.
In particular, for $n = 1$, set $s \coloneq \mu_0/p$ in $\Lambda_{R, 0}$, and using Remark \ref{rem:gamma0_act_lambdar0} and Lemma \ref{lem:gamma0_act_s}, note that we have a well-defined operator
\begin{equation}\label{eq:nablaq_s}
	\begin{aligned}
		\nabla_{q, s} \colon \Lambda_{R, 0} &\longrightarrow \Lambda_{R, 0}\\
				x &\longmapsto \tfrac{(\gamma_0-1)x}{(\gamma_0-1)s}.
	\end{aligned}
\end{equation}

\begin{rem}\label{rem:nablaq_stilde_qconnection}
	Considering $\stilde$ as a parameter, the operator $\nabla_{q, \stilde}$ in \eqref{eq:nablaq_stilde}, may be considered as a $q\textrm{-differential}$ operator in non-logarithmic coordinates, in the sense of Definition \ref{defi:qdeRham_complex} and Remark \ref{rem:nablaqi_nota}, where $q\Omega^1_{A/D}$ identifies with the $(p, p_1(\mu))\adic$ completion of the module of K\"ahler differentials of $\LambdatildeRo$, with respect to $p_1 \colon \AR \rightarrow \LambdatildeRo$.
	Similarly, considering $s$ as a parameter, the operator $\nabla_{q, s}$ in \eqref{eq:nablaq_s}, may also be considered as a non-logarithmic $q\textrm{-differential}$ operator in the sense of Definition \ref{defi:qdeRham_complex} and Remark \ref{rem:nablaqi_nota}, where $q\Omega^1_{A/D}$ identifies with the $p\adic$ completion of the module of K\"ahler differentials of $\Lambda_{R, 0}$, with respect to $R$.
\end{rem}

For each $n \geqslant 1$, the operator $\nabla_{q, \stilde}$ is an endomorphism of $\LambdatildeRo/(p_1(\mu))^n$, so we may define the following two term Koszul complex:
\begin{equation}\label{eq:stilde_qderham}
	K_{\LambdatildeRo/(p_1(\mu))^n}(\nabla_{q, \stilde}) \coloneq \big[\LambdatildeRo/(p_1(\mu))^n \xrightarrow{\hspace{1mm} \nabla_{q, \stilde} \hspace{1mm}} \LambdatildeRo/(p_1(\mu))^n\big].
\end{equation}
For $n=1$, we have the following claim:
\begin{lem}\label{lem:h1_qderham_lambdar0}
	The cohomology of the complex $K_{\Lambda_{R, 0}}(\nabla_{q, s})$ vanishes in degree 1, i.e.\ $H^1(K_{\Lambda_{R, 0}}(\nabla_{q, s})) = 0$.
\end{lem}
\begin{proof}
	The claim follows from the more general statement in Propositions \ref{prop:mlambda0_connection} and \ref{prop:abarone_comp}, which will be proven for finite $\Lambda_{R, 0}\modules$ admitting a continuous action of $\Gamma_0$ (trivial modulo $\mu_0$), in particular, $\Lambda_{R, 0}$ itself.
\end{proof}

\begin{proof}[Proof of Lemma \ref{lem:a1_gamma0_exact}]
	Using the $(\varphi, \Gamma_R \times \Gamma_0)\equivariant$ isomorphism of rings from \eqref{eq:lambdatildero_mod_p1mun}, the exact sequence in \eqref{eq:a1_fpx_exact} may be rewritten as follows:
	\begin{equation}\label{eq:proof_fpx_exact}
		0 \longrightarrow \Lambda_{R, 0} \xrightarrow{\hspace{1mm} p_1(\mu)^n \hspace{1mm}} \LambdatildeRo/(p_1(\mu))^{n+1} \longrightarrow \LambdatildeRo/(p_1(\mu))^{n} \longrightarrow 0.
	\end{equation}
	Then, using the operator $\nabla_{q, \stilde}$ in \eqref{eq:nablaq_stilde} and the Koszul complex defined in \eqref{eq:stilde_qderham}, we obtain the following exact sequence of Koszul complexes:
	\begin{equation}\label{eq:proof_koszul_exact}
		0 \longrightarrow K_{\Lambda_{R, 0}}(\nabla_{q, s}) \xrightarrow{\hspace{1mm} p_1(\mu)^n \hspace{1mm}} K_{\LambdatildeRo/(p_1(\mu))^{n+1}}(\nabla_{q, \stilde}) \longrightarrow K_{\LambdatildeRo/(p_1(\mu))^{n}}(\nabla_{q, \stilde}) \longrightarrow 0.
	\end{equation}
	Considering the associated long exact sequence, and noting that $H^1(K_{\Lambda_{R, 0}}(\nabla_{q, s})) = 0$ from Lemma \ref{lem:h1_qderham_lambdar0}, we obtain the following exact sequence:
	\begin{equation}\label{eq:proof_horizontal_exact}
		0 \longrightarrow \Lambda_{R, 0}^{\nabla_{q, s}=0} \xrightarrow{\hspace{1mm} p_1(\mu)^n \hspace{1mm}} (\LambdatildeRo/(p_1(\mu))^{n+1})^{\nabla_{q, \stilde}=0} \longrightarrow (\LambdatildeRo/(p_1(\mu))^{n})^{\nabla_{q, \stilde}=0} \longrightarrow 0.
	\end{equation}
	Since the action of $1 \times \Gamma_0$ on $\LambdatildeRo$ is continuous for the $(p, p_1(\mu))\adic$ topology, therefore from Lemma \ref{lem:cont_coh_disc}, we have that $(\LambdatildeRo/(p_1(\mu))^{n+1})^{\nabla_{q, \stilde}=0} = (\LambdatildeRo/(p_1(\mu))^{n+1})^{1 \times \Gamma_0}$, for each $n \in \NN$.
	Hence, from the preceding exact sequence we obtain that the sequence in \eqref{eq:lambdatilder0_gamma0_exact} is exact, and therefore, it follows that the sequence in \eqref{eq:a1_gamma0_exact} is also exact.
	This completes our proof.
\end{proof}

\subsubsection{The case \texorpdfstring{$p=2$}{-}}\label{subsubsec:p=2_gammaF_action}

In this section, our goal is to prove a statement similar to Lemma \ref{lem:a1_gamma0_exact}, for $p=2$.
From \eqref{eq:gammaf_es}, recall that we have $\Gamma_0 \isomorphic 1 + 4\ZZ_2$ and $\Gamma_{\textrm{tor}} \isomorphic \{\pm 1\}$ as groups.
We will first look at the action of $\Gamma_{\textrm{tor}}$ on $\LambdatildeR$.
Let $\sigma$ denote a generator of $\Gamma_{\textrm{tor}}$.
Then, from \eqref{eq:plus_minus_decomp}, recall that by setting $\LambdatildeRplus \coloneq \{x \in \LambdatildeR \textrm{ such that } \sigma(x) = x\}$ and $\LambdatildeRminus \coloneq \{x \in \LambdatildeR \textrm{ such that } \sigma(x) = -x\}$, we have a natural injective homomorphism of $\LambdatildeRplus\modules$
\begin{equation}\label{eq:lambdatilder_pm}
	\LambdatildeRplus \oplus \LambdatildeRminus \longrightarrow \LambdatildeR,
\end{equation}
given as $(x, y) \mapsto x+y$.
Note that the action of $1 \times \Gamma_F$ on $\LambdatildeR$ is continuous for the $(p, p_1(\mu))\adic$ topology, so it follows that $\LambdatildeRplus$ is a $(p, p_1(\mu))\adically$ complete subring of $\LambdatildeR$.
Moreover, the $(\varphi, \Gamma_F)\textrm{-action}$ on $\MLambda$ induces a natural Frobenius and a continuous $\Gamma_0\textrm{-action}$ on $\LambdatildeRplus$ and $\LambdatildeRminus$, respectively.
Equipping $\LambdatildeRplus$ and $\LambdatildeRminus$ with the induced structures, we see that \eqref{eq:lambdatilder_pm} is $(\varphi, \Gamma_R \times \Gamma_F)\equivariant$.
Furthermore, from \eqref{eq:lambdatilder_mod_p1mu}, recall that $\LambdatildeR/(p_1(\mu)) \isomorphic \Lambda_R$.
Then, by setting $\LambdaRplus \coloneq \{x \in \Lambda_R \textrm{ such that } \sigma(x) = x\}$ and $\LambdaRminus \coloneq \{x \in \Lambda_R \textrm{ such that } \sigma(x) = -x\}$, and from \eqref{eq:plus_minus_decomp}, we get a natural $(\varphi, \Gamma_F)\equivariant$ injective homomorphism of $\LambdaRplus\modules$
\begin{equation}\label{eq:lambdar_pm}
	\LambdaRplus \oplus \LambdaRminus \longrightarrow \Lambda_R.
\end{equation}
\begin{lem}\label{lem:lambdar_pm}
	The natural map in \eqref{eq:lambdar_pm} is bijective.
\end{lem}
\begin{proof}
	Recall that $t = \log(1+\mu)$ is an element of $\Lambda_F$ and $t/\mu$ is a unit in $\Lambda_F$ (see Lemma \ref{lem:tovermu_unit}), so we may write, $\Lambda_R = R[t, \{(t^{p-1}/p)^{[k]}\}_{k \in \NN}]_p^{\wedge} = R[\{(t/2)^{[k]}\}_{k \in \NN}]_p^{\wedge}$.
	Since $\sigma(t) = -t$, therefore, it easily follows that $\LambdaRplus = R[\{(t/2)^{[k]}\}_{k \in 2\NN}]_p^{\wedge}$ and $\LambdaRminus = R[\{(t/2)^{[k]}\}_{k \in 2\NN+1}]_p^{\wedge}$.
	Then, it is immediate that $\Lambda_R = \LambdaRplus \oplus \LambdaRminus$.
	Hence, \eqref{eq:lambdar_pm} is bijective.
\end{proof}

Next, we will give another useful description of $\LambdaRplus$ and $\LambdaRminus$.
Let us consider the following element in $\AF$ from \cite[Section 5.2.5]{fontaine-corps-periodes}:
\begin{equation}\label{eq:nu_defi}
	\nu \coloneq q-1+\sigma(q-1) = q+q^{-1}-2 = \tfrac{(q-1)^2}{q} = \tfrac{\mu^2}{1+\mu}.
\end{equation}
Using Lemma \ref{lem:tovermu_unit}, note that the element $\nu$ is the product of $t^2$ with a unit in $\Lambda_F$, and set $\tau \coloneq \nu/8$.
\begin{lem}\label{lem:lambdar_pm_nu}
	The element $\tau$ is the product of $t^2/8$ with a unit in $\LambdaFplus$, i.e.\ we have $\tau = ut^2/8$, for some unit $u$ in $\LambdaFplus$.
	In particular, we have that
	\begin{equation}\label{eq:lambdar_pm_nu}
		\begin{aligned}
			\LambdaRplus &= R[\{(t^2/8)^{[k]}\}_{k \in \NN}]_p^{\wedge} = R[\{\tau^{[k]}\}_{k \in \NN}]_p^{\wedge},\\
			\LambdaRminus &= R[\{(t/2)(t^2/8)^{[k]}\}_{k \in \NN}]_p^{\wedge} = R[\{(t/2)\tau^{[k]}\}_{k \in \NN}]_p^{\wedge} = (t/2)\LambdaRplus,
		\end{aligned}
	\end{equation}
\end{lem}
\begin{proof}
	Recall that $\sigma(t) = -t$, so $t^2/8$ is an element of $\LambdaFplus$.
	Moreover, $\nu = u t^2$, for some unit $u$ in $\Lambda_F$, so $\tau$ is an element of $\Lambda_F$.
	Since, $\sigma(\nu) = \nu$, it follows that $\tau = \nu/8$ belongs to $\LambdaFplus$, and therefore, $u$ is a unit in $\LambdaFplus$.
	Next, from the proof of Lemma \ref{lem:lambdar_pm}, note that $\LambdaRplus = R[\{(t/2)^{[n]}\}_{n \in 2\NN}]_p^{\wedge}$ and $\LambdaRminus = R[\{(t/2)(t/2)^{[n]}\}_{n \in 2\NN}]_p^{\wedge}$.
	Now, let $n = 2k$, for $k \in \NN$, and note that
	\begin{align*}
		(\tfrac{t}{2})^{[n]} = \tfrac{t^{2k}}{(2k)!4^k} = \tfrac{k! 2^k}{(2k)!} (\tfrac{t^2}{8})^{[k]} = \tfrac{k! 2^k}{(2k)! u^k} \tau^{[k]}.
	\end{align*}
	The equalities in \eqref{eq:lambdar_pm_nu} now follow from the discussion above because an easy computation shows that we have $\upsilon_2(2^k) + \upsilon_2(k!) = \upsilon_2((2k)!)$, where $\upsilon_2$ denotes the $2\textrm{-adic}$ valuation on $\ZZ_2$.
\end{proof}

Let us now consider a lifting of the element $\tau = \nu/8$ to $\LambdatildeR$, under the surjective map $\LambdatildeR \rightarrow \LambdatildeR/(p_1(\mu)) \isomorphic \Lambda_R$, where the isomorphism is the $(\varphi, \Gamma_F)\equivariant$ isomorphism in \eqref{eq:lambdatilder_mod_p1mu}.
We have the following element in $\LambdatildeR$:
\begin{equation}\label{eq:tautilde_defi}
	\tautilde \coloneq \tfrac{1}{p_2(q)} \delta\big(\tfrac{p_2([p]_q)}{p_1([p]_q)}\big) = \tfrac{1}{1 \otimes q} \delta\big(\tfrac{1 \otimes [p]_q}{[p]_q \otimes 1}\big) = \tfrac{1}{1 \otimes q} \delta\big(\tfrac{1 \otimes \ptilde}{\ptilde \otimes 1}\big),
\end{equation}
where the last equality follows from Remark \ref{rem:mu0_ptilde_nota}.
\begin{lem}\label{lem:tautilde}
	The element $\tautilde$ belongs to $\LambdatildeRplus$ and we have $\tautilde = \tau \textrm{ mod } p_1(\mu)\LambdatildeR$.
\end{lem}
\begin{proof}
	Let $\sigma$ be a generator of $1 \times \Gamma_{\textrm{tor}}$ and note that $\sigma(1 \otimes q) = 1 \otimes q^{-1}$.
	Moreover, since the action of $1 \times \Gamma_R$ commutes with the $\delta\textrm{-structure}$ on $\LambdatildeR$, we have
	\begin{equation*}
		\sigma(\tautilde) = \tfrac{1}{\sigma(1 \otimes q)} \delta\big(\tfrac{\sigma(1 \otimes [p]_q)}{[p]_q \otimes 1}\big) = (1 \otimes q) \delta\big(\tfrac{1}{1 \otimes q}\tfrac{1 \otimes [p]_q)}{[p]_q \otimes 1}\big) = \tfrac{1}{1 \otimes q} \delta\big(\tfrac{1 \otimes [p]_q}{[p]_q \otimes 1}\big) = \tautilde,
	\end{equation*}
	where the third equality follows from using the product formula for $\delta\textrm{-structures}$ from \eqref{defi:delta_ring} and the fact that $\delta(1 \otimes q) = 0$.
	Therefore, $\tautilde$ is an element of $\LambdatildeRplus$.
	Next, since the isomorphism $\LambdatildeR/(p_1(\mu)) \isomorphic \Lambda_R$ in \eqref{eq:lambdatilder_mod_p1mu} is compatible with the respective Frobenii, in particular, with respective $\delta\textrm{-structures}$, we have the following:
	\begin{equation*}
		\tautilde \textrm{ mod } p_1(\mu) = \tfrac{1}{p_2(q)} \delta\big(\tfrac{p_2([p]_q)}{p_1([p]_q)}\big) \textrm{ mod } p_1(q-1) = \tfrac{1}{q} \delta\big(\tfrac{[p]_q}{2} \big) = \tfrac{(q-1)^2}{8q} = \tfrac{\nu}{8} = \tau.
	\end{equation*}
	This proves the second claim.
\end{proof}

\begin{lem}\label{lem:lambdatildeplus_mod_p1mun}
	For each $n \geqslant 1$, reduction modulo $p_1(\mu)^n$ of \eqref{eq:lambdatilder_pm}, induces a natural $(\varphi, \Gamma_R \times \Gamma_0)\equivariant$ isomorphism
	\begin{equation}\label{eq:lambdatildeplus_mod_p1mun}
		\LambdatildeRplus/(p_1(\mu))^n \isomorphic (\LambdatildeR/(p_1(\mu))^n)^{1 \times \Gamma_{\textup{tor}}}.
	\end{equation}
	Moreover, for $n = 1$, the $(\varphi, \Gamma_F)\equivariant$ isomorphism $\LambdatildeR/(p_1(\mu)) \isomorphic \Lambda_R$ from \eqref{eq:lambdatilder_mod_p1mu}, induces a natural $(\varphi, \Gamma_F)\equivariant$ isomorphism
	\begin{equation}\label{eq:lambdatildeplus_mod_p1mu}
		\LambdatildeRplus/(p_1(\mu)) \isomorphic \LambdaRplus.
	\end{equation}
\end{lem}
\begin{proof}
	Consider the following natural $(\varphi, \Gamma_R \times \Gamma_0)\equivariant$ commutative diagram with exact rows:
	\begin{equation}\label{eq:lambdatildeplus_mod}
		\begin{tikzcd}
			0 \arrow[r] & \LambdatildeRplus/(p_1(\mu)) \arrow[r, "p_1(\mu)^n"] \arrow[d, "\eqref{eq:lambdatilder_pm}"] & \LambdatildeRplus/(p_1(\mu))^{n+1} \arrow[r] \arrow[d, "\eqref{eq:lambdatilder_pm}"] & \LambdatildeRplus/(p_1(\mu))^n \arrow[r] \arrow[d, "\eqref{eq:lambdatilder_pm}"] & 0\\
			0 \arrow[r] & (\LambdatildeR/(p_1(\mu)))^{1 \times \Gamma_{\textrm{tor}}} \arrow[r, "p_1(\mu)^n"] & (\LambdatildeR/(p_1(\mu))^{n+1})^{1 \times \Gamma_{\textrm{tor}}} \arrow[r] & (\LambdatildeR/(p_1(\mu))^n)^{1 \times \Gamma_{\textrm{tor}}},
		\end{tikzcd}
	\end{equation}
	where the vertical maps are injective because we have $p_1(\mu)^n \LambdatildeR \cap \LambdatildeRplus = p_1(\mu)^n \LambdatildeRplus$, as $p_1(\mu)$ is invariant under the action of $1 \times \Gamma_F$.
	Composing the left vertical arrow in \eqref{eq:lambdatildeplus_mod} with the $(\varphi, \Gamma_F)\equivariant$ isomorphism $\LambdatildeR/(p_1(\mu)) \isomorphic \Lambda_R$ from \eqref{eq:lambdatilder_mod_p1mu}, we obtain a natural $(\varphi, \Gamma_F)\equivariant$ injective map $\LambdatildeRplus/(p_1(\mu)) \rightarrow \LambdaRplus$, and we claim that it is surjective as well.

	Indeed, note that $\LambdatildeRplus/(p_1(\mu))$ is a $p\torsionfree$ ring equipped with an induced Frobenius, in particular, a $\delta\textrm{-structure}$, and so we see that the left vertical map in \eqref{eq:lambdatildeplus_mod} is compatible with the respective $\delta\textrm{-structures}$.
	Now, from \eqref{eq:lambdar_pm_nu} in Lemma \ref{lem:lambdar_pm_nu}, recall that we have $\LambdaRplus = R[\{\tau^{[k]}\}_{k \in \NN}]_p^{\wedge}$.
	If we denote by $\tau$, its preimage under the isomorphism \eqref{eq:lambdatilder_mod_p1mu}, then from Lemma \ref{lem:tautilde} we have that $\tau$ is an element of $\LambdatildeRplus/(p_1(\mu))$, and we need to show that $\tau^{[k]}$ belongs to $\LambdatildeRplus/(p_1(\mu))$, for each $k \in \NN$.
	Moreover, using \cite[Lemma 2.35]{bhatt-scholze-prisms}, we see that it is enough to show that $\tfrac{\tau^2}{2}$ is an element of $\LambdatildeRplus/(p_1(\mu))$.
	Since, $\LambdatildeRplus/(p_1(\mu))$ is a $\delta\textrm{-ring}$, we have that
	\begin{equation*}
		\delta(\tau) = \delta\big(\tfrac{(q-1)^2}{8q}\big) = \tfrac{1}{2q^2}\big(\tfrac{(q^2-1)^2}{8} - \tfrac{(q-1)^4}{64}\big) = \tfrac{1}{2q}\big(\tau(q+1)^2 - \tfrac{\tau^2}{q}\big).
	\end{equation*}
	As $[2]_q = q+1$ may be written as the product of $2$ with a unit in $\Lambda_F$ (see Lemma \ref{lem:tovermu_unit}), therefore, it follows that $\tfrac{\tau^2}{2} = \tfrac{q\tau(q+1)^2}{2} - q^2\delta(\tau)$ is an element of $\LambdatildeRplus/(p_1(\mu))$.
	Hence, we conclude that $\LambdatildeRplus/(p_1(\mu)) \isomorphic \LambdaRplus$, in particular, the composition in \eqref{eq:lambdatildeplus_mod_p1mu} and the left vertical arrow in \eqref{eq:lambdatildeplus_mod} are bijective.
	Now, using the diagram \eqref{eq:lambdatildeplus_mod} and an easy induction on $n \geqslant 1$ shows that the right vertical arrow is bijective, the bottom right horizontal arrow is surjective and the middle vertical arrow is bijective, for each $n \geqslant 1$.
	Hence, the natural $(\varphi, \Gamma_R \times \Gamma_0)\equivariant$ map $\LambdatildeRplus/(p_1(\mu))^n \rightarrow (\LambdatildeR/(p_1(\mu))^n)^{1 \times \Gamma_{\textrm{tor}}}$, induced by \eqref{eq:lambdatilder_pm}, is bijective for each $n \geqslant 1$, thus allowing us to conclude.
\end{proof}

\begin{rem}
	Lemma \ref{lem:lambdatildeplus_mod_p1mun} may be proven using an alternative method as in the proof of Lemma \ref{lem:mlambdatildeplus_mod_p1mun}, where a crucial input is Lemma \ref{lem:sigma_action_surjective}.
	However, the proof given above is conceptually more satisfactory.
\end{rem}

From Lemma \ref{lem:lambdatildeplus_mod_p1mun}, we obtain the following:
\begin{lem}\label{lem:a1_pm_exact}
	For each $n \geqslant 1$, the following natural $(\varphi, \Gamma_R \times \Gamma_0)\equivariant$ sequence is exact:
	\begin{equation}\label{eq:a1_pm_exact}
		0 \longrightarrow (\LambdatildeR/(p_1(\mu)))^{1 \times \Gamma_{\textup{tor}}} \xrightarrow{\hspace{1mm} p_1(\mu)^n \hspace{1mm}} (\LambdatildeR/(p_1(\mu))^{n+1})^{1 \times \Gamma_{\textup{tor}}} \longrightarrow (\LambdatildeR/(p_1(\mu))^n)^{1 \times \Gamma_{\textup{tor}}} \longrightarrow 0.
	\end{equation}
\end{lem}
\begin{proof}
	The sequence \eqref{eq:a1_pm_exact} is the same as the second row of the diagram \eqref{eq:lambdatildeplus_mod_p1mun}, which was shown to be exact in the proof of Lemma \ref{lem:lambdatildeplus_mod_p1mun}.
\end{proof}

Next, we will look at the action of $1 \times \Gamma_0 \isomorphic 1 \times (1 + 4\ZZ_2)$ on $\LambdatildeRplus$ and prove a result similar to Lemma \ref{lem:a1_gamma0_exact}, for $p=2$.
In particular, we will show the following:
\begin{lem}\label{lem:a1_gamma0_exact_2}
	For each $n \geqslant 1$, the following natural $(\varphi, \Gamma_R \times 1)\equivariant$ sequence is exact:
	\begin{equation}\label{eq:a1_gamma0_exact_2}
		0 \longrightarrow (\LambdatildeR/(p_1(\mu)))^{1 \times \Gamma_F} \xrightarrow{\hspace{1mm} p_1(\mu)^n \hspace{1mm}} (\LambdatildeR/(p_1(\mu))^{n+1})^{1 \times \Gamma_F} \longrightarrow (\LambdatildeR/(p_1(\mu))^{n})^{1 \times \Gamma_F} \longrightarrow 0.
	\end{equation}
\end{lem}

\begin{rem}\label{rem:a1_gammaF_exact_2}
	Via the natural $(\varphi, \Gamma_R \times \Gamma_F)\equivariant$ isomorphism $\iota_{\Lambdatilde}$ in Lemma \ref{lem:iota_lambdatilde_iso} (see \eqref{eq:iota_lambdatilde}), we see that the exact sequence in \eqref{eq:a1_gamma0_exact_2} is the $(1 \times \Gamma_F)\textrm{-invariant}$ of the exact sequence in \eqref{eq:a1_geometric_exact} in Lemma \ref{lem:a1_geometric_exact}, for $p = 2$.
\end{rem}

Note that using the $(\varphi, \Gamma_R \times \Gamma_0)\equivariant$ isomorphism in \eqref{eq:lambdatildeplus_mod_p1mun} and \eqref{eq:lambdatildeplus_mod_p1mu}, the sequence in \eqref{eq:a1_gamma0_exact_2} may be rewritten as the following $(\varphi, \Gamma_R \times \Gamma_0)\equivariant$ sequence:
\begin{equation}\label{eq:lambdatildeplus_gamma0_exact}
	0 \longrightarrow \LambdaRplus^{\Gamma_0} \xrightarrow{\hspace{1mm} p_1(\mu)^n \hspace{1mm}} (\LambdatildeRplus/(p_1(\mu))^{n+1})^{1 \times \Gamma_0} \longrightarrow (\LambdatildeRplus/(p_1(\mu))^{n})^{1 \times \Gamma_0} \longrightarrow 0.
\end{equation}
In order to show that the sequence \eqref{eq:lambdatildeplus_gamma0_exact} is exact, we start with the following observation:
\begin{lem}\label{lem:lambdatilde+_triv_modnu}
	The action of $1 \times \Gamma_0$ is trivial on $\LambdatildeRplus/(p_2(\nu))$ and the action of $\Gamma_R \times 1$ is trivial on $\LambdatildeRplus/(p_1(\mu))$.
\end{lem}
\begin{proof}
	For the first claim, note that we have $p_2(\nu)\LambdatildeR \cap \LambdatildeRplus = p_2(\nu)\LambdatildeRplus$.
	So, if $x$ is an element of $\LambdatildeRplus$ and $g$ any element of $1 \times \Gamma_0$, then it is enough to show that $(g-1)x$ is an element of $p_2(\nu)\LambdatildeR$.
	Moreover, recall that $\nu$ is the product of $\mu^2$ with a unit in $\AF$.
	Therefore, we are reduced to showing that $(g-1)x$ is an element of $p_2(\mu)^2\LambdatildeR$.
	Now, using Lemma \ref{lem:lambdatildeR_triv_modmu}, we may write $(g-1)x = p_2(\mu)y$, for some $y$ in $\LambdatildeR$.
	Let $\sigma$ be a generator of $1 \times \Gamma_{\textrm{tor}}$ and note that $\sigma(x) = x$.
	Then, we have $\sigma(p_2(\mu)) \sigma(y) = p_2(\mu) y$, in particular, $(\sigma-1)y = -(2+p_2(\mu))y$.
	Again, using Lemma \ref{lem:lambdatildeR_triv_modmu}, we may write $-p_2([p]_q)y = (\sigma-1)y = p_2(\mu)z$, for some $z$ in $\LambdatildeR$.
	So, we get that $-py = 0 \textrm{ mod }p_2(\mu) \LambdatildeR$.
	Note that $(p_2(\mu), p)$ is a regular sequence on $\LambdatildeR$, since $p_2 \colon \AF \rightarrow \LambdatildeR$ is flat (see Construction \ref{const:lambdatilder}).
	Therefore, we conclude that $y = 0 \textrm{ mod }p_2(\mu) \LambdatildeR$, i.e.\ $y$ is an element of $p_2(\mu)\LambdatildeR$ and $(g-1)x = p_2(\mu)y$ is an element of $p_2(\mu)^2\LambdatildeR$, as claimed.
	The triviality of the action of $\Gamma_R \times 1$ on $\LambdatildeRplus/(p_1(\mu))$ easily follows from Lemma \ref{lem:lambdatildeR_triv_modmu}.
\end{proof}

\begin{rem}\label{rem:gamma0_act_lambdarplus}
	From Lemma \ref{lem:lambdatilde+_triv_modnu}, note that the action of $1 \times \Gamma_0$ is trivial on $\LambdatildeRplus/(p_2(\nu))$, and multiplication by $p_1(\mu)$ on $\LambdatildeRplus$ is equivariant for this action.
	Therefore, for any $g$ in $1 \times \Gamma_0$ and any $x$ in $\LambdatildeRplus/(p_1(\mu))^n$, we get that $(g-1)x$ is an element of $p_2(\nu)\LambdatildeRplus/(p_1(\mu))^n$.
	In particular, for $n=1$, using the isomorphism $\LambdatildeRplus/(p_1(\mu)) \isomorphic \LambdaRplus$ from \eqref{eq:lambdatildeplus_mod_p1mu}, we get that for any $g$ in $\Gamma_0$ and any $x$ in $\LambdaRplus$, the element $(g-1)x$ belongs to $\nu\LambdaRplus$.
\end{rem}

Using the action of $1 \times \Gamma_0$ on $\LambdatildeRplus$, let us define a $q\textrm{-de Rham}$ complex (see Definition \ref{defi:qdeRham_complex}), by considering the element $\tautilde$ in $\LambdatildeRplus$ as a parameter.
We start with the following observation:
\begin{lem}\label{lem:gamma0_act_tautilde}
	Let $\gamma_0$ be any element of $1 \times \Gamma_0$.
	Then, $(\gamma_0-1)\tautilde = u \hspace{0.5mm} p_2(\nu)$, for some unit $u$ in $\LambdatildeRplus$ depending on $\gamma_0$.
\end{lem}
\begin{proof}
	Note that it is enough to show the claim for a topological generator $\gamma_0$ of $1 \times \Gamma_0$ such that $\chi(\gamma_0) = 1 + 4a$, for a unit $a$ in $\ZZ_2$.
	Let us set $v \coloneq \tfrac{p_2([p]_q)}{p_1([p]_q)}$ and note that we have $\tautilde = \tfrac{\delta(v)}{p_2(q)}$.
	Now, observe that
	\begin{align*}
		(\gamma_0-1)\tautilde = (\gamma_0-1)\big(\tfrac{\delta(v)}{p_2(q)}\big) &= (\gamma_0-1)(\tfrac{1}{p_2(q)}) \delta(v) + \tfrac{(\gamma_0-1)\delta(v)}{\gamma_0(p_2(q))}\\
		&= \tfrac{1}{p_2(q^5)} ((\gamma_0-1)\delta(v) - (p_2(q^4)-1)\delta(v)).
	\end{align*}
	Since $\delta$ and $\gamma_0$ commute with each other, an easy computation shows that
	\begin{align*}
		(\gamma_0-1)\delta(v) &= \delta((\gamma_0-1)v) - v(\gamma_0-1)v,\\
		(p_2(q^4)-1)\delta(v) &= \delta((p_2(q^4)-1)v) - v^2(p_2(q^2)-1).
	\end{align*}
	Then, we note that
	\begin{align*}
		v(\gamma_0-1)v - v^2(p_2(q^2)-1) &= v(p_2(q^2)-1)\big(\tfrac{p_2(q^3)+p_2(q)}{p_1(q)+1} - v) = -v^2(p_2(q)-1)(p_2(q^3)-1)\\
		&= -v^2 p_2(q)(p_2(q^2) + p_2(q) + 1)p_2(\nu) = p_2(\nu) x_1,
	\end{align*}
	for some $x_1$ in $\LambdatildeR$ and $(\gamma_0-1)v = p_2(q)(p_2(q)-1)(p_2(q^2)+1)v$.

	Now, let $a = p_2(q)(p_2(q^2)+1)$, $b = p_2(q)+1$ and $c = (p_2(q)-1)v$, and note that $a-b$, $\delta(a-b)$ and $\delta(c)$ are elements of $(p_2(q)-1)\LambdatildeR$, since the latter is a $\delta\textrm{-stable}$ ideal of $\LambdatildeR$, in the sense of \cite[Example 2.10]{bhatt-scholze-prisms}.
	So, we obtain that
	\begin{align*}
		\delta((\gamma_0-1)v) - \delta((p_2(q^4)-1)v) &= \delta(ac) - \delta(bc) = \delta(ac-bc) + b^2c^2 + abc^2\\
		&= \delta(a-b) c^2 + (a-b)^2\delta(c) + 2\delta(a-b)\delta(c) + b^2c^2 + abc^2 = p_2(\nu) x_2,
	\end{align*}
	for some $x_2$ in $\LambdatildeR$, and in the third equality we have used that $\nu = \tfrac{(q-1)^2}{q}$ from \eqref{eq:nu_defi}.
	Set $u \coloneq \tfrac{x_1+x_2}{p_2(q^5)}$ in $\LambdatildeR$, and by putting everything together, we have that
	\begin{equation}\label{eq:gamma0_act_tautilde}
		(\gamma_0-1)\tautilde = up_2(\nu).
	\end{equation}
	Using that $\sigma(\tautilde) = \tautilde$ (see Lemma \ref{lem:tautilde}), $\sigma(\nu) = \nu$, the group $1 \times \Gamma_F$ is commutative and $\LambdatildeRplus$ is $p_2(\nu)\torsionfree$, we get that $u$ is an element of $\LambdatildeRplus$.
	So to show the claim, it is enough to show that $u$ is a unit in $\LambdatildeR$.
	Now, note that from the discussion after \eqref{eq:lambdatilder_pm}, the ring $\LambdatildeRplus$ is $p_1(\mu)\adically$ complete and we have a $(\varphi, \Gamma_0)\equivariant$ isomorphism $\LambdatildeRplus/(p_1(\mu)) \isomorphic \LambdaRplus$ (see \eqref{eq:lambdatildeplus_mod_p1mu} in Lemma \ref{lem:lambdatildeplus_mod_p1mun}), therefore, we are reduced to showing that $\overline{u}$, the image of $u$ under the preceding isomorphism, is a unit.
	By reducing the equalities in \eqref{eq:gamma0_act_tautilde}, modulo $p_1(\mu)$, we obtain the following expression in $\LambdaRplus$:
	\begin{equation*}
		(\gamma_0-1)\tau = (\gamma_0-1)\tfrac{\nu}{8} = \overline{u} \nu.
	\end{equation*}
	But from Lemma \ref{lem:gamma0_act_s}, we see that $\overline{u}$ must be a unit in $\LambdaRplus$.
	Hence, $u$ is a unit in $\LambdatildeRplus$, as claimed.
\end{proof}

The following observation was used above:
\begin{lem}\label{lem:gamma0_act_tau}
	Let $\gamma_0$ be a topological generator of $\Gamma_0$ such that $\chi(\gamma_0) = 1+4a$, for some unit $a$ in $\ZZ_2$.
	Then, we have $(\gamma_0-1)\tau = (\gamma_0-1)\tfrac{\nu}{8} = u\nu$, for some unit $u$ in $\LambdaFplus = \Lambda_F^{\Gamma_{\textup{tor}}}$ depending on $\gamma_0$.
\end{lem}
\begin{proof}
	From Lemma \ref{lem:lambdar_pm_nu}, recall that $\nu/8$ is the product of $t^2/8$ with a unit in $\LambdaFplus$.
	So, let us write $\nu/8 = e t^2/8$, for some unit $e$ in $\LambdaFplus$.
	Note that from Remark \ref{rem:gamma0_act_lambdarplus}, we have that $(\gamma_0-1) e = \nu x$, for some $x$ in $\LambdaFplus$.
	Therefore,
	\begin{align*}
		(\gamma_0-1)\tfrac{\nu}{8} = (\gamma_0-1)\tfrac{et^2}{8} &= \tfrac{t^2}{8}(\gamma_0-1)e  + \gamma_0(e) (\gamma_0-1)\tfrac{t^2}{8}\\
			&= \tfrac{t^2}{8} \nu x  + \gamma_0(e) (\chi(\gamma_0)^2-1)\tfrac{t^2}{8} = \nu \big(\tfrac{t^2}{8}x + \gamma_0(e) e^{-1}a(2a+1)\big) = \nu u,
	\end{align*}
	where $u = \big(\tfrac{t^2}{8}x + \gamma_0(e) e^{-1}a(2a+1)\big)$ is a unit in $\LambdaFplus$ because $\gamma_0(e) e^{-1}a(2a+1)$ is a unit and $t^2/8$ is $p\adically$ nilpotent in $\LambdaFplus$.
	Hence, the lemma is proved.
\end{proof}

In the rest of the Section \ref{subsubsec:p=2_gammaF_action}, we will fix a topological generator $\gamma_0$ of $1 \times \Gamma_0$ such that $\chi(\gamma_0) = 1+4a$, for a unit $a$ in $\ZZ_2$.
Similar to \eqref{eq:nablaq_stilde}, let us now consider the following operator on $\LambdatildeRplus$:
\begin{equation}\label{eq:nablaq_tautilde}
	\begin{aligned}
		\nabla_{q, \tautilde} \colon \LambdatildeRplus &\longrightarrow \LambdatildeRplus\\
					x &\longmapsto \tfrac{(\gamma_0-1)x}{(\gamma_0-1)\tautilde}.
	\end{aligned}
\end{equation}
From the triviality of the action of $1 \times \Gamma_0$ on $\LambdatildeRplus/(p_2(\nu))$ (see Lemma \ref{lem:lambdatilde+_triv_modnu}) and from Lemma \ref{lem:gamma0_act_tautilde}, it follows that the operator $\nabla_{q, \tautilde}$ is well defined.
For each $n \geqslant 1$, using Remark \ref{rem:gamma0_act_lambdarplus}, the operator in \eqref{eq:nablaq_tautilde} induces well-defined operators $\nabla_{q, \tautilde} \colon \LambdatildeRplus/(p_1(\mu))^n \longrightarrow \LambdatildeRplus/(p_1(\mu))^n$.
In particular, for $n = 1$, we have $\tau = \nu/8$ in $\Lambda_{R, 0}$, and using Remark \ref{rem:gamma0_act_lambdarplus} and Lemma \ref{lem:gamma0_act_tau}, we have a well-defined operator
\begin{equation}\label{eq:nablaq_tau}
	\begin{aligned}
		\nabla_{q, \tau} \colon \LambdaRplus &\longrightarrow \LambdaRplus\\
				x &\longmapsto \tfrac{(\gamma_0-1)x}{(\gamma_0-1)\tau}.
	\end{aligned}
\end{equation}

\begin{rem}\label{rem:nablaq_tautilde_qconnection}
	Considering $\tautilde$ as a parameter, the operator $\nabla_{q, \tautilde}$ in \eqref{eq:nablaq_stilde}, may be considered as a $q\textrm{-differential}$ operator in non-logarithmic coordinates, in the sense of Definition \ref{defi:qdeRham_complex} and Remark \ref{rem:nablaqi_nota}, where $q\Omega^1_{A/D}$ identifies with the $(p, p_1(\mu))\adic$ completion of the module of K\"ahler differentials of $\LambdatildeRplus$, with respect to $p_1 \colon \AR \rightarrow \LambdatildeRplus$.
	Similarly, considering $\tau$ as a variable, the operator $\nabla_{q, \tau}$ in \eqref{eq:nablaq_tau}, may also be considered as a non-logarithmic $q\textrm{-differential}$ operator in the sense of Definition \ref{defi:qdeRham_complex} and Remark \ref{rem:nablaqi_nota}, where the $q\Omega^1_{A/D}$ identifies with the $p\adic$ completion of the module of K\"ahler differentials of $\LambdaRplus$, with respect to $R$.
\end{rem}

For each $n \geqslant 1$, the operator $\nabla_{q, \tautilde}$ is an endomorphism of $\LambdatildeRplus/(p_1(\mu))^n$, so we may define the following two term Koszul complex:
\begin{equation}\label{eq:tautilde_qderham}
	K_{\LambdatildeRplus/(p_1(\mu))^n}(\nabla_{q, \tautilde}) \colon \big[\LambdatildeRplus/(p_1(\mu))^n \xrightarrow{\hspace{1mm} \nabla_{q, \tautilde} \hspace{1mm}} \LambdatildeRplus/(p_1(\mu))^n\big].
\end{equation}
For $n=1$, we have the following claim:
\begin{lem}\label{lem:h1_qderham_lambdarplus}
	The cohomology of the complex $K_{\LambdaRplus}(\nabla_{q, \tau})$ vanishes in degree 1, i.e.\ $H^1(K_{\LambdaRplus}(\nabla_{q, \tau})) = 0$.
\end{lem}
\begin{proof}
	Similar to Lemma \ref{lem:h1_qderham_lambdar0}, the claim follows from the more general statement in Propositions \ref{prop:mlambda+_connection} and \ref{prop:lambdar+_comp} (also see Propositions \ref{prop:mlambda0_connection} and \ref{prop:abarone_comp}), which will be proven for finite $\Lambda_{R, +}\modules$ admitting a continuous action of $\Gamma_0$ (trivial modulo $\nu$), in particular, $\Lambda_{R, +}$ itself.
\end{proof}

\begin{proof}[Proof of Lemma \ref{lem:a1_gamma0_exact_2}]
	The proof follows by the same arguments as in the proof of Lemma \ref{lem:a1_gamma0_exact}.
	Indeed, by using the $(\varphi, \Gamma_R \times \Gamma_0)\equivariant$ isomorphism in \eqref{eq:lambdatildeplus_mod_p1mun}, in the exact sequence of \eqref{eq:a1_pm_exact}, we obtain an exact sequence (similar to \eqref{eq:proof_fpx_exact}) on which we use the operator $\nabla_{q, \tautilde}$ in \eqref{eq:nablaq_tautilde} and the Koszul complex defined in \eqref{eq:tautilde_qderham} to obtain an exact sequence of Koszul complexes (similar to \eqref{eq:proof_koszul_exact}).
	Considering the associated long exact sequence, and noting that $H^1(K_{\LambdaRplus}(\nabla_{q, \tau})) = 0$ from Lemma \ref{lem:h1_qderham_lambdarplus}, we obtain the following exact sequence (similar to \eqref{eq:proof_horizontal_exact}):
	\begin{equation*}
		0 \longrightarrow \LambdaRplus^{\nabla_{q, \tau}=0} \xrightarrow{\hspace{1mm} p_1(\mu)^n \hspace{1mm}} (\LambdatildeRplus/(p_1(\mu))^{n+1})^{\nabla_{q, \tautilde}=0} \longrightarrow (\LambdatildeRplus/(p_1(\mu))^{n})^{\nabla_{q, \tautilde}=0} \longrightarrow 0.
	\end{equation*}
	Since the action of $1 \times \Gamma_0$ is continuous on $\LambdatildeRplus$ for the $(p, p_1(\mu))\adic$ topology, therefore from Lemma \ref{lem:cont_coh_disc}, we have that $(\LambdatildeRplus/(p_1(\mu))^{n+1})^{\nabla_{q, \tautilde}=0} = (\LambdatildeRplus/(p_1(\mu))^{n+1})^{1 \times \Gamma_0}$, for each $n \in \NN$.
	Hence, from the preceding exact sequence we obtain that the sequence in \eqref{eq:lambdatildeplus_gamma0_exact} is exact, and therefore, it follows that the sequence in \eqref{eq:a1_gamma0_exact_2} is also exact.
	This completes our proof.
\end{proof}

\subsubsection{Proof of Proposition \ref{prop:delta_modmun}}

Note that from the explicit description of $\AR(1)/(p_1(\mu))$ in Proposition \ref{prop:a1modp1mu_pdring}, it is easy to see that reduction modulo $\mu$ of $p_1 \colon \AR \rightarrow \AR(1)$ induces an isomorphism $p_1 \colon R \isomorphic (\AR(1)/(p_1(\mu)))^{1 \times \Gamma_R}$.
More generally, we have the following:
\begin{lem}\label{lem:a1modp1mun}
	For $n \in \NN_{\geqslant 1}$, reduction modulo $\mu^n$ of the $(\varphi, \Gamma_R^2)\equivariant$ map $p_1 \colon \AR \rightarrow \AR(1)$ induces a $(\varphi, \Gamma_R \times 1)\equivariant$ isomorphism $p_1 \colon \AR/(\mu)^n \isomorphic (\AR(1)/(p_1(\mu))^n)^{1 \times \Gamma_R}$.
\end{lem}
\begin{proof}
	Note that $p_1$ is $(\varphi, \Gamma_R^2)\equivariant$, and $p_1$ modulo $\mu^n$ factors as $p_1 \colon \AR/(\mu)^n \rightarrow (\AR(1)/(p_1(\mu))^n)^{1 \times \Gamma_R} \hookrightarrow \AR(1)/(p_1(\mu))^n$, with the composition being faithfully flat by Lemma \ref{lem:pi_rj_ff}, and therefore, injective.
	In particular, we see that the induced $(\varphi, \Gamma_R \times 1)\equivariant$ homomorphism $p_1 \colon \AR/(\mu)^n \rightarrow (\AR(1)/(p_1(\mu))^n)^{1 \times \Gamma_R}$ is injective.
	Now, consider the following commutative diagram with exact rows:
	\begin{equation}\label{eq:a1modp1mun}
		\begin{tikzcd}
			0 \arrow[r] & \AR/(\mu) \arrow[r, "\mu^n"] \arrow[d, "p_1", "\wr"'] & \AR/(\mu)^{n+1} \arrow[r] \arrow[d, "p_1"] & \AR/(\mu)^n \arrow[r] \arrow[d, "p_1"] & 0\\
			0 \arrow[r] & (\AR(1)/(p_1(\mu)))^{1 \times \Gamma_R} \arrow[r, "p_1(\mu)^n"] & (\AR/(p_1(\mu))^{n+1})^{1 \times \Gamma_R} \arrow[r] & (\AR(1)/(p_1(\mu))^n)^{1 \times \Gamma_R},
		\end{tikzcd}
	\end{equation}
	and note that an easy induction on $n \geqslant 1$ shows that the right vertical arrow is bijective, the bottom right horizontal arrow is surjective and the middle vertical arrow is bijective, for each $n \geqslant 1$.
	Hence, we obtain the claimed $(\varphi, \Gamma_R \times 1)\equivariant$ isomorphism $p_1 \colon \AR/(\mu)^n \isomorphic (\AR(1)/(p_1(\mu))^n)^{1 \times \Gamma_R}$, for each $n \geqslant 1$.
\end{proof}

\begin{rem}\label{rem:a2modp1mun}
	From the description of the action of $\Gamma_R^3$ on $\AR(2)$ in Remark \ref{rem:phigamma_act_ar+n}, we note that there is an induced action of $\Gamma_R^3$ on $\AR(2)/(p_1(\mu))$, where the action of the first component is the identity.
	Moreover, we have a $(\varphi, \Gamma_R^3)\equivariant$ map $r_1 \colon \AR \rightarrow \AR(2)$, where $\AR$ is equipped with an action of $\Gamma_R^3$ via projection on to the first coordinate.
	Then, similar to above, it may easily be shown that the reduction modulo $\mu$ of $r_1 \colon \AR \rightarrow \AR(2)$ induces an isomorphism $r_1 \colon R \isomorphic (\AR(2)/(p_1(\mu)))^{1 \times \Gamma_R \times \Gamma_R}$.
\end{rem}

Now, recall that we have the $\varphi\equivariant$ multiplication map $\Delta \colon \AR(1) \rightarrow \AR$.
The map $\Delta$ induces an $\AR/(\mu)^n\linear$ (via $p_1$) and $\varphi\equivariant$ maps $\Delta \colon \AR(1)/(p_1(\mu))^n \rightarrow \AR/(\mu)^n$, for $n \in \NN_{\geqslant 1}$.
For $n=1$, using Lemma \ref{lem:a1modp1mun}, the map $\Delta$ restricts to a $\varphi\equivariant$ isomorphism $(\AR(1)/(p_1(\mu)))^{1 \times \Gamma_R} \isomorphic R$.
More generally, we have the following:
\begin{lem}\label{lem:delta_modmun}
	For $n \in \NN_{\geqslant 1}$, reduction modulo $p_1(\mu)^n$ of the $\varphi\equivariant$ homomorphism $\Delta \colon \AR(1) \rightarrow \AR$ restricts to a $\varphi\equivariant$ isomorphism $(\AR(1)/(p_1(\mu))^n)^{1 \times \Gamma_R} \isomorphic \AR/(\mu)^n$.
\end{lem}
\begin{proof}
	Let us first note that using Lemma \ref{lem:a1_geometric_exact}, together with Lemma \ref{lem:a1_fpx_exact}, Lemma \ref{lem:a1_gamma0_exact} and Remark \ref{rem:a1_gammaF_exact}, for $p \geqslant 3$, and Lemma \ref{lem:a1_pm_exact}, Lemma \ref{lem:a1_gamma0_exact_2} and Remark \ref{rem:a1_gammaF_exact_2}, for $p = 2$, we obtain that for each $n \geqslant 1$, the following $(\varphi, \Gamma_R \times 1)\equivariant$ sequence is exact:
	\begin{equation}\label{eq:a1_gammaR_exact}
		0 \longrightarrow (\AR(1)/(p_1(\mu)))^{1 \times \Gamma_R} \xrightarrow{\hspace{1mm} p_1(\mu)^n \hspace{1mm}} (\AR(1)/(p_1(\mu))^{n+1})^{1 \times \Gamma_R} \longrightarrow (\AR(1)/(p_1(\mu))^{n})^{1 \times \Gamma_R} \longrightarrow 0.
	\end{equation}
	Next, we note that $\Delta$ is $\varphi\equivariant$, so it is enough to show that the map modulo $p_1(\mu)^n$ is bijective.
	Now, consider the following natural commutative diagram:
	\begin{center}
		\begin{tikzcd}
			0 \arrow[r] & (\AR(1)/(p_1(\mu)))^{1 \times \Gamma_R} \arrow[r, "p_1(\mu)^n"] \arrow[d, "\Delta", "\wr"'] & (\AR/(p_1(\mu))^{n+1})^{1 \times \Gamma_R} \arrow[r] \arrow[d, "\Delta"] & (\AR(1)/(p_1(\mu))^n)^{1 \times \Gamma_R} \arrow[d, "\Delta"] \arrow[r] & 0\\
			0 \arrow[r] & \AR/(\mu) \arrow[r, "\mu^n"] & \AR/(\mu)^{n+1} \arrow[r] & \AR/(\mu)^n \arrow[r] & 0,
		\end{tikzcd}
	\end{center}
	where the top row is the exact sequence in \eqref{eq:a1_gammaR_exact}.
	Using the diagram, an easy induction on $n \geqslant 1$ gives the $\varphi\equivariant$ isomorphism $\Delta \colon (\AR(1)/(p_1(\mu))^{n+1})^{1 \times \Gamma_R} \isomorphic \AR/(\mu)^{n+1}$, as claimed.
\end{proof}

Finally, recall that the ring $\AR(1)$ is equipped with an action of $(\varphi, \Gamma_R^2)$, and the rings $\AR(1)^{1 \times \Gamma_R}$ and $(\AR(1)/(p_1(\mu))^n)^{1 \times \Gamma_R}$ are equipped with a residual action of $\Gamma_R = \Gamma_R \times 1 \subset \Gamma_R^2$.
Then, the following observation proves Proposition \ref{prop:delta_modmun}:
\begin{lem}\label{lem:delta_modmun_phigamma}
	For $n \in \NN_{\geqslant 1}$ the isomorphism $(\AR(1)/(p_1(\mu))^n)^{1 \times \Gamma_R} \isomorphic \AR/(\mu)^n$ of Proposition \ref{prop:delta_modmun}, is compatible with the respective $(\varphi, \Gamma_R)\action$.
	Passing to the limit over $n$ gives a $(\varphi, \Gamma_R)\equivariant$ isomorphism $\AR(1)^{1 \times \Gamma_R} \isomorphic \AR$.
\end{lem}
\begin{proof}
	The isomorphism in Proposition \ref{prop:delta_modmun} is $\varphi\equivariant$.
	To check $\Gamma_R\textrm{-equivariance}$, note that if $g$ is in $\Gamma_R$ and $a$ in $\AR(1)/(p_1(\mu))^n$, then we have that $\Delta((g, g)a) = g(a)$.
	So if $a$ is $(1 \times \Gamma_R)\textrm{-invariant}$, then for $g_1$, $g_2$ in $\Gamma_R$, we have that $\Delta((g_1, g_2)f) = \Delta((g_1, g_1)f) = g_1(\Delta(f))$.
	This proves the first claim.
	Next, as inverse limit commutes with taking $(1 \times \Gamma_R)\textrm{-invariants}$, therefore, it follows that we have a $(\varphi, \Gamma_R)\equivariant$ isomorphism $\AR(1)^{1 \times \Gamma_R} = (\lim_n \AR(1)/(p_1(\mu))^n)^{1 \times \Gamma_R} = \lim_n (\AR(1)/(p_1(\mu))^n)^{1 \times \Gamma_R} \isomorphic \lim_n \AR/(\mu)^n = \AR$.
	This proves the second claim.
\end{proof}

\section{An integral comparison isomorphism}\label{sec:integral_comp}

In this section, we shall prove an integral comparison isomorphism for Wach modules, which will be the most important input in building a stratification on Wach modules in Section \ref{subsubsec:wachmod_strat}.
We will use the setup and notations of Section \ref{subsec:setup_nota} and Section \ref{sec:theprism_ar+}.
\begin{defi}[Wach modules, {\cite[Definition 1.1, Lemma 3.7]{abhinandan-relative-wach-ii}}]\label{defi:wach_mods_relative}
	A \textit{Wach module} over $\AR$ is a finitely generated $\AR\textrm{-module}$ $N$ equipped with a semilinear action of $\Gamma_R$, and satisfying the following conditions:
	\begin{enumarabicup}
		\item The sequences $\{p, \mu\}$ and $\{\mu, p\}$ are regular on $N$.

		\item The action of $\Gamma_R$ is trivial on $N/\mu N$.

		\item $N$ is equipped with a Frobenius of finite $\pqheight$, i.e.\ an $\AR\linear$ and $\Gamma_R\equivariant$ isomorphism $\varphi_N \colon (\varphi^*N)[1/[p]_q] = (\AR \otimes_{\varphi, \AR} N)[1/[p]_q] \isomorphic N[1/[p]_q]$.
	\end{enumarabicup}
	Say that $N$ is \textit{effective} if $\varphi_N$ carries $\varphi^*N$ into $N$.
	Denote the category of Wach modules over $\AR$ as $(\varphi, \Gamma)\Mod_{\AR}^{[p]_q}$ with morphisms between objects being $\AR\linear$, $\Gamma_R\equivariant$ and $\varphi_N\equivariant$ (after inverting $[p]_q$).
\end{defi}

\begin{rem}\label{rem:wachmod_props}
	Note that the action of $\Gamma_R$ is automatically continuous on $N$ for the $(p, \mu)\adic$ topology (see \cite[Lemma 3.4]{abhinandan-relative-wach-ii}).
	Moreover, from \cite[Proposition 3.8]{abhinandan-relative-wach-ii}, note that for a Wach module $N$ over $\AR$, the $\AR[1/p]\module$ $N[1/p]$ is finite projective, the $\AR[1/\mu]\module$ $N[1/\mu]$ is finite projective and by \cite[Remark 3.9]{abhinandan-relative-wach-ii} the $\AR[1/[p]_q]\module$ $N[1/[p]_q]$ is finite projective.
	Furthermore, from loc.\ cit., the sequences $\{p, [p]_q\}$ and $\{[p]_q, p\}$ are regular on $N$ and equivalent to condition (1) in Definition \ref{defi:wach_mods_relative}.
\end{rem}

\begin{rem}
	For $R = O_F$, a Wach module over $\AF$ is necessarily finite free (see \cite[Remark 1.4]{abhinandan-relative-wach-ii}).
\end{rem}

From Remark \ref{rem:phigamma_act_ar+n}, recall that we have a $(\varphi, \Gamma_R^2)\equivariant$ maps $p_i \colon \AR \rightarrow \AR(1)$ for $i = 1, 2$, where $\AR$ is equipped with a $\Gamma_R^2\action$ via projection onto the $i^{\textrm{th}}$ coordinate.
Moreover, we note that there is an induced action of $\Gamma_R^2$ on $\AR(1)/(p_1(\mu))$, where the action of the first component is trivial.
In this section, we shall identify $\Gamma_R$ with $1 \times \Gamma_R$ and say that $\AR(1)/(p_1(\mu))$ is equipped with a natural continuous action of $\Gamma_R$.

\begin{nota}
	In this section, by the $\varphi\textrm{-equivariance}$ of a morphism we always mean $\varphi\textrm{-equivariance}$ after inverting $p$.
	However, we will not always mention this explicitly.
\end{nota}

The goal of this section is to prove Theorem \ref{thm:integral_comp_relative} below, which is an important ingredient for the proof of Theorem \ref{thm:strat_wach_comp_relative}.
\begin{thm}\label{thm:integral_comp_relative}
	Let $N$ be Wach module over $\AR$ and set $M \coloneq (\AR(1)/(p_1(\mu)) \otimes_{p_2, \AR} N)^{\Gamma_R}$ as an $R\module$ equipped with the tensor product Frobenius.
	Then, there exists a natural $(\varphi, \Gamma_R)\equivariant$ isomorphism
	\begin{equation}\label{eq:integral_comp_relative}
		\AR(1)/(p_1(\mu)) \otimes_{p_1, R} M \isomorphic \AR(1)/(p_1(\mu)) \otimes_{p_2, \AR} N.
	\end{equation}
	Moreover, $M$ is a finitely generated $p\torsionfree$ $R\module$, and the multiplication map $\Delta \colon \AR(1) \rightarrow \AR$ from Construction \ref{const:cech_nerve} induces a $\varphi\equivariant$ isomorphism of $R\modules$ $M \isomorphic N/\mu N$.
\end{thm}

\begin{rem}
	For $R = O_F$, the $R\module$ $M$ in Theorem \ref{thm:integral_comp_relative} is finite free of rank $= \rank_{\AF} N$.
\end{rem}

\begin{rem}\label{rem:frob_mpinverse}
	Let $N$ be a Wach module over $\AR$.
	By extending the isomorphism in Definition \ref{defi:wach_mods_relative} (3) along the $(\varphi, \Gamma_R)\equivariant$ homomorphism $p_2 \colon \AR \rightarrow \AR(1)/(p_1(\mu))$, we obtain a $(\varphi, \Gamma_R)\textrm{-equivariant}$ isomorphism $(\AR(1)/(p_1(\mu)) \otimes_{p_2 \circ \varphi, \AR} N)[1/[p]_q] \isomorphic (\AR(1)/(p_1(\mu)) \otimes_{p_2, \AR} N)[1/[p]_q]$.
	Now, let us note that $p_1([p]_q) = p \mod p_1(\mu) \AR(1)$ and $p_1([p]_q)/p_2([p]_q)$ is a unit in $\AR(1)$ (see Lemma \ref{lem:pi_rj_ff}).
	So, by setting $M \coloneq (\AR(1)/(p_1(\mu)) \otimes_{p_2, \AR} N)^{\Gamma_R}$, taking the $\Gamma_R\textrm{-invariants}$ in the preceding isomorphism, and using Theorem \ref{thm:integral_comp_relative} gives an $R[1/p]\linear$ isomorphism $(\varphi^*M)[1/p] = (R \otimes_{\varphi, R} M)[1/p] \isomorphic M[1/p]$.
\end{rem}

In Sections \ref{subsec:geometric_descent} and \ref{subsec:arithmetic_descent}, we shall prove the comparison isomorphism \eqref{eq:integral_comp_relative} claimed in Theorem \ref{thm:integral_comp_relative}.
Our proof is broadly divided into three main steps: a geometric descent for the action of $\Gamma_R'$ (see Proposition \ref{prop:abarone_comp}) and a ``two-step'' arithmetic descent for the action of $\Gamma_F$ (see Propositions \ref{prop:mlambda_fpx_descent} and \ref{prop:lambdar0_comp}, for $p \geqslant 3$, and Propositions \ref{prop:mlambda_pm_descent} and \ref{prop:lambdar+_comp}, for $p=2$).
Then in Section \ref{subsec:proof_integral_comp_relative}, we will put everything together to complete the proof of Theorem \ref{thm:integral_comp_relative}.
We begin by interpreting the action of $\Gamma_R$ on a Wach module as a $q\textrm{-connection}$.

\subsection{Wach modules and \texorpdfstring{$\qconnections$}{-}}\label{subsec:wachmod_qconnection}

In this section, our goal is to interpret Wach modules and its scalar extension to $\AR(1)/(p_1(\mu))$ as modules with $\qconnections$ similar to \cite[Section 5]{abhinandan-relative-wach-ii}.
We shall work with the notation described in Section \ref{subsubsec:a1modp1mu_qconnection}.

\begin{defi}[$\qconnection$, {\cite[Definition 2.2]{morrow-tsuji}}]\label{defi:qconnection}
	A module with $\textit{q-connection}$ over $D$ is a right $D\module$ $N$ equipped with an $A\linear$ map $\nabla_q \colon N \rightarrow N \otimes q\Omega^1_{D/A}$ satisfying the Leibniz rule $\nabla_q(xf) = \nabla_q(x)f + x \otimes d_q(f)$ for all $f$ in $D$ and $x$ in $N$.
	The $\qconnection$ $\nabla_q$ extends uniquely to a map of graded $A\modules$ $\nabla_q \colon N \otimes q\Omega^{\bullet}_{D/A} \rightarrow N \otimes q\Omega^{\bullet+1}_{D/A}$ satisfying $\nabla_q((n \otimes \omega) \cdot \omega') = \nabla_q(n \otimes \omega) \cdot \omega' + (-1)^{\textrm{deg } \omega} (n \otimes \omega) \cdot d_q(\omega')$.
	The $\qconnection$ $\nabla_q$ is said to be \textit{flat} or \textit{integrable} if $\nabla_q \circ \nabla_q = 0$.
\end{defi}

\begin{exam}\label{exam:wach_qconnection}
	From Example \ref{exam:ar+_qconnection}, take $D = \AF = O_F\llbracket \mu \rrbracket$, $A = \AR$ equipped with the action of $\Gamma_R$ and $\{\gamma_1, \ldots, \gamma_d\}$ as topological generators of $\Gamma_R'$ (see Section \ref{subsec:important_rings}).
	Taking $q = 1+\mu$ and $U_i = [X_i^{\flat}]$, for $1 \leqslant i \leqslant d$, we know that $\AR$ satisfies the hypotheses of Definition \ref{defi:qdeRham_complex}.
	In particular, $\AR$ is equipped with an $\AF\linear$ $\qconnection$ $\nabla_q \colon \AR \rightarrow q\Omega^1_{\AR/\AF}$, given as $f \mapsto \sum_{i=1}^d \frac{\gamma_i(f)-f}{p_2(\mu)} \dlog(p_2([X_i^{\flat}]))$.
	Now, let $N$ be a Wach module over $\AR$.
	Then, from \cite[Proposition 5.3]{abhinandan-relative-wach-ii}, the geometric $\qconnection$ $\nabla_q \colon N \rightarrow N \otimes_{\AR} \Omega^1_{\AR/\AF}$, given as $x \mapsto \sum_{i=1}^d \tfrac{\gamma_i(x)-x}{\mu} \dlog([X_i^{\flat}])$, describes $(N, \nabla_q)$ as a $\varphi\module$ equipped with a $(p, [p]_q)\textrm{-adically}$ quasi-nilpotent flat $\qconnection$ over $\AR$ (see \cite[Section 5.2]{abhinandan-relative-wach-ii} for details).
\end{exam}

\begin{exam}\label{exam:wach_arpd_qconnection}
	From Example \ref{exam:arpd_qconnection}, take $D = \Lambda_F$ and $A = \ARpd$ as a $\Lambda_F\algebra$.
	Note that $\ARpd$ is equipped with a $\Lambda_F\linear$ action of $\Gamma_R'$ and we have $\{\gamma_1, \ldots, \gamma_d\}$ as topological generators of $\Gamma_R'$ (see Section \ref{subsec:important_rings}).
	Then, by setting $q = 1+\mu$ and $U_i = [X_i^{\flat}]$, for $1 \leqslant i \leqslant d$, we know that $\ARpd$ satisfies the hypotheses of Definition \ref{defi:qdeRham_complex}.
	In particular, $\ARpd$ is equipped with a $\Lambda_F\linear$ $\qconnection$ $\nabla_q \colon \ARpd(1) \rightarrow q\Omega^1_{\ARpd(1)/\Lambda_F}$, given as $f \mapsto \sum_{i=1}^d \frac{\gamma_i(f)-f}{p_2(\mu)} \dlog(p_2([X_i^{\flat}]))$.
	Now, let $N$ be a Wach module over $\AR$ and set $N^{\PD} \coloneq \ARpd \otimes_{\AR} N$, equipped with the tensor product Frobenius and the tensor product action of $\Gamma_R$.
	Then, note that for any $f \otimes y$ in $N^{\PD}$ and $g$ in $\Gamma_R$, we have that $(g-1)(f \otimes y) = (g-1)f \otimes y + g(f) \otimes (g-1)y$ is in $\mu N^{\PD}$.
	Therefore, the operator $\nabla_q \colon N^{\PD} \rightarrow N^{\PD} \otimes_{\ARpd} \Omega^1_{\ARpd/\Lambda_F}$, given as $x \mapsto \sum_{i=1}^d \tfrac{\gamma_i(x)-x}{\mu} \dlog([X_i^{\flat}])$, satisfies the assumptions of Definition \ref{defi:qconnection}.
	Moreover, the $\qconnection$ $\nabla_q$ on $N^{\PD}$ is $p\adically$ quasi-nilpotent using Example \ref{exam:arpd_qconnection} and Example \ref{exam:wach_qconnection}, and it is flat because $\gamma_i$ commute with each other.
\end{exam}

\begin{exam}\label{exam:wach_a1modp1mu_qconnection}
	From Example \ref{exam:a1modp1mu_qconnection}, take $D$ to be $\Lambda_R$ and $A$ to be $\Abar(1) = \AR(1)/(p_1(\mu))$ as a $\Lambda_R\algebra$ via the morphism of rings $\iota_{\Lambda} \colon \Lambda_R \rightarrow \Abar(1)$.
	Then, the $\Lambda_R\algebra$ $\Abar(1)$ is equipped with a $\Lambda_R\linear$ (via $\iota_{\Lambda}$) action of $\Gamma_R'$ and we take $\{\gamma_1, \ldots, \gamma_d\}$ as topological generators of $\Gamma_R'$ (see Section \ref{subsec:important_rings}).
	Moreover, by setting $q = 1+\mu$ and $U_i = [X_i^{\flat}]$, for $1 \leqslant i \leqslant d$, we see that $\Abar(1)$ satisfies the hypotheses of Definition \ref{defi:qdeRham_complex}.
	In particular, $\Abar(1)$ is equipped with a $\Lambda_R\linear$ $\qconnection$ $\nabla_q \colon \Abar(1) \rightarrow q\Omega^1_{\Abar(1)/\Lambda_R}$, given as $f \mapsto \sum_{i=1}^d \frac{\gamma_i(f)-f}{p_2(\mu)} \dlog(p_2([X_i^{\flat}]))$.
	Now, let $N$ be a Wach module over $\AR$ and set $N(1) \coloneq \AR(1) \otimes_{p_2, \AR} N$, equipped with the tensor product Frobenius and the tensor product action of $\Gamma_R^2$, where $\Gamma_R^2$ acts on $N$ via projection onto the second coordinate.
	Let $\Nbar(1) \coloneq N(1)/(p_1(\mu)) = \Abar(1) \otimes_{p_2, \AR} N$, equipped with the induced action of Frobenius and $\Gamma_R = 1 \times \Gamma_R$.
	Note that for any $f \otimes y$ in $\Nbar(1)$ and $g$ in $\Gamma_R$, we have that $(g-1)(f \otimes y) = (g-1)y \otimes y + g(f) \otimes (g-1)y$ is in $p_2(\mu)\Nbar(1)$.
	Therefore, the operator $\nabla_q \colon \Nbar(1) \rightarrow \Nbar(1) \otimes_{\Abar(1)} \Omega^1_{\Abar(1)/\Lambda_R}$, given as $x \mapsto \sum_{i=1}^d \tfrac{\gamma_i(x)-x}{\mu} \dlog([X_i^{\flat}])$ satisfies the assumptions of Definition \ref{defi:qconnection}.
	Moreover, the $\qconnection$ $\nabla_q$ on $\Nbar(1)$ is $p\adically$ quasi-nilpotent using Example \ref{exam:a1modp1mu_qconnection} and Example \ref{exam:wach_qconnection}, and it is flat because $\gamma_i$ commute with each other.
\end{exam}

\begin{lem}\label{lem:nr1_pctf}
	Let $N$ be a Wach module over $\AR$ and $\Abar(1) = \AR(1)/(p_1(\mu))$.
	Then, the $\Abar(1)\module$ $\Abar(1) \otimes_{p_2, \AR} N$ is $p\adically$ complete and $p\torsionfree$.
\end{lem}
\begin{proof}
	Note that the morphism $p_2 \colon \AR \rightarrow \AR(1)$ is faithfully flat from Lemma \ref{lem:pi_rj_ff}.
	Moreover, from Definition \ref{defi:wach_mods_relative} recall that $\{p, \mu\}$ and $\{\mu, p\}$ are regular sequences on $N$.
	Then, from the flatness of $p_2$, it follows that $\{p, p_2(\mu)\}$ and $\{p_2(\mu), p\}$ are regular sequences on $\AR(1) \otimes_{p_2, \AR} N$.
	Using the fact that $p_1([p]_q)/p_2([p]_q)$ is a unit in $\AR(1)$ (see Lemma \ref{lem:pi_rj_ff}), we get that $p_1(\mu)^{p-1}$ is an element of $(p, p_2(\mu))\AR(1)$.
	Now, since both $\{p, p_2(\mu)\}$ and $\{p_2(\mu), p\}$ are regular sequences on $\AR(1) \otimes_{p_2, \AR} N$, therefore, it follows that both $\{p, p_1(\mu)^{p-1}\}$ and $\{p_1(\mu)^{p-1}, p\}$ are regular sequences on $\AR(1) \otimes_{p_2, \AR} N$.
	In particular, using \cite[\href{https://stacks.math.columbia.edu/tag/07DV}{Tag 07DV}]{stacks-project} we get that both $\{p, p_1(\mu)\}$ and $\{p_1(\mu), p\}$ are regular sequences on $\AR(1) \otimes_{p_2, \AR} N$, and we conclude that $\Abar(1) \otimes_{p_1, \AR} N$ is $p\torsionfree$.
	Next, note that $N$ is a finitely generated and $(p, \mu)\adically$ complete $\AR\module$.
	So, it follows that $\AR(1) \otimes_{p_2, \AR} N$ is a finitely generated and $(p, p_2(\mu))\adically = (p, p_1(\mu))\adically$ complete $\AR(1)\module$, in particular, it is $(p, p_1(\mu))\adically$ separated.
	Consequently, we get that $\Abar(1) \otimes_{p_2, \AR} N$ is $p\adically$ separated and it is clearly finitely generated as an $\Abar(1)\module$, therefore, $p\adically$ complete.
	This completes our proof.
\end{proof}

\begin{rem}\label{rem:npd_pctf}
	Let $N$ be a Wach module over $\AR$ and let $\ARpd$ be the ring defined in Example \ref{exam:arpd_qconnection}.
	Moreover, recall that we defined a ring $E$ in Remark \ref{rem:qdeform_arpd}, which admits faithfully flat maps $p_1 \colon \AF \rightarrow E$ and $p_2 \colon \AR \rightarrow E$, and we have that $E/(p_1(\mu)) \isomorphic \ARpd$.
	Then, by an argument similar to the proof of Lemma \ref{lem:nr1_pctf}, it follows that the $\ARpd\module$ $\ARpd \otimes_{\AR} N$ is $p\adically$ complete and $p\torsionfree$.
\end{rem}

From Lemma \ref{lem:tovermu_unit}, we have that $t = \log(1+\mu)$ converges in $\mu \Lambda_F \subset \mu \Lambda_R$ and $t/\mu$ is a unit.
\begin{prop}\label{prop:n1modp1mu_connection}
	Let $N$ be a Wach module over $\AR$ and $\Abar(1) = \AR(1)/(p_1(\mu))$.
	Then, for $1 \leqslant i \leqslant d$, the series of operators $\nabla_i^{\log} = \frac{\log \gamma_i}{t} = \frac{1}{t}\sum_{k \in \NN} (-1)^k \frac{(\gamma_i-1)^{k+1}}{k+1}$ converge $p\adically$ on $\Nbar(1) = \Abar(1) \otimes_{p_2, \AR} N$.
	This defines a $\Lambda_R\linear$ $p\adically$ quasi-nilpotent flat connection on $\Nbar(1)$, denoted as $\nabla \colon \Nbar(1) \rightarrow \Nbar(1) \otimes_{\Abar(1)} \Omega^1_{\Abar(1)/\Lambda_R}$ and given as $x \mapsto \sum_{i=1}^d \nabla_i^{\log}(x) \dlog([X_i^{\flat}])$.
	The data of the connection $\nabla$ on $\Nbar(1)$ is equivalent to the data of the $\qconnection$ $\nabla_q$ described in Example \ref{exam:wach_a1modp1mu_qconnection}, i.e.\ either may be recovered from the other.
	Moreover, the $q\textrm{-de Rham}$ complex $\Nbar(1) \otimes_{\Abar(1)} q\Omega^{\bullet}_{\Abar(1)/\Lambda_R}$ is naturally quasi-isomorphic to the de Rham complex $\Nbar(1) \otimes_{\Abar(1)} \Omega^{\bullet}_{\Abar(1)/\Lambda_R}$.
	In particular, we have that $\Nbar(1)^{\nabla_q=0} = \Nbar(1)^{\nabla=0}$.
\end{prop}
\begin{proof}
	Our proof employs ideas similar to that of \cite[Lemmas 4.36 \& 4.38]{abhinandan-relative-wach-i}, \cite[Theorem 4.2]{morrow-tsuji} and \cite[Corollary 12.5]{bhatt-morrow-scholze-1}.
	Note that the $\Abar(1)\module$ $\Nbar(1)$ is $p\adically$ complete and $p\torsionfree$ by Lemma \ref{lem:nr1_pctf}.
	Let us first show that for $1 \leqslant i \leqslant d$ the series of operators $\nabla_i^{\log} = \frac{1}{t}\sum_{k \in \NN} (-1)^k \frac{(\gamma_i-1)^{k+1}}{k+1}$ converge on $\Nbar(1)$.
	As $t/\mu$ is a unit in $\Lambda_R$, we may write 
	\begin{equation*}
		\tfrac{1}{t}\textstyle\sum_{k \in \NN} (-1)^k\tfrac{(\gamma_i-1)^{k+1}}{k+1} = \tfrac{\mu}{t} \textstyle\sum_{k \in \NN}\tfrac{\mu^k}{k+1} \tfrac{(\gamma_i-1)^{k+1}}{\mu^{k+1}},
	\end{equation*}
	where the right hand term converges on $\Nbar(1)$ because $\mu^k/(k+1)$ converges $p\adically$ to $0$ in $\Lambda_R$ as $k$ goes to $+\infty$.

	Next, let us check that the operator $\nabla_i^{\log}$ satisfies the Leibniz rule.
	To show this, we claim that for $x$ in $\Nbar(1)$, we have that
	\begin{equation}\label{eq:log_as_lim}
		\lim_{m \rightarrow +\infty} \tfrac{\gamma_i^{p^m}-1}{p^m} (x) = t\nabla_i^{\log}(x).
	\end{equation}
	Indeed, note that since $t^k/k!$ converges $p\adically$ to $0$ in $\Abar(1)$ as $k$ goes to $+\infty$, so we may write $\gamma_i^n = \exp(nt\nabla_i^{\log})$, for $n \in \ZZ$.
	Expanding the preceding exponential, we see that
	\begin{equation*}
		\tfrac{\gamma_i^n-1}{n} = t\nabla_i^{\log} + n \textstyle\sum_{k \geqslant 2}n^{k-2} \tfrac{t^k}{k!} (\nabla_i^{\log})^k \colon \Nbar(1) \rightarrow \Nbar(1),
	\end{equation*}
	is well defined.
	Taking $n = p^m$ and letting $m \rightarrow +\infty$, we get the formula in \eqref{eq:log_as_lim}.
	Now, for any $f$ in $\Abar(1)$ and $x$ in $\Nbar(1)$, we have that 
	\begin{equation*}
		(\gamma_i^{p^m} - 1)(fx) = (\gamma_i^{p^m} - 1)(f) \cdot x + \gamma_i^{p^m}(f) (\gamma_i^{p^m} - 1)(x).
	\end{equation*}
	Dividing out the preceding equality by $tp^m$, letting $m \rightarrow +\infty$ and using \eqref{eq:log_as_lim} we get that $\nabla_i^{\log}(fx) = \nabla_i^{\log}(f)x + f\nabla_i^{\log}(x)$, where the first operator on the right is $\nabla_i^{\log} \coloneq \log(\gamma_i)/t \colon \Abar(1) \rightarrow \Abar(1)$, whose well definedness and the equality in \eqref{eq:log_as_lim} may be checked similar to above.
	In particular, we have shown that the operators $\nabla_i^{\log}$ are well defined and satisfy a Leibniz rule.

	To show that $\nabla \colon \Nbar(1) \rightarrow \Nbar(1) \otimes_{\Abar(1)} \Omega^1_{\Abar(1)/\Lambda_R}$, given as $x \mapsto \sum_{i=1}^d \nabla_i^{\log}(x) \dlog([X_i^{\flat}])$, is a well-defined connection, we need to show that $\nabla \colon \Abar(1) \rightarrow \Omega^1_{\Abar(1)/\Lambda_R}$ is the usual de Rham differential $d \colon \Abar(1) \rightarrow \Omega^1_{\Abar(1)/\Lambda_R}$.
	As each $\nabla_i^{\log}$ is a continuous $\Lambda_R\linear$ derivation, we may write $\nabla_i^{\log} = h_i \circ d$ for some unique continuous $\Lambda_R\linear$ map $h_i \colon \Omega^1_{\Abar(1)/\Lambda_R} \rightarrow \Lambda_R$.
	Then, it is easy to see that $h_i(d[X_i^{\flat}]) = \nabla_i^{\log}([X_i^{\flat}]) = [X_i^{\flat}]$, for $i = j$, and $0$, otherwise.
	Therefore, we have $d = \sum_i [X_i^{\flat}]^{-1} \nabla_i^{\log}(-) \otimes d[X_i^{\flat}] = \sum_i \nabla_i^{\log}(-) \otimes d\log([X_i^{\flat}])$, as desired.

	Next, let us show that the operators $\nabla_i = \nabla_i^{\log}/[X_i^{\flat}] = (\log \gamma_i)/(t[X_i^{\flat}])$ are $p\adically$ quasi-nilpotent.
	Indeed, we first note that from the commutativity of $\varphi$ and $\gamma_i$, it follows that $\log \gamma_i \circ \varphi = \varphi \circ \log \gamma_i$.
	Therefore, it is easy to see that $\nabla_i \circ \varphi = p [X_i^{\flat}]^{p-1} \varphi \circ \nabla_i$.
	Recall that $N$ is equipped with an $\AR\linear$ isomorphism $\varphi^*(N)[1/[p]_q] \isomorphic N[1/[p]_q]$, and $[p]_q/p$ is a unit in $\Lambda_F$ (see Lemma \ref{lem:tovermu_unit}), therefore, $\Nbar(1)$ is equipped with an $\Abar(1)\linear$ isomorphism $\varphi^*(\Nbar(1))[1/p] \isomorphic \Nbar(1)[1/p]$.
	In particular, for any $x$ in $\Nbar(1)$, there exists $r \in \NN$ large enough, such that $p^rx$ belongs to $\varphi^*(\Nbar(1))$.
	Then, from the relation $\nabla_i \circ \varphi = p \varphi \circ \nabla_i$, we see that $\nabla_i^k(p^r x)$ converges $p\adically$ to $0$ as $k \rightarrow +\infty$.
	Hence, it follows that $\nabla_i^k(x) = p^{-r} \nabla_i^k(p^r x)$ converges $p\adically$ to 0 as $k \rightarrow +\infty$, in particular, $\nabla_i$ are $p\adically$ quasi-nilpotent.

	So far, we have shown that $\nabla \colon \Nbar(1) \rightarrow \Nbar(1) \otimes_{\Abar(1)} \Omega^1_{\Abar(1)/\Lambda_R}$ is a $p\adically$ quasi-nilpotent connection and it is flat since $\gamma_i$, and therefore $\nabla_i^{\log}$, commute with each other.
	Moreover, we defined the connection $\nabla$ using the action of $\Gamma_R'$ and conversely, we have shown that the action of $\Gamma_R'$ may be recovered by the formula $\gamma_i \coloneq \exp(t\nabla_i^{\log})$ and it remains to check that, the action of $\gamma_i$ thus obtained, is semilinear.
	Note that the Leibniz rule implies that $\frac{1}{k!}(t\nabla_i^{\log})^k(xf) = \sum_{a+b=k} \frac{1}{a!}(t\nabla_i^{\log})^a(x) \frac{1}{b!}(t\nabla_i^{\log})^b(f)$ for $x$ in $\Nbar(1)$ and $f$ in $\Abar(1)$.
	Now, taking the sum over all $k > 0$, we get that $\exp(t\nabla_i^{\log}(xf)) = \exp(t\nabla_i^{\log})(x)\exp(t\nabla_i^{\log})(f) = \exp(t\nabla_i^{\log})(x)\gamma_i(f)$, as required.

	From the discussion above, it is clear that we have a $q\textrm{-de Rham}$ complex $\Nbar(1) \otimes_{\Abar(1)} q\Omega^{\bullet}_{\Abar(1)/\Lambda_R}$ and a de Rham complex $\Nbar(1) \otimes_{\Abar(1)} \Omega^{\bullet}_{\Abar(1)/\Lambda_R}$.
	We claim that these complexes are naturally quasi-isomorphic.
	Indeed, let us first note that the endomorphisms $\nabla_{q, i}^{\log} = \tfrac{\gamma_i-1}{\mu} \colon \Nbar(1) \rightarrow \Nbar(1)$ are commuting.
	So, let $K_{\Nbar(1)}(\nabla_{q, 1}^{\log}, \ldots, \nabla_{q, d}^{\log})$ denote the corresponding Koszul complex (see Definition \ref{defi:koszul_complex}).
	Then, we have a natural identification of complexes 
	\begin{equation*}
		K_{\Nbar(1)}(\nabla_{q, 1}^{\log}, \ldots, \nabla_{q, d}^{\log}) = \Nbar(1) \otimes_{\Abar(1)} q\Omega^{\bullet}_{\Abar(1)/\Lambda_R}.
	\end{equation*}
	Next, let $t = \log(1+\mu)$, and again note that the endomorphisms $\nabla_i^{\log} = \frac{\log \gamma_i}{t} \colon \Nbar(1) \rightarrow \Nbar(1)$ are commuting.
	So, let $K_{\Nbar(1)}(\nabla_1^{\log}, \ldots, \nabla_d^{\log})$ denote the corresponding Koszul complex, and again we have a natural identification of complexes 
	\begin{equation*}
		K_{\Nbar(1)}(\nabla_1^{\log}, \ldots, \nabla_d^{\log}) = \Nbar(1) \otimes_{\Abar(1)} \Omega^{\bullet}_{\Abar(1)/\Lambda_R}.
	\end{equation*}
	Now, from the discussion above, recall that we have the relation $\gamma_i = \exp(t\nabla_i^{\log})$ on $\Nbar(1)$.
	Therefore, we may write
	\begin{equation}\label{eq:nablaqi_nablai}
		\nabla_{q, i}^{\log} = \nabla_i^{\log}\big(\tfrac{t}{\mu} + \textstyle \sum_{k \geqslant 2} \tfrac{t^k}{k!\mu} (\nabla_i^{\log})^{k-1}\big),
	\end{equation}
	where the term inside the parentheses is invertible, because $t/\mu$ is a unit in $\Abar(1)$ (see Lemma \ref{lem:tovermu_unit}) and $\mu^{k-1}/(k!)$ is topologically nilpotent because $(\mu^{p-1}/p)^k$ is topologically nilpotent in $\Abar(1)$ (see Proposition \ref{prop:a1modp1mu_pdring}).
	Now, in the notation of Lemma \ref{lem:koszul_complex_qiso}, let us set $M = \Nbar(1)$, $f_i = \nabla_i^{\log}$ and take $h_i$ to be the formula inside the parentheses in \eqref{eq:nablaqi_nablai}, in particular, $f_ih_i = \nabla_{q, i}^{\log}$.
	Then, from Lemma \ref{lem:koszul_complex_qiso}, we obtain a natural quasi-isomorphism of complexes
	\begin{equation*}
		\Nbar(1) \otimes_{\Abar(1)} q\Omega^{\bullet}_{\Abar(1)/\Lambda_R} = K_{\Nbar(1)}(\nabla_{q, 1}^{\log}, \ldots, \nabla_{q, d}^{\log}) \isomorphic K_{\Nbar(1)}(\nabla_1^{\log}, \ldots, \nabla_d^{\log}) = \Nbar(1) \otimes_{\Abar(1)} \Omega^{\bullet}_{\Abar(1)/\Lambda_R}.
	\end{equation*}
	Finally, from the quasi-isomorphism above, it follows that we have $\Nbar(1)^{\Gamma_R'} = \Nbar(1)^{\nabla_q=0} \isomorphic \Nbar(1)^{\nabla=0}$.
	This concludes our proof.
\end{proof}

\begin{rem}\label{rem:npd_connection}
	From Remark \ref{rem:npd_pctf}, recall that the $\ARpd\module$ $N^{\PD} = \ARpd \otimes_{\AR} N$ is $p\adically$ complete and $p\torsionfree$.
	Then, in Proposition \ref{prop:n1modp1mu_connection}, replacing $\Lambda_R$ by $\Lambda_F$, $\Abar(1)$ by $\ARpd$ and $\Nbar(1)$ by $N^{\PD} = \ARpd \otimes_{\AR} N$, and using essentially the same arguments, we see that for $1 \leqslant i \leqslant d$, the series of operators $\nabla_i^{\log} = \frac{\log \gamma_i}{t} = \frac{1}{t}\sum_{k \in \NN} (-1)^k \frac{(\gamma_i-1)^k}{k+1}$ converge $p\adically$ on $N^{\PD}$.
	This defines a $\Lambda_F\linear$ $p\adically$ quasi-nilpotent flat connection on $N^{\PD}$, denoted as $\nabla \colon N^{\PD} \rightarrow N^{\PD} \otimes_{\ARpd} \Omega^1_{\ARpd/\Lambda_F}$ and given as $x \mapsto \sum_{i=1}^d \nabla_i^{\log}(x) \dlog([X_i^{\flat}])$.
	The data of the connection $\nabla$ on $N^{\PD}$ is equivalent to the data of the $\qconnection$ $\nabla_q$ described in Example \ref{exam:wach_arpd_qconnection}, i.e.\ either may be recovered from the other.
	Moreover, the $q\textrm{-de Rham}$ complex $N^{\PD} \otimes_{A^{\PD}} q\Omega^{\bullet}_{A^{\PD}/\Lambda_F}$ is naturally quasi-isomorphic to the de Rham complex $N^{\PD} \otimes_{A^{\PD}} \Omega^{\bullet}_{A^{\PD}/\Lambda_F}$.
	In particular, we have $(N^{\PD})^{\nabla_q=0} \isomorphic (N^{\PD})^{\nabla=0}$.
\end{rem}

\begin{rem}\label{rem:npd_nbarone_connection_comp}
	Let $\Abar(1) = \AR(1)/(p_1(\mu))$ and $\ARpd$ as in Example \ref{exam:arpd_qconnection}.
	Recall that we have a natural injective $(\varphi, \Gamma_R)\equivariant$ homomorphism of rings $p_2 \colon \ARpd \rightarrow \Abar(1)$ (see Remark \ref{rem:a1modp1mu_oacris_connection}).
	Now, let $N$ be a Wach module over $\AR$, then we have that $\Abar(1) \otimes_{p_2, \ARpd} N^{\PD} \isomorphic \Abar(1) \otimes_{p_2, \AR} N = \Nbar(1)$, compatible with the action of $(\varphi, \Gamma_R)$.
	So, using the compatibility of the corresponding connections on $\Abar(1)$ and $\ARpd$ (see Remark \ref{rem:arpd_abarone_connection_comp}) and the Leibniz rule for the connection on $\Nbar(1)$, proven in Proposition \ref{prop:n1modp1mu_connection}, it follows that the respective connections on $\Nbar(1)$ and $N^{\PD}$ are compatible, in particular, the connection on $\Nbar(1)$ is given as the tensor product of the respective connections on $\Abar(1)$ and $N^{\PD}$.
\end{rem}

\subsection{Geometric descent: the action of \texorpdfstring{$\Gamma_R'$}{-}}\label{subsec:geometric_descent}

Let us begin by recalling that we have an isomorphism of $\Lambda_F\algebras$ $\Lambda_R = R[\mu, \{(\mu^{p-1}/p)^{[k]}\}_{k \in \NN}]_p^{\wedge} \isomorphic \ARpd$ (see the isomorphism $\iota$ in \eqref{eq:pd_iota}), and $\Abar(1) = \AR(1)/(p_1(\mu))$ as in Example \ref{exam:a1modp1mu_qconnection} admitting a natural $(\varphi, \Gamma_R)\equivariant$ homomorphism $p_2 \colon \ARpd \rightarrow \Abar(1)$ (see Remark \ref{rem:arpd_abarone_connection_comp}).
Moreover, from Section \ref{subsubsec:geo_gamma_action} recall that we have a $(\varphi, 1 \times \Gamma_F)\equivariant$ map $\Delta' \colon \AR(1) \rightarrow \LambdatildeR$, and from the proof of Lemma \ref{lem:delta_prime_iso} we have a $(\varphi, \Gamma_F)\equivariant$ diagram 
\begin{equation*}
	\Lambda_R \xrightarrow[\sim]{\hspace{1mm} \iota_{\Lambda} \hspace{1mm}} \AR(1)/(p_1(\mu)) \xrightarrow{\hspace{1mm} \Delta' \hspace{1mm}} \LambdatildeR/(p_1(\mu)) \xrightarrow[\sim]{\hspace{1mm} \eqref{eq:lambdatilder_mod_p1mu} \hspace{1mm}} \Lambda_R,
\end{equation*}
such that the composition is the identity on $\Lambda_R$.
By abusing notations, we will denote the composition of the middle and right arrows again by $\Delta'$.
Moreover, note that the $(\varphi, \Gamma_F)\equivariant$ composition $\ARpd \xrightarrow{p_2} \Abar(1) \xrightarrow{\Delta'} \Lambda_R$ coincides with the inverse of the $(\varphi, \Gamma_F)\equivariant$ isomorphism $\iota \colon \Lambda_R \isomorphic \ARpd$.

Let $N$ be a Wach module over $\AR$ and from Example \ref{exam:wach_a1modp1mu_qconnection} recall that the $\Abar(1)\module$ $\Nbar(1) = \Abar(1) \otimes_{p_2, \AR} N$ is equipped with a Frobenius-semilinear endomorphism $\varphi$ (after inverting $p$) and an action of $1 \times \Gamma_R'$, which induces a $\Lambda_R\linear$ $p\textrm{-adically}$ quasi-nilpotent flat $\qconnection$ and thus a $\Lambda_R\linear$ $p\textrm{-adically}$ quasi-nilpotent flat connection $\nabla \colon \Nbar(1) \rightarrow \Nbar(1) \otimes_{\Abar(1)} \Omega^1_{\Abar(1)/\Lambda_R}$ (see Proposition \ref{prop:n1modp1mu_connection}).
Moreover, from Example \ref{exam:wach_arpd_qconnection} and Remark \ref{rem:npd_connection}, note that $N^{\PD} \coloneq \ARpd \otimes_{\AR} N$ is a finite module over $\ARpd$, equipped with a $p\adically$ quasi-nilpotent flat connection and a Frobenius-semilinear endomorphism $\varphi$ (after inverting $p$), and from Remark \ref{rem:npd_nbarone_connection_comp} we have a natural $\Abar(1)\linear$ isomorphism $\Abar(1) \otimes_{p_2, \ARpd} N^{\PD} \isomorphic \Abar(1) \otimes_{p_2, \AR} N = \Nbar(1)$ compatible with the respective Frobenii and connections, and we shall identify these objects in the following.

Now, by tensoring the map $\Delta'$ from above with the $\ARpd\module$ $N^{\PD}$, we obtain a natural $(\varphi, \Gamma_F)\equivariant$ homomorphism $\Delta'_N \colon \Nbar(1) \rightarrow \Lambda_R \otimes_{\ARpd} N^{\PD} = \Lambda_R \otimes_{\AR} N$, and we make the following claim:
\begin{prop}\label{prop:abarone_comp}
	Let $N$ be a Wach module over $\AR$, and set $\MLambda \coloneq (\Abar(1) \otimes_{p_2, \AR} N)^{\nabla=0}$ as a $\Lambda_R\module$ via $\iota_{\Lambda} \colon \Lambda_R \isomorphic \Abar(1)^{\nabla=0}$ and equipped with an induced $(\varphi, \Gamma_F)\action$.
	Then, there exists a natural $(\varphi, \nabla, \Gamma_F)\equivariant$ isomorphism
	\begin{equation}\label{eq:abarone_comp}
		\Abar(1) \otimes_{\Lambda_R} \MLambda \isomorphic \Abar(1) \otimes_{p_2, \AR} N,
	\end{equation}
	induced by $\Delta'_N \colon \Nbar(1) \rightarrow \Lambda_R \otimes_{\AR} N$.
	Moreover, we have $(\varphi, \Gamma_F)\equivariant$ isomorphisms of finitely generated, $p\adically$ complete and $p\torsionfree$ modules over $\Lambda_R$
	\begin{equation}\label{eq:mlambda_arpdn}
		\MLambda \isomorphic \Lambda_R \otimes_{\AR} N.
	\end{equation}
	Furthermore, the de Rham complex $\Nbar(1) \otimes_{\Abar(1)} \Omega^{\bullet}_{\Abar(1)/\Lambda_R}$ is acyclic in positive degrees.
\end{prop}
\begin{proof}
	To obtain the isomorphism in \eqref{eq:abarone_comp}, our strategy is to translate the connection on $\Nbar(1)$ into a stratification (see the isomorphism $\overline{\varepsilon}$ below) and use its properties to deduce \eqref{eq:abarone_comp}.
	Although one could try to do this directly, the proof becomes more efficient and conceptually clear with the use of crystalline site and $F\crystals$.
	We refer the reader to Section \ref{subsec:crys_site_crystal} for a quick recollection.

	In order to use the statements from Section \ref{subsec:crys_site_crystal}, we will work with the notation of Remark \ref{rem:arpd_selfprod_abarone} below, i.e.\ we set $\Sigma_m = \Spec(\Lambda_F/p^m)$ and $X_m = Y_m = \Spec(A^{\PD}/p^m)$.
	Note that the $\ARpd\module$ $N^{\PD}$ is $p\textrm{-adically}$ complete, and equipped with a $\Lambda_F\linear$ $p\adically$ quasi-nilpotent flat connection and a Frobenious structure which is compatible with the connection (see Example \ref{exam:wach_arpd_qconnection} and Remark \ref{rem:npd_connection}), so $N^{\PD}$ is an object of $\MIC^{\varphi}(\ARpd)$ (see Section \ref{subsec:crys_site_crystal}).
	Then, from the equivalence in \eqref{eq:crys_fcrystal_connections}, there exists a quasi-coherent $F\crystal$ $\paze$ on $X/\Sigma$ such that we have $\paze(\ARpd) = N^{\PD}$.

	Next, recall that we have a $(\varphi, \Gamma_F)\equivariant$ isomorphism of rings $\iota \colon \Lambda_R \isomorphic \ARpd$ from \eqref{eq:pd_iota}.
	Moreover, $\Abar(1)$ is the $\padic$ completion of a PD-polynomial algebra over $\ARpd$ in variables $T_1, \ldots, T_d$, and it is equipped with a $\Lambda_R\linear$ flat connection $\nabla \colon \Abar(1) \rightarrow \Omega^1_{\Abar(1)/\Lambda_R} = \Abar(1) \otimes_{\ARpd} \Omega^1_{\ARpd/\Lambda_F}$ (see Propositions \ref{prop:a1modp1mu_pdring} and \ref{prop:a1modp1mu_connection}, and Remarks \ref{rem:a1modp1mu_lambda} and \ref{rem:arpd_abarone_connection_comp}).
	Now, consider the PD-thickening $\Spec(\Lambda_R/p^m) \xhookrightarrow{\Delta'} \Spec(\Abar(1)/p^m)$ induced by the $(\varphi, \Gamma_F)\equivariant$ surjection $\Delta' \colon \Abar(1) \twoheadrightarrow \Lambda_R$, and note that the composition $\Lambda_R \xrightarrow{\hspace{1mm} \iota_{\Lambda} \hspace{1mm}} \Abar(1) \twoheadrightarrow \Lambda_R$ is the identity and $(\varphi, \Gamma_F)\equivariant$ (see Remark \ref{rem:arpd_selfprod_abarone} below).
	So, we have the following $(\varphi, \Gamma_F)\equivariant$ morphisms in $\CRYS(X_m/\Sigma_m)$, for $m \geqslant 1$:
	\begin{equation}\label{eq:rin_arone}
		\begin{tikzcd}
			\Spec(\Lambda_R/p^m) \arrow[r, hookrightarrow] \arrow[d, "\iota"] & \Spec(\Abar(1)/p^m) \arrow[d, "p_2"] & \Spec(\Lambda_R/p^m) \arrow[r, hookrightarrow] \arrow[d, "id"] & \Spec(\Abar(1)/p^m) \arrow[d, "T_i^{[k_i]} \mapsto 0"]\\
			\Spec(\ARpd/p^m) \arrow[r, "id"] & \Spec(\ARpd/p^m) & \Spec(\Lambda_R/p^m) \arrow[r, hookrightarrow] & \Spec(\Abar(1)/p^m).
		\end{tikzcd}
	\end{equation}

	Evaluating the $F\crystal$ $\paze$ on $\Spec(\Lambda_R/p^m) \hookrightarrow \Spec(\Abar(1)/p^m)$ for each $m \geqslant 1$, and taking the limit over $m$, yields a $p\textrm{-adically}$ complete $\Abar(1)\module$ $\paze(\Abar(1))$ equipped with a $(\varphi, \Gamma_F)\action$ and a $\Lambda_R\linear$ $p\adically$ quasi-nilpotent flat connection (see the proofs of \cite[\href{https://stacks.math.columbia.edu/tag/07JG}{Tag 07JG}, \href{https://stacks.math.columbia.edu/tag/07JH}{Tag 07JH}]{stacks-project} for $\padic$ completeness and the connection).
	Then, using the diagram on the left of \eqref{eq:rin_arone} and the fact that $\paze$ is an $F\crystal$, we see that $\paze(\Abar(1)) \lisomorphic \Abar(1) \otimes_{p_2, \ARpd} N^{\PD} = \Nbar(1)$, compatible with the respective $(\varphi, \Gamma_F)\textrm{-actions}$ and $\Lambda_R\linear$ connections.
	Moreover, note that the right vertical arrow of the diagram on the right of \eqref{eq:rin_arone} factors through $\Spec(\Abar(1)/p^m) \xrightarrow{\iota_{\Lambda}} \Spec(\Lambda_R/p^m) \xhookrightarrow{\Delta'} \Spec(\Abar(1)/p^m)$.
	Therefore, from the fact that $\paze$ is an $F\crystal$, we obtain $(\varphi, \Gamma_F)\equivariant$ isomorphisms
	\begin{equation*}
		\overline{\varepsilon} \colon \Abar(1) \otimes_{\iota_{\Lambda}, \Lambda_R} \paze(\Lambda_R) \isomorphic \paze(\Abar(1)) \lisomorphic \Abar(1) \otimes_{p_2, \ARpd} N^{\PD} = \Nbar(1),
	\end{equation*}
	and that $\Lambda_R \otimes_{\Delta', \Abar(1)} \Nbar(1) \isomorphic \paze(\Lambda_R)$, where $\paze(\Lambda_R)$ denotes the limit over $m$ of the evaluation of the $F\crystal$ $\paze$ on $\Spec(\Lambda_R/p^m) \xrightarrow{id} \Spec(\Lambda_R/p^m)$.
	Additionally, note that the isomorphism $\iota \colon \Lambda_R \isomorphic \ARpd$ induces an isomorphism $\paze(\Lambda_R) \isomorphic \Lambda_R \otimes_{\iota^{-1}, \ARpd} \paze(\ARpd) = N^{\PD}$, and so we see that $\paze(\Lambda_R)$ is a $p\textrm{-adically}$ complete finite $\Lambda_R\module$.

	Using that the composition $\ARpd \xrightarrow{p_2} \Abar(1) \xrightarrow{\Delta'} \Lambda_R$ coincides with the inverse of $\iota \colon \Lambda_R \isomorphic \ARpd$, we observe that the base change of $\overline{\varepsilon}$ along $\Delta' \colon \Abar(1) \rightarrow \Lambda_R$ is the identity on $\paze(\Lambda_R)$.
	So, from the interpretation of $\Abar(1)/p^m$ as the self product of $\Lambda_R/p^m$ in $\CRYS(X_m/\Sigma_m)$ (see Remark \ref{rem:arpd_selfprod_abarone} below), we see that the isomorphism $\overline{\varepsilon}$ is just the stratification of $\paze(\Lambda_R)$ over $\Abar(1)$.
	Furthermore, note that the isomorphism $\overline{\varepsilon}$ is compatible with the respective $\Lambda_R\linear$ connections, where the $\Lambda_R\linear$ connection on the source is given as $\nabla \otimes 1$ with $\nabla$ being the $\Lambda_R\linear$ connection on $\Abar(1)$ described above.
	Indeed, from the compatibility of $\overline{\varepsilon}$ with the $\Lambda_R\linear$ connection on $\Abar(1)$ given via the explicit formula described in Proposition \ref{prop:a1modp1mu_connection}, it is enough to show that $\overline{\varepsilon}(\paze(\Lambda_R)) \subset \Nbar(1)^{\nabla=0}$.
	This follows from the explicit formula for $\overline{\varepsilon}$ on $\paze(\Lambda_R)$, given as $\overline{\varepsilon}(x) = \sum_{j_1, \ldots, j_d \geqslant 0} (-1)^{j_1+\cdots+j_d} \prod_{i=1}^d \nabla_i^{j_i}(x) \otimes T_i^{[j_i]}$, for $x$ in $\paze(\Lambda_R)$ and $\nabla_i = [X_i^{\flat}]^{-1} \nabla_i^{\log}$ (see Proposition \ref{prop:n1modp1mu_connection} for the definition of $\nabla_i^{\log}$).

	Now, consider the de Rham complex $\Omega^{\bullet}_{\Abar(1)/\Lambda_R} \isomorphic \Abar(1) \otimes_{P_R} \Omega^{\bullet}_{P_R/R}$ (see Remark \ref{rem:lambda_connection} for the isomorphism and Remark \ref{rem:a1modp1mu_lambda} for the definition of $P_R$), regarded as a complex of $P_R\modules$ via the map $P_R \rightarrow \Abar(1)$ induced by the inverse of the map \eqref{eq:a1modp1mu_lambda}.
	Then, from the proof of \cite[Chapitre V, Lemme 2.1.2]{berthelot-cohomologie-cristalline} or \cite[\href{https://stacks.math.columbia.edu/tag/07LC}{Tag 07LC}]{stacks-project}, we see that the de Rham complex $\Omega^{\bullet}_{P_R/R}$ is acyclic in positive degrees.
	Hence, from the following isomorphism of de Rham complexes:
	\begin{equation*}
		(\Abar(1) \otimes_{\Lambda_R} \paze(\Lambda_R)) \otimes_{\Abar(1)} \Omega^{\bullet}_{\Abar(1)/\Lambda_R} \isomorphic \Nbar(1) \otimes_{\Abar(1)} \Omega^{\bullet}_{\Abar(1)/\Lambda_R},
	\end{equation*}
	and the acyclicity in positive degrees of the left hand complex, it follows that the de Rham complex $\Nbar(1) \otimes_{\Abar(1)} \Omega^{\bullet}_{\Abar(1)/\Lambda_R}$ is also acyclic in positive degrees.

	Using the preceding discussion and by taking the horizontal sections for the respective $\Lambda_R\linear$ connections in the isomorphism $\overline{\varepsilon}$, we obtain $(\varphi, \Gamma_F)\equivariant$ isomorphisms of $p\textrm{-adically}$ complete and $p\torsionfree$ finite $\Lambda_R\modules$ $\paze(\Lambda_R) \isomorphic \paze(\Abar(1))^{\nabla=0} \lisomorphic \Nbar(1)^{\nabla=0}$, as claimed in \eqref{eq:mlambda_arpdn}.
	Composing the preceding isomorphism with the $(\varphi, \Gamma_F)\equivariant$ composition $\Nbar(1)^{\nabla=0} \subset \Nbar(1) \xrightarrow{\Delta'_N} \Lambda_R \otimes_{\ARpd} N^{\PD} = \paze(\Lambda_R)$ yields the identity on $\paze(\Lambda_R)$.
	Thus, we have the following $(\varphi, \nabla, \Gamma_F)\equivariant$ diagram:
	\begin{equation}\label{eq:crys_strat_horz_sec}
		\begin{tikzcd}
			\Abar(1) \otimes_{\Lambda_R} \paze(\Lambda_R) \arrow[r, "\overline{\varepsilon}", "\sim"'] & \Abar(1) \otimes_{\ARpd} N^{\PD} \arrow[d, equal]\\
			\Abar(1) \otimes_{\Lambda_R} \Nbar(1)^{\nabla=0} \arrow[r] \arrow[u, "1 \otimes \Delta'_N", "\wr"'] & \Abar(1) \otimes_{p_2, \AR} N,
		\end{tikzcd}
	\end{equation}
	where the bottom horizontal arrow is taken to be the composition of the left vertical, top horizontal and right vertical arrows.
	In particular, the bottom horizontal arrow of \eqref{eq:crys_strat_horz_sec} constructs the claimed isomorphism \eqref{eq:abarone_comp} (a simple diagram chase in \eqref{eq:crys_strat_horz_sec} shows that \eqref{eq:abarone_comp} is the $\AR(1)/(p_1(\mu))\linear$ extension of the natural $\Lambda_R\linear$ inclusion $\MLambda \subset \AR(1)/(p_1(\mu)) \otimes_{p_2, \AR} N$).
	This allows us to conclude.
\end{proof}

\begin{rem}\label{rem:arpd_selfprod_abarone}
	In the notation of Section \ref{subsec:crys_site_crystal}, we may take $A$ to be $\Lambda_F$ with the PD-structure on it given by $p(\Lambda_F/p^m) + J\Lambda_F/p^m$, where $J = (\{(\mu^{p-1}/p)^{[k]}\}_{k \geqslant 1}) \subset \Lambda_F$, and $B = D$ to be $A^{\PD}$ which is a $p\textrm{-adically}$ completed \'etale algebra over $\Lambda_F$ (see Remark \ref{rem:polynom_powerseries_comp} and Example \ref{exam:arpd_qconnection}).
	In particular, we set $\Sigma_m = \Spec(\Lambda_F/p^m)$ and $X_m = Y_m = \Spec(A^{\PD}/p^m)$.
	Now, recall that we have an isomorphism of rings $\iota \colon \Lambda_R \isomorphic \ARpd$ from \eqref{eq:pd_iota}.
	Moreover, from Proposition \ref{prop:a1modp1mu_pdring} and Remark \ref{rem:a1modp1mu_lambda}, recall that $\Abar(1)$ is the $\padic$ completion of a PD-polynomial algebra over $\ARpd$ in variables $T_1, \ldots, T_d$, so we have a PD-thickening $\Spec(\Lambda_R/p^m) \hookrightarrow \Spec(\Abar(1)/p^m)$ induced by the surjection $\Delta' \colon \Abar(1) \twoheadrightarrow \Lambda_R$ (see before the statement of Proposition \ref{prop:abarone_comp}) with the composition $\Lambda_R \xrightarrow{\hspace{1mm} \iota_{\Lambda} \hspace{1mm}} \Abar(1) \twoheadrightarrow \Lambda_R$ being the identity.
	In particular, we see that $\Spec(\Lambda_R/p^m) \hookrightarrow \Spec(\Abar(1)/p^m)$ is an object of $\CRYS(X_m/\Sigma_m)$.
	Furthermore, by identifying the rings $\Lambda_R$ and $\ARpd$ via the isomorphism $\iota$, it is easy to see that $\Abar(1)/p^m$ is the self product of $\Lambda_R/p^m$ in $\CRYS(X_m/\Sigma_m)$ with the two maps $\iota_{\Lambda} \colon \Lambda_R \rightarrow \Abar(1)$ and $p_2 \colon \ARpd \rightarrow \Abar(1)$ describing the two projection maps.
\end{rem}

\subsection{Arithmetic descent}\label{subsec:arithmetic_descent}

In this section, we shall carry out descent for the arithmetic part of $\Gamma_R$, i.e.\ $\Gamma_F$.
This will be worked out in two steps, namely, the torsion part, i.e.\ $\Gamma_{\textrm{tor}}$, and the (free) arithmetic part, i.e.\ $\Gamma_0$ (see Section \ref{subsec:setup_nota} for notations and explanations).
However, since $\Gamma_F$ has different shape for $p \geqslant 3$, compared to $p = 2$, therefore, in Sections \ref{subsubsec:fpx_descent} and \ref{subsubsec:gamma0_descent} we will assume that $p \geqslant 3$.
The case of $p =2$ will be handled in Section \ref{subsubsec:a1_galact_p=2}.

\subsubsection{The action of \texorpdfstring{$\FF_p^{\times}$}{-}}\label{subsubsec:fpx_descent}

Our first goal is to obtain the descent of Wach modules, for the action of $\FF_p^{\times}$, in the case $p \geqslant 3$.
Let $N$ be a Wach module over $\AR$ and let us consider it as a module over $R\llbracket\mu\rrbracket$ via the $(\varphi, \Gamma_F)\equivariant$ isomorphism of rings $\iota \colon R \llbracket\mu\rrbracket \isomorphic \AR$ (see Section \ref{subsubsec:ring_ar+}).
In particular, by abusing notations, we will consider $N$ to be an $R\llbracket\mu\rrbracket\module$ equipped with an $R\llbracket\mu\rrbracket\linear$ Frobenius isomorphism $\varphi^*(N)[1/[p]_q] \isomorphic N[1/[p]_q]$, and an $R\linear$ and continuous action of $\Gamma_F$ commuting with the Frobenius and such that the action of $\Gamma_F$ is trivial on $N/\mu N$.
Next, from Sections \ref{subsubsec:ring_ar+} and \ref{subsubsec:rmu0}, recall that we have $\mu_0 = -p + \sum_{a \in \FF_p} (1+\mu)^{[a]}$ and $\ptilde = \mu_0+p$, as elements of $R\llbracket\mu\rrbracket^{\FF_p^{\times}}$.
Moreover, from Lemma \ref{lem:rmu0_fpxinv}, we have a $(\varphi, \Gamma_0)\equivariant$ isomorphism of rings $R\llbracket\mu_0\rrbracket \isomorphic R\llbracket\mu\rrbracket^{\FF_p^{\times}}$.
Then, we claim the following:

\begin{prop}\label{prop:wachmod_fpx_descent}
	Let $N$ be a Wach module over $\AR$.
	Then, $N_0 \coloneq N^{\FF_p^{\times}}$ is a finitely generated $R\llbracket\mu_0\rrbracket\module$, equipped with a continuous and semilinear action of $\Gamma_0$, such that the action of $\Gamma_0$ is trivial on $N_0/\mu_0 N_0 \isomorphic N/\mu N$, and we have a natural $(\varphi, \Gamma_F)\equivariant$ isomorphism of $R\llbracket\mu\rrbracket\modules$ $R\llbracket\mu\rrbracket \otimes_{R\llbracket\mu_0\rrbracket} N_0 \isomorphic N$.
	Moreover, the sequences $\{p, \mu_0\}$ and $\{\mu_0, p\}$ are regular on $N_0$, and it is equipped with an $R\llbracket\mu_0\rrbracket\linear$ isomorphism $\varphi^*(N_0)[1/\ptilde] \isomorphic N_0[1/\ptilde]$, compatible with the respective natural actions of $\Gamma_0$.
\end{prop}
\begin{proof}
	From \eqref{eq:fpx_decomp}, note that we have an $\FF_p^{\times}\textrm{-decomposition}$ $N = \bigoplus_{i=0}^{p-2} N_i$, where each $N_i$ is an $R\llbracket\mu_0\rrbracket\module$ equipped with a continuous and $R\linear$ action of $\Gamma_0$.
	Now, recall that $R\llbracket\mu\rrbracket$ is flat and finite of degree $p-1$ over $R\llbracket\mu_0\rrbracket$ (see Remark \ref{rem:mu0_mu_ff}), so it follows that $N$ is finitely generated as an $R\llbracket\mu_0\rrbracket\module$.
	Since $R\llbracket\mu_0\rrbracket$ is noetherian, therefore, the $R\llbracket\mu_0\rrbracket\textrm{-submodule}$ $N_i \subset N$ is finitely generated, hence, $(p, \mu_0)\textrm{-adically}$ complete, for each $0 \leqslant i \leqslant p-1$.
	Moreover, recall that $\mu_0$ is the product of $\mu^{p-1}$ with a unit in $R\llbracket \mu \rrbracket$ (see Lemma \ref{lem:mu0_mup1_unit}), so we see that $N$ is $\mu_0\torsionfree$, therefore, it follows that each $N_i$ is $\mu_0\torsionfree$ as well.

	Next, note that by extending the natural $R\llbracket\mu_0\rrbracket\linear$ and $(\varphi, \Gamma_F)\equivariant$ inclusion $N_0 \subset N$, along the map $R\llbracket\mu_0\rrbracket \rightarrow R\llbracket\mu\rrbracket$, we obtain a natural $(\varphi, \Gamma_F)\equivariant$ map 
	\begin{equation}\label{eq:rmun0_is_n}
		N' \coloneq R\llbracket\mu\rrbracket \otimes_{R\llbracket\mu_0\rrbracket} N_0 \longrightarrow N,
	\end{equation}
	and our goal is to show that \eqref{eq:rmun0_is_n} is bijective.
	As the natural map $R\llbracket\mu_0\rrbracket \rightarrow R\llbracket\mu\rrbracket$ is flat and $N_0$ is $\mu_0\torsionfree$, therefore, we get that $N'$ is $\mu_0\torsionfree = \mu^{p-1}\torsionfree$, hence, $\mu\torsionfree$.
	Moreover, as $N_0$ is $(p, \mu_0)\adically$ complete, it also follows that $N'$ is $(p, \mu_0)\adically = (p, \mu)\adically$ complete.

	Let us first show that \eqref{eq:rmun0_is_n} is surjective.
	Indeed, note that we have an $R\linear$ and $\Gamma_F\equivariant$ surjective map $N \rightarrow N/\mu N$.
	Then, from the $\FF_p^{\times}\textrm{-decomposition}$ in \eqref{eq:fpx_decomp}, we may rewrite the preceding map as $\bigoplus_{i=0}^{p-2} N_i \rightarrow \bigoplus_{i=0}^{p-2} (N/\mu N)_i$ which is $R\linear$ and $\Gamma_F\equivariant$, in particular, it is termwise surjective, i.e.\ the induced $R\linear$ map $N_i \rightarrow (N/\mu N)_i$ is surjective for each $0 \leqslant i \leqslant p-2$.
	However, since the action of $\Gamma_F$ is trivial on $N/\mu N$, therefore, we obtain that $N/\mu N = \bigoplus_{i=0}^{p-2} (N/\mu N)_i = (N/\mu N)_0$ and it follows that the natural $R\linear$ map $N_0 \rightarrow N/\mu N$ is surjective.
	As $\mu$ is in the Jacobson radical of $R\llbracket\mu\rrbracket$, therefore, by Nakayama Lemma we see that the natural $(\varphi, \Gamma_F)\equivariant$ map $R\llbracket\mu\rrbracket \otimes_{R\llbracket\mu_0\rrbracket} N_0 \rightarrow N$ is surjective as well.

	Next, note that $N$ is $(p, \mu)\adically$ complete and $\mu\torsionfree$, therefore, to show that \eqref{eq:rmun0_is_n} is bijective it is enough to show that \eqref{eq:rmun0_is_n} is bijective modulo $\mu$.
	Now, consider the following commutative diagram:
	\begin{equation}\label{eq:n0mu0_to_nmu}
		\begin{tikzcd}
			N'/\mu N' \arrow[r, twoheadrightarrow] \arrow[d, equal] & N/\mu N\\
			N_0/\mu_0 N_0 \arrow[ru],
		\end{tikzcd}
	\end{equation}
	where the top arrow is surjective by the discussion above, the slanted arrow is the natural $R\linear$ and $\Gamma_F\equivariant$ map induced by the inclusion $N_0 \subset N$ and the left vertical equality follows because we have that
	\begin{equation*}
		N'/\mu N' = (N'/\mu_0 N')/\mu = (R\llbracket\mu\rrbracket/(\mu^{p-1}) \otimes_R N_0/\mu_0 N_0)/\mu = N_0/\mu_0 N_0.
	\end{equation*}
	To show that \eqref{eq:rmun0_is_n} is injective modulo $\mu$, it is enough to show that the slanted arrow in \eqref{eq:n0mu0_to_nmu} is injective.
	So, set $N'' \coloneq \mu N \cap N_0 \subset N$ as an $R\llbracket \mu_0 \rrbracket\module$ and note that we have a natural $(\varphi, \Gamma_F)\equivariant$ inclusion $N'' \subset \mu N$.
	The preceding inclusion induces a $\Gamma_F\equivariant$ map $N'' \rightarrow \mu N/\mu^2 N$, where the source admits a trivial action of $\FF_p^{\times}$ and the target admits a non-trivial action of $\FF_p^{\times}$ (see Remark \ref{rem:fpx_action_munmu2}).
	So it follows that $N'' = \mu^2 N \cap N_0 \subset N$.
	Iterating the preceding argument $p-2$ times, we obtain that $N'' = \mu^{p-1} N \cap N_0 = \mu_0 N \cap N_0 \subset N = \mu_0 N_0$, where the last equality follows because for any $x$ in $N$ and $g$ in $\FF_p^{\times}$, we have that $g(\mu_0 x) = \mu_0 x$ if and only if $g(x) = x$, i.e.\ $x$ is in $N_0$.
	Thus, from these observations, it follows that the natural map $N_0/\mu_0 N_0 \rightarrow N/\mu N$ is injective.
	Hence, from \eqref{eq:n0mu0_to_nmu}, we conclude that $N_0/\mu_0 N_0 \isomorphic N/\mu N$ as $R\modules$ and \eqref{eq:rmun0_is_n} is bijective.
	Moreover, we get that the action of $\Gamma_0$ is trivial on $N_0/\mu_0 N_0$, and from \eqref{eq:rmun0_is_n} and \eqref{eq:n0mu0_to_nmu}, it also follows that the sequences $\{p, \mu_0\}$ and $\{\mu_0, p\}$ are regular on $N_0$.

	Finally, let us show the Frobenius finite height condition on $N_0$.
	Note that since $\ptilde$ is the product of $[p]_q$ with a unit in $R\llbracket\mu\rrbracket$ (see Lemma \ref{lem:pq_mu0p_unit}), therefore, the Frobenius finite height condition on $N$ may also be stated as an $R\llbracket\mu\rrbracket\linear$ isomorphism $\varphi^*(N)[1/\ptilde] \isomorphic N[1/\ptilde]$.
	Now, as the Frobenius on $N$ commutes with the action of $\Gamma_F$, therefore, taking the invariants of the preceding isomorphism under the action of $\FF_p^{\times}$ and using the $(\varphi, \Gamma_F)\equivariant$ isomorphism in \eqref{eq:rmun0_is_n}, we obtain that $N_0$ is equipped with an $R\llbracket\mu_0\rrbracket\linear$ isomorphism $\varphi^*(N_0)[1/\ptilde] \isomorphic N_0[1/\ptilde]$, compatible with the natural action of $\Gamma_0$ on each side.
	This allows us to conclude.
\end{proof}

\begin{rem}\label{rem:fpx_action_munmu2}
	Let $N$ be a Wach module over $\AR$.
	Then, as the action of $\Gamma_R$ is trivial on $N/\mu N$, therefore, we see that for each $k \in \NN$, over the quotient $\mu^k N/ \mu^{k+1} N$, the action of $\Gamma_R$ is given via the $\padic$ cyclotomic character.
	In particular, it follows that $\FF_p^{\times}$ has a non-trivial action on $\mu^k N/ \mu^{k+1} N$, for $1 \leqslant k \leqslant p-2$.
\end{rem}

\begin{rem}
	In Proposition \ref{prop:wachmod_fpx_descent}, for $R = O_F$, note that the $O_F\llbracket\mu_0\rrbracket\module$ $N_0$ is $p\torsionfree$ and $\mu_0\torsionfree$.
	Moreover, $N_0/\mu_0 N_0 \isomorphic N/\mu N$ is $p\torsionfree$.
	Therefore, from \cite[Lemma 3.3]{abhinandan-relative-wach-ii} and \cite[Proposition B.1.2.4]{fontaine-phigamma}, it follows that $N_0$ is finite free over $O_F\llbracket\mu_0\rrbracket$.
\end{rem}

\begin{prop}\label{prop:mlambda_fpx_descent}
	Let $N$ be a Wach module over $\AR$ and let $\MLambda = (\AR(1)/(p_1(\mu)) \otimes_{p_2, \AR} N)^{1 \times \Gamma_R'}$ be the $\Lambda_R\module$ from Proposition \ref{prop:abarone_comp}.
	Then, $\MLambdao \coloneq \MLambda^{1 \times \FF_p^{\times}}$ is a finitely generated, $p\textrm{-adically}$ complete and $p\torsionfree$ module over $\Lambda_{R, 0} = \Lambda_R^{\FF_p^{\times}}$, equipped with an induced semilinear and continuous action of $\Gamma_0$, such that the action of $\Gamma_0$ is trivial on $\MLambdao/\mu_0 \MLambdao$.
	Moreover, $\MLambdao$ is equipped with an induced $\Lambda_{R, 0}\linear$ isomorphism $\varphi^*(\MLambdao)[1/p] \isomorphic \MLambdao[1/p]$, compatible with the natural action of $\Gamma_0$ on each side.
	Furthermore, the $\Lambda_R\linear$ extension of the $\Lambda_{R,0}\linear$ inclusion $\MLambdao \subset \MLambda$, yields a natural $(\varphi, \Gamma_F)\equivariant$ isomorphism of $\Lambda_R\modules$
	\begin{equation*}
		\Lambda_R \otimes_{\Lambda_{R, 0}} \MLambdao \isomorphic \MLambda.
	\end{equation*}
\end{prop}
\begin{proof}
	From Proposition \ref{prop:abarone_comp}, recall that $\MLambda$ is a finitely generated, $p\textrm{-adically}$ complete and $p\torsionfree$ $\Lambda_R\module$ and we have a $(\varphi, \Gamma_F)\equivariant$ isomorphism of $\Lambda_R\modules$ $\MLambda \isomorphic \Lambda_R \otimes_{\AR} N$ via the composition $\AR \rightarrow A^{\PD} \isomorphic \Lambda_R$, where the last isomorphism is the inverse of \eqref{eq:pd_iota}.
	Now, consider $N$ as an $R\llbracket\mu\rrbracket\module$ via the $(\varphi, \Gamma_F)\equivariant$ isomorphism of rings $\iota \colon R \llbracket\mu\rrbracket \isomorphic \AR$ (see Section \ref{subsubsec:ring_ar+}), equipped with a $(\varphi, \Gamma_F)\action$ (see Section \ref{subsubsec:fpx_descent}).
	Then, from Proposition \ref{prop:wachmod_fpx_descent}, we have that $N_0 \coloneq N^{\FF_p^{\times}}$ is a finitely generated and $(p, \mu_0)\textrm{-adically}$ complete $R\llbracket\mu_0\rrbracket\module$, equipped with an induced action of $(\varphi, \Gamma_0)$.
	Moreover, from Proposition \ref{prop:wachmod_fpx_descent}, note that the $R\llbracket\mu\rrbracket\linear$ extension of the natural inclusion $N_0 \subset N$, induces a $(\varphi, \Gamma_F)\equivariant$ isomorphism of $R\llbracket\mu\rrbracket\modules$ $R\llbracket\mu\rrbracket \otimes_{R\llbracket\mu_0\rrbracket} N_0 \isomorphic N$.
	Combining this with the preceding discussion, it follows that we have a $(\varphi, \Gamma_F)\equivariant$ isomorphism of $\Lambda_R\modules$ $\MLambda \isomorphic \Lambda_R \otimes_{R\llbracket\mu_0\rrbracket} N_0$.
	Using that $\Lambda_{R, 0} = \Lambda_R^{\FF_p^{\times}}$ (see Construction \ref{const:lambdatilder0}), together with the preceding isomorphism, we obtain a $(\varphi, \Gamma_0)\equivariant$ isomorphism of $\Lambda_{R, 0}\modules$ 
	\begin{equation}\label{eq:mlambda0_n0}
		\MLambdao = \MLambda^{1 \times \FF_p^{\times}} \isomorphic \Lambda_{R, 0} \otimes_{R\llbracket\mu_0\rrbracket} N_0.
	\end{equation}
	In particular, the $\Lambda_{R, 0}\module$ $\MLambdao$ is finitely generated, $p\torsionfree$ and $p\textrm{-adically}$ separated (since $\MLambda$ is $p\torsionfree$ and $p\textrm{-adically}$ separated), hence, $p\textrm{-adically}$ complete.

	Now, let $g$ be any element of $\Gamma_0$.
	Then, note that for any $f$ in $\Lambda_{R, 0}$ and $y$ in $N_0$, we have that
	\begin{equation*}
		(\gamma_0-1)fy = (\gamma_0-1)f \cdot y + \gamma_0(f) (\gamma_0-1)y \in \mu_0 (\Lambda_{R, 0} \otimes_{\AR} N).
	\end{equation*}
	From the $(\varphi, \Gamma_0)\equivariant$ isomorphism in \eqref{eq:mlambda0_n0}, it follows that for any $x$ in $\MLambdao$, we have that $(g-1)x$ is an element of $\mu_0\MLambdao$.
	Furthermore, from \eqref{eq:mlambda0_n0} it also follows that the $\Lambda_R\linear$ extension of the natural $\Lambda_{R,0}\linear$ inclusion $\MLambdao \subset \MLambda$, gives a $(\varphi, \Gamma_F)\equivariant$ isomorphism of $\Lambda_R\modules$ 
	\begin{equation}\label{eq:lambdaR_comp}
		\Lambda_R \otimes_{\Lambda_{R, 0}} \MLambdao \isomorphic \MLambda.
	\end{equation}

	Finally, note that from Proposition \ref{prop:wachmod_fpx_descent}, we have an $R\llbracket\mu_0\rrbracket\linear$ isomorphism $\varphi^*(N_0)[1/\ptilde] \isomorphic N_0[1/\ptilde]$ compatible with the action of $\Gamma_0$ on each side.
	Then, by extending this isomorphism $\Lambda_{R, 0}\textrm{-linearly}$, using the $(\varphi, \Gamma_F)\equivariant$ isomorphism in \eqref{eq:mlambda0_n0} and noting that $\ptilde/p$ is a unit in $\Lambda_{F, 0}$ from Lemma \ref{lem:tovermu_unit}, we obtain a $\Lambda_{R, 0}\linear$ isomorphism $\varphi^*(\MLambdao)[1/p] \isomorphic \MLambdao[1/p]$ compatible with the action of $\Gamma_0$ on each side.
	Hence, the proposition is proved.
\end{proof}

\subsubsection{The action of \texorpdfstring{$1 + p\ZZ_p$}{-}}\label{subsubsec:gamma0_descent}

We will assume that $p \geqslant 3$ and show the descent step, for the action of $1 \times \Gamma_0 \isomorphic 1 \times (1 + p\ZZ_p)$.
Let $N$ be a Wach module over $\AR$, let $\MLambda = (\AR(1)/(p_1(\mu)) \otimes_{p_2, \AR} N)^{1 \times \Gamma_R'}$ be the $\Lambda_R\module$ from Proposition \ref{prop:abarone_comp} and let $\MLambdao = \MLambda^{1 \times \FF_p^{\times}}$ be the $\Lambda_{R, 0}\module$ from Proposition \ref{prop:mlambda_fpx_descent}.
Let $\gamma_0$ be a topological generator of $\Gamma_0$ such that $\chi(\gamma_0) = 1+pa$, for a unit $a$ in $\ZZ_p$.
Then, from Proposition \ref{prop:mlambda_fpx_descent} note that for any $x$ in $\MLambdao$, we have that $(\gamma_0-1)x$ is an element of $\mu_0 \MLambdao$.

Set $s \coloneq \mu_0/p$ in $\Lambda_{R, 0}$, and from Lemma \ref{lem:gamma0_act_s} recall that $(\gamma_0-1)s = u \mu_0$, for some unit $u$ in $\Lambda_{F, 0}$ depending on $\gamma_0$.
Therefore, we see that the following operator is well defined:
\begin{equation}\label{eq:nablaq_s_mlambda0}
	\begin{aligned}
		\nabla_{q, s} \colon \MLambdao &\longrightarrow \MLambdao\\
				x &\mapsto \tfrac{(\gamma_0-1)x}{(\gamma_0-1)s}.
	\end{aligned}
\end{equation}
As the operator $\nabla_{q, s}$ is an endomorphism of $\MLambdao$, we may define the following two term Koszul complex:
\begin{equation}\label{eq:s_qderham_mlambda0}
	K_{\MLambdao}(\nabla_{q, s}) \colon \big[\MLambdao \xrightarrow{\hspace{1mm} \nabla_{q, s} \hspace{1mm}} \MLambdao\big].
\end{equation}

\begin{rem}
	Considering $s$ as a variable, similar to Remark \ref{rem:nablaq_stilde_qconnection}, the operator $\nabla_{q, s}$ in \eqref{eq:nablaq_s_mlambda0}, may also be considered as a $\qconnection$ in non-logarithmic coordinates, in the sense of Definition \ref{defi:qconnection} and Remark \ref{rem:nablaqi_nota}.
	Then, \eqref{eq:s_qderham_mlambda0} is the $q\textrm{-de Rham}$ complex arising from such a $\qconnection$.
\end{rem}

\begin{prop}\label{prop:mlambda0_connection}
	The series of operators $\nabla_0^{\log} = \frac{\log \gamma_0}{\log(\chi(\gamma_0))} = \frac{1}{\log(\chi(\gamma_0))}\sum_{k \in \NN} (-1)^k \frac{(\gamma_0-1)^{k+1}}{k+1}$ converge $p\adically$ on $\MLambdao$.
	Let $z \coloneq t^{p-1}/p$ in $\Lambda_{R, 0}$, then the operator $\nabla_0 \coloneq \tfrac{1}{(p-1)z}\nabla_0^{\log}$ defines an $R\linear$ $p\adically$ quasi-nilpotent flat connection on $\MLambdao$, denoted by $\nabla \colon \MLambdao \rightarrow \MLambdao \otimes_{\Lambda_{R, 0}} \Omega^1_{\Lambda_{R, 0}/R}$ and given as $x \mapsto \nabla_0(x) dz$.
	The data of the connection $\nabla$ on $\MLambdao$ is equivalent to the data of the $\qconnection$ $\nabla_{q, s}$ from \eqref{eq:nablaq_s_mlambda0}, i.e.\ either may be recovered from the other.
	Moreover, the $q\textrm{-de Rham}$ complex $K_{\MLambdao}(\nabla_{q, s})$ in \eqref{eq:s_qderham_mlambda0} is naturally quasi-isomorphic to the de Rham complex $\MLambdao \otimes_{\Lambda_{R, 0}} \Omega^{\bullet}_{\Lambda_{R, 0}/R}$.
	In particular, we have that $\MLambdao^{\nabla_{q, s}=0} \isomorphic \MLambdao^{\nabla_0=0}$.
\end{prop}
\begin{proof}
	Recall that $s = \mu_0/p = ut^{p-1}/p$, for a unit $u$ in $\Lambda_{F, 0}$ (see Construction \ref{const:lambdatilder0}, in particular, \eqref{eq:lambdar0_explicit} and the discussion preceding it).
	Moreover, from Lemma \ref{lem:gamma0_act_stilde} we have that $(\gamma_0-1)s = (\gamma_0-1)\tfrac{\mu_0}{p} = v \mu_0$, for a unit $v$ in $\Lambda_{F, 0}$.
	Now, let $z = t^{p-1}/p$ and we write $(1+pa)^{p-1} = 1 + pb$, where $b$ is a unit in $\ZZ_p$.
	Then, note that we have 
	\begin{equation}\label{eq:gamma0_z}
		(\gamma_0-1)z = (\gamma_0-1)\tfrac{t^{p-1}}{p} = (\chi(\gamma_0)^{p-1}-1)\tfrac{t^{p-1}}{p} = ((1+pa)^{p-1}-1)\tfrac{t^{p-1}}{p} = b t^{p-1} = u^{-1}b \mu_0.
	\end{equation}
	Therefore, it follows that the complex $K_{\MLambdao}(\nabla_{q, s})$ is quasi-isomorphic to the following complex:
	\begin{equation}\label{eq:z_qderham_mlambda0}
		K_{\MLambdao}(\nabla_{q, z}) \colon \big[\MLambdao \xrightarrow{\hspace{1mm} \nabla_{q, z} \hspace{1mm}} \MLambdao\big].
	\end{equation}
	Rest of the proof is similar to the proof of Proposition \ref{prop:n1modp1mu_connection}, with some changes.
	To avoid confusion, let us provide additional details.

	We will first show that $\nabla_0^{\log} \coloneq \tfrac{\log(\gamma_0)}{\log(\chi(\gamma_0))} = \tfrac{1}{\log(\chi(\gamma_0))} \sum_{k \in \NN} (-1)^k \frac{(\gamma_0-1)^{k+1}}{k+1}$, converge as a series of operators on $\MLambdao$.
	Indeed, let $x$ be any element of $\MLambdao$, then using Proposition \ref{prop:mlambda_fpx_descent}, we may write $(\gamma_0-1)x = t^{p-1} x_1$, for some $x_1$ in $\MLambdao$.
	Let us note that $\log(\chi(\gamma_0)) = \log(1+pa) = p c$, where $c$ is a unit in $\ZZ_p$, and we also have $(\gamma_0-1)t^{p-1} = ((1+pa)^{p-1}-1)t^{p-1} = p b t^{p-1}$, where $b$ is a unit in $\ZZ_p$.
	Therefore, an easy induction on $k \in \NN$, shows that $(\gamma_0-1)^{k+1} x = p^k t^{p-1} x_{k+1}$, for some $x_{k+1}$ in $\MLambdao$.
	In particular, we get that
	\begin{align*}
		\nabla_0^{\log}(x) &= \tfrac{1}{\log(\chi(\gamma_0))} \textstyle\sum_{k \in \NN} (-1)^k \tfrac{(\gamma_0-1)^{k+1}(x)}{k+1}\\
			&= \tfrac{1}{\log(\chi(\gamma_0))} \textstyle\sum_{k \in \NN} (-1)^k \tfrac{p^k t^{p-1} x_{k+1}}{k+1} = \tfrac{t^{p-1}}{pc} \textstyle\sum_{k \in \NN} (-1)^k \tfrac{p^k x_{k+1}}{k+1},
	\end{align*}
	which converges in $\tfrac{t^{p-1}}{p} \MLambdao = z \MLambdao$, since $\MLambdao$ is $p\adically$ complete (see Proposition \ref{prop:mlambda_fpx_descent}).
	Moreover, it also follows that the operator $\nabla_0 \coloneq \tfrac{1}{(p-1)z}\nabla_0^{\log}$ is well defined on $\MLambdao$.
	Next, similar to the case of $\nabla_i^{\log}$ in the proof of Proposition \ref{prop:n1modp1mu_connection}, it can be shown that $\nabla_0^{\log}$, and therefore $\nabla_0 = \tfrac{1}{(p-1)z}\nabla_0^{\log}$ satisfies a Leibniz rule, i.e.\ $\nabla_0(fx) = \nabla_0(f)x + f\nabla_0(x)$, where the first operator on the right is $\nabla_0 \coloneq \tfrac{\log(\gamma_0)}{(p-1)z\log(\chi(\gamma_0))} \colon \Lambda_{R, 0} \rightarrow \Lambda_{R, 0}$, whose well definedness can be checked similar to above.
	Moreover, note that the operator $\nabla_0$ is flat by definition.
	Furthermore, similar to the case of $\nabla_i$ in the proof of Proposition \ref{prop:n1modp1mu_connection}, it can be shown that the operator $\nabla_0 \colon \Lambda_{R, 0} \rightarrow \Omega^1_{\Lambda_{R, 0}/R}$ is the continuous de Rham differential operator $d \colon \Lambda_{R, 0} \rightarrow \Omega^1_{\Lambda_{R, 0}/R}$.
	So, in particular, the operator $\nabla \colon \MLambdao \rightarrow \MLambdao \otimes_{\Lambda_{R, 0}} \Omega^1_{\Lambda_{R, 0}/R}$, given as $x \mapsto \nabla_0(x) dz$, is a well-defined flat connection.

	Next, let us show that the operator $\nabla_0$ is $p\adically$ quasi-nilpotent.
	Indeed, we first note that from the commutativity of $\varphi$ and $\gamma_0$, it follows that $\nabla_0^{\log} \circ \varphi = \varphi \circ \nabla_0^{\log}$.
	Therefore, it is easy to see that $\nabla_0 \circ \varphi = p^{p-1} \varphi \circ \nabla_0$.
	Recall that $\MLambdao$ is equipped with a $\Lambda_{R, 0}\linear$ isomorphism $\varphi^*(\MLambdao)[1/p] \isomorphic \MLambdao[1/p]$, compatible with the action of $\Gamma_0$ on each side.
	In particular, for any $x$ in $\MLambdao$, there exists $r \in \NN$ large enough, such that $p^rx$ belongs to $\varphi^*(\MLambdao)$.
	Then, from the relation $\nabla_0 \circ \varphi = p^{p-1} \varphi \circ \nabla_0$, we see that $\nabla_0^k(p^r x)$ converges $p\adically$ to $0$ as $k \rightarrow +\infty$.
	Hence, it follows that $\nabla_0^k(x) = p^{-r} \nabla_0^k(p^r x)$ converges $p\adically$ to 0 as $k \rightarrow +\infty$, in particular, $\nabla_0$ is $p\adically$ quasi-nilpotent.

	So far, we have defined the $p\adically$ quasi-nilpotent flat connection $\nabla$ using the action of $\Gamma_0$ and conversely, we have shown that the action of $\Gamma_0$ may be recovered by the formula $\gamma_0 \coloneq \exp(\log(\chi(\gamma_0))\nabla_0^{\log})$.
	Again, similar to case of $\gamma_i$ in the proof of Proposition \ref{prop:n1modp1mu_connection}, using the Leibniz rule for $\nabla_0$, it can be checked that the action of $\gamma_0$ thus obtained, is semilinear.

	Finally, it remains to compare the $q\textrm{-de Rham}$ complex in \eqref{eq:s_qderham_mlambda0} with the de Rham complex $\MLambdao \otimes_{\Lambda_{R, 0}} \Omega^{\bullet}_{\Lambda_{R, 0}/R} = \MLambdao \otimes_{\Lambda_{F, 0}} \Omega^{\bullet}_{\Lambda_{F, 0}/O_F}$.
	As $\nabla_0$ is an endomorphism of $\MLambdao$, let $K_{\MLambdao}(\nabla_0)$ denote the corresponding Koszul complex in the sense of Definition \ref{defi:koszul_complex}.
	Then, we have an identification of complexes:
	\begin{equation*}
		\MLambdao \otimes_{\Lambda_{F, 0}} \Omega^{\bullet}_{\Lambda_{F, 0}/O_F} = K_{\MLambdao}(\nabla_0) \colon \big[\Lambda_{R, 0} \xrightarrow{\hspace{1mm} \nabla_0 \hspace{1mm}} \Lambda_{R, 0}\big].
	\end{equation*}
	Now, recall that we have $\gamma_0 = \exp(\log(\chi(\gamma_0)) \nabla_0^{\log})$.
	Therefore, we may write
	\begin{equation}\label{eq:nablaqz_nabla0}
		\nabla_{q, z} = \tfrac{\gamma_0-1}{(\gamma_0-1)z} = \tfrac{(p-1)z \log(\chi(\gamma_0))}{(\gamma_0-1)z} \nabla_0 \big( 1 + \textstyle \sum_{k \geqslant 1} \tfrac{(\log\chi(\gamma_0))^k}{(k+1)!} (\nabla_0^{\log})^k\big).
	\end{equation}
	Recall that $(\gamma_0-1)z = pbz$ and $\log(\chi(\gamma_0)) = pc$, for units $b$ and $c$ in $\ZZ_p$.
	Therefore, we have that $\tfrac{\log(\chi(\gamma_0)) (p-1)z}{(\gamma_0-1)z} = \tfrac{c(p-1)}{b}$ is a unit, and it is clear that the term inside the parentheses, on the right hand side of \eqref{eq:nablaqz_nabla0}, converges $p\adically$ to a unit.
	Now, in the notation of Lemma \ref{lem:koszul_complex_qiso}, let us set $i = 1$, $M = \MLambdao$, $f_1 = \nabla_0$ and take $h_1$ to be the product of $\tfrac{c(p-1)}{b}$ with the formula in parentheses on the right hand side of \eqref{eq:nablaqz_nabla0}, in particular, $f_1h_1 = \nabla_{q, z}$.
	Then, from Lemma \ref{lem:koszul_complex_qiso}, we obtain a natural quasi-isomorphism of complexes
	\begin{equation*}
		K_{\MLambdao}(\nabla_{q, s}) \isomorphic K_{\MLambdao}(\nabla_{q, z}) \isomorphic K_{\MLambdao}(\nabla_0).
	\end{equation*}
	In particular, we get that $\MLambdao^{\nabla_{q, s}=0} \isomorphic \MLambdao^{\nabla_0=0}$.
	This allows us to conclude.
\end{proof}

\begin{prop}\label{prop:lambdar0_comp}
	Set $M \coloneq \MLambdao^{\nabla_0=0}$ as an $R\module$ via the isomorphism $R \isomorphic \Lambda_{R, 0}^{\nabla_0=0}$.
	Then, $M$ is a finitely generated $p\torsionfree$ module over $R$.
	Moreover, the $\Lambda_{R,0}\linear$ extension of the natural $R\linear$ inclusion $M \subset \MLambdao$ yields a natural $(\varphi, \nabla_0)\equivariant$ isomorphism
	\begin{equation}\label{eq:lambdar0_comp}
		\Lambda_{R, 0} \otimes_{R} M \isomorphic \MLambdao.
	\end{equation}
	Furthermore, the de Rham complex $\MLambdao \otimes_{\Lambda_{F, 0}} \Omega^{\bullet}_{\Lambda_{F, 0}/O_F}$ is acyclic in positive degrees.
	In particular, from Proposition \ref{prop:mlambda0_connection}, we get that $H^1(K_{\MLambdao}(\nabla_{q, s})) = 0$.
\end{prop}
\begin{proof}
	The proof is similar to the proof of Proposition \ref{prop:abarone_comp}.
	To avoid confusion and for the sake of completeness, let us provide additional details.

	We shall use the results from Section \ref{subsec:crys_site_crystal}, by setting $\Sigma_m = X_m = \Spec(R/p^m)$ and $Y_m = \Spec(\Lambda_{R, 0}/p^m)$ (from Construction \ref{const:lambdatilder0}, recall that $\Lambda_{R, 0}$ is the $\padic$ completion of a PD-polynomial algebra over $R$ in the variable $s$, see \eqref{eq:lambdar0_explicit}).
	Let us first note that from Proposition \ref{prop:mlambda0_connection}, the $\Lambda_{R, 0}\module$ $\MLambdao$ is equipped with an $R\linear$ $p\adically$ quasi-nilpotent flat connection $\nabla \colon \MLambdao \rightarrow \MLambdao \otimes_{\Lambda_{R, 0}} \Omega^1_{\Lambda_{R, 0}/R} = \MLambdao \otimes_{\Lambda_{F, 0}} \Omega^1_{\Lambda_{F, 0}/O_F}$ and a Frobenius structure (after inverting $p$) compatible with the connection, in particular, $\MLambdao$ is an object of $\MIC^{\varphi}(\Lambda_{R, 0})$.
	Then, from the equivalence in \eqref{eq:crys_fcrystal_connections}, there exists a quasi-coherent $F\crystal$ $\paze$ on $X/\Sigma$ such that $\paze(\Lambda_{R, 0}) = \MLambdao$.
	Now, consider the PD-thickening $\Spec(R/p^m) \hookrightarrow \Spec(\Lambda_{R, 0}/p^m)$ induced by the $\varphi\equivariant$ surjection $\Lambda_{R, 0} \twoheadrightarrow R$, sending $s^{[k]} \mapsto 0$ for each $k \geqslant 1$, and note that the composition $R \rightarrow \Lambda_{R, 0} \twoheadrightarrow R$ is the identity and $\varphi\equivariant$.
	So, we have the following $\varphi\equivariant$ diagram in $\CRYS(X_m/\Sigma_m)$, for $m \geqslant 1$:
	\begin{equation}\label{eq:rin_lambdar0}
		\begin{tikzcd}[column sep=large]
			\Spec(R/p^m) \arrow[r, hookrightarrow, "s^{[k]} \mapsto 0"] \arrow[d, "id"] & \Spec(\Lambda_{R, 0}/p^m) \arrow[d] \\
			\Spec(R/p^m) \arrow[r, "id"] & \Spec(R/p^m).
		\end{tikzcd}
	\end{equation}

	Evaluating the $F\crystal$ $\paze$ on $\Spec(R/p^m) \hookrightarrow \Spec(\Lambda_{R, 0}/p^m)$, for each $m \geqslant 1$, and taking the limit over $m$, gives the $\Lambda_{R, 0}\module$ $\paze(\Lambda_{R, 0}) = \MLambdao$.
	Now, using diagram \eqref{eq:rin_lambdar0} and the fact that $\paze$ is an $F\crystal$, we have the following $\varphi\equivariant$ isomorphism:
	\begin{equation*}
		f \colon \Lambda_{R, 0} \otimes_R \paze(R) \isomorphic \paze(\Lambda_{R, 0}) = \MLambdao,
	\end{equation*}
	where $\paze(R)$ denotes the limit over $m$ of the evaluation of $\paze$ on $\Spec(R/p^m) \xrightarrow{id} \Spec(R/p^m)$.
	Moreover, it is easy to see that the base change of $f$ along the map $\Lambda_{R, 0} \twoheadrightarrow R$ yields a $\varphi\equivariant$ isomorphism $\paze(R) \isomorphic R \otimes_{\Lambda_{R, 0}} \paze(\Lambda_{R, 0})$, in particular, $\paze(R)$ is a finitely generated $R\module$.
	Furthermore, $f$ is compatible with the respective $R\linear$ connections, where the $R\linear$ connection on the source is given as $\nabla \otimes 1$ with $\nabla$ being the $R\linear$ connection on $\Lambda_{R, 0}$ described in the proof of Proposition \ref{prop:mlambda0_connection}.
	Indeed, from the compatibility of $f$ with the $R\linear$ connection on $\Lambda_{R, 0}$ given via the explicit formula described in Proposition \ref{prop:mlambda0_connection}, it is enough to show that $f(\paze(R)) \subset \MLambdao^{\nabla=0}$.
	This follows from the explicit formula for $f$ on $\paze(R)$, given as $f(x) = \sum_j (-1)^j\nabla_0^j(x) \otimes z^{[j]}$, for $x$ in $\paze(R)$ and $\nabla_0 = \tfrac{1}{(p-1)z}\nabla_0^{\log}$ (see Proposition \ref{prop:mlambda0_connection} for the definition of $\nabla_0^{\log}$).

	Next, consider the de Rham complex $\Omega^{\bullet}_{\Lambda_{R, 0}/R} = \Lambda_{R, 0} \otimes_{\Lambda_{F, 0}} \Omega^{\bullet}_{\Lambda_{F, 0}/O_F}$, regarded as a complex of $\Lambda_{F, 0}\modules$ via the natural map $\Lambda_{F, 0} \rightarrow \Lambda_{R, 0}$.
	Recall that $\Lambda_{F, 0} = O_F[\{s^{[k]}\}_{k \in \NN}]_p^{\wedge}$, i.e.\ the $\padic$ completion of a PD-polynomial algebra over $O_F$ in one variable, and from this it follows that the de Rham complex $\Omega^{\bullet}_{\Lambda_{R, 0}/R} = \Lambda_{R, 0} \otimes_{\Lambda_{F, 0}} \Omega^{\bullet}_{\Lambda_{F, 0}/O_F}$ is acyclic in positive degrees (also see \cite[Chapitre V, Lemme 2.1.2]{berthelot-cohomologie-cristalline} or \cite[\href{https://stacks.math.columbia.edu/tag/07LC}{Tag 07LC}]{stacks-project}).
	Hence, from the following isomorphism of de Rham complexes:
	\begin{equation*}
		(\Lambda_{R, 0} \otimes_R \paze(R)) \otimes_{\Lambda_{R, 0}} \Omega^{\bullet}_{\Lambda_{R, 0}/R} \isomorphic \MLambdao \otimes_{\Lambda_{R, 0}} \Omega^{\bullet}_{\Lambda_{R, 0}/R},
	\end{equation*}
	and the acylicity in positive degrees of the left hand complex, it follows that the de Rham complex $\MLambdao \otimes_{\Lambda_{R, 0}} \Omega^{\bullet}_{\Lambda_{R, 0}/R}$ is also acyclic in positive degrees.

	Using the preceding discussion and by taking horizontal sections for the respective connections in the isomorphism $f$ from above, we obtain a $\varphi\equivariant$ isomorphism of finite $R\modules$ $\paze(R) \isomorphic \paze(\Lambda_{R, 0})^{\nabla=0} = \MLambdao^{\nabla_0=0}$.
	Moreover, as $\MLambdao$ is $p\torsionfree$, it follows that $\MLambdao^{\nabla_0=0}$ is $p\torsionfree$.
	Furthermore, from the isomorphism $f$, we deduce that $\Lambda_{R, 0}\textrm{-linearly}$ extending the natural inclusion $\MLambdao^{\nabla_0=0} \subset \MLambdao$, we obtain the following $(\varphi, \nabla_0)\equivariant$ commutative diagram:
	\begin{center}
		\begin{tikzcd}
			\Lambda_{R, 0} \otimes_R \paze(R) \arrow[r, "f", "\sim"'] \arrow[d, "1 \otimes f", "\wr"'] & \paze(\Lambda_{R, 0}) \arrow[d, equal]\\
			\Lambda_{R, 0} \otimes_{\Lambda_R} \MLambdao^{\nabla_0=0} \arrow[r, "\eqref{eq:lambdar0_comp}"] & \MLambdao.
		\end{tikzcd}
	\end{center}
	Hence, it follows that \eqref{eq:lambdar0_comp} is an isomorphism.
	This concludes our proof.
\end{proof}

\subsubsection{The case \texorpdfstring{$p = 2$}{-}}\label{subsubsec:a1_galact_p=2}

In this section, we shall prove statements analogous to Proposition \ref{prop:mlambda_fpx_descent}, Proposition \ref{prop:mlambda0_connection} and Proposition \ref{prop:lambdar0_comp}, for $p=2$.
Let $N$ be a Wach module over $\AR$ and let $\MLambda = (\AR(1)/(p_1(\mu)) \otimes_{p_2, \AR} N)^{1 \times \Gamma_R'}$ be the $\Lambda_R\module$ from Proposition \ref{prop:abarone_comp} equipped with an induced action of $(\varphi, \Gamma_F)$.
First, we will look at the action of $\Gamma_{\textrm{tor}}$ on $\MLambda$.
Let $\sigma$ denote a generator of $\Gamma_{\textrm{tor}}$.
Then, from \eqref{eq:plus_minus_decomp}, recall that by setting $\MLambdaplus \coloneq \{x \in \MLambda \textrm{ such that } \sigma(x) = x\}$ and $\MLambdaminus \coloneq \{x \in \MLambda \textrm{ such that } \sigma(x) = -x\}$, we have a natural injective map of torsion-free modules over $\LambdaRplus$ (see Section \ref{subsubsec:p=2_gammaF_action}),
\begin{equation}\label{eq:mlambda_pm}
	\MLambdaplus \oplus \MLambdaminus \longrightarrow \MLambda,
\end{equation}
given as $(x, y) \mapsto x+y$.
Note that the action of $(\varphi, \Gamma_F)$ on $\MLambda$ induces a natural Frobenius structure (after inverting $p$) and a continuous $\Gamma_0\textrm{-action}$ on $\MLambdaplus$ and $\MLambdaminus$, respectively.
Equipping $\MLambdaplus$ and $\MLambdaminus$ with the induced structures, we see that \eqref{eq:mlambda_pm} is $(\varphi, \Gamma_F)\equivariant$.
\begin{prop}\label{prop:mlambda_pm_descent}
	The $\LambdaRplus\module$ $\MLambdaplus$ is $p\textrm{-adically}$ complete and $p\torsionfree$.
	Moreover, we have that $\MLambdaminus \isomorphic (t/2)\MLambdaplus$ as $\LambdaRplus\modules$ compatible with the respective $(\varphi, \Gamma_0)\textrm{-actions}$.
	Furthermore, the natural map in \eqref{eq:mlambda_pm} is bijective and we have the following natural $\Lambda_R\linear$ and $(\varphi, \Gamma_F)\equivariant$ isomorphism:
	\begin{equation}\label{eq:mlambda_pm_descent}
		\Lambda_R \otimes_{\LambdaRplus} \MLambdaplus \isomorphic \MLambda.
	\end{equation}
\end{prop}
\begin{proof}
	Let us first recall that we have $\Lambda_R = \LambdaRplus \oplus (t/2)\LambdaRplus$ from Lemma \ref{lem:lambdar_pm} and Lemma \ref{lem:lambdar_pm_nu}, in particular, $\Lambda_R$ is finite over $\LambdaRplus$, which implies that $M$ is a finite $\LambdaRplus\module$.
	Moreover, by using Lemma \ref{lem:sigma_action_decomp}, note that we have a natural $\LambdaRplus\linear$ surjection $\MLambda \twoheadrightarrow \MLambdaplus$.
	So, it follows that $\MLambdaplus$ is finitely generated over $\LambdaRplus$ and it is $p\torsionfree$ and $p\textrm{-adically}$ separated (since $\MLambda$ is $p\torsionfree$ and $p\textrm{-adically}$ separated), therefore, $p\textrm{-adically}$ complete (see \cite[\href{https://stacks.math.columbia.edu/tag/031B}{Tag 031B}]{stacks-project}).
	Furthermore, as the action of $\Gamma_F$ is continuous for the $(p, p_1(\mu))\adic$ topology on $\MLambda$, so it follows that the induced action of $\Gamma_0$ on $\MLambdaplus$ is continuous as well.
	Next, note that by $\Lambda_R\textrm{-linearly}$ extending the $\LambdaRplus\linear$ and $(\varphi, \Gamma_F)\equivariant$ injective map $\MLambdaplus \rightarrow \MLambda$ from \eqref{eq:mlambda_pm}, we obtain the $\Lambda_R\linear$ and $(\varphi, \Gamma_F)\equivariant$ map in \eqref{eq:mlambda_pm_descent}.
	Moreover, we have that $(t/2)\MLambdaplus \subset \MLambdaminus$, so from the injectivity of \eqref{eq:mlambda_pm}, it follows that \eqref{eq:mlambda_pm_descent} is injective.
	To prove that \eqref{eq:mlambda_pm} is bijective, it is enough to show that \eqref{eq:mlambda_pm_descent} is surjective.

	Let $I$ denote the kernel of the surjective map $\Lambda_R \twoheadrightarrow R$.
	Then, from the explicit description of $\Lambda_R$ in Proposition \ref{prop:polynom_to_powerseries}, it follows that $I$ is the ideal of $\Lambda_R$ generated by the divided powers of $t/2$.
	As the divided powers of $t/2$ are topologically nilpotent, it follows that $I$ is contained in the Jacobson radical of $\Lambda_R$.
	Moreover, since the ideal $I$ is stable under the action of $(\varphi, \Gamma_F)$ on $\Lambda_R$, therefore, it follows that the map $\Lambda_R \rightarrow R$ is $(\varphi, \Gamma_F)\equivariant$.
	Next, from \eqref{eq:mlambda_arpdn} in Proposition \ref{prop:abarone_comp}, recall that we have a $(\varphi, \Gamma_F)\equivariant$ isomorphism of $\Lambda_R\modules$ $\MLambda \isomorphic \Lambda_R \otimes_{\AR} N$.
	Base changing this isomorphism along the surjective map $\Lambda_R \twoheadrightarrow R$, we obtain a $(\varphi, \Gamma_F)\equivariant$ isomorphism of $R\modules$ $\MLambda/I\MLambda \isomorphic N/\mu N$.
	Precomposing it with \eqref{eq:mlambda_pm_descent} gives a $\Lambda_R\linear$ map
	\begin{equation}\label{eq:lambdarplu_surj}
		\Lambda_R \otimes_{\LambdaRplus} \MLambdaplus \longrightarrow \MLambda/I\MLambda \isomorphic N/\mu N.
	\end{equation}

	As $I$ is in the Jacobson radical of $\Lambda_R$, therefore, by Nakayama Lemma, to show that \eqref{eq:mlambda_pm_descent} is surjective, it is enough to show that the first map in \eqref{eq:lambdarplu_surj} is surjective.
	So, let $\overline{x}$ be any element of $\MLambda/I\MLambda$ and let $x$ in $\MLambda$ be a lift of $\overline{x}$.
	We first claim that $(\sigma+1)x$ is an element of $2\MLambda$.
	Indeed, let us write $x = \sum_i a_i \otimes x_i$, for some $a_i$ in $\Lambda_R$ and $x_i$ in $N$.
	Then, from the description of $\Lambda_R$ in Lemma \ref{lem:lambdar_pm} and Lemma \ref{lem:lambdar_pm_nu}, we have that $(\sigma-1)a_i = tb_i$, for some $b_i$ in $\LambdaRplus$, and from the triviality of the action of $\Gamma_R$ on $N/\mu N$, we may write $(\sigma-1)x_i = \mu y_i$, for some $y_i$ in $N$.
	So, we have the following:
	\begin{align*}
		(\sigma+1)x = (\sigma-1)x + 2x &= (\sigma-1)(\textstyle\sum_i a_i \otimes x_i) + 2x\\
		&= \textstyle\sum_i (\sigma-1)a_i \otimes x_i + \textstyle\sum_i \sigma(a_i) \otimes (\sigma-1)x_i + 2x\\
		&= \textstyle\sum_i tbi \otimes x_i + \textstyle\sum_i \sigma(a_i) \otimes \mu y_i + 2x,
	\end{align*}
	which is clearly in $2 \MLambda$ (since $t/\mu$ is a unit in $\Lambda_R$, see Lemma \ref{lem:tovermu_unit}).
	As $\MLambda$ is $p\torsionfree$, therefore, we set $x' \coloneq \tfrac{(\sigma+1)}{2}(x)$ in $\MLambda$ and note that $\sigma(x') = x'$, i.e.\ $x'$ is in $\MLambdaplus$.
	From the computation of $(\sigma+1)x$ above and the fact that $\MLambda/I\MLambda \isomorphic N/\mu N$ (see \eqref{eq:lambdarplu_surj}) is $p\torsionfree$, we see that $x' = \overline{x} \textrm{ mod } I\MLambda$.
	In particular, we conclude that the first map of \eqref{eq:lambdarplu_surj} is surjective, hence, \eqref{eq:mlambda_pm_descent} is bijective.

	Finally, from the decomposition $\Lambda_R = \LambdaRplus \oplus (t/2)\LambdaRplus$, the inclusion $(t/2)\MLambdaplus \subset \MLambdaminus$ and the bijectivity of \eqref{eq:mlambda_pm_descent}, it follows that \eqref{eq:mlambda_pm} is bijective and $\MLambdaminus \isomorphic (t/2)\MLambdaplus$ compatible with the respective $(\varphi, \Gamma_0)\textrm{-actions}$.
	This completes our proof.
\end{proof}

Next, we will look at the action of $\Gamma_0 \isomorphic 1 + 4\ZZ_2$ on $\MLambdaplus$.
From \eqref{eq:nu_defi}, recall that $\nu = \tfrac{\mu^2}{1+\mu}$ is an element of $\LambdaFplus$ and we claim the following:
\begin{lem}\label{lem:mlambda+_gamma0_modnu}
	The action of $\Gamma_0$ is trivial on $\MLambdaplus/\nu \MLambdaplus$.
\end{lem}
\begin{proof}
	The claim follows from Remark \ref{rem:gamma0_act_mlambda+}.
	Note that the arguments in Remark \ref{rem:gamma0_act_mlambda+} do not depend on Lemma \ref{lem:mlambda+_gamma0_modnu} and the subsequent results of the current Section \ref{subsubsec:a1_galact_p=2}.
\end{proof}


\begin{proof}[Alternative proof of Lemma \ref{lem:mlambda+_gamma0_modnu}]
	We provide a more explicit proof of the claim.
	Note that $\Gamma_0$ acts continuously on $\MLambdaplus$, so it is enough to show the claim for a topological generator of $\Gamma_0$.
	So, let us fix $\gamma_0$ to be a topological generator of $\Gamma_0$ such that $\chi(\gamma_0) = 1 + 4a$, for a unit $a$ in $\ZZ_2$.
	Moreover, using Lemma \ref{lem:tovermu_unit}, it is easy to see that $\nu$ is the product of $t^2$ with a unit in $\LambdaFplus$.
	So, for any $m$ in $\MLambdaplus$, it is enough to show that $(\gamma_0-1)m$ is an element of $t^2\MLambdaplus$.
	Now, from Lemma \ref{lem:mlambda_gamma0_modnu} below, let us write $(\gamma_0-1)m = 2tm'$, for some $m'$ in $\MLambda$.
	Since, $\sigma(m) = m$, therefore, we get that $\sigma(m') = -m'$, i.e.\ $m'$ is in $\MLambdaminus$.
	Then, by using Proposition \ref{prop:mlambda_pm_descent}, we may write $m' = (t/2)n$, for some $n$ in $\MLambdaplus$.
	Hence, we get that $(\gamma_0-1)m = t^2 n$, as claimed.
\end{proof}

The following observation was used above:
\begin{lem}\label{lem:mlambda_gamma0_modnu}
	Let $\gamma_0$ be a topological generator of $\Gamma_0$ such that $\chi(\gamma_0) = 1 + 4a$, for a unit $a$ in $\ZZ_2$.
	Then, for any $m$ in $\MLambda$, the element $(\gamma_0-1)m$ belongs to $2t\MLambda$.
\end{lem}
\begin{proof}
	Let $f$ be an element of $\Lambda_R$.
	Then, using that $\Lambda_R = \LambdaRplus \oplus (t/2)\LambdaRplus$ from Lemma \ref{lem:lambdar_pm} and Lemma \ref{lem:lambdar_pm_nu}, and the fact that the action of $\gamma_0$ is trivial on $\LambdaRplus/\nu \LambdaRplus$, where $\nu$ is the product of $t^2$ with a unit in $\LambdaFplus$, it follows that $(\gamma_0-1)f$ is an element of $2t\Lambda_R$.

	Next, let $x$ be any element of $\MLambdaplus \subset \MLambda$, then from the $(\varphi, \Gamma_F)\equivariant$ isomorphism $\MLambda \isomorphic \Lambda_R \otimes_{\AR} N$ (see \eqref{eq:mlambda_arpdn} in Proposition \ref{prop:abarone_comp}), the triviality of the action of $\Gamma_F$ on $\Lambda_R/(\mu)$ (see Lemma \ref{lem:lambdatildeR_triv_modmu}) and $N/\mu N$ (see Definition \ref{defi:wach_mods_relative}), and the fact that $t/\mu$ is a unit in $\Lambda_F$ (see Lemma \ref{lem:tovermu_unit}), it follows that we may write$(\gamma_0-1)x = tx_1$, for some $x_1$ in $\MLambda$.
	As $x$ is in $\MLambdaplus$, we have $\sigma(x) = x$, and since $\Gamma_F$ is commutative and $\MLambda$ is $t\torsionfree$, therefore, it follows that $\sigma(x_1) = -x_1$, i.e.\ $x_1$ is an element of $\MLambdaminus$.
	Then, using Proposition \ref{prop:mlambda_pm_descent}, we may write $x_1 = (t/2)x_2$, for some $x_2$ in $\MLambdaplus$, and get that $(\gamma_0-1)x$ is an element of $(t^2/2)\MLambdaplus$.

	Next, let $y$ be any element of $\MLambdaminus$ and using Proposition \ref{prop:mlambda_pm_descent}, write $y = (t/2)y_1$, for some $y_1$ in $\MLambdaplus$.
	From the preceding discussion, note that we have $(\gamma_0-1)y_1 = (t^2/2)y_2$, for some $y_2$ in $\MLambdaplus$.
	So, we get that,
	\begin{align*}
		(\gamma_0-1) y = (\gamma_0-1)\big(\tfrac{t}{2}y_1\big) &= \tfrac{1}{2}\big(y_1(\gamma_0-1)t + \gamma_0(t) (\gamma_0-1)y_1\big)\\
		&= \tfrac{1}{2}(4aty_2 + (1+4a)ty_2\tfrac{t^2}{2}) = 2t(ay_1 + (1+4a)y_2\tfrac{t^2}{8}),
	\end{align*}
	is an element of $2t\MLambdaplus$.

	Next, note that since $N[1/p]$ is finite projective over $\AR[1/p]$ (see Remark \ref{rem:wachmod_props}) and the map $\AR \rightarrow \Lambda_R$ is injective, therefore, it follows that the $\AR\linear$ and $(\varphi, \Gamma_F)\equivariant$ map $N \rightarrow \Lambda_R \otimes_{\AR} N \lisomorphic \MLambda$ is injective.
	Now, let $z$ be an element of $N$ and let us denote its image in $\MLambda$, again by $z$.
	Then, by Definition \ref{defi:wach_mods_relative}, we have that $(\gamma_0-1)z = \mu z_0$, for some $z_0$ in $N \hookrightarrow \MLambda$.
	Using Proposition \ref{prop:mlambda_pm_descent}, let us write $z = x + y$, for some $x$ in $\MLambdaplus$ and $y$ in $\MLambdaminus = (t/2)\MLambdaplus$.
	Then, from the preceding discussions and the fact that $t/\mu$ is a unit in $\Lambda_F$, we may write $(\gamma_0-1)x = (\mu^2/2) x'$ and $(\gamma_0-1)y = 2\mu y'$, for some $x'$ and $y'$ in $\MLambda$.
	In particular, since $\MLambda$ is $\mu\torsionfree$, therefore, we get that $z_0 = (\mu/2) x' + 2 y'$ in $\MLambda$.
	Reducing the preceding equality modulo $I\MLambda$, where $I$ is the kernel of the surjective map $\Lambda_R \rightarrow R$ (see the proof of Proposition \ref{prop:mlambda_pm_descent}), we see that $z_0 = 2y' \textrm{ mod } I\MLambda$, in $\MLambda/I\MLambda \isomorphic N/\mu N$ (see \eqref{eq:lambdarplu_surj}).
	Since $y' \textrm{ mod } I\MLambda$ is an element of $\MLambda/I\MLambda \isomorphic N/\mu N$ and $z_0$ is in $N$, therefore, we get that $z_0 \textrm{ mod } \mu N$ is an element of $2(N/\mu N)$.
	In particular, we may write $z_0 = 2z_1 + \mu z_2$ for some $z_1$, $z_2$ in $N$.
	So, we have that
	\begin{equation*}
		(\gamma_0-1)z = \mu z_0 = 2\mu z_1 + \mu^2 z_2 = 2\mu\big(z_1 + \tfrac{\mu}{2}z_2\big),
	\end{equation*}
	is an element of $2t\MLambda$, since $t/\mu$ is a unit in $\Lambda_F$.

	Now, let $f \otimes z$ be an element of $\Lambda_R \otimes_{\AR} N$, and using the discussion at the beginning of the proof, let us write $(\gamma_0-1)f = 2t e$, for some $e$ in $\Lambda_R$.
	Moreover, from the discussion above we may write $(\gamma_0-1)z = 2tz'$, for some $z'$ in $\MLambda$.
	In particular, we see that 
	\begin{equation*}
		(\gamma_0-1)(f \otimes z) = ((\gamma_0-1)f) \otimes z + \gamma_0(f) \otimes (\gamma_0-1)z = 2t e + \gamma_0(f) 2tz' = 2t(e + \gamma_0(f) z'),
	\end{equation*}
	is an element of $2t\MLambda$.
	Using the $(\varphi, \Gamma_F)\equivariant$ isomorphism $\MLambda \isomorphic \Lambda_R \otimes_{\AR} N$ (see \eqref{eq:mlambda_arpdn} in Proposition \ref{prop:abarone_comp}) and the preceding observation, we conclude that for any $m$ in $\MLambda$, the element $(\gamma_0-1)m$ belongs to $2t\MLambda$.
	Hence, the lemma is proved.
\end{proof}

In the rest of this section, let us fix $\gamma_0$ to be a topological generator of $\Gamma_0$ such that $\chi(\gamma_0) = 1+4a$, for a unit $a$ in $\ZZ_2$.
Then, from Lemma \ref{lem:mlambda+_gamma0_modnu}, note that for any $x$ in $\MLambdaplus$, we have that $(\gamma_0-1)x$ is an element of $\nu \MLambdaplus$.
Set $\tau \coloneq \nu/8$ in $\LambdaRplus$, then from Lemma \ref{lem:gamma0_act_tau} we know that $(\gamma_0-1)\tau = u \nu$, for some unit $u$ in $\LambdaFplus$ depending on $\gamma_0$.
Therefore, we see that the following operator is well defined:
\begin{equation}\label{eq:nablaq_tau_mlambda+}
	\begin{aligned}
		\nabla_{q, \tau} \colon \MLambdaplus &\longrightarrow \MLambdaplus\\
				x &\mapsto \tfrac{(\gamma_0-1)x}{(\gamma_0-1)\tau}.
	\end{aligned}
\end{equation}
As the operator $\nabla_{q, \tau}$ is an endomorphism of $\MLambdaplus$, we may define the following two term Koszul complex:
\begin{equation}\label{eq:tau_qderham_mlambda+}
	K_{\MLambdaplus}(\nabla_{q, \tau}) \colon \big[\MLambdaplus \xrightarrow{\hspace{1mm} \nabla_{q, \tau} \hspace{1mm}} \MLambdaplus\big].
\end{equation}

\begin{rem}
	Considering $\tau$ as a variable, similar to Remark \ref{rem:nablaq_tautilde_qconnection}, the operator $\nabla_{q,\tau}$ in \eqref{eq:nablaq_tau_mlambda+}, may also be considered as a $\qconnection$ in non-logarithmic coordinates, in the sense of Definition \ref{defi:qconnection} and Remark \ref{rem:nablaqi_nota}.
	Then, \eqref{eq:tau_qderham_mlambda+} is the $q\textrm{-de Rham}$ complex arising from such a $\qconnection$.
\end{rem}

\begin{prop}\label{prop:mlambda+_connection}
	The series of operators $\nabla_0^{\log} = \frac{\log \gamma_0}{\log(\chi(\gamma_0))} = \frac{1}{\log(\chi(\gamma_0))}\sum_{k \in \NN} (-1)^k \frac{(\gamma_0-1)^{k+1}}{k+1}$ converge $p\adically$ on $\MLambdaplus$.
	Let $w \coloneq t^2/8$ in $\LambdaRplus$, then the operator $\nabla_0 \coloneq \tfrac{1}{2w}\nabla_0^{\log}$ defines an $R\linear$ $p\adically$ quasi-nilpotent flat connection on $\MLambdaplus$, denoted by $\nabla \colon \MLambdaplus \rightarrow \MLambdaplus \otimes_{\LambdaRplus} \Omega^1_{\LambdaRplus/R}$ and given as $x \mapsto \nabla_0(x) dw$.
	The data of the connection $\nabla$ on $\MLambdaplus$ is equivalent to the data of the $\qconnection$ $\nabla_{q, \tau}$ from \eqref{eq:nablaq_tau_mlambda+}, i.e.\ either may be recovered from the other.
	Moreover, the $q\textrm{-de Rham}$ complex $K_{\MLambdaplus}(\nabla_{q,\tau})$ in \eqref{eq:tau_qderham_mlambda+} is naturally quasi-isomorphic to the de Rham complex $\MLambdaplus \otimes_{\LambdaRplus} \Omega^{\bullet}_{\LambdaRplus/R}$.
	In particular, we have that $\MLambdaplus^{\nabla_{q, \tau}=0} \isomorphic \MLambdaplus^{\nabla_0=0}$.
\end{prop}
\begin{proof}
	The idea of the proof is similar to the proof of Proposition \ref{prop:mlambda0_connection}, but with different computations.
	We sketch it below for completeness.
	Recall that $\tau = e t^2/8$, for a unit $e$ in $\LambdaFplus$ (see Lemma \ref{lem:lambdar_pm_nu}).
	Moreover, from Lemma \ref{lem:gamma0_act_tau} we have that $(\gamma_0-1)\tau = (\gamma_0-1)\tfrac{\nu}{8} = u\nu$, for a unit $u$ in $\LambdaFplus$.
	Now, let $w = t^2/8$ and we write $(1+4a)^2 = 1 + 8b$, noting that $b = a(2a+1)$ is a unit in $\ZZ_2$.
	So, we have that
	\begin{equation}\label{eq:gamma0_w}
		(\gamma_0-1)w = (\gamma_0-1)\tfrac{t^2}{8} = (\chi(\gamma_0)^2-1)\tfrac{t^2}{8} = ((1+4a)^2-1)\tfrac{t^2}{8} = b t^2 = e^{-1}b(2a+1) \nu.
	\end{equation}
	Therefore, it follows that the complex $K_{\MLambdaplus}(\nabla_{q, \tau})$ is quasi-isomorphic to the following complex:
	\begin{equation}\label{eq:w_qderham_mlambda+}
		K_{\MLambdaplus}(\nabla_{q, w}) \colon \big[\MLambdaplus \xrightarrow{\hspace{1mm} \nabla_{q, w} \hspace{1mm}} \MLambdaplus\big].
	\end{equation}

	Now, let us show that $\nabla_0^{\log} \coloneq \tfrac{\log(\gamma_0)}{\log(\chi(\gamma_0))} = \tfrac{1}{\log(\chi(\gamma_0))} \sum_{k \in \NN} (-1)^k \frac{(\gamma_0-1)^{k+1}}{k+1}$, converge as a series of operators on $\MLambdaplus$.
	Indeed, let $x$ be any element of $\MLambdaplus$, then using Lemma \ref{lem:mlambda+_gamma0_modnu}, we may write $(\gamma_0-1)x = t^2 x_1$, for some $x_1$ in $\MLambdaplus$.
	Let us note that $\log(\chi(\gamma_0)) = \log(1+4a) = 4 c$, where $c$ is a unit in $\ZZ_2$, and we also have $(\gamma_0-1)t^2 = ((1+4a)^2-1)t^{p-1} = 8b t^2$, where $b$ is a unit in $\ZZ_2$.
	Therefore, an easy induction on $k \in \NN$, shows that $(\gamma_0-1)^{k+1} x = 8^k t^2 x_{k+1}$, for some $x_{k+1}$ in $\MLambdaplus$.
	In particular, we get that
	\begin{align*}
		\nabla_0^{\log}(x) &= \tfrac{1}{\log(\chi(\gamma_0))} \textstyle\sum_{k \in \NN} (-1)^k \tfrac{(\gamma_0-1)^{k+1}(x)}{k+1}\\
		&= \tfrac{1}{\log(\chi(\gamma_0))} \textstyle\sum_{k \in \NN} (-1)^k \tfrac{8^k t^2 x_{k+1}}{k+1} = \tfrac{t^2}{4c} \textstyle\sum_{k \in \NN} (-1)^k \tfrac{8^k x_{k+1}}{k+1},
	\end{align*}
	which converges in $\tfrac{t^2}{4} \MLambdaplus = 2w \MLambdaplus$, since $\MLambdaplus$ is $2\adically$ complete.
	Then, it also follows that the operator $\nabla_0 \coloneq \tfrac{1}{2w}\nabla_0^{\log}$ is well defined on $\MLambdaplus$.
	Next, similar to the case of $\nabla_i^{\log}$ in the proof of Proposition \ref{prop:n1modp1mu_connection}, it can be shown that $\nabla_0^{\log}$, and therefore $\nabla_0 = \tfrac{1}{2w}\nabla_0^{\log}$ satisfies a Leibniz rule, i.e.\ $\nabla_0(fx) = \nabla_0(f)x + f\nabla_0(x)$, where the first operator on the right is $\nabla_0 \coloneq \tfrac{\log(\gamma_0)}{2w\log(\chi(\gamma_0))} \colon \LambdaRplus \rightarrow \LambdaRplus$ whose well definedness can be checked similar to above.
	Moreover, note that the operator $\nabla_0$ is flat by definition.
	Furthermore, similar to the case of $\nabla_i$ in the proof of Proposition \ref{prop:n1modp1mu_connection}, it can be shown that the operator $\nabla_0 \colon \LambdaRplus \rightarrow \Omega^1_{\LambdaRplus/R}$ is the usual de Rham differential $d \colon \LambdaRplus \rightarrow \Omega^1_{\LambdaRplus/R}$.
	So, in particular, the map $\nabla \colon \MLambdaplus \rightarrow \MLambdaplus \otimes_{\LambdaRplus} \Omega^1_{\LambdaRplus/R}$, given as $x \mapsto \nabla_0(x) dw$, is a well-defined connection.

	Now, let us show that the operator $\nabla_0$ is $p\adically$ quasi-nilpotent.
	Indeed, we first note that from the commutativity of $\varphi$ and $\gamma_0$, it follows that $\nabla_0^{\log} \circ \varphi = \varphi \circ \nabla_0^{\log}$.
	Therefore, it is easy to see that $\nabla_0 \circ \varphi = 4 \varphi \circ \nabla_0$.
	Next, recall that $\MLambda$ is equipped with a $\Lambda_R\linear$ isomorphism $\varphi^*(\MLambda)[1/p] \isomorphic \MLambda[1/p]$, compatible with the action of $\Gamma_F$ on each side.
	Then, by taking the $\Gamma_{\textrm{tor}}\textrm{-invariants}$ of the preceding isomorphism and using the isomorphism \eqref{eq:mlambda_pm_descent} in Proposition \ref{prop:mlambda_pm_descent}, we easily obtain that $\MLambdaplus$ is equipped with a $\LambdaRplus\linear$ isomorphism $\varphi^*(\MLambdaplus)[1/p] \isomorphic \MLambdaplus[1/p]$, compatible with the action of $\Gamma_0$ on each side.
	Using this observation and the relation $\nabla_0 \circ \varphi = 4 \varphi \circ \nabla_0$, similar to the proof of Proposition \ref{prop:mlambda0_connection}, it can be shown that for any $x$ in $\MLambdaplus$, the sequence $\nabla_0^k(x)$ converges $p\adically$ to 0 as $k \rightarrow +\infty$, in particular, $\nabla_0$ is $p\adically$ quasi-nilpotent.
	Furthermore, note that so far we have defined the $p\adically$ quasi-nilpotent flat connection $\nabla$ using the action of $\Gamma_0$, and conversely, we have shown that the action of $\Gamma_0$ may be recovered by the formula $\gamma_0 \coloneq \exp(\log(\chi(\gamma_0))\nabla_0^{\log})$.
	Again, similar to case of $\gamma_i$ in Proposition \ref{prop:n1modp1mu_connection}, using the Leibniz rule for $\nabla_0$, it can be checked that the action of $\gamma_0$ thus obtained, is semilinear.

	Finally, it remains to compare the $q\textrm{-de Rham}$ complex in \eqref{eq:s_qderham_mlambda0} with the de Rham complex $\MLambdaplus \otimes_{\LambdaRplus} \Omega^{\bullet}_{\LambdaRplus/R} = \MLambdaplus \otimes_{\LambdaFplus} \Omega^{\bullet}_{\LambdaFplus/O_F}$.
	As $\nabla_0$ is an endomorphism of $\MLambdaplus$, let $K_{\MLambdaplus}(\nabla_0)$ denote the corresponding Koszul complex in the sense of Definition \ref{defi:koszul_complex}.
	Then, we have an identification of complexes
	\begin{equation*}
		\MLambdaplus \otimes_{\LambdaFplus} \Omega^{\bullet}_{\LambdaFplus/O_F} = K_{\MLambdaplus}(\nabla_0) \colon \big[\LambdaRplus \xrightarrow{\hspace{1mm} \nabla_0 \hspace{1mm}} \LambdaRplus\big].
	\end{equation*}
	Now, recall that we have $\gamma_0 = \exp(\log(\chi(\gamma_0)) \nabla_0^{\log})$.
	Therefore, we may write
	\begin{equation}\label{eq:nablaqw_nabla0}
		\nabla_{q, w} = \tfrac{\gamma_0-1}{(\gamma_0-1)w} = \tfrac{2w \log(\chi(\gamma_0))}{(\gamma_0-1)w} \nabla_0 \big(1 + \textstyle \sum_{k \geqslant 1} \tfrac{(\log\chi(\gamma_0))^k}{(k+1)!} (\nabla_0^{\log})^k\big).
	\end{equation}
	Recall that $(\gamma_0-1)w = 8bw$ and $\log(\chi(\gamma_0)) = 4c$, for units $b$ and $c$ in $\ZZ_2$.
	Therefore, we have that $\tfrac{\log(\chi(\gamma_0)) 2w}{(\gamma_0-1)w} = \tfrac{c}{b}$ is a unit, and it is clear that the term inside the parentheses, on the right hand side of \eqref{eq:nablaqw_nabla0}, converges $p\adically$ to a unit.
	Now, in the notation of Lemma \ref{lem:koszul_complex_qiso}, let us set $i = 1$, $M = \MLambdaplus$, $f_1 = \nabla_0$ and take $h_1$ to be the product of $\tfrac{c}{b}$ with the formula in the parentheses in \eqref{eq:nablaqw_nabla0}, in particular, $f_1h_1 = \nabla_{q, w}$.
	Then, from Lemma \ref{lem:koszul_complex_qiso}, we obtain a natural quasi-isomorphism of complexes
	\begin{equation*}
		K_{\MLambdaplus}(\nabla_{q, \tau}) \isomorphic K_{\MLambdaplus}(\nabla_{q, w}) \isomorphic K_{\MLambdaplus}(\nabla_0).
	\end{equation*}
	In particular, we get that $\MLambdaplus^{\nabla_{q, \tau}=0} \isomorphic \MLambdaplus^{\nabla_0=0}$.
	This allows us to conclude.
\end{proof}

\begin{prop}\label{prop:lambdar+_comp}
	Set $M \coloneq \MLambdaplus^{\nabla_0=0}$ as an $R\module$ via the isomorphism $R \isomorphic \LambdaRplus^{\nabla_0=0}$.
	Then, $M$ is a finitely generated $p\torsionfree$ module over $R$.
	Moreover, the $\LambdaRplus\linear$ extension of the natural $R\linear$ inclusion $M \subset \MLambdaplus$ yields a natural $(\varphi, \nabla_0)\equivariant$ isomorphism
	\begin{equation}\label{eq:lambdar+_comp}
		\LambdaRplus \otimes_R M \isomorphic \MLambdaplus.
	\end{equation}
	Furthermore, the de Rham complex $\MLambdaplus \otimes_{\LambdaFplus} \Omega^{\bullet}_{\LambdaFplus/O_F}$ is acyclic in positive degrees.
	In particular, from Proposition \ref{prop:mlambda+_connection}, we have $H^1(K_{\MLambdaplus}(\nabla_{q, \tau})) = 0$.
\end{prop}
\begin{proof}
	The proof works in exactly the same manner as the proof of Proposition \ref{prop:lambdar0_comp} by changing the notations as follows:
	replace $\Lambda_{R, 0}$ with $\LambdaRplus = R[\tau^{[k]}, k \in \NN]_p^{\wedge}$, and replace $\MLambdao$ with $\MLambdaplus$ equipped with a Frobenius endomorphism (after inverting $p$) and a $p\adically$ quasi-nilpotent flat connection as in Proposition \ref{prop:mlambda+_connection} compatible with the Frobenius, where in the definition of the connection we replace the parameter $z$ with $w$ and the operator $\nabla_0 = \tfrac{1}{(p-1)z} \nabla_0^{\log}$ with $\nabla_0 = \tfrac{1}{2w} \nabla_0^{\log}$.
	Then, the proof of the isomorphism \eqref{eq:lambdar+_comp} and the acyclicity of the de Rham complex $\MLambdaplus \otimes_{\LambdaFplus} \Omega^{\bullet}_{\LambdaFplus/O_F}$ in positive degrees, respectively, follow by arguments similar to the ones given in Proposition \ref{prop:lambdar0_comp} for the isomorphism \eqref{eq:lambdar0_comp} and the acyclicity of the de Rham complex $\MLambdaplus \otimes_{\LambdaFplus} \Omega^{\bullet}_{\LambdaFplus/O_F}$ in positive degrees.
	This completes our proof.
\end{proof}

\subsection{Proof of Theorem \ref{thm:integral_comp_relative}}\label{subsec:proof_integral_comp_relative}

Let $N$ be a Wach module over $\AR$ as above and consider the $\AR(1)\module$ $\AR(1) \otimes_{p_2, \AR} N$ equipped with the tensor product Frobenius and tensor product action of $\Gamma_R^2$, where $\Gamma_R^2$ acts on $N$ via projection onto the second coordinate.
Then, in Example \ref{exam:wach_a1modp1mu_qconnection}, using the action of $1 \times \Gamma_R'$, we equipped $\AR(1)/(p_1(\mu)) \otimes_{p_2, \AR} N$ with a $\qconnection$, denoted as $\nabla_q$.
Moreover, in Proposition \ref{prop:n1modp1mu_connection}, we equipped $\AR(1)/(p_1(\mu)) \otimes_{p_2, \AR} N$ with a connection, denoted as $\nabla$, and showed that we have the second equality in the following set of equalities:
\begin{equation*}
	(\AR(1)/(p_1(\mu)) \otimes_{p_2, \AR} N)^{1 \times \Gamma_R'} = (\AR(1)/(p_1(\mu)) \otimes_{p_2, \AR} N)^{\nabla_q=0} = (\AR(1)/(p_1(\mu)) \otimes_{p_2, \AR} N)^{\nabla=0},
\end{equation*}
and where the first equality follows because the action of $1 \times \Gamma_R'$ on $\AR(1)/(p_1(\mu)) \otimes_{p_2, \AR} N$ is continuous for the $\padic$ topology (see Lemma \ref{lem:cont_coh_disc}).

Next, from Proposition \ref{prop:abarone_comp}, let us recall that $\MLambda \coloneq (\AR(1)/(p_1(\mu)) \otimes_{p_2, \AR} N)^{\nabla=0}$ is a finitely generated $\Lambda_R\module$, equipped with an induced action of $(\varphi, \Gamma_F)$.
Then, from \eqref{eq:abarone_comp} we have a natural $(\varphi, \nabla, \Gamma_F)\equivariant$, or equivalently, a $(\varphi, \Gamma_R)\equivariant$ isomorphism
\begin{equation*}
	\AR(1)/(p_1(\mu)) \otimes_{\Lambda_R} \MLambda \isomorphic \AR(1)/(p_1(\mu)) \otimes_{p_2, \AR} N.
\end{equation*}
(A simple diagram chase in \eqref{eq:crys_strat_horz_sec} shows that the preceding map is the $\AR(1)/(p_1(\mu))\linear$ extension of the natural $\Lambda_R\linear$ inclusion $\MLambda \subset \AR(1)/(p_1(\mu)) \otimes_{p_2, \AR} N$).

Now, let us set $M \coloneq \MLambda^{\Gamma_F}$ as an $R\module$.
Then, by using Proposition \ref{prop:mlambda_fpx_descent} and Proposition \ref{prop:lambdar0_comp}, for $p \geqslant 3$, and by using Proposition \ref{prop:mlambda_pm_descent} and Proposition \ref{prop:lambdar+_comp}, for $p = 2$, we see that $M$ is a finitely generated $R\module$.
Moreover, since $\AR(1)/(p_1(\mu)) \otimes_{p_2, \AR} N$ is $p\torsionfree$ by Lemma \ref{lem:nr1_pctf}, therefore, we get that $M$ is $p\torsionfree$ as well.
Furthermore, for $p \geqslant 3$, using \eqref{eq:lambdaR_comp} in the proof of Proposition \ref{prop:mlambda_fpx_descent} and \eqref{eq:lambdar0_comp} in Proposition \ref{prop:lambdar0_comp}, and for $p = 2$, using \eqref{eq:mlambda_pm_descent} in Proposition \ref{prop:mlambda_pm_descent} and \eqref{eq:lambdar+_comp} in Proposition \ref{prop:lambdar+_comp}, we see that the $\Lambda_R\linear$ extension of the natural inclusion $M \subset \MLambda$ yields $(\varphi, \Gamma_F)\equivariant$ isomorphism of $\Lambda_R\modules$
\begin{equation}\label{eq:mlambda_m_comp}
	\Lambda_R \otimes_R M \isomorphic \MLambda.
\end{equation}

Putting everything together from the discussion above, we have the following $(\varphi, 1 \times \Gamma_R)\equivariant$ diagram:
\begin{center}
	\begin{tikzcd}[column sep=large]
		\AR(1)/(p_1(\mu)) \otimes_{p_1, R} M \arrow[r] \arrow[d, "\wr"', "\eqref{eq:mlambda_m_comp}"] & \AR(1)/(p_1(\mu)) \otimes_{p_2, \AR} N\\
		\AR(1)/(p_1(\mu)) \otimes_{\Lambda_R} \MLambda \arrow[ru, "\sim", "\eqref{eq:abarone_comp}"'].
	\end{tikzcd}
\end{center}
The top horizontal arrow is defined to be the composition of the left vertical arrow and slanted arrow (by definition of the arrows, it follows that the top horizontal arrow is the $\AR(1)/(p_1(\mu))\linear$ extension of the natural $R\linear$ inclusion $M \subset \AR(1)/(p_1(\mu)) \otimes_{p_2, \AR} N$).
Hence, the diagram constructs the claimed isomorphism in \eqref{eq:integral_comp_relative}, proving the first part of Theorem \ref{thm:integral_comp_relative}.

Finally, it remains to show that we have a $\varphi\equivariant$ isomorphism of $R\modules$ $M \isomorphic N/\mu N$.
To show this claim, let us note that the multiplication map $\Delta \colon \AR(1) \rightarrow \AR$ is $(\varphi, \Gamma_R \times 1)\equivariant$, where $\Gamma_R \times 1$ acts on $\AR$ via projection onto the first coordinate (see Lemma \ref{lem:delta_modmun_phigamma}).
Then, the multiplication map $\Delta$ induces a $\varphi\equivariant$ map $\Delta_N \colon \AR(1) \otimes_{p_2, \AR} N \rightarrow N$.
Reducing it modulo $p_1(\mu)^n$, we obtain an $\AR/(\mu)^n\linear$ and $\varphi\equivariant$ map $\Delta_N \colon \AR(1)/(p_1(\mu))^n \otimes_{p_2, \AR} N \rightarrow N/\mu^n N$.
For $n=1$, we claim the following:
\begin{prop}\label{prop:deltan_modmu_relative}
	Let $M \coloneq (\AR(1)/(p_1(\mu)) \otimes_{p_2, \AR} N)^{1 \times \Gamma_R}$ as a finitely generated $\varphi\module$ over $R$.
	Then, the map $\Delta_N$ restricts to a $\varphi\equivariant$ $R\linear$ isomorphism
	\begin{equation}\label{eq:deltan_modmu_relative}
		M = (\AR(1)/(p_1(\mu)) \otimes_{p_2, \AR} N)^{1 \times \Gamma_R} \isomorphic N/\mu N.
	\end{equation}
\end{prop}
\begin{proof}
	By definition, we have that $\Delta_N$ is $\varphi\equivariant$, so we only need to check the bijectivity of the map in claim.
	Let us note that the composition $R \xrightarrow{p_1} \AR(1)/(p_1(\mu)) \xrightarrow{\Delta} R$ is the identity.
	Moreover, the composition $\AR \xrightarrow{p_2} \AR(1) \twoheadrightarrow \AR(1)/(p_1(\mu)) \xrightarrow{\Delta} R$ coincides with the map $\AR \twoheadrightarrow \AR/(\mu) = R$.
	So, we have that
	\begin{align*}
		\Delta_N \colon M &= R \otimes_{\Delta, \AR(1)/(p_1(\mu))} \big(\AR(1)/(p_1(\mu)) \otimes_{p_1, R} M \big)\\
		&\isomorphic R \otimes_{\Delta, \AR(1)/(p_1(\mu))} \big(\AR(1)/(p_1(\mu)) \otimes_{p_2, \AR} N\big) \isomorphic \AR/(\mu) \otimes_{\AR} N = N/\mu N,
	\end{align*}
	where the first isomorphism follows from the isomorphism in Theorem \ref{thm:integral_comp_relative} and the second isomorphism follows from the discussion above.
	Hence, we get the claim.
\end{proof}

This completes the proof of the second part of Theorem \ref{thm:integral_comp_relative}.

\begin{rem}\label{rem:p1p2_modmu_comp_relative}
	Let $N$ be a Wach module over $\AR$.
	Then, we have a $\varphi\equivariant$ commutative diagram
	\begin{equation}\label{eq:p1p2_modmu_comp_relative}
		\begin{tikzcd}
			\AR(1)/(p_1(\mu)) \otimes_{p_1, \AR} N \arrow[r, "\sim"] \arrow[d, equal] & \AR(1)/(p_1(\mu)) \otimes_{p_2, \AR} N\\
			\AR(1)/(p_1(\mu)) \otimes_{p_1, R} N/\mu N \arrow[r, "\sim"] & \AR(1)/(p_1(\mu)) \otimes_{p_1, R} M \arrow[u, "\wr"'],
		\end{tikzcd}
	\end{equation}
	where the bottom horizontal arrow is the extension along $p_1 \colon R \rightarrow \AR(1)/(p_1(\mu))$ of the inverse of the isomorphism $\Delta_N$ in \eqref{eq:deltan_modmu_relative}, the right vertical arrow is \eqref{eq:integral_comp_relative} and the top horizontal arrow is the composition of the left vertical, the bottom horizontal and the right vertical arrows.
	Let us note that in Proposition \ref{prop:deltan_modmun_relative}, we shall show that the map $\Delta_N$ is equivariant for the natural action of $1 \times \Gamma_R$ on the source and the natural action of $\Gamma_R$ on the target.
	Then, it will follow that the bottom arrow in the diagram \eqref{eq:p1p2_modmu_comp_relative} is $(\varphi, 1 \times \Gamma_R)\equivariant$, and therefore, the diagram \eqref{eq:p1p2_modmu_comp_relative} is also $(\varphi, 1 \times \Gamma_R)\equivariant$.
\end{rem}

\section{Prismatic \texorpdfstring{$F\crystals$}{-} and Wach modules}\label{sec:prismatic_wach}

In this section, let $X \coloneq \Spf R$ denote a $\padic$ formal scheme, where $R$ is the ring from Section \ref{subsec:setup_nota}.
Our main goal in this section is to relate analytic/completed prismatic $F\crystals$ on the absolute prismatic site $(\Spf R)_{\prism}$ (see Section \ref{subsec:prismatic_fcrystals}) to Wach modules over $\AR$ (see Definition \ref{defi:wach_mods_relative}).
We start with the following construction:

\begin{prop}\label{prop:afcrystal_eval_ar+}
	Let $\paze$ be an object of $\Vect^{\an}(X_{\prism})$ and set $N \coloneq \paze(\AR, [p]_q)$.
	Then, $N$ is a finitely generated $(p, \mu)\textrm{-adically}$ complete $\AR\module$ equipped with a continuous action of $\Gamma_R$ such that the action of $\Gamma_R$ is trivial on $N/\mu N$.
	Moreover, if $\paze$ is an object of $\Vect^{\an, \varphi}(X_{\prism})$, then $N$ is a Wach module over $\AR$.
\end{prop}
\begin{proof}
	From Lemma \ref{lem:analyticrys_ar+_fingen} below, with $\paze(\AR, [p]_q)[1/p]$ as a finite projective $\AR[1/p]\module$, and $\paze(\AR, [p]_q)[1/\mu]$ as a finite projective $\AR[1/\mu]\module$, we see that the $\AR\module$
	\begin{equation*}
		N \coloneq \paze(\AR, [p]_q) = \paze(\AR, [p]_q)[1/p] \cap \paze(\AR, [p]_q)[1/\mu] \subset \paze(\AR, [p]_q)[1/p, 1/\mu],
	\end{equation*}
	is finitely generated and the sequences $\{p, \mu\}$ and $\{\mu, p\}$ are regular on $N$.

	Next, let $g$ be any element of $\Gamma_R$ and from Lemma \ref{lem:gammar_act_prisaut} recall that $g$ is an automorphism of $(\AR, [p]_q)$ in $X_{\prism}$, i.e.\ we have that $g \colon (\AR, [p]_q) \isomorphic (\AR, [p]_q)$.
	Since, $\paze$ is a crystal, it follows that base changing $N = \paze(\AR, [p]_q)$ along $g$, induces an $\AR\textrm{-semilinear}$ isomorphism $g \colon \AR \otimes_{g, \AR}N \isomorphic N$, in particular, a semilinear action of $\Gamma_R$.
	Note that $\AR/(\mu) \isomorphic R$, the pair $(R, p)$ is a prism and an object of $(\Spf R)_{\prism}$ and $\Gamma_R$ acts trivially on $R$.
	Therefore, the reduction modulo $\mu$ of the isomorphism $g \colon \AR \otimes_{g, \AR}N \isomorphic N$ yields an isomorphism $g \textrm{ mod } \mu \colon R \otimes_{g, R} N/\mu N \isomorphic N/\mu N$, which is the identity, in particular, we see that the action of $\Gamma_R$ is trivial on $N/\mu N$.
	Now, additionally assume that $\paze$ is an $F\crystal$, i.e.\ an object of $\Vect^{\an, \varphi}(X_{\prism})$.
	Then, from Definition \ref{defi:analytic_fcrystal}, we have an $\AR\linear$ isomorphism $\varphi_N \colon (\varphi^*N)[1/[p]_q] \isomorphic N[1/[p]_q]$.
	Again, using that $\paze$ is a crystal and any $g$ in $\Gamma_R$ is an automorphism of $(\AR, [p]_q)$ in $X_{\prism}$, it is easy to show that $\varphi_N$ is equivariant for the action of $\Gamma_R$ on $N$ described above.
	So, we see that $N =  \paze(\AR, [p]_q)$ satisfies all the axioms of Definition \ref{defi:wach_mods_relative}, i.e.\ it is a Wach module over $\AR$.
\end{proof}

\begin{rem}\label{rem:cfcrystal_eval_ar+}
	A claim similar to Proposition \ref{prop:afcrystal_eval_ar+} holds for completed prismatic $F\crystals$.
	More precisely, let $\paze$ be an object of $\CRhat(X_{\prism})$ and set $N \coloneq \paze(\AR, [p]_q)$.
	Then, $N$ is a finitely generated $\AR\module$ equipped with a continuous action of $\Gamma_R$ such that the action of $\Gamma_R$ is trivial on $N/\mu N$.
	Moreover, if $\paze$ is an object of $\CRhatphi(X_{\prism})$, then $N$ is a Wach module over $\AR$.
\end{rem}
\begin{proof}
	By Definition \ref{defi:completed_crystal}, note that $N \coloneq \paze(\AR, [p]_q)$ is a finitely generated $(p, [p]_q)\textrm{-adically} = (p, \mu)\adically$ complete $\AR\module$.
	Then, similar to the proof of Proposition \ref{prop:afcrystal_eval_ar+}, using that $\paze$ is a crystal and the fact that any $g$ in $\Gamma_R$ is an automorphism of $(\AR, [p]_q)$ in $X_{\prism}$ with the action of $\Gamma_R$ being trivial on $\AR/(\mu) \isomorphic R$, we get that $N$ is equipped with a semilinear action of $\Gamma_R$ such that the action of $\Gamma_R$ is trivial on $N/\mu N$.
	Now, additionally assume that $\paze$ is an $F\crystal$, i.e.\ an object of $\CRhatphi(X_{\prism})$.
	So, from Definition \ref{defi:completed_fcrystal}, we have an $\AR\linear$ isomorphism $\varphi_N \colon (\varphi^*N)[1/[p]_q] \isomorphic N[1/[p]_q]$ and similar to above, it is easy to see that $\varphi_N$ is $\Gamma_R\equivariant$.
	It remains to show that the sequences $\{p, \mu\}$ and $\{\mu, p\}$ are regular on $N$.

	Let $\frakS \coloneq R\llbracket u \rrbracket$, and recall that the Breuil-Kisin prism $(\frakS, u-p)$ from Section \ref{defi:completed_crystal} is an object of $X_{\prism}$ and a cover of the final object of the topos $\Shv(X_{\prism})$ (see \cite[Section 3.3]{du-liu-moon-shimizu}).
	In particular, there exists a prism $(B, J)$ in $X_{\prism}$ so that it is $(p, [p]_q)\textrm{-completely}$ faithfully flat over $(\AR, [p]_q)$ and admits a map from $(\frakS, u-p)$ in $X_{\prism}$.
	Moreover, as $\AR$ is noetherian, it follows that the map $\AR \rightarrow B$ is faithfully flat.
	Now, using an argument similar to \cite[Lemma 3.19]{du-liu-moon-shimizu} we have that $B \otimes_{\AR} N \isomorphic \paze(B, J)$, and therefore
	\begin{equation*}
		(B \otimes_{\AR} N)[1/p] \isomorphic \paze(B, J)[1/p] \lisomorphic B \otimes_{\frakS} \paze(\frakS, u-p)[1/p],
	\end{equation*}
	where the last isomorphism follows from \cite[Lemma 3.24]{du-liu-moon-shimizu}.
	As $\paze(\frakS, u-p)[1/p]$ is finite projective over $\frakS[1/p]$ and $\AR \rightarrow B$ is faithfully flat, therefore, it follows that $N[1/p]$ is finite projective over $\AR[1/p]$.
	A similar argument (by inverting the prismatic ideal instead of $p$) shows that $N[1/[p]_q]$ is finite projective over $\AR[1/[p]_q]$.
	Then, by using Lemma \ref{lem:analyticrys_ar+_fingen}, we see that the sequences $\{p, [p]_q\}$ and $\{[p]_q, p\}$ are regular on $N$, and therefore, the sequences $\{p, \mu\}$ and $\{\mu, p\}$ are regular on $N$ by \cite[Lemma 3.2]{abhinandan-relative-wach-ii}.
	Hence, it follows that $N$ satisfies all the axioms of Definition \ref{defi:wach_mods_relative}, in particular, it is a Wach module over $\AR$.
\end{proof}

Let $a$ be an element of $\AR$ such that the sequences $\{p, a\}$ and $\{a, p\}$ are regular on $\AR$ and we have an equality of zero loci $V(p, a) = V(p, [p]_q)$ inside $\Spec(\AR)$.
For example, we may take $a = \mu$ or $[p]_q$.
Then, we have the following:
\begin{lem}\label{lem:analyticrys_ar+_fingen}
	Let $M$ be a finite projective $\AR[1/p]\module$ and $D$ a finite projective $\AR[1/a]\module$, equipped with an $\AR[1/p, 1/a]\linear$ isomorphism $f \colon \AR[1/p, 1/a] \otimes_{\AR[1/p]} M \isomorphic \AR[1/p, 1/a] \otimes_{\AR[1/a]} D$.
	Then, $N \coloneq M \cap D \subset \AR[1/p, 1/a] \otimes_{\AR[1/a]} D$ is a finitely generated $\AR\module$, the sequences $\{p, a\}$ and $\{a, p\}$ are regular on $N$, and we have natural isomorphisms $N[1/p] \isomorphic M$ and $N[1/a] \isomorphic D$.
\end{lem}
\begin{proof}
	Let $S = \Spec(\AR)$, $Z = V(p, a) \subset S$ as the zero locus of the ideal $(p, a) \subset \AR$ (the same as the zero locus of $(p, [p]_q) \subset \AR$) and set $U \coloneq S \setminus Z \xhookrightarrow{j} S$.
	Then, note that the natural map $\Spec(\AR[1/a]) \cup \Spec(\AR[1/p]) \rightarrow \Spec(\AR)\setminus V(p, a) = U$ is a flat cover.
	Therefore, from the data $(M, D, f)$ and faithfully flat descent, we obtain a vector bundle $\pazf$ over $U \subset S$.
	Now, as $U$ is irreducible and $\pazf$ is a vector bundle on it, therefore, the only associated point of $\pazf$ is the generic point of $U$.
	Moreover, $U$ is the complement of a codimension 2 closed subset $Z \subset S$, so from \cite[\href{https://stacks.math.columbia.edu/tag/0BK3}{Proposition 0BK3}]{stacks-project} it follows that $j_{\ast} \pazf$ is coherent, i.e.\ $N \coloneq H^0(S, j_{\ast} \pazf) = H^0(U, \pazf) = M \cap D \subset \AR[1/p, 1/a] \otimes_{\AR[1/a]} D$ is a finitely presented $\AR\module$.
	Furthermore, \cite[\href{https://stacks.math.columbia.edu/tag/0BK0}{Tag 0BK0}]{stacks-project} implies that $j^*j_{\ast} \pazf \isomorphic \pazf$ on $U$ and $H^0_Z(N) = H^1_Z(N) = 0$.
	Then, an easy argument shows that the sequences $\{p, a\}$ and $\{a, p\}$ are regular on $N$, and $N[1/p] \isomorphic M$ and $N[1/a] \isomorphic D$ (for example, see \cite[Remark A.6 and Lemma 3.3]{abhinandan-relative-wach-ii} for $a = \mu$).
	This allows us to conclude.
\end{proof}

\begin{rem}\label{rem:regsec_saturation}
	Let $N$ be a finitely generated $\AR\module$, let $a$ in $\AR$ be as above, and assume that the sequence $\{p, a\}$ and $\{a, p\}$ are regular on $N$.
	Then, we have equalities $(N/pN)[a] = (N/aN)[p] = 0$, and $(N[1/\mu])/p = (N/pN)[1/\mu]$.
	So, it follows that we have natural inclusions $N/p^nN \hookrightarrow (N/p^nN)[1/\mu]$ for each $n \geqslant 1$, and therefore \cite[Lemma A.1]{abhinandan-relative-wach-ii} implies that we have $N = N[1/p] \cap N[1/a] \subset N[1/p, 1/a]$.
\end{rem}

From Proposition \ref{prop:afcrystal_eval_ar+}, note that we have a well-defined evaluation functor for analytic $F\crystals$:
\begin{equation}\label{eq:afcrystal_eval_ar+}
	\begin{aligned}
		\eval_{\AR}^{\prism} \colon \Vect^{\an, \varphi}(X_{\prism}) &\longrightarrow (\varphi, \Gamma_R)\Mod_{\AR}^{[p]_q}\\
				\paze &\longmapsto \paze(\AR, [p]_q).
	\end{aligned}
\end{equation}
Similarly, from Remark \ref{rem:cfcrystal_eval_ar+}, we also obtain a well-defined evaluation functor for completed $F\crystals$:
\begin{equation}\label{eq:cfcrystal_eval_ar+}
	\begin{aligned}
		\eval_{\AR}^{\prism} \colon \CRhatphi(X_{\prism}) &\longrightarrow (\varphi, \Gamma_R)\Mod_{\AR}^{[p]_q}\\
				\paze &\longmapsto \paze(\AR, [p]_q).
	\end{aligned}
\end{equation}
From Lemma \ref{lem:ar+_cover_fo}, recall that $(\AR, [p]_q)$ covers the final object of the topos $\Shv(X_{\prism})$, and we claim the following:
\begin{thm}\label{thm:fcrystals_wachmod_relative}
	The evaluation functors in \eqref{eq:afcrystal_eval_ar+} and \eqref{eq:cfcrystal_eval_ar+} induce equivalences of categories.
\end{thm}
\begin{proof}
	The claim follows from Propositions \ref{prop:prisfcrys_strat_wach_relative} and \ref{prop:afcrystal_eval_ar+strat}, and Theorem \ref{thm:strat_wach_comp_relative} shown below.
\end{proof}

In the rest of this section, we shall state and prove Proposition \ref{prop:prisfcrys_strat_wach_relative} and Theorem \ref{thm:strat_wach_comp_relative}.
We start by building some necessary tools to state these claims.

\subsection{\texorpdfstring{$\AR\modules$}{-} with stratification}\label{subsec:modules_with_strat}

In order to prove Theorem \ref{thm:fcrystals_wachmod_relative} we shall interpret crystals in terms of modules with stratification.
We begin with the definition of some cosimplicial objects in $\Shv(X_{\prism})$.

\subsubsection{Stratifications}

Using the simplicial object $(\AR(\bullet), I(\bullet))$ in $(\Spf R)_{\prism}$ as described in Construction \ref{const:cech_nerve}, we define modules with stratifications as follows:
\begin{defi}[Prismatic stratification]\label{defi:ar+mod_strat}
	A \textit{stratification} on an $\AR\module$ $N$ with respect to $\AR(\bullet)$ is an $\AR(1)\linear$ isomorphism $\varepsilon \colon \AR(1) \otimes_{p_1, \AR} N \isomorphic \AR(1) \otimes_{p_2, \AR} N$ satisfying the following conditions:
	\begin{enumerate}
		\item[(1)] The scalar extension $\Delta^*(\varepsilon)$ of $\varepsilon$ by $\Delta \colon \AR(1) \rightarrow \AR$ is the identity map on $N$;

		\item[(2)] Cocycle condition over $\AR(2)$, i.e.\ $p_{23}^*(\varepsilon) \circ p_{12}^*(\varepsilon) = p_{13}^*(\varepsilon) \colon \AR(2) \otimes_{r_1, \AR} N \isomorphic \AR(2) \otimes_{r_3, \AR} N$.
	\end{enumerate}
	Let $\Strat(\AR(\bullet))$ denote the category of $\AR\modules$ equipped with a stratification with respect to $\AR(\bullet)$.
	Additionally, we will say that $N$ is \textit{analytic} if the following holds:
	\begin{enumerate}
		\item[(3)] The sequences $\{p, \mu\}$ and $\{\mu, p\}$ are regular on $N$.
	\end{enumerate}
	Let $\Strat^{\an}(\AR(\bullet))$ denote the full subcategory of analytic objects in $\Strat(\AR(\bullet))$.
	Furthermore, we will say that $N$ is a $\varphi\module$ over $\AR$ equipped with a stratification if $N$ is finitely generated and equipped with,
	\begin{enumerate}
		\item[(4)] An $\AR\linear$ isomorphism $\varphi_N \colon \varphi^*(N)[1/[p]_q] \isomorphic N[1/[p]_q]$ compatible with the stratification $\varepsilon$.
	\end{enumerate}
	Let $\Strat^{\an, \varphi}(\AR(\bullet))$ denote the category of analytic $\varphi\modules$ over $\AR$ equipped with a stratification.
\end{defi}

Now, let us relate the category $\Vect^{\an, \varphi}(X_{\prism})$ of analytic prismatic $F\crystals$ over $X_{\prism}$ to the category $\Strat^{\an, \varphi}(\AR(\bullet))$ of analytic $\varphi\modules$ over $\AR$ equipped with a stratification.
\begin{const}\label{const:afcrystal_eval_ar+strat}
	Let $\paze$ be an object of $\Vect^{\an, \varphi}(X_{\prism})$ and set $N \coloneq \paze(\AR(0))$ as an $\AR\module$.
	Then, by using Proposition \ref{prop:afcrystal_eval_ar+}, we have that $N = \paze(\AR, [p]_q)$ is a finitely generated $\AR\module$ such that the sequences $\{p, \mu\}$ and $\{\mu, p\}$ are regular on $N$.
	Moreover, $N$ is equipped with an $\AR\linear$ isomorphism $\varphi_N \colon \varphi^*(N)[1/[p]_q] \isomorphic N[1/[p]_q]$.
	Next, set $D \coloneq \paze[1/\pazi_{\prism}](\AR, [p]_q) = N[1/[p]_q]$ as a finite projective $\AR[1/[p]_q]\module$ and note that $D$ is equipped with a stratification with respect to $\AR(\bullet)$ given as the composition
	\begin{equation}\label{eq:strat_d}
		\varepsilon_D \colon \AR(1) \otimes_{p_1, \AR} D \xrightarrow[\sim]{\hspace{1mm} p_1 \hspace{1mm}} \paze[1/\pazi_{\prism}](\AR(1)) \xrightarrow[\sim]{\hspace{1mm} p_2^{-1} \hspace{1mm}} \AR(1) \otimes_{p_2, \AR} D,
	\end{equation}
	such that $\Delta^*(\varepsilon_D) = id$ on $D$ and $\varepsilon_D$ satisfies the cocycle condition over $\AR(2)$.
	Similarly, set $M \coloneq \paze[1/p](\AR, [p]_q) = N[1/p]$ as a finite projective $\AR[1/p]\module$ and note that $M$ is equipped with a stratification with respect to $\AR(\bullet)$ given as the composition
	\begin{equation}\label{eq:strat_m}
		\varepsilon_M \colon \AR(1) \otimes_{p_1, \AR} M \xrightarrow[\sim]{\hspace{1mm} p_1 \hspace{1mm}} \paze[1/p](\AR(1)) \xrightarrow[\sim]{\hspace{1mm} p_2^{-1} \hspace{1mm}} \AR(1) \otimes_{p_2, \AR} M,
	\end{equation}
	such that $\Delta^*(\varepsilon_M) = id$ on $M$ and $\varepsilon_M$ satisfies the cocycle condition over $\AR(2)$.
	Note that by definition, we have an $\AR[1/p, 1/[p]_q]\linear$ isomorphism $f \colon \AR[1/p, 1/[p]_q] \otimes_{\AR[1/p]} M \isomorphic \AR[1/p, 1/[p]_q] \otimes_{\AR[1/[p]_q]} D$, and \eqref{eq:strat_d} and \eqref{eq:strat_m} are compatible with the preceding isomorphism, i.e.\ $f \circ (\varepsilon_M[1/[p]_q]) = (\varepsilon_D[1/p]) \circ f$.
	Therefore, by taking the intersection of \eqref{eq:strat_d} with \eqref{eq:strat_m} inside $\AR[1/p, 1/[p]_q] \otimes_{\AR[1/[p]_q]} D$, and noting that $N = M \cap D$ (see Remark \ref{rem:regsec_saturation}) and that the maps $p_1, p_2 \colon \AR \rightarrow \AR(1)$ are faithfully flat (see Lemma \ref{lem:pi_rj_ff}), we obtain an isomorphism
	\begin{equation}\label{eq:strat_n}
		\varepsilon \colon \AR(1) \otimes_{p_1, \AR} N \xrightarrow[\sim]{\hspace{1mm} p_1 \hspace{1mm}} \paze(\AR(1)) \xrightarrow[\sim]{\hspace{1mm} p_2^{-1} \hspace{1mm}} \AR(1) \otimes_{p_2, \AR} N,
	\end{equation}
	such that $\Delta^*(\varepsilon) = id$ on $N$ and $\varepsilon$ satisfies the cocycle condition over $\AR(2)$.
	Hence, $N$ is an analytic $\varphi\module$ over $\AR$ equipped with a stratification and it is clear that our construction is functorial in $\paze$.
	In particular, we have described a well-defined natural functor
	\begin{equation}\label{eq:afcrystal_eval_ar+strat}
		\eval_{\AR(\bullet)}^{\prism} \colon \Vect^{\an, \varphi}(X_{\prism}) \longrightarrow \Strat^{\an, \varphi}(\AR(\bullet)),
	\end{equation}
	by sending an analytic $F\textrm{-crystal}$ $\paze$ to the finitely generated $\AR\module$ $N \coloneq \paze(\AR(0))$ equipped with a stratification as in \eqref{eq:strat_n}.
\end{const}

\subsubsection{Equivalence between crystals and stratifications}

The main goal of this section is to prove the following claim:
\begin{prop}\label{prop:afcrystal_eval_ar+strat}
	The functor in \eqref{eq:afcrystal_eval_ar+strat} induces a natural equivalence of categories.
\end{prop}
\begin{proof}
	From Lemma \ref{lem:ar+_cover_fo}, recall that the prism $(\AR, [p]_q)$ in $X_{\prism}$ is a cover of the final object of the topos $\Shv(X_{\prism})$, therefore, from Construction \ref{const:cech_nerve} and the limit description of analytic $F\crystals$ (see Definition \ref{defi:analytic_fcrystal}), it follows that we have a descent equivalence
	\begin{equation}\label{eq:anfcrys_descent_equiv}
		\begin{aligned}
			\Vect^{\an, \varphi}(X_{\prism}) &\isomorphic \lim \Vect^{\an, \varphi}(\AR(\bullet)).
		\end{aligned}
	\end{equation}
	From Definition \ref{defi:analytic_fcrystal}, the target category may be described as follows: let $U \coloneq \Spec(\AR) \setminus V(p, [p]_q)$, $U(1) \coloneq \Spec(\AR(1)) \setminus V(p, [p]_q)$ equipped with projection maps $p_1, p_2 \colon U(1) \rightarrow U$ and a diagonal map $\Delta \colon U \rightarrow U(1)$, and set $U(2) \coloneq \Spec(\AR(2)) \setminus V(p, [p]_q)$ equipped with projection maps $p_{12}, p_{23}, p_{13} \colon U(2) \rightarrow U(1)$ (see Construction \ref{const:cech_nerve} for the maps).
	Then, an object in the category $\lim \Vect^{\an, \varphi}(\AR(\bullet))$ is a vector bundle $\paze$ on $U$ equipped with an isomorphism $\varepsilon \colon p_1^*\paze \isomorphic p_2^*\paze$ of vector bundles over $U(1)$ such that $\Delta^*(\varepsilon) = id$ on $\paze$ and $\varepsilon$ satisfies the cocycle condition over $U(2)$, i.e.\ $p_{23}^*(\varepsilon) \circ p_{12}^*(\varepsilon) = p_{13}^*(\varepsilon)$.
	Moreover, $\paze$ is equipped with an $\AR[1/[p]_q]\linear$ isomorphism $\varphi_{\paze} \colon (\varphi^*\paze)[1/[p]_q] \isomorphic \paze[1/[p]_q]$ (see Definition \ref{defi:analytic_fcrystal}), compatible with $\varepsilon$.

	Now, let us consider the following diagram of functors between categories:
	\begin{equation}\label{eq:afcrystal_vb_strataf+}
		\begin{tikzcd}[column sep=large]
			\Vect^{\an, \varphi}(X_{\prism}) \arrow[r, "\eqref{eq:anfcrys_descent_equiv}", "\sim"'] \arrow[rd, "\eqref{eq:afcrystal_eval_ar+strat}"'] & \lim \Vect^{\an, \varphi}(\AR(\bullet)) \arrow[d]\\
			& \Strat^{\an, \varphi}(\AR(\bullet)),
		\end{tikzcd}
	\end{equation}
	where the right vertical arrow is defined by sending an object $\paze$ to the finite $\AR\module$ $N \coloneq H^0(\paze, U)$ equipped with a Frobenius structure and a stratification constructed similar to \eqref{eq:strat_n} in Construction \ref{const:afcrystal_eval_ar+strat}.
	It is clear that the diagram \eqref{eq:afcrystal_vb_strataf+} commutes by definition.

	We claim that the right vertical arrow in \eqref{eq:afcrystal_vb_strataf+} induces a natural equivalence of categories.
	Indeed, let us describe its quasi-inverse:
	let $N$ be an analytic $\varphi\module$ over $\AR$ equipped with a stratification $\varepsilon_N$ with respect to $\AR(\bullet)$.
	Then, from Construction \ref{const:strat_eval_ar+} below we see that $N$ admits the structure of a Wach module over $\AR$.
	In particular, from Remark \ref{rem:wachmod_props} it follows that the $\AR[1/p]\module$ $N[1/p]$ is finite projective and the $\AR[1/[p]_q]\module$ $N[1/[p]_q]$ is finite projective, and from Remark \ref{rem:regsec_saturation} we have that $N = M \cap D \subset D[1/p]$.
	Then, from the proof of Lemma \ref{lem:analyticrys_ar+_fingen} we may obtain a vector bundle $\pazf$ on $U$ such that $H^0(\pazf, U) \isomorphic N$.
	It remains to construct a stratification on $\pazf$ using the stratification on $N$.

	We being by noting that $\Spec(\AR(1)[1/p]) \cup \Spec(\AR(1)[1/[p]_q]) \rightarrow U(1)$ is a flat cover, and similarly, $\Spec(\AR(2)[1/p]) \cup \Spec(\AR(2)[1/[p]_q]) \rightarrow U(2)$ is a flat cover.
	Let us denote by $i \colon \Spec(\AR(1)[1/p]) \rightarrow U(1)$ and $j \colon \Spec(\AR(1)[1/p]) \rightarrow U(1)$ the respective natural maps.
	Then, by faithfully flat descent, we see that the data of a stratification $\varepsilon_{\pazf}$ on $\pazf$, i.e.\ an isomorphism of vector bundles $\varepsilon_{\pazf} \colon p_1^*\pazf \isomorphic p_2^*\pazf$ over $U(1)$ with $\Delta^*(\varepsilon_{\pazf}) = id$ on $\pazf$ and $\varepsilon_{\pazf}$ satisfying the cocycle condition over $U(2)$, is equivalent to the following two pieces of data:
	\begin{enumerate}
		\item[(1)] a stratification $\varepsilon_{i^*\pazf}$ on $i^*\pazf$ over $\Spec(\AR(1)[1/p])$ with $\Delta^*(\varepsilon_{i^*\pazf}) = id$ on $i^*\pazf$ and $\varepsilon_{i^*\pazf}$ satisfies the cocycle condition over $\Spec(\AR(2)[1/p])$,

		\item[(2)] a stratification $\varepsilon_{j^*\pazf}$ on $j^*\pazf$ over $\Spec(\AR(1)[1/[p]_q])$ with $\Delta^*(\varepsilon_{j^*\pazf}) = id$ on $j^*\pazf$ and $\varepsilon_{j^*\pazf}$ satisfies the cocycle condition over $\Spec(\AR(2)[1/[p]_q])$,
	\end{enumerate}
	and such that (1) and (2) agree over the intersection $\Spec(\AR(2)[1/p, 1/[p]_q])$.
	In other words, the two respective pullback stratifications along the two maps $\Spec(\AR[1/p,1/[p]_q]) \xhookrightarrow{u} \Spec(\AR[1/p])$ and $\Spec(\AR[1/p,1/[p]_q]) \xhookrightarrow{v} \Spec(\AR[1/[p]_q])$ must agree, i.e.\ $u^*(\varepsilon_{i^*\pazf}) = v^*(\varepsilon_{j^*\pazf})$ on $u^*i^*\pazf = v^*j^*\pazf$.

	As $\Spec(\AR(1)[1/p])$ is an affine scheme and for the vector bundle $i^*\pazf$ on $\Spec(\AR(1)[1/p])$ it is clear that we have $H^0(\Spec(\AR(1)[1/p]), i^*\pazf) = H^0(U, \pazf)[1/p] \isomorphic N[1/p] = M$, therefore, the piece of data in (1) is equivalent to having an $\AR(1)\linear$ isomorphism
	\begin{equation*}
		\varepsilon_M \colon \AR(1) \otimes_{p_1, \AR} M \isomorphic \AR(1) \otimes_{p_2, \AR} M,
	\end{equation*}
	such that $\Delta^*(\varepsilon_M) = id$ on $M$ and $\varepsilon_M$ satisfies the cocycle condition over $\AR(2)[1/p]$.
	This is obtained by setting $\varepsilon_M = \varepsilon_N[1/p]$.
	Similarly, as $\Spec(\AR(1)[1/[p]_q])$ is an affine scheme and for the vector bundle $j^*\pazf$ on $\Spec(\AR(1)[1/[p]_q])$ it is clear that we have $H^0(\Spec(\AR(1)[1/[p]_q]), j^*\pazf) = H^0(U, \pazf)[1/[p]_q] \isomorphic N[1/[p]_q] = D$, therefore, the piece of data in (2) is equivalent to having an $\AR(1)\linear$ isomorphism
	\begin{equation*}
		\varepsilon_D \colon \AR(1) \otimes_{p_1, \AR} D \isomorphic \AR(1) \otimes_{p_2, \AR} D,
	\end{equation*}
	such that $\Delta^*(\varepsilon_D) = id$ on $D$ and $\varepsilon_D$ satisfies the cocycle condition over $\AR(2)[1/[p]_q]$.
	This is obtained by setting $\varepsilon_D = \varepsilon_N[1/[p]_q]$.
	Furthermore, as $\Spec(\AR[1/p,1/[p]_q])$ is affine, to show that $u^*(\varepsilon_{i^*\pazf}) = v^*(\varepsilon_{j^*\pazf})$ on $u^*i^*\pazf = v^*j^*\pazf$, it is enough to show that $\varepsilon_M[1/p] = \varepsilon_D[1/[p]_q]$ on $M[1/[p]_q] = N[1/p,1/[p]_q] = D[1/p]$, but this is obvious because $\varepsilon_M[1/p] = \varepsilon_N[1/p,1/[p]_q] = \varepsilon_D[1/[p]_q]$, by definition.
	Hence, the stratification on $N$ naturally induces a stratification on $\pazf$.

	It is easy to check that the construction above yields a quasi-inverse to the right vertical arrow in \eqref{eq:afcrystal_vb_strataf+}, thus establishing a natural equivalence of categories.
	Consequently, we get that the slanted arrow in \eqref{eq:afcrystal_vb_strataf+}, i.e.\ \eqref{eq:afcrystal_eval_ar+strat}, induces a natural equivalence of categories.
	This allows us to conclude.
\end{proof}

\begin{rem}
	Following arguments similar to Construction \ref{const:afcrystal_eval_ar+strat} and Proposition \ref{prop:afcrystal_eval_ar+strat}, it can also be shown that we have a natural equivalence of categories
	\begin{equation}\label{eq:cfcrystal_eval_ar+strat}
		\eval_{\AR(\bullet)}^{\prism} \colon \CRhatphi(X_{\prism}) \isomorphic \Strat^{\an, \varphi}(\AR(\bullet)),
	\end{equation}
	obtained by sending a completed $F\textrm{-crystal}$ $\paze$ to the finitely generated $\AR\module$ $N = \paze(\AR(0))$ equipped with a stratification similar to \eqref{eq:strat_n} in Construction \ref{const:afcrystal_eval_ar+strat}.
\end{rem}

Our next goal is to show that the functor $\eval_{\AR(\bullet)}^{\prism}$ in \eqref{eq:afcrystal_eval_ar+strat} (resp.\ \eqref{eq:cfcrystal_eval_ar+strat}) is suitably compatible with the functor $\eval_{\AR}^{\prism}$ in \eqref{eq:afcrystal_eval_ar+} (resp.\ \eqref{eq:cfcrystal_eval_ar+}), similar to \cite[Lemma 3.16]{morrow-tsuji}.
From Remark \ref{rem:phigamma_act_ar+n}, recall that for $n \in \NN$, the product $\Gamma_R^{\times (n+1)}$ of $n+1$ copies of $\Gamma_R$ naturally acts on $(\AR(n), [p]_q)$.
Moreover, from Proposition \ref{prop:arn_gammar_triv_modmu}, we have that the action of the $i^{\textrm{th}}$ component of $\Gamma_R^{\times (n+1)}$ is trivial on $\AR(n)/(n_i(\mu))$.

\begin{const}\label{const:strat_eval_ar+}
	Let $(N, \varepsilon)$ be an object of $\Strat^{\an, \varphi}(\AR(\bullet))$, i.e.\ $N$ is a finitely generated $\AR\module$ on which the sequences $\{p, \mu\}$ and $\{\mu, p\}$ are regular, $N$ is equipped with a stratification $\varepsilon$ with respect to $\AR(\bullet)$ and an $\AR\linear$ isomorphism $\varphi_N \colon \varphi^*(N)[1/[p]_q] \isomorphic N[1/[p]_q]$.
	We define a functor
	\begin{equation}\label{eq:strat_eval_ar+}
		\eval_{\AR}^{\Strat} \colon \Strat^{\an, \varphi}(\AR(\bullet)) \longrightarrow (\varphi, \Gamma_R)\Mod_{\AR}^{[p]_q},
	\end{equation}
	by setting the underlying $\AR\module$ to be $N$ satisfying (1) and (3) of Definition \ref{defi:wach_mods_relative}.
	We equip $N$ with an action of $\Gamma_R$ by noting that the $\Gamma_R^2\action$ on $\AR(1)$ induces a semilinear action of $\Gamma_R$ on $N$.
	Indeed, for each $g$ in $\Gamma_R$, the base change of the stratification $\varepsilon$ along $\AR(1) \xrightarrow{(g, 1)} \AR(1) \xrightarrow{\Delta} \AR$ defines an isomorphism $\AR \otimes_{g, \AR} N \isomorphic N$, i.e.\ a semilinear action of the element $g$ on $N$.
	Moreover, for any $g'$ in $\Gamma_R$, base change of the aforementioned isomorphism along $g' \colon \AR \isomorphic \AR$ is the base change of $\varepsilon$ along $g' \circ \Delta \circ g = \Delta \circ (g'g, g')$.
	So, from Definition \ref{defi:ar+mod_strat}, it follows that the isomorphisms $\AR \otimes_{g, \AR} N \isomorphic N$, for $g$ in $\Gamma_R$, define a semilinear action of $\Gamma_R$ on $N$.
	Finally, since the action of $\Gamma_R \times 1$ is trivial on $\AR(1)/(p_1(\mu))$, therefore, it follows that the induced action of $\Gamma_R$ on $N/\mu N$ is also trivial (in particular, $\Gamma_R$ acts continuously on $N$ by \cite[Lemma 3.4]{abhinandan-relative-wach-ii}).
	Hence, we conclude that $N$ is a Wach module over $\AR$ in the sense of Definition \ref{defi:wach_mods_relative}.
\end{const}

In Theorem \ref{thm:strat_wach_comp_relative}, we shall show that the functor in \eqref{eq:strat_eval_ar+} induces a natural equivalence of categories (the proof of which will be independent of the claims proven in this section).
We have the following compatibility between the functors of \eqref{eq:afcrystal_eval_ar+}, \eqref{eq:afcrystal_eval_ar+strat} and \eqref{eq:strat_eval_ar+} (resp.\ \eqref{eq:cfcrystal_eval_ar+}, \eqref{eq:cfcrystal_eval_ar+strat} and \eqref{eq:strat_eval_ar+}):
\begin{prop}\label{prop:prisfcrys_strat_wach_relative}
	The following diagram is commutative up to canonical isomorphisms:
	\begin{center}
		\begin{tikzcd}[row sep=huge, column sep=huge]
			\Vect^{\an, \varphi}(X_{\prism}) \arrow[r, "\sim"', "\eval_{\AR(\bullet)}^{\prism}\eqref{eq:afcrystal_eval_ar+strat}"] \arrow[rd, "\eval_{\AR}^{\prism}\eqref{eq:afcrystal_eval_ar+}"'] & \Strat^{\an, \varphi}(\AR(\bullet)) \arrow[d, "\eval_{\AR}^{\Strat}", "\eqref{eq:strat_eval_ar+}"'] & \CRhatphi(X_{\prism}) \arrow[l, "\sim", "\eval_{\AR(\bullet)}^{\prism}\eqref{eq:cfcrystal_eval_ar+strat}"'] \arrow[ld, "\eval_{\AR}^{\prism}\eqref{eq:afcrystal_eval_ar+}"]\\
			& (\varphi, \Gamma_R)\Mod_{\AR}^{[p]_q}.
		\end{tikzcd}
	\end{center}
\end{prop}
\begin{proof}
	We will only prove the commutativity of the left triangle; commutativity of the right triangle follows by a similar argument.
	Let $\paze$ be an object of $\Vect^{\an, \varphi}(X_{\prism})$ and set $(N, \varepsilon) \coloneq \eval_{\AR(\bullet)}^{\prism}(\paze)$.
	Since $N = \paze(\AR(0))$ and $\AR(0) = \AR$, therefore, we note that the $\AR\module$ obtained by composing the top horizontal arrow and the middle vertical arrow is isomorphic (as an $\AR\module$) to the module obtained via the slanted arrow on the left hand side.
	Moreover, from Lemma \ref{lem:analyticrys_ar+_fingen}, we know that the sequences $\{p, \mu\}$ and $\{\mu, p\}$ are regular on $N$.
	Furthermore, the conditions on $\varphi_N$ as explained in Proposition \ref{prop:afcrystal_eval_ar+} and Construction \ref{const:strat_eval_ar+} ensure that the Frobenius on $N$ is of finite $\pqheight$.
	So, it remains to check the compatibility of the action of $\Gamma_R$.
	Note that for each $g$ in $\Gamma_R$, the following diagrams commute:
	\begin{center}
		\begin{tikzcd}
			\AR(1) \arrow[r, "(g {,} 1)"] & \AR(1) \arrow[d, "\Delta"] & & \AR(1) \arrow[r, "(g {,} 1)"] & \AR(1) \arrow[d, "\Delta"]\\
			\AR(0) \arrow[u, "p_1"] \arrow[r, "g"'] & \AR(0), & & \AR(0) \arrow[u, "p_2"] \arrow[r, "id"'] & \AR(0).
		\end{tikzcd}
	\end{center}
	Therefore, the action of $g$ on $N = \paze(\AR(0))$ induced by the stratification $\varepsilon$ (see Construction \ref{const:strat_eval_ar+}) coincides with the action of $g$ on $\paze(\AR(0))$ induced by the action on $\AR(0)$ and the crystal property of $\paze$ (see Proposition \ref{prop:afcrystal_eval_ar+}).
	Hence, $\eval_{\AR}^{\Strat}(N, \varepsilon) \isomorphic \paze(\AR(0)) = \paze(\AR) = \eval_{\AR}^{\prism}(\paze)$, thus allowing us to conclude.
\end{proof}

From Proposition \ref{prop:prisfcrys_strat_wach_relative}, we note that in order to show that the functor $\eval_{\AR}^{\prism}$ induces an equivalence of categories, it is enough to show that the functor $\eval_{\AR}^{\Strat}$ induces an equivalence of categories.

\subsection{Constructing stratification on Wach modules}\label{subsec:wach_and_strat}

In this section, we shall show the following claim:
\begin{thm}\label{thm:strat_wach_comp_relative}
	The functor in \eqref{eq:strat_eval_ar+} induces a natural equivalence of categories 
	\begin{equation*}
		\eval_{\AR}^{\Strat} \colon \Strat^{\an, \varphi}(\AR(\bullet)) \isomorphic (\varphi, \Gamma_R)\Mod_{\AR}^{[p]_q}.
	\end{equation*}
\end{thm}
\begin{proof}
	In \eqref{eq:strat_from_wach_relative}, we will define a functor $\Strat_{\AR(\bullet)}$ from the category of Wach modules to the category of $\varphi\modules$ over $\AR$ equipped with a stratification with respect to $\AR(\bullet)$.
	Moreover, in Proposition \ref{prop:strat_from_wach_relative}, we shall show that $\Strat_{\AR(\bullet)}$ is a quasi-inverse to the functor $\eval_{\AR}^{\Strat}$, thus establishing the claimed equivalence.
\end{proof}

Our first goal is to construct a stratification on a Wach module.
Our strategy will be similar to that of \cite[Section 3.2]{morrow-tsuji}.
However, note that we are working with the Galois group $\Gamma_R$, which has ``$\textrm{arithmetic} = \Gamma_F$'' and ``$\textrm{geometric} = \Gamma_R'$'' parts, instead of being a truly ``geometric'' Galois group as considered in loc.\ cit., therefore, our arguments are of different nature and require different computations (see Section \ref{subsec:setup_nota} for notation and explanations on the group $\Gamma_R$).

Let us set $N(1) \coloneq \AR(1) \otimes_{p_2, \AR} N$ equipped with the tensor product Frobenius and tensor product action of $\Gamma_R^2$, where $\Gamma_R^2$ acts on $N$ via projection to the second coordinate.
We will start by proving some results on the cohomology of $N(1)$ for the action of $1 \times \Gamma_R$.
The steps will be similar to the \textit{3-step} argument presented in Section \ref{subsec:gamma_act_ar1}.
In particular, let us note that the results for the action of the geometric part of $\Gamma_R$ in Section \ref{subsubsec:geo_gamma_action_n1}, are applicable for all primes $p$.
However, the structure of $\Gamma_F$ is different for $p \geqslant 3$ compared to $p = 2$, therefore, in Sections \ref{subsubsec:fpx_action_n1} and \ref{subsubsec:gamma0_action_n1}, we will assume that $p \geqslant 3$.
For $p =2$, the action of $\Gamma_F$ will be handled in Section \ref{subsubsec:p=2_gammaF_action_n1}.
Finally, we will put everything together in Section \ref{subsubsec:wachmod_strat} to construct a stratification using the action of $\Gamma_R$ on Wach modules (see Proposition \ref{prop:strat_from_wach_relative}).

\subsubsection{The action of \texorpdfstring{$\Gamma_R'$}{-}}\label{subsubsec:geo_gamma_action_n1}

Our first goal is to show the following claim:
\begin{lem}\label{lem:n1_geometric_exact}
	For each $n \geqslant 1$, the following natural $(\varphi, \Gamma_R \times \Gamma_F)\equivariant$ sequence is exact:
	\begin{equation}\label{eq:n1_geometric_exact}
		0 \longrightarrow (N(1)/p_1(\mu))^{1 \times \Gamma_R'} \xrightarrow{\hspace{1mm} p_1(\mu)^n \hspace{1mm}} (N(1)/p_1(\mu)^{n+1})^{1 \times \Gamma_R'} \longrightarrow (N(1)/p_1(\mu)^{n})^{1 \times \Gamma_R'} \longrightarrow 0.
	\end{equation}
\end{lem}
\begin{proof}
	The proof is similar to the proof of Lemma \ref{lem:a1_geometric_exact}.
	To lighten notations, let us denote by $A(1) \coloneq \AR(1)$, $\Abar(1) \coloneq \AR(1)/p_1(\mu)$ and $\Nbar(1) \coloneq N(1)/p_1(\mu)$.
	Instead of working with the action of $1 \times \Gamma_R'$, we will work with the $\qconnection$ arising from this action.
	More precisely, in the notation of Definition \ref{defi:qdeRham_complex}, take $D$ to be $\LambdatildeR \isomorphic A(1)^{1 \times \Gamma_R'}$ (see Lemma \ref{lem:iota_lambdatilde_iso}), set $A$ to be $A(1)$ equipped with a $\LambdatildeR\linear$ action of $1 \times \Gamma_R'$, and let $\{\gamma_1, \ldots, \gamma_d\}$ be the topological generators of $\Gamma_R'$ (see Section \ref{subsec:important_rings}).
	Then, by setting $q = 1+p_2(\mu)$ and $U_i = p_2([X_i^{\flat}])$, for $1 \leqslant i \leqslant d$, we know that $A(1)$ satisfies the hypothesis of Definition \ref{defi:qdeRham_complex} (see the proof of Lemma \ref{lem:a1_geometric_exact}).
	In particular, $A(1)$ is equipped with a $\LambdatildeR\linear$ $\qconnection$ $\nabla_q \colon A(1) \rightarrow q\Omega^1_{A(1)/\LambdatildeR}$, given as $f \mapsto \sum_{i=1}^d \frac{\gamma_i(f)-f}{p_2(\mu)} \dlog(p_2([X_i^{\flat}]))$.

	Next, we have that $N$ is a Wach module over $\AR$ and $N(1) = A(1) \otimes_{p_2, \AR} N$ is equipped with the tensor product Frobenius and the tensor product action of $\Gamma_R^2$.
	Note that for any $f \otimes y$ in $N(1)$ and $g$ in $1 \times \Gamma_R$, we have that $(g-1)(f \otimes y) = (g-1)y \otimes y + g(f) \otimes (g-1)y$ is in $p_2(\mu)N(1)$.
	Therefore, the operator 
	\begin{align*}
		\nabla_q \colon N(1) &\longrightarrow N(1) \otimes_{A(1)} \Omega^1_{A(1)/\LambdatildeR}\\
		x &\longmapsto \textstyle\sum_{i=1}^d \tfrac{\gamma_i(x)-x}{p_2(\mu)} \dlog([X_i^{\flat}]),
	\end{align*}
	satisfies the assumptions of Definition \ref{defi:qconnection}.
	Moreover, from the proof of Lemma \ref{lem:iota_lambdatilde_iso} and Example \ref{exam:wach_qconnection}, we see that the $\qconnection$ $\nabla_q$ on $N(1)$ is $(p, \mu)\adically$ quasi-nilpotent, and it is flat because $\gamma_i$ commute with each other.
	Let us also note that the action of $1 \times \Gamma_R'$ is trivial on $p_1(\mu)$ and $\LambdatildeR/(p_1(\mu)) \isomorphic \Lambda_R \isomorphic \Abar(1)^{1 \times \Gamma_R'}$ (see \eqref{eq:lambdatilder_mod_p1mu}, Example \ref{exam:a1modp1mu_qconnection} and the proof of Lemma \ref{lem:iota_lambdatilde_iso}).
	Therefore, we see that the $\qconnection$ on $\Nbar(1)$, induced by reducing the $\qconnection$ on $N(1)$ modulo $p_1(\mu)$, coincides with the $\qconnection$ on $\Nbar(1)$ described in Example \ref{exam:wach_a1modp1mu_qconnection}.
	Now, consider the following exact sequence of $q\textrm{-de Rham}$ complexes:
	\begin{align*}
		0 \longrightarrow \Nbar(1) \otimes_{\Abar(1)} q\Omega^{\bullet}_{\Abar(1)/\Lambda_R} &\xrightarrow{\hspace{1mm} p_1(\mu)^n \hspace{1mm}} N(1)/p_1(\mu)^{n+1} \otimes_{A(1)} q\Omega^{\bullet}_{A(1)/\LambdatildeR} \longrightarrow \\
		&\longrightarrow N(1)/p_1(\mu)^n \otimes_{A(1)} q\Omega^{\bullet}_{A(1)/\LambdatildeR} \longrightarrow 0.
	\end{align*}
	Since the action of $1 \times \Gamma_R'$ is continuous on $N(1)$, therefore by Lemma \ref{lem:cont_coh_disc}, we see that $(N(1)/p_1(\mu)^n)^{1 \times \Gamma_R'} = (N(1)/p_1(\mu)^n)^{\nabla_q=0}$.
	In particular, showing that \eqref{eq:n1_geometric_exact} is exact, is equivalent to showing that $H^1\big(\Nbar(1) \otimes_{\Abar(1)} q\Omega^{\bullet}_{\Abar(1)/\Lambda_R}\big) = 0$.
	From Proposition \ref{prop:n1modp1mu_connection}, recall that the $q\textrm{-de Rham}$ complex $\Nbar(1) \otimes_{\Abar(1)} q\Omega^{\bullet}_{\Abar(1)/\Lambda_R}$ is naturally quasi-isomorphic to the de Rham complex $\Nbar(1) \otimes_{\Abar(1)} \Omega^{\bullet}_{\Abar(1)/\Lambda_R}$.
	Moreover, in Proposition \ref{prop:abarone_comp}, we have shown that the de Rham complex $\Nbar(1) \otimes_{\Abar(1)} \Omega^{\bullet}_{\Abar(1)/\Lambda_R}$ is acyclic in positive degrees, in particular, we see that $H^1\big(\Nbar(1) \otimes_{\Abar(1)} q\Omega^{\bullet}_{\Abar(1)/\Lambda_R}\big) = H^1\big(\Nbar(1) \otimes_{\Abar(1)} \Omega^{\bullet}_{\Abar(1)/\Lambda_R}\big) = 0$.
	Hence, it follows that \eqref{eq:n1_geometric_exact} is exact.
\end{proof}

Let $\MLambdatilde \coloneq N(1)^{1 \times \Gamma_R'}$ as a module over $A(1)^{1 \times \Gamma_R'} \lisomorphic \LambdatildeR$ (see Lemma \ref{lem:iota_lambdatilde_iso}), equipped with the induced Frobenius and an induced semilinear and continuous action of $\Gamma_R \times \Gamma_F$.
Then, we note the following:
\begin{lem}\label{lem:mlambdatilde_p2mup_reg}
	The sequence $\{p_2(\mu), p\}$ is regular on $\MLambdatilde$.
\end{lem}
\begin{proof}
	From Definition \ref{defi:wach_mods_relative} recall that $\{\mu, p\}$ is a regular sequence on $N$ and $p_2 \colon \AR \rightarrow \AR(1)$ is faithfully flat by Lemma \ref{lem:pi_rj_ff}.
	Therefore, it follows that $\{p_2(\mu), p\}$ is a regular sequence on $\AR(1) \otimes_{p_2, \AR} N$.
	Since $p_2(\mu)$ is invariant under the action of $1 \times \Gamma_R'$, it follows that $p_2(\mu)N(1) \cap \MLambdatilde = p_2(\mu) \MLambdatilde$.
	In particular, the natural map $\MLambdatilde/p_2(\mu)\MLambdatilde \rightarrow N(1)/p_2(\mu)N(1)$ is injective.
	Hence, $\MLambdatilde/p_2(\mu)\MLambdatilde$ is $p\torsionfree$, as claimed.
\end{proof}

Our next goal is to describe the $\LambdatildeR\module$ $\MLambdatilde$, as well as, the action of $1 \times \Gamma_F \subset \Gamma_R \times \Gamma_F$ on it, more explicitly.
So, let us consider the following $(\varphi, 1 \times \Gamma_F)\equivariant$ composition:
\begin{equation*}
	\rho \colon R\llbracket\mu\rrbracket \xrightarrow[\sim]{\hspace{1mm} \iota \hspace{1mm}} \AR \xrightarrow{\hspace{1mm} p_2 \hspace{1mm}} \AR(1) \xrightarrow{\hspace{1mm} \Delta' \hspace{1mm}} \LambdatildeR,
\end{equation*}
where the map $\iota$ is described in Section \ref{subsec:important_rings}, the map $p_2$ is described in Section \ref{subsec:cech_nerve_ar+} and the map $\Delta'$ is described in Section \ref{subsubsec:geo_gamma_action} (see before Lemma \ref{lem:delta_prime_iso}), and we set $\rho \coloneq \Delta' \circ p_2 \circ \iota$.
In particular, we have that $\rho(\mu) = p_2(\mu)$ and $\rho(X_i) = p_1([X_i^{\flat}])$, for all $1 \leqslant i \leqslant d$.
Moreover, by the definition of $\Delta'$, we have that $\Delta' \circ p_1 = \Delta' \circ p_2 \colon \AR \rightarrow \LambdatildeR$ as ring homomorphisms, and the composition $\Delta' \circ p_1$ coincides with the structure map $p_1 \colon \AR \rightarrow \LambdatildeR$, which is faithfully flat (see Construction \ref{const:lambdatilder}).
Therefore, it follows that the composition $\Delta' \circ p_2$, and hence, $\rho \coloneq \Delta' \circ p_2 \circ \iota \colon R\llbracket\mu\rrbracket \rightarrow \LambdatildeR$ is also faithfully flat.

Now, tensoring the $(\varphi, 1 \times \Gamma_F)\equivariant$ map $\Delta' \colon \AR(1) \rightarrow \LambdatildeR$ with the Wach module $N$ over $\AR$, we obtain the following $(\varphi, 1 \times \Gamma_F)\equivariant$ map:
\begin{equation*}
	\Delta_N' \colon N(1) = \AR(1) \otimes_{p_2, \AR} N \longrightarrow \LambdatildeR \otimes_{\Delta' \circ p_2, \AR} N = \LambdatildeR \otimes_{\rho, R\llbracket\mu\rrbracket} N,
\end{equation*}
where we consider $N$ as an $R\llbracket\mu\rrbracket\module$ via the isomorphism $\iota \colon R\llbracket\mu\rrbracket \isomorphic \AR$, equipped with a $(\varphi, 1 \times \Gamma_F)\action$ (see Section \ref{subsubsec:fpx_descent}).
For each $n \geqslant 1$, reducing the map $\Delta_N'$ modulo $p_1(\mu)^n$ and taking $(1 \times \Gamma_R')\textrm{-invariants}$ of the source, induces the following $(\varphi, 1 \times \Gamma_F)\equivariant$ map:
\begin{equation}\label{eq:delta_prime_n}
	\Delta_N' \colon (N(1)/p_1(\mu)^n)^{1 \times \Gamma_R'} \longrightarrow (\LambdatildeR \otimes_{\rho, R\llbracket\mu\rrbracket} N)/p_1(\mu)^n.
\end{equation}

\begin{lem}\label{lem:delta_prime_n_iso}
	For each $n \geqslant 1$, the map $\Delta_N'$ in \eqref{eq:delta_prime_n} is a $(\varphi, 1 \times \Gamma_F)\equivariant$ isomorphism of $\LambdatildeR\modules$.
	Moreover, \eqref{eq:delta_prime_n} induces a $\LambdatildeR\linear$ and $(\varphi, 1 \times \Gamma_F)\equivariant$ isomorphism
	\begin{equation}\label{eq:delta_prime_N}
		\Delta_N' \colon \MLambdatilde \coloneq N(1)^{1 \times \Gamma_R'} \isomorphic \LambdatildeR \otimes_{\rho, R\llbracket\mu\rrbracket} N.
	\end{equation}
	In particular, for each $n \geqslant 1$, we have natural $\LambdatildeR\linear$ and $(\varphi, \Gamma_R \times \Gamma_F)\equivariant$ isomorphisms
	\begin{equation}\label{eq:mlambdatilde_mod_p1mun}
		\MLambdatilde/p_1(\mu)^n \isomorphic (N(1)/p_1(\mu)^n)^{1 \times \Gamma_R'}.
	\end{equation}
\end{lem}
\begin{proof}
	Note that for $n = 1$, from \eqref{eq:delta_prime_n}, we have the following $\Lambda_R\linear$ and $(\varphi, \Gamma_F)\equivariant$ composition:
	\begin{equation*}
		\MLambda = \Nbar(1)^{1 \times \Gamma_R'} \longrightarrow \Lambda_R \otimes_{\rho, R\llbracket\mu\rrbracket} N,
	\end{equation*}
	where $\Nbar(1) \coloneq N(1)/p_1(\mu)$, and we have used that $\Nbar(1)^{1 \times \Gamma_R'} = \Nbar(1)^{\nabla_q=0} = \Nbar(1)^{\nabla=0} = M_{\Lambda}$ (in the notations of Proposition \ref{prop:n1modp1mu_connection} and Proposition \ref{prop:abarone_comp}), where the first equality follows because the action of $1 \times \Gamma_R'$ is continuous on $\Nbar(1)$ (see Lemma \ref{lem:cont_coh_disc}).
	Now, from Proposition \ref{prop:abarone_comp} we see that the composition above is precisely the isomorphism in \eqref{eq:mlambda_arpdn}, in particular, \eqref{eq:delta_prime_n} is an isomorphism for $n = 1$,.
	So, consider the following $\LambdatildeR\linear$ and $(\varphi, 1 \times \Gamma_F)\equivariant$ commutative diagram:
	\begin{center}
		\begin{tikzcd}
			0 \arrow[r] & \MLambda \arrow[r, "p_1(\mu)^n"] \arrow[d, "\eqref{eq:delta_prime_n}", "\wr"'] & (N(1)/p_1(\mu)^{n+1})^{1 \times \Gamma_R'} \arrow[r] \arrow[d, "\eqref{eq:delta_prime_n}"] & (N(1)/p_1(\mu)^n)^{1 \times \Gamma_R'} \arrow[r] \arrow[d, "\eqref{eq:delta_prime_n}"] & 0\\
			0 \arrow[r] & \Lambda_R \otimes_{\rho, R\llbracket\mu\rrbracket} N \arrow[r, "p_1(\mu)^n"] & (\LambdatildeR \otimes_{\rho, R\llbracket\mu\rrbracket} N)/p_1(\mu)^{n+1} \arrow[r] & (\LambdatildeR \otimes_{\rho, R\llbracket\mu\rrbracket} N)/p_1(\mu)^n \arrow[r] & 0,
		\end{tikzcd}
	\end{center}
	where the bottom row is exact and the top row is the exact sequence from \eqref{eq:n1_geometric_exact}.
	Using the diagram, an easy induction on $n \geqslant 1$, shows that the natural $\LambdatildeR\linear$ and $(\varphi, 1 \times \Gamma_F)\equivariant$ map in \eqref{eq:delta_prime_n} is an isomorphism, i.e.\ $\Delta_N' \colon \MLambdatilde/p_1(\mu)^n \isomorphic (\LambdatildeR \otimes_{\rho, R\llbracket\mu\rrbracket} N)/p_1(\mu)^n$.
	Moreover, as both $N(1)$ and $(\LambdatildeR \otimes_{\rho, R\llbracket\mu\rrbracket} N)$ are $p_1(\mu)\adically$ complete, taking the limit over $n \geqslant 1$, and noting that inverse limit commutes with taking $(1 \times \Gamma_R')\textrm{-invariants}$, we obtain the $\LambdatildeR\linear$ and $(\varphi, 1 \times \Gamma_F)\equivariant$ isomorphism in \eqref{eq:delta_prime_N}, i.e.\ $\Delta'_N \colon \MLambdatilde = N(1)^{1 \times \Gamma_R'} \isomorphic \LambdatildeR \otimes_{\rho, R\llbracket\mu\rrbracket} N$.

	Finally, note that for each $n \geqslant 1$, it is clear that we have a natural $\LambdatildeR\linear$ and $(\varphi, \Gamma_R \times \Gamma_F)\equivariant$ inclusions $\MLambdatilde/p_1(\mu)^n \subset (N(1)/p_1(\mu)^n)^{1 \times \Gamma_R'}$, in particular, the map in \eqref{eq:mlambdatilde_mod_p1mun} is injective.
	To obtain the surjectivity of \eqref{eq:mlambdatilde_mod_p1mun}, consider the following $\LambdatildeR\linear$ and $(\varphi, 1 \times \Gamma_F)\equivariant$ morphisms:
	\begin{equation*}
		\MLambdatilde/p_1(\mu)^n \xrightarrow{\hspace{1mm}\eqref{eq:mlambdatilde_mod_p1mun}\hspace{1mm}} (N(1)/p_1(\mu)^n)^{1 \times \Gamma_R'} \xrightarrow[\sim]{\eqref{eq:delta_prime_n}} (\LambdatildeR \otimes_{\rho, R\llbracket\mu\rrbracket} N)/p_1(\mu)^n,
	\end{equation*}
	where it is easy to see that the composition is reduction modulo $p_1(\mu)^n$ of the isomorphism $\MLambdatilde \isomorphic \LambdatildeR \otimes_{\rho, R\llbracket\mu\rrbracket} N$.
	Since the composition is bijective, therefore, we get that \eqref{eq:mlambdatilde_mod_p1mun} is bijective as well.
	This allows us to conclude.
\end{proof}

\subsubsection{The action of \texorpdfstring{$\FF_p^{\times}$}{-}}\label{subsubsec:fpx_action_n1}

In this section, we will assume that $p \geqslant 3$ and consider the invariants of the exact sequence \eqref{eq:n1_geometric_exact}, for the action of $1 \times \FF_p^{\times}$.
More precisely, we claim the following:
\begin{lem}\label{lem:n1_fpx_exact}
	For each $n \geqslant 1$, the following natural $(\varphi, \Gamma_R \times \Gamma_0)\equivariant$ sequence is exact:
	\begin{equation}\label{eq:n1_fpx_exact}
		0 \longrightarrow \MLambda^{\FF_p^{\times}} \xrightarrow{\hspace{1mm} p_1(\mu)^n \hspace{1mm}} (\MLambdatilde/p_1(\mu)^{n+1})^{1 \times \FF_p^{\times}} \longrightarrow (\MLambdatilde/p_1(\mu)^{n})^{1 \times \FF_p^{\times}} \longrightarrow 0.
	\end{equation}
\end{lem}
\begin{proof}
	Using the $\LambdatildeR\linear$ and $(\varphi, \Gamma_R \times \Gamma_F)\equivariant$ isomorphism \eqref{eq:mlambdatilde_mod_p1mun}, the exact sequence \eqref{eq:n1_geometric_exact} may be written as the following $\LambdatildeR\linear$ and $(\varphi, \Gamma_R \times \Gamma_F)\equivariant$ exact sequence:
	\begin{equation*}
		0 \longrightarrow \MLambda \xrightarrow{\hspace{1mm} p_1(\mu)^n \hspace{1mm}} \MLambdatilde/p_1(\mu)^{n+1} \longrightarrow \MLambdatilde/p_1(\mu)^{n} \longrightarrow 0.
	\end{equation*}
	By considering the associated long exact sequence for the cohomology of $(1 \times \FF_p^{\times})\action$ and noting that $H^1(1 \times \FF_p^{\times}, \MLambda) = 0$, since $p-1$ is invertible in $\ZZ_p$, we obtain that the sequence in \eqref{eq:n1_fpx_exact} is exact.
\end{proof}

Let us describe the $\LambdatildeR\modules$ in Lemma \ref{lem:n1_fpx_exact} more explicitly.
Recall that from Construction \ref{const:lambdatilder0}, we have the ring $\LambdatildeRo = \LambdatildeR^{1 \times \FF_p^{\times}}$, which is $(p, p_1(\mu))\textrm{-adically}$ complete and equipped with an induced Frobenius and continuous action of $\Gamma_R \times \Gamma_F$.
Moreover, $\MLambdatilde = N(1)^{1 \times \Gamma_R'}$ is a finite and $(p, p_1(\mu))\textrm{-adically}$ complete $\LambdatildeR\module$, equipped with an induced action of $(\varphi, \Gamma_R \times \Gamma_F)$.
Set $\MLambdatildeo \coloneq \MLambdatilde^{1 \times \FF_p^{\times}}$ as a $\LambdatildeRo\module$, equipped with an induced Frobenius and an induced semilinear and continuous action of $\Gamma_R \times \Gamma_0$.
Furthermore, from \eqref{eq:mlambdatilde_mod_p1mun} in Lemma \ref{lem:delta_prime_n_iso}, we have a natural $\Lambda_R\linear$ and $(\varphi, 1 \times \Gamma_F)\equivariant$ isomorphism $\MLambdatilde/p_1(\mu) \isomorphic \MLambda$, and we set $\MLambdao \coloneq \MLambda^{1 \times \FF_p^{\times}}$ as a module over $\Lambda_{R, 0} \lisomorphic \LambdatildeRo/(p_1(\mu))$ (see \eqref{eq:lambdatildero_mod_p1mu}), equipped with an induced Frobenius and an induced semilinear and continuous action of $\Gamma_0$.
Then, from the discussion in Proposition \ref{prop:mlambda_fpx_descent}, we have that $\MLambdao$ is a finitely generated, $p\textrm{-adically}$ complete and $p\torsionfree$ $\Lambda_{R, 0}\module$, and $\Lambda_R\textrm{-linearly}$ extending the natural $\Lambda_{R,0}\linear$ inclusion $\MLambdao \subset \MLambda$, induces a $(\varphi, \Gamma_F)\equivariant$ isomorphism of $\Lambda_R\modules$ $\Lambda_R \otimes_{\Lambda_{R, 0}} \MLambdao \isomorphic \MLambda$.
\begin{lem}\label{lem:lambdatildeR_comp}
	The $\LambdatildeRo\module$ $\MLambdatildeo$ is finitely generated and $(p, p_1(\mu))\textrm{-adically}$ complete.
	Moreover, by $\LambdatildeR\textrm{-linearly}$ extending the natural $\LambdatildeRo\linear$ inclusion $\MLambdatildeo \subset \MLambdatilde$, we obtain a $(\varphi, \Gamma_R \times \Gamma_F)\equivariant$ isomorphism
	\begin{equation}\label{eq:lambdatildeR_comp}
		\LambdatildeR \otimes_{\LambdatildeRo} \MLambdatildeo \isomorphic \MLambdatilde.
	\end{equation}
	Furthermore, we have a natural $(\varphi, \Gamma_0)\equivariant$ isomorphism $\MLambdatildeo/p_1(\mu) \isomorphic \MLambda^{1 \times \FF_p^{\times}} = \MLambdao$.
\end{lem}
\begin{proof}
	From \eqref{eq:delta_prime_N} in Lemma \ref{lem:delta_prime_n_iso}, recall that we have a $(\varphi, 1 \times \Gamma_F)\equivariant$ isomorphism of $\LambdatildeR\modules$ $\Delta'_N \colon \MLambdatilde \isomorphic \LambdatildeR \otimes_{\rho, R\llbracket\mu\rrbracket} N$.
	Moreover, from Proposition \ref{prop:wachmod_fpx_descent} recall that $N_0 \coloneq N^{\FF_p^{\times}}$ is a finitely generated module over $R\llbracket\mu_0\rrbracket = R\llbracket\mu\rrbracket^{\FF_p^{\times}}$ (see Lemma \ref{lem:rmu0_fpxinv}), equipped with the induced action of $(\varphi, \Gamma_0)$, and we have a natural $(\varphi, \Gamma_F)\equivariant$ isomorphism of $R\llbracket\mu\rrbracket\modules$ $R\llbracket\mu\rrbracket \otimes_{R\llbracket\mu_0\rrbracket} N_0 \isomorphic N$.
	Then, taking invariants of \eqref{eq:delta_prime_N} under the action of $1 \times \FF_p^{\times}$, induces a $(\varphi, 1 \times \Gamma_0)\equivariant$ isomorphism of $\LambdatildeRo\modules$
	\begin{equation}\label{eq:delta_prime_n0_iso}
		\MLambdatildeo \isomorphic \LambdatildeRo \otimes_{\rho, R\llbracket\mu_0\rrbracket} N_0.
	\end{equation}
	In particular, it follows that $\MLambdatildeo$ is finitely generated over $\LambdatildeRo$, and $(p, p_1(\mu))\textrm{-adically}$ separated, hence, $(p, p_1(\mu))\textrm{-adically}$ complete.
	Now, consider the following natural $\LambdatildeR\linear$ diagram:
	\begin{center}
		\begin{tikzcd}
			\LambdatildeR \otimes_{\LambdatildeRo} \MLambdatildeo \arrow[r, "\eqref{eq:lambdatildeR_comp}"] \arrow[d, "\eqref{eq:delta_prime_n0_iso}", "\wr"'] & \MLambdatilde \arrow[d, "\eqref{eq:delta_prime_n}", "\wr"']\\
			\LambdatildeR \otimes_{\rho, R\llbracket\mu_0\rrbracket} N_0 \arrow[r, "\sim"] & \LambdatildeR \otimes_{\rho, R\llbracket\mu\rrbracket} N,
		\end{tikzcd}
	\end{center}
	where the bottom horizontal arrow is the extension along the $(\varphi, 1 \times \Gamma_F)\equivariant$ map $\rho \colon R\llbracket\mu\rrbracket \rightarrow \LambdatildeR$ (see the discussion before Lemma \ref{lem:delta_prime_n_iso}) of the $(\varphi, \Gamma_F)\equivariant$ and $R\llbracket\mu\rrbracket\linear$ isomorphism $R\llbracket\mu\rrbracket \otimes_{R\llbracket\mu_0\rrbracket} N_0 \isomorphic N$ from Proposition \ref{prop:wachmod_fpx_descent}.
	The diagram commutes by definition, and it follows that the top horizontal arrow, i.e.\ \eqref{eq:lambdatildeR_comp} is also an isomorphism.
	Finally, it is easy to see that the isomorphism \eqref{eq:lambdatildeR_comp} induces a natural $(\varphi, \Gamma_0)\equivariant$ isomorphism $\MLambdatildeo/p_1(\mu) = \MLambdatilde^{1 \times \FF_p^{\times}}/p_1(\mu) \isomorphic \MLambda^{1 \times \FF_p^{\times}} = \MLambdao$, thus allowing us to conclude.
\end{proof}

Let us note an important observation for the action of $\Gamma_R \times \Gamma_F$ on $\MLambdatilde$.
\begin{lem}\label{lem:mlambdatilde_triv_modmu}
	The action of $1 \times \Gamma_F$ is trivial on $\MLambdatilde/p_2(\mu)$ and the action of $\Gamma_R \times 1$ is trivial on $\MLambdatilde/p_1(\mu)$.
\end{lem}
\begin{proof}
	From \eqref{eq:delta_prime_N} in Lemma \ref{lem:delta_prime_n_iso}, recall that we have a $(\varphi, 1 \times \Gamma_F)\equivariant$ isomorphism $\Delta_N' \colon \MLambdatilde \isomorphic \LambdatildeR \otimes_{\rho, R\llbracket\mu\rrbracket} N$.
	Now, let $g$ be any element of $\Gamma_F = 1 \times \Gamma_F$, then for any $f \otimes y$ in $\LambdatildeR \otimes_{\rho, R\llbracket\mu\rrbracket} N$, we have that
	\begin{equation*}
		(g-1)(f \otimes y) = ((g-1)f) \otimes y + g(f) \otimes (g-1)y.
	\end{equation*}
	Since the action of $1 \times \Gamma_F$ is trivial on $\LambdatildeR/(p_2(\mu))$ (see Lemma \ref{lem:lambdatildeR_triv_modmu}), as well as, on $N/\mu N$ by definition, therefore, it follows that $(g-1)(f \otimes x)$ is an element of $p_2(\mu)\LambdatildeR \otimes_{\rho, R\llbracket\mu\rrbracket} N$.
	Then, using the $(\varphi, 1 \times \Gamma_F)\equivariant$ isomorphism \eqref{eq:delta_prime_N}, it follows that for any $x$ in $\MLambdatilde$, we have that $(g-1)x$ is an element of $p_2(\mu)\MLambdatilde$, in particular, the action of $1 \times \Gamma_F$ is trivial on $\MLambdatilde/p_2(\mu)$.
	Next, since the action of $\Gamma_R \times 1$ is trivial on $N(1)/p_1(\mu)$ (see the proof of Lemma \ref{lem:n1_geometric_exact}), therefore, it easily follows that the induced action of $\Gamma_R \times 1$ trivial on $\MLambdatilde/p_1(\mu) \isomorphic (N(1)/p_1(\mu))^{1 \times \Gamma_R'}$ (see \eqref{eq:mlambdatilde_mod_p1mun} in Lemma \ref{lem:delta_prime_n_iso}).
	This concludes our proof.
\end{proof}

\subsubsection{The action of \texorpdfstring{$1 + p\ZZ_p$}{-}}\label{subsubsec:gamma0_action_n1}

We will assume that $p \geqslant 3$ and consider the invariants of the exact sequence \eqref{eq:n1_fpx_exact}, for the action of $1 \times \Gamma_0 \isomorphic 1 \times (1 + p\ZZ_p)$, and show the following:
\begin{lem}\label{lem:n1_gamma0_exact}
	For each $n \geqslant 1$, the following natural $(\varphi, \Gamma_R \times 1)\equivariant$ sequence is exact:
	\begin{equation}\label{eq:n1_gamma0_exact}
		0 \longrightarrow \MLambda^{1 \times \Gamma_F} \xrightarrow{\hspace{1mm} p_1(\mu)^n \hspace{1mm}} (\MLambdatilde/p_1(\mu)^{n+1})^{1 \times \Gamma_F} \longrightarrow (\MLambdatilde/p_1(\mu)^n)^{1 \times \Gamma_F} \longrightarrow 0.
	\end{equation}
\end{lem}

For each $n \geqslant 1$, note that by reducing modulo $p_1(\mu)^n$, the $(\varphi, \Gamma_R \times \Gamma_F)\equivariant$ isomorphism \eqref{eq:lambdatildeR_comp} from Lemma \ref{lem:lambdatildeR_comp}, taking its $(1 \times \FF_p^{\times})\textrm{-invariants}$, and observing that $p_1(\mu)$ is invariant under the action of $1 \times \Gamma_F$, we obtain a $(\varphi, \Gamma_R \times \Gamma_0)\equivariant$ isomorphism $\MLambdatildeo/p_1(\mu)^n \isomorphic (\MLambdatilde/p_1(\mu)^n)^{1 \times \FF_p^{\times}}$.
Consequently, the sequence in \eqref{eq:n1_gamma0_exact} may be rewritten as the following $(\varphi, \Gamma_R \times \Gamma_0)\equivariant$ sequence:
\begin{equation}\label{eq:mlambdatilde0_gamma0_exact}
	0 \longrightarrow \MLambdao^{\Gamma_0} \xrightarrow{\hspace{1mm} p_1(\mu)^n \hspace{1mm}} (\MLambdatildeo/p_1(\mu)^{n+1})^{1 \times \Gamma_0} \longrightarrow (\MLambdatildeo/p_1(\mu)^{n})^{1 \times \Gamma_0} \longrightarrow 0.
\end{equation}
In order to prove that \eqref{eq:mlambdatilde0_gamma0_exact} is exact, let us now look at the action of $\Gamma_R \times \Gamma_F$ on $\MLambdatilde$ and $\MLambdatildeo$, respectively.
We start with the following observation:
\begin{lem}\label{lem:mlambdatilde0_triv_modmu0}
	The action of $1 \times \Gamma_0$ is trivial on $\MLambdatildeo/p_2(\mu_0)$ and the action of $\Gamma_R \times 1$ is trivial on $\MLambdatildeo/p_1(\mu)$.
\end{lem}
\begin{proof}
	Using the triviality of the action of $1 \times \Gamma_F$ on $\MLambdatilde/p_2(\mu)$ from Lemma \ref{lem:mlambdatilde_triv_modmu}, the $(\varphi, 1 \times \Gamma_0)\equivariant$ isomorphism $\MLambdatildeo \isomorphic \LambdatildeRo \otimes_{\rho, R\llbracket\mu_0\rrbracket} N_0$ (see \eqref{eq:lambdatildeR_comp} in Lemma \ref{lem:n1_fpx_exact}), and the fact that the action of $1 \times \Gamma_0$ is trivial on $\LambdatildeRo/(p_2(\mu_0))$ (see Lemma \ref{lem:lambdatilde0_triv_modmu0}) and on $N_0/\mu_0 N_0$ (see Proposition \ref{prop:wachmod_fpx_descent}), it follows that the action of $1 \times \Gamma_0$ is trivial on $\MLambdatildeo/p_2(\mu_0)$.
	Next, observe that the action of $\Gamma_R \times 1$ is trivial on $N(1)/p_1(\mu)$ (see the proof of Lemma \ref{lem:n1_geometric_exact}), therefore, it easily follows that the induced action of $\Gamma_R \times 1$ is trivial on $\MLambdatildeo/p_1(\mu) \isomorphic \MLambda^{1 \times \FF_p^{\times}}$ (see Lemma \ref{lem:lambdatildeR_comp}).
\end{proof}

\begin{rem}\label{rem:gamma0_act_mlambda0}
	From Lemma \ref{lem:mlambdatilde0_triv_modmu0}, note that the action of $1 \times \Gamma_0$ is trivial on $\MLambdatildeo/p_2(\mu_0)$ and multiplication by $p_1(\mu)$ on $\MLambdatildeo$ is equivariant for this action.
	Therefore, for any $g$ in $\Gamma_0$ and any $x$ in $\MLambdatildeo/p_1(\mu)^n$, we see that $(g-1)x$ is an element of $p_2(\mu_0)\MLambdatildeo/p_1(\mu)^n$.
	In particular, for $n=1$, using the isomorphism $\MLambdatildeo/p_1(\mu) \lisomorphic \MLambdao$ from Lemma \ref{lem:lambdatildeR_comp}, we get that for any $g$ in $\Gamma_0$ and any $x$ in $\MLambdao$, the element $(g-1)x$ belongs to $\mu_0\MLambdao$.
\end{rem}

Now, using the action of $1 \times \Gamma_0$ on $\MLambdatildeo$, let us define a $\qconnection$ (see Definition \ref{defi:qconnection}).
Recall that, to define a $q\textrm{-de Rham}$ complex over $\LambdatildeRo$, we fixed the following element in $\LambdatildeRo$ as a parameter (see \eqref{eq:elem_stilde}):
\begin{equation*}
	\stilde = \tfrac{1 \otimes \ptilde - \ptilde \otimes 1}{\ptilde \otimes 1} = \tfrac{p_2(\ptilde) - p_1(\ptilde)}{p_1(\ptilde)}.
\end{equation*}
Moreover, if $\gamma_0$ is an element of $1 \times \Gamma_0$, then from Lemma \ref{lem:gamma0_act_stilde} we have that $(\gamma_0-1)\stilde = u \hspace{0.5mm} p_2(\mu_0)$, for some unit $u$ in $\LambdatildeRo$ depending on $\gamma_0$.
So, let us fix the choice of a topological generator $\gamma_0$ of $1 \times \Gamma_0$ such that $\chi(\gamma_0) = 1+pa$, for a unit $a$ in $\ZZ_p$, and consider the following operator on $\MLambdatildeo$:
\begin{equation}\label{eq:nablaq_stilde_M}
	\begin{aligned}
		\nabla_{q, \stilde} \colon \MLambdatildeo &\longrightarrow \MLambdatildeo\\
					x &\longmapsto \tfrac{(\gamma_0-1)x}{(\gamma_0-1)\stilde}.
	\end{aligned}
\end{equation}
From the triviality of the action of $1 \times \Gamma_0$ on $\MLambdatildeo/p_2(\mu_0)$ (see Lemma \ref{lem:mlambdatilde0_triv_modmu0}) and from Lemma \ref{lem:gamma0_act_stilde}, it follows that the operator $\nabla_{q, \stilde}$ is well defined.
For each $n \geqslant 1$, using Remark \ref{rem:gamma0_act_mlambda0}, the operator in \eqref{eq:nablaq_stilde_M} induces well-defined operators $\nabla_{q, \stilde} \colon \MLambdatildeo/p_1(\mu)^n \longrightarrow \MLambdatildeo/p_1(\mu)^n$.
As the operator $\nabla_{q, \stilde}$ is an endomorphism of $\MLambdatildeo/p_1(\mu)^n$, we may define the following two term Koszul complex:
\begin{equation}\label{eq:stilde_qderham_M}
	K_{\MLambdatildeo/p_1(\mu)^n}(\nabla_{q, \stilde}) \colon \big[\MLambdatildeo/p_1(\mu)^n \xrightarrow{\hspace{1mm} \nabla_{q, \stilde} \hspace{1mm}} \MLambdatildeo/p_1(\mu)^n\big].
\end{equation}
In particular, for $n = 1$ we set $s \coloneq \mu_0/p$ in $\Lambda_{R, 0}$, and using Remark \ref{rem:gamma0_act_mlambda0} and the fact that $(\gamma_0-1)s = v\mu_0$, for some unit $v$ in $\Lambda_{F,0}$ depending on $\gamma_0$ (see Lemma \ref{lem:gamma0_act_s}), we have a well-defined operator
\begin{equation*}
	\begin{aligned}
		\nabla_{q, s} \colon \MLambdao &\longrightarrow \MLambdao\\
				x &\longmapsto \tfrac{(\gamma_0-1)x}{(\gamma_0-1)s}.
	\end{aligned}
\end{equation*}
Note that the operator above coincides with the operator defined in \eqref{eq:nablaq_s_mlambda0}, and the complex from \eqref{eq:stilde_qderham_M} for $n=1$ coincides with the complex from \eqref{eq:s_qderham_mlambda0}.
Therefore, from Proposition \ref{prop:lambdar0_comp}, we have that the cohomology of the Koszul complex $K_{\MLambdao}(\nabla_{q, s})$ vanishes in degree 1, i.e.\ $H^1(K_{\MLambdao}(\nabla_{q, s})) = 0$.

\begin{rem}
	Considering $\stilde$ as a parameter, similar to Remark \ref{rem:nablaq_stilde_qconnection}, the operator $\nabla_{q, \stilde}$ in \eqref{eq:nablaq_stilde_M} can be considered as a $\qconnection$ in non-logarithmic coordinates, in the sense of Definition \ref{defi:qconnection} and Remark \ref{rem:nablaqi_nota}.
	Then, \eqref{eq:s_qderham_mlambda0} is the $q\textrm{-de Rham}$ complex arising from such a $\qconnection$.
	Similarly, considering $s$ as a parameter, the operator $\nabla_{q, s}$ on $\MLambdao$ can also be considered as a $\qconnection$ in non-logarithmic coordinates, in the sense of Definition \ref{defi:qconnection} and Remark \ref{rem:nablaqi_nota}.
\end{rem}

\begin{proof}[Proof of Lemma \ref{lem:n1_gamma0_exact}]
	The proof follows by the same arguments as in the proof of Lemma \ref{lem:a1_gamma0_exact}.
	Indeed, by using the $(\varphi, \Gamma_R \times \Gamma_0)\equivariant$ isomorphism $\MLambdao/p_1(\mu)^n \isomorphic (\MLambda/p_1(\mu)^n)^{1 \times \FF_p^{\times}}$ (see the discussion before \eqref{eq:mlambdatilde0_gamma0_exact}), in the exact sequence of \eqref{eq:n1_fpx_exact}, we obtain an exact sequence (similar to \eqref{eq:proof_fpx_exact}) on which we use the operator $\nabla_{q, \stilde}$ in \eqref{eq:nablaq_stilde_M} and the Koszul complex defined in \eqref{eq:stilde_qderham_M} to obtain an exact sequence of Koszul complexes (similar to \eqref{eq:proof_koszul_exact}).
	Considering the associated long exact sequence and noting that $H^1(K_{\MLambdao}(\nabla_{q, s})) = 0$, from Proposition \ref{prop:lambdar0_comp}, we obtain the following exact sequence (similar to \eqref{eq:proof_horizontal_exact}):
	\begin{equation*}
		0 \longrightarrow \MLambdao^{\nabla_{q, s}=0} \xrightarrow{\hspace{1mm} p_1(\mu)^n \hspace{1mm}} (\MLambdatildeo/p_1(\mu)^{n+1})^{\nabla_{q, \stilde}=0} \longrightarrow (\MLambdatildeo/p_1(\mu)^{n})^{\nabla_{q, \stilde}=0} \longrightarrow 0.
	\end{equation*}
	As the action of $1 \times \Gamma_0$ is continuous on $\MLambdatildeo$ for the $(p, p_1(\mu))\adic$ topology, therefore from Lemma \ref{lem:cont_coh_disc}, we obtain that $(\MLambdatildeo/p_1(\mu)^{n+1})^{\nabla_{q, \stilde}=0} = (\MLambdatildeo/p_1(\mu)^{n+1})^{1 \times \Gamma_0}$, for each $n \in \NN$.
	Hence, from the preceding exact sequence we obtain that the sequence in \eqref{eq:mlambdatilde0_gamma0_exact} is exact, and therefore, the sequence in \eqref{eq:n1_gamma0_exact} is exact as well.
\end{proof}

\subsubsection{The case \texorpdfstring{$p=2$}{-}}\label{subsubsec:p=2_gammaF_action_n1}

In this section, our goal is to prove a statement similar to Lemma \ref{lem:n1_gamma0_exact}, for $p=2$.
From \eqref{eq:gammaf_es}, recall that we have $\Gamma_0 \isomorphic 1 + 4\ZZ_2$ and $\Gamma_{\textrm{tor}} \isomorphic \{\pm 1\}$ as groups.
Let us first look at the action of $\Gamma_{\textrm{tor}}$ on $\MLambdatilde$.
Let $\sigma$ denote a generator of $\Gamma_{\textrm{tor}}$.
Then, from \eqref{eq:plus_minus_decomp}, recall that by setting $\MLambdatildeplus \coloneq \{x \in \MLambdatilde \textrm{ such that } \sigma(x) = x\}$ and $\MLambdatildeminus \coloneq \{x \in \MLambdatilde \textrm{ such that } \sigma(x) = -x\}$, we have a natural injective map of $\LambdatildeRplus\modules$
\begin{equation}\label{eq:mlambdatilde_pm}
	\MLambdatildeplus \oplus \MLambdatildeminus \longrightarrow \MLambdatilde,
\end{equation}
given as $(x, y) \mapsto x+y$.
Note that the action of $1 \times \Gamma_F$ is continuous for the $(p, p_1(\mu))\adic$ topology on $\MLambdatilde$, so it follows that $\MLambdatildeplus$ is a $(p, p_1(\mu))\adically$ complete $\LambdatildeRplus\textrm{-submodule}$ of $\MLambdatilde$, stable under the action of $(\varphi, \Gamma_R \times \Gamma_F)$ on $\MLambdatilde$, and similarly, $\MLambdatildeminus$ is a $\LambdatildeRplus\textrm{-submodule}$ of $\MLambdatilde$, stable under the action of $(\varphi, \Gamma_R \times \Gamma_F)$.
Equipping $\MLambdatildeplus$ and $\MLambdatildeminus$ with induced structures, we see that \eqref{eq:mlambdatilde_pm} is $(\varphi, \Gamma_R \times \Gamma_F)\equivariant$.
Furthermore, from \eqref{eq:mlambdatilde_mod_p1mun}, recall that $\MLambdatilde/p_1(\mu) \isomorphic (N(1)/p_1(\mu))^{1 \times \Gamma_R'} = \MLambda$, and similar to above, in Section \ref{subsubsec:a1_galact_p=2} we defined $\Lambda_{R,+}\modules$ $\MLambdaplus$ and $\MLambdaminus$ and showed that their natural inclusions in $\MLambda$ induce a natural $(\varphi, \Gamma_F)\equivariant$ isomorphism of $\LambdaRplus\modules$ $\MLambdaplus \oplus \MLambdaminus \isomorphic \MLambdaplus$ (see \eqref{eq:mlambda_pm} in Lemma \ref{prop:mlambda_pm_descent}).

\begin{lem}\label{lem:mlambdatildeplus_mod_p1mun}
	For each $n \geqslant 1$, the reduction modulo $p_1(\mu)^n$ of \eqref{eq:mlambdatilde_pm} induces a natural $(\varphi, \Gamma_R \times \Gamma_0)\equivariant$ isomorphism
	\begin{equation}\label{eq:mlambdatildeplus_mod_p1mun}
		\MLambdatildeplus/p_1(\mu)^n \isomorphic (\MLambdatilde/p_1(\mu)^n)^{1 \times \Gamma_{\textup{tor}}}.
	\end{equation}
	Moreover, for $n = 1$, the $(\varphi, \Gamma_F)\equivariant$ isomorphism $\MLambdatilde/p_1(\mu) \isomorphic \MLambda$ from \eqref{eq:mlambdatilde_mod_p1mun} induces a natural $(\varphi, \Gamma_F)\equivariant$ isomorphism
	\begin{equation}\label{eq:mlambdatildeplus_mod_p1mu}
		\MLambdatildeplus/p_1(\mu) \isomorphic (\MLambdatilde/p_1(\mu))^{1 \times \Gamma_{\textup{tor}}} \isomorphic \MLambdaplus.
	\end{equation}
\end{lem}
\begin{proof}
	Let us consider the following natural $(\varphi, \Gamma_R \times \Gamma_0)\equivariant$ commutative diagram:
	\begin{equation}\label{eq:mlambdatildeplus_mod}
		\begin{tikzcd}
			0 \arrow[r] & \MLambdatildeplus/p_1(\mu) \arrow[r, "p_1(\mu)^n"] \arrow[d, "\eqref{eq:mlambdatilde_pm}"] & \MLambdatildeplus/p_1(\mu)^{n+1} \arrow[r] \arrow[d, "\eqref{eq:mlambdatilde_pm}"] & \MLambdatildeplus/p_1(\mu)^n \arrow[r] \arrow[d, "\eqref{eq:mlambdatilde_pm}"] & 0\\
			0 \arrow[r] & (\MLambdatilde/p_1(\mu))^{1 \times \Gamma_{\textrm{tor}}} \arrow[r, "p_1(\mu)^n"] & (\MLambdatilde/p_1(\mu)^{n+1})^{1 \times \Gamma_{\textrm{tor}}} \arrow[r] & (\MLambdatilde/p_1(\mu)^n)^{1 \times \Gamma_{\textrm{tor}}} \arrow[r] & 0,
		\end{tikzcd}
	\end{equation}
	where the top row is exact, and the vertical maps are injective because we have $p_1(\mu)^n \MLambdatilde \cap \MLambdatildeplus = p_1(\mu)^n \MLambdatildeplus$, as $p_1(\mu)$ is invariant under the action of $1 \times \Gamma_F$.
	Since the action of $1 \times \Gamma_{\textrm{tor}} \subset 1 \times \Gamma_F$ is trivial on $\MLambdatilde/p_2(\mu)$ from Lemma \ref{lem:mlambdatilde_triv_modmu}, therefore, for each $n \geqslant 1$ and $x$ in $\MLambdatilde/p_1(\mu)^n$ we see that $(\sigma-1)x$ is an element of $p_2(\mu)(\MLambdatilde/p_1(\mu)^n)$.
	Then, from Lemma \ref{lem:sigma_action_surjective} it follows that the bottom right horizontal arrow in \eqref{eq:mlambdatildeplus_mod} is surjective, in particular, the bottom row is exact.

	Next, by composing the left vertical arrow in \eqref{eq:mlambdatildeplus_mod} with the $(\varphi, \Gamma_F)\equivariant$ isomorphism $\MLambdatilde/p_1(\mu) \isomorphic \MLambda$ from \eqref{eq:mlambdatilde_mod_p1mun}, we obtain a natural $(\varphi, \Gamma_F)\equivariant$ injective map $\MLambdatildeplus/p_1(\mu) \rightarrow \MLambdaplus$, i.e.\ \eqref{eq:mlambdatildeplus_mod_p1mu}, and we shall show that it is surjective as well.
	Indeed, since the action of $1 \times \Gamma_F$ is trivial on $\MLambdatilde/p_2(\mu)$ from Lemma \ref{lem:mlambdatilde_triv_modmu} and the action of $\Gamma_F$ is trivial on $\MLambda/\mu\MLambda$ using \eqref{eq:mlambda_arpdn} in Proposition \ref{prop:abarone_comp}, therefore, by using Lemma \ref{lem:sigma_action_surjective} it follows that the $\LambdatildeR\linear$ and $(\varphi, 1 \times \Gamma_F)\equivariant$ surjective map $\MLambdatilde \twoheadrightarrow \MLambda$ from \eqref{eq:mlambdatilde_mod_p1mun}, induces a $\LambdatildeRplus\linear$ and $(\varphi, 1 \times \Gamma_0)\equivariant$ surjective map $\MLambdatildeplus \twoheadrightarrow \MLambdaplus$, which factors through \eqref{eq:mlambdatildeplus_mod_p1mu}.
	In particular, we get that the composition in \eqref{eq:mlambdatildeplus_mod_p1mu} is bijective, therefore, the left vertical arrow in \eqref{eq:mlambdatildeplus_mod} is also bijective.
	Now, using the diagram \eqref{eq:mlambdatildeplus_mod}, an easy induction on $n \geqslant 1$ shows that the right vertical arrow is also bijective.
	Hence, it follows that the natural $(\varphi, \Gamma_R \times \Gamma_0)\equivariant$ map $\MLambdatildeplus/p_1(\mu)^n \rightarrow (\MLambdatilde/p_1(\mu)^n)^{1 \times \Gamma_{\textrm{tor}}}$, induced by \eqref{eq:mlambdatilde_pm}, is bijective for each $n \geqslant 1$, thus allowing us to conclude.
\end{proof}

From Lemma \ref{lem:mlambdatildeplus_mod_p1mun}, we obtain the following:
\begin{lem}\label{lem:n1_pm_exact}
	For each $n \geqslant 1$, the following natural $(\varphi, \Gamma_R \times \Gamma_0)\equivariant$ sequence is exact:
	\begin{equation}\label{eq:n1_pm_exact}
		0 \longrightarrow (\MLambdatilde/p_1(\mu))^{1 \times \Gamma_{\textup{tor}}} \xrightarrow{\hspace{1mm} p_1(\mu)^n \hspace{1mm}} (\MLambdatilde/p_1(\mu)^{n+1})^{1 \times \Gamma_{\textup{tor}}} \longrightarrow (\MLambdatilde/p_1(\mu)^n)^{1 \times \Gamma_{\textup{tor}}} \longrightarrow 0.
	\end{equation}
\end{lem}
\begin{proof}
	The sequence \eqref{eq:n1_pm_exact} is the same as the second row of the diagram \eqref{eq:mlambdatildeplus_mod_p1mun}, which was shown to be exact in the proof of Lemma \ref{lem:mlambdatildeplus_mod_p1mun}.
\end{proof}

Next, we will look at the action of $1 \times \Gamma_0 \isomorphic 1 \times (1 + 4\ZZ_2)$ on $\MLambdatildeplus$ and show the following:
\begin{lem}\label{lem:n1_gamma0_exact_2}
	For each $n \geqslant 1$, the following natural $(\varphi, \Gamma_R \times 1)\equivariant$ sequence is exact:
	\begin{equation}\label{eq:n1_gamma0_exact_2}
		0 \longrightarrow (\MLambdatilde/p_1(\mu))^{1 \times \Gamma_F} \xrightarrow{\hspace{1mm} p_1(\mu)^n \hspace{1mm}} (\MLambdatilde/p_1(\mu)^{n+1})^{1 \times \Gamma_F} \longrightarrow (\MLambdatilde/p_1(\mu)^{n})^{1 \times \Gamma_F} \longrightarrow 0.
	\end{equation}
\end{lem}

For each $n \geqslant 1$, note that the reduction modulo $p_1(\mu)^n$ of the $(\varphi, \Gamma_R \times \Gamma_F)\equivariant$ isomorphism \eqref{eq:lambdatildeR_comp} from Lemma \ref{lem:mlambdatildeplus_mod_p1mun}, and taking the $(1 \times \Gamma_{\textup{tor}})\textrm{-invariants}$ of the target yields a $(\varphi, \Gamma_R \times \Gamma_0)\equivariant$ isomorphism $\MLambdatildeplus/p_1(\mu)^n \isomorphic (\MLambdatilde/p_1(\mu)^n)^{1 \times \Gamma_{\textup{tor}}}$, because $p_1(\mu)$ is invariant under the action of $1 \times \Gamma_F$.
Consequently, the sequence in \eqref{eq:n1_gamma0_exact_2} may be rewritten as the following $(\varphi, \Gamma_R \times \Gamma_0)\equivariant$ sequence:
\begin{equation}\label{eq:mlambdatilde+_gamma0_exact}
	0 \longrightarrow \MLambdaplus^{\Gamma_0} \xrightarrow{\hspace{1mm} p_1(\mu)^n \hspace{1mm}} (\MLambdatildeplus/p_1(\mu)^{n+1})^{1 \times \Gamma_0} \longrightarrow (\MLambdatildeplus/p_1(\mu)^{n})^{1 \times \Gamma_0} \longrightarrow 0.
\end{equation}
In order to prove that \eqref{eq:mlambdatilde+_gamma0_exact} is exact, let us now look at the action of $\Gamma_R \times \Gamma_F$ on $\MLambdatilde$ and $\MLambdatildeplus$, respectively.
From \eqref{eq:nu_defi} recall that we set $\nu = \tfrac{\mu^2}{1+\mu}$ in $\AF$, and now we make the following observation:
\begin{lem}\label{lem:mlambdatilde+_triv_modnu}
	The action of $1 \times \Gamma_0$ is trivial on $\MLambdatildeplus/p_2(\nu)$ and the action of $\Gamma_R \times 1$ is trivial on $\MLambdatildeplus/p_1(\mu)$.
\end{lem}
\begin{proof}
	For the first claim, note that $\nu$ is invariant under the action of $\Gamma_{\textrm{tor}}$ and we have that $p_2(\nu)\MLambdatilde \cap \MLambdatildeplus = p_2(\nu)\MLambdatildeplus$.
	So, if $x$ is an element of $\MLambdatildeplus$ and $g$ any element of $1 \times \Gamma_0$, then it is enough to show that $(g-1)x$ is an element of $p_2(\nu)\MLambdatilde$.
	Moreover, note that $\nu$ and $\mu^2$ differ by a unit in $\AF$.
	Therefore, we are reduced to showing that $(g-1)x$ is an element of $p_2(\mu)^2\MLambdatilde$.
	Now, using Lemma \ref{lem:mlambdatilde_triv_modmu} we may write $(g-1)x = p_2(\mu)y$, for some $y$ in $\MLambdatilde$.
	Let $\sigma$ be a generator of $1 \times \Gamma_{\textrm{tor}}$ and note that $\sigma(x) = x$.
	Then, we have that $\sigma(p_2(\mu)) \sigma(y) = p_2(\mu) y$, in particular, $(\sigma-1)y = -(2+p_2(\mu))y$.
	Again, using Lemma \ref{lem:mlambdatilde_triv_modmu} we may write $-p_2([p]_q)y = (\sigma-1)y = p_2(\mu)z$, for some $z$ in $\MLambdatilde$.
	So, we get that $-py = 0 \textrm{ mod }p_2(\mu) \MLambdatilde$.
	Note that $(p_2(\mu), p)$ is a regular sequence on $\MLambdatilde$ (see Lemma \ref{lem:mlambdatilde_p2mup_reg}).
	Therefore, we conclude that $y = 0 \textrm{ mod }p_2(\mu) \LambdatildeR$, i.e.\ $y$ is an element of $p_2(\mu)\LambdatildeR$ and $(g-1)x = p_2(\mu)y$ is an element of $p_2(\mu)^2\LambdatildeR$, as claimed.
	The second claim easily follows from Lemma \ref{lem:lambdatildeR_triv_modmu}.
\end{proof}

\begin{rem}\label{rem:gamma0_act_mlambda+}
	From Lemma \ref{lem:mlambdatilde+_triv_modnu}, note that the action of $1 \times \Gamma_0$ is trivial on $\MLambdatildeplus/p_2(\nu)$ and multiplication by $p_1(\mu)$ on $\MLambdatildeplus$ is equivariant for this action.
	Therefore, it follows that for any $g$ in $1 \times \Gamma_0$ and any $x$ in $\MLambdatildeplus/p_1(\mu)^n$, we have that $(g-1)x$ is an element of $p_2(\nu)\MLambdatildeplus/p_1(\mu)^n$.
	In particular, for $n=1$, using the isomorphism $\MLambdatildeplus/p_1(\mu) \isomorphic \MLambdaplus$ from \eqref{eq:mlambdatildeplus_mod_p1mu}, we get that for any $g$ in $\Gamma_0$ and any $x$ in $\MLambdaplus$, the element $(g-1)x$ is in $\nu\MLambdaplus$.
\end{rem}

Now, using the action of $1 \times \Gamma_0$ on $\MLambdatildeplus$, let us define a $\qconnection$ (see Definition \ref{defi:qconnection}).
Recall that in Section \ref{subsubsec:p=2_gammaF_action}, to define a $q\textrm{-de Rham}$ complex over $\LambdatildeR$, we fixed the following element in $\LambdatildeRplus$ as a parameter (see \eqref{eq:tautilde_defi}):
\begin{equation*}
	\tautilde = \tfrac{1}{p_2([p]_q)} \delta\big(\tfrac{p_2([p]_q)}{p_1([p]_q)}.
\end{equation*}
Moreover, if $\gamma_0$ is any element of $1 \times \Gamma_0$, then from Lemma \ref{lem:gamma0_act_tautilde} we have that $(\gamma_0-1)\tautilde = u \hspace{0.5mm} p_2(\nu)$, for some unit $u$ in $\LambdatildeRplus$ depending on $\gamma_0$.
So, let us fix the choice of a topological generator $\gamma_0$ of $1 \times \Gamma_0$ such that $\chi(\gamma_0) = 1+4a$, for a unit $a$ in $\ZZ_2$, and consider the following operator on $\MLambdatildeplus$:
\begin{equation}\label{eq:nablaq_tautilde_M}
	\begin{aligned}
		\nabla_{q, \tautilde} \colon \MLambdatildeplus &\longrightarrow \MLambdatildeplus\\
					x &\longmapsto \tfrac{(\gamma_0-1)x}{(\gamma_0-1)\tautilde}.
	\end{aligned}
\end{equation}
From the triviality of the action of $1 \times \Gamma_0$ on $\MLambdatildeplus/p_2(\nu)$ (see Lemma \ref{lem:mlambdatilde+_triv_modnu}) and from Lemma \ref{lem:gamma0_act_tautilde}, it follows that the operator $\nabla_{q, \tautilde}$ is well defined.
For each $n \geqslant 1$, using Remark \ref{rem:gamma0_act_mlambda+}, the operator in \eqref{eq:nablaq_tautilde_M} induces well-defined operators $\nabla_{q, \tautilde} \colon \MLambdatildeplus/p_1(\mu)^n \longrightarrow \MLambdatildeplus/p_1(\mu)^n$.
As the operator $\nabla_{q, \tautilde}$ is an endomorphism of $\MLambdatildeplus/p_1(\mu)^n$, we may define the following two term Koszul complex:
\begin{equation}\label{eq:tautilde_qderham_M}
	K_{\MLambdatildeplus/p_1(\mu)^n}(\nabla_{q, \tautilde}) \colon \big[\MLambdatildeplus/p_1(\mu)^n \xrightarrow{\hspace{1mm} \nabla_{q, \tautilde} \hspace{1mm}} \MLambdatildeplus/p_1(\mu)^n\big].
\end{equation}
In particular, for $n = 1$ we set $\tau \coloneq \nu/8$ in $\LambdaRplus$, and using Remark \ref{rem:gamma0_act_mlambda+} and the fact that $(\gamma_0-1)s = v\nu$, for some unit $v$ in $\LambdaFplus$ depending on $\gamma_0$ (see Lemma \ref{lem:gamma0_act_tau}), we have a well-defined operator
\begin{equation*}
	\begin{aligned}
		\nabla_{q, \tau} \colon \MLambdaplus &\longrightarrow \MLambdaplus\\
				x &\longmapsto \tfrac{(\gamma_0-1)x}{(\gamma_0-1)\tau}.
	\end{aligned}
\end{equation*}
Note that the operator above coincides with the operator defined in \eqref{eq:nablaq_tau_mlambda+}, and the complex from \eqref{eq:tautilde_qderham_M} for $n=1$ coincides with the complex from \eqref{eq:tau_qderham_mlambda+}.
Therefore, from Proposition \ref{prop:lambdar+_comp}, we have that the cohomology of the Koszul complex $K_{\MLambdaplus}(\nabla_{q, s})$ vanishes in degree 1, i.e.\ $H^1(K_{\MLambdaplus}(\nabla_{q, \tau})) = 0$.

\begin{rem}
	Considering $\tautilde$ as a parameter, similar to Remark \ref{rem:nablaq_tautilde_qconnection}, the operator $\nabla_{q, \tautilde}$ in \eqref{eq:nablaq_tautilde_M} can be considered as a $\qconnection$ in non-logarithmic coordinates, in the sense of Definition \ref{defi:qconnection} and Remark \ref{rem:nablaqi_nota}.
	Then, \eqref{eq:tau_qderham_mlambda+} is the $q\textrm{-de Rham}$ complex arising from such a $\qconnection$.
	Similarly, considering $\tau$ as a parameter, the operator $\nabla_{q, \tau}$ on $\MLambdaplus$ can also be considered as a $\qconnection$ in non-logarithmic coordinates, in the sense of Definition \ref{defi:qconnection} and Remark \ref{rem:nablaqi_nota}.
\end{rem}

\begin{proof}[Proof of Lemma \ref{lem:n1_gamma0_exact_2}]
	The proof follows by the same arguments as in the proof of Lemma \ref{lem:a1_gamma0_exact}.
	Indeed, by using the $(\varphi, \Gamma_R \times \Gamma_0)\equivariant$ isomorphism $\MLambdaplus/p_1(\mu)^n \isomorphic (\MLambda/p_1(\mu)^n)^{1 \times \Gamma_{\textrm{tor}}}$ (see the discussion before \eqref{eq:mlambdatilde+_gamma0_exact}) in the exact sequence of \eqref{eq:n1_pm_exact}, we obtain an exact sequence (similar to \eqref{eq:proof_fpx_exact}) on which we use the operator $\nabla_{q, \tautilde}$ in \eqref{eq:nablaq_tautilde_M} and the Koszul complex defined in \eqref{eq:tautilde_qderham_M} to obtain an exact sequence of Koszul complexes (similar to \eqref{eq:proof_koszul_exact}).
	Considering the associated long exact sequence, and noting that $H^1(K_{\MLambdaplus}(\nabla_{q, \tau})) = 0$ from Proposition \ref{prop:lambdar+_comp}, we obtain the following exact sequence (similar to \eqref{eq:proof_horizontal_exact}):
	\begin{equation*}
		0 \longrightarrow \MLambdaplus^{\nabla_{q, \tau}=0} \xrightarrow{\hspace{1mm} p_1(\mu)^n \hspace{1mm}} (\MLambdatildeplus/p_1(\mu)^{n+1})^{\nabla_{q, \tautilde}=0} \longrightarrow (\MLambdatildeplus/p_1(\mu)^{n})^{\nabla_{q, \tautilde}=0} \longrightarrow 0.
	\end{equation*}
	Since the action of $1 \times \Gamma_0$ is continuous on $\MLambdatildeplus$ for the $(p, p_1(\mu))\adic$ topology, therefore from Lemma \ref{lem:cont_coh_disc}, we get that $(\MLambdatildeplus/p_1(\mu)^{n+1})^{\nabla_{q, \tautilde}=0} = (\MLambdatildeplus/p_1(\mu)^{n+1})^{1 \times \Gamma_0}$, for each $n \in \NN$.
	Hence, from the preceding exact sequence we obtain that the sequence in \eqref{eq:mlambdatilde+_gamma0_exact} is exact, and therefore, the sequence in \eqref{eq:n1_gamma0_exact_2} is also exact.
\end{proof}

\subsubsection{Stratification on Wach modules}\label{subsubsec:wachmod_strat}

In this section, we shall construct the promised stratification on a Wach module over $\AR$ and prove Proposition \ref{prop:strat_from_wach_relative} stated below.
Let $N$ be a Wach module over $\AR$ and recall that $N(1) = \AR(1) \otimes_{p_2, \AR} N$ is equipped with a $(\varphi, \Gamma_R^2)\action$ (see the discussion before Section \ref{subsubsec:geo_gamma_action_n1}).
Moreover, from the discussion before Proposition \ref{prop:deltan_modmu_relative}, recall that we have a $\varphi\equivariant$ homomorphism $\Delta_N \colon N(1) \rightarrow N$ induced by tensoring $\Delta \colon \AR(1) \rightarrow \AR$ with $N$.
After reducing $\Delta_N$ modulo $p_1(\mu)^n$, we claim the following:
\begin{prop}\label{prop:deltan_modmun_relative}
	Let $n \in \NN_{\geqslant 1}$, then $\Delta_N$ modulo $p_1(\mu)^n$ restricts to a $(\varphi, \Gamma_R \times 1)\equivariant$ isomorphism of $\AR/(\mu)^n\modules$
	\begin{equation*}
		\Delta_N \colon (N(1)/p_1(\mu)^n)^{1 \times \Gamma_R} \isomorphic N/\mu^n N,
	\end{equation*}
	where the $\AR\module$ structure on the source is defined via the map $p_1 \colon \AR \rightarrow \AR(1)$.
\end{prop}
\begin{proof}
	Let us first consider the following $(\varphi, \Gamma_R \times \Gamma_R)\equivariant$ short exact sequence of $\AR(1)\modules$:
	\begin{equation*}
		0 \longrightarrow N(1)/p_1(\mu) \xrightarrow{\hspace{1mm} p_1(\mu)^n \hspace{1mm}} N(1)/p_1(\mu)^{n+1} \longrightarrow N(1)/p_1(\mu)^n \longrightarrow 0.
	\end{equation*}
	Then, for $p \geqslant 3$, using Lemma \ref{lem:n1_geometric_exact}, Lemma \ref{lem:n1_fpx_exact} and Lemma \ref{lem:n1_gamma0_exact}, and for $p=2$, using Lemma \ref{lem:n1_geometric_exact}, Lemma \ref{lem:n1_pm_exact} and Lemma \ref{lem:n1_gamma0_exact_2}, it follows that the following $(\varphi, \Gamma_R \times 1)\equivariant$ sequence is exact:
	\begin{equation}\label{eq:n1_gammaR_exact}
		0 \longrightarrow (N(1)/p_1(\mu))^{1 \times \Gamma_R} \xrightarrow{\hspace{1mm} p_1(\mu)^n \hspace{1mm}} (N(1)/p_1(\mu)^{n+1})^{1 \times \Gamma_R} \longrightarrow (N(1)/p_1(\mu)^n)^{1 \times \Gamma_R} \longrightarrow 0.
	\end{equation}
	Now, consider the following $\varphi\equivariant$ commutative diagram with exact rows:
	\begin{center}
		\begin{tikzcd}
			0 \arrow[r] & (N(1)/p_1(\mu))^{1 \times \Gamma_R} \arrow[r, "p_1(\mu)^n"] \arrow[d, "\wr"', "\Delta_N"] & (N(1)/p_1(\mu)^{n+1})^{1 \times \Gamma_R} \arrow[r] \arrow[d, "\Delta_N"] & (N(1)/p_1(\mu)^n)^{1 \times \Gamma_R} \arrow[d, "\Delta_N"] \arrow[r] & 0\\
			0 \arrow[r] & N/\mu N \arrow[r, "\mu^n"] & N/\mu^{n+1} N \arrow[r] & N/\mu^n N \arrow[r] & 0,
		\end{tikzcd}
	\end{center}
	where the bottom row is exact and the top row is the exact sequence in \eqref{eq:n1_gammaR_exact}.
	For $n=1$, from Theorem \ref{thm:integral_comp_relative} (see Proposition \ref{prop:deltan_modmu_relative}), recall that we have the isomorphism $\Delta_N \colon (N(1)/p_1(\mu))^{1 \times \Gamma_R} \isomorphic N/\mu N$.
	Then, by using the diagram, an easy induction on $n \geqslant 1$, gives the following $\varphi\equivariant$ isomorphism:
	\begin{equation*}
		\Delta_N \colon (\AR(1)/p_1(\mu)^{n+1} \otimes_{p_2, \AR} N)^{1 \times \Gamma_R} \isomorphic N/\mu^{n+1} N.
	\end{equation*}
	Finally, we need to check the $(\Gamma_R \times 1)\textrm{-equivariance}$ of $\Delta_N$, the proof of which is similar to that of \cite[Lemma 3.19]{morrow-tsuji}.
	From Lemma \ref{lem:delta_modmun_phigamma} recall that for any $g$ in $\Gamma_R$ and $a$ in $\AR(1)/(p_1(\mu))^{n+1}$, we have that $\Delta((g, g)a) = g(a)$ which implies that $\Delta_N((g, g)(x)) = g(\Delta_N(x))$, for any $g$ in $\Gamma_R$ and $x$ in $N(1)/p_1(\mu)^{n+1}$.
	So, if $x$ is $(1 \times \Gamma_R)\textrm{-invariant}$, then for any $g_1$, $g_2$ in $\Gamma_R$, we have that $\Delta_N((g_1, g_2)x) = \Delta_N((g_1, g_1)x) = g_1(\Delta_N(x))$.
	This concludes our proof.
\end{proof}

An immediate consequence of Proposition \ref{prop:deltan_modmun_relative} is the following:
\begin{prop}\label{prop:diagonal_m1_relative}
	Let $N$ be a Wach module over $\AR$.
	Then, the natural $(\varphi, \Gamma_R)\equivariant$ homomorphism $\Delta_N \colon \AR(1) \otimes_{p_2, \AR} N \rightarrow N$ induced by $\Delta \colon \AR(1) \rightarrow \AR$, restricts to a $(\varphi, \Gamma_R \times 1)\equivariant$ isomorphism of $\AR\modules$
	\begin{equation*}
		\Delta_N \colon (\AR(1) \otimes_{p_2, \AR} N)^{1 \times \Gamma_R} \isomorphic N,
	\end{equation*}
	where the $\AR\module$ structure on the source is defined via the map $p_1 \colon \AR \rightarrow \AR(1)$.
\end{prop}
\begin{proof}
	Note that by construction $\Delta_N$ is $(\varphi, \Gamma_R \times 1)\equivariant$.
	Moreover, $\AR$ is $\mu\adically$ complete and $N$ is a finite $\AR\module$, in particular, we have that $N = \lim_n N/\mu^n N$.
	Furthermore, we have that $p_2 \colon \AR \rightarrow \AR(1)$ is faithfully flat and $\AR(1)$ is $p_1(\mu)\adically$ complete, so we get that $\AR(1) \otimes_{p_2, \AR} N = \lim_n \big(\AR(1)/p_1(\mu)^n \otimes_{p_2, \AR} N\big)$.
	Now, recall that inverse limit commutes with taking $(1 \times \Gamma_R)\textrm{-invariants}$, so we get that
	\begin{align*}
		(\AR(1) \otimes_{p_2, \AR} N)^{1 \times \Gamma_R} &= \big(\lim_n \AR(1)/p_1(\mu)^n \otimes_{p_2, \AR} N\big)^{1 \times \Gamma_R}\\
		&= \lim_n \big(\AR(1)/p_1(\mu)^n \otimes_{p_2, \AR} N\big)^{1 \times \Gamma_R} \xrightarrow[\Delta]{\hspace{1mm} \sim \hspace{1mm}} \lim_n N/\mu^n N = N,
	\end{align*}
	where the isomorphism follows from Proposition \ref{prop:deltan_modmun_relative}.
	This concludes our proof.
\end{proof}

Using Proposition \ref{prop:diagonal_m1_relative} we may define a natural stratification on a Wach module over $\AR$ as follows:
\begin{prop}\label{prop:strat_from_wach_relative}
	Let $N$ be a Wach module over $\AR$ and let $\varepsilon \colon \AR(1) \otimes_{p_1, \AR} N \rightarrow \AR(1) \otimes_{p_2, \AR} N$ be the $\AR(1)\linear$ homomorphism induced by the inverse of the isomorphism in Proposition \ref{prop:diagonal_m1_relative}.
	Then,
	\begin{enumarabicup}
	\item The homomorphism $\varepsilon$ is a stratification on $N$ with respect to $\AR(\bullet)$.

	\item The action of $\Gamma_R$ on $\eval_{\AR}^{\Strat}(N, \varepsilon)$, whose underlying $\AR\module$ is $N$, coincides with the original action of $\Gamma_R$ on $N$, i.e.\ $\eval_{\AR}^{\Strat}(N, \varepsilon) \isomorphic N$ as Wach modules over $\AR$.
	\end{enumarabicup}
\end{prop}

Let us first note that the construction of the stratification $\varepsilon$ in Proposition \ref{prop:strat_from_wach_relative} is functorial in $N$.
In particular, we have a well-defined natural functor
\begin{equation}\label{eq:strat_from_wach_relative}
	\Strat_{\AR(\bullet)} \colon (\varphi, \Gamma_R)\Mod_{\AR}^{[p]_q} \longrightarrow \Strat^{\an, \varphi}(\AR(\bullet)).
\end{equation}
Moreover, from the statement of Proposition \ref{prop:strat_from_wach_relative} it is clear that the functor $\Strat_{\AR(\bullet)}$ in \eqref{eq:strat_from_wach_relative} is a quasi-inverse to the functor $\eval_{\AR}^{\Strat}$ in \eqref{eq:strat_eval_ar+}.
We note the following:
\begin{lem}\label{lem:strat_gamma_equiv_relative}
	The $\AR(1)\linear$ homomorphism $\varepsilon$ in Proposition \ref{prop:strat_from_wach_relative} is $\Gamma_R^2\equivariant$.
\end{lem}
\begin{proof}
	It is enough to show that $\varepsilon(1 \otimes g_1(y)) = (g_1, g_2)(\varepsilon(1 \otimes y))$, for all $y$ in $N$ and $g_1$, $g_2$ in $\Gamma_R$.
	Similar to the last part of the proof of Proposition \ref{prop:deltan_modmun_relative}, we note that for any $x$ in $\AR(1) \otimes_{p_1, \AR} N$, we have that $\Delta_N((g_1, g_2)x) = g_1(\Delta_N(x))$.
	Then, by setting $y = \Delta_N(x)$, we get the claim.
\end{proof}

The goal of the rest of this section is to prove Proposition \ref{prop:strat_from_wach_relative}.
To that end, we need a result analogous to Proposition \ref{prop:diagonal_m1_relative} over $\AR(2)$.
Recall that we have $(\varphi, \Gamma_R^3)\equivariant$ maps $r_i \colon \AR \rightarrow \AR(2)$, for $i = 1, 2, 3$, where $\AR(2)$ is equipped with an action of $\Gamma_R^3$ as discussed before Construction \ref{const:strat_eval_ar+}, and $\AR$ is equipped with an action of $\Gamma_R^3$ via projection onto the $i^{\textrm{th}}\textrm{-coordinate}$.
Similarly, we have natural $(\varphi, \Gamma_R^3)\equivariant$ maps $p_{ij} \colon \AR(1) \rightarrow \AR(2)$ for $(i, j) \in \{(1, 2), (2, 3), (1, 3)\}$, and where $\AR(1)$ is equipped with an action of $\Gamma_R^3$ via projection onto the $(i, j)^{\textrm{th}}\textrm{-coordinate}$.

Now, let $N$ be a Wach module over $\AR$ as above and consider the $\AR(2)\module$ $\AR(2) \otimes_{r_3, \AR} N$ equipped with the tensor product Frobenius and the tensor product action of $\Gamma_R^3$, where $\Gamma_R^3$ acts on $N$ via projection onto the third coordinate.
Note that the multiplication map $\Delta \colon \AR(2) \rightarrow \AR$ is $(\varphi, \Gamma_R \times 1 \times 1)\equivariant$, where $\Gamma_R \times 1 \times 1$ acts on $\AR$ via projection onto the first coordinate.
The multiplication map $\Delta$ induces a $(\varphi, \Gamma_R \times 1 \times 1)\equivariant$ map $\Delta_N \colon \AR(2) \otimes_{r_3, \AR} N \rightarrow N$, where $\Gamma_R \times 1 \times 1$ acts on $N$ (in the target) via projection onto the first coordinate.
\begin{prop}\label{prop:diagonal_m2_relative}
	The homomorphism $\Delta_N \colon \AR(2) \otimes_{r_3, \AR} N \rightarrow N$ induced by $\Delta \colon \AR(2) \rightarrow \AR$ restricts to an injective map
	\begin{equation*}
		\Delta_N \colon (\AR(2) \otimes_{r_3, \AR} N)^{1 \times \Gamma_R \times \Gamma_R} \longrightarrow N.
	\end{equation*}
\end{prop}
\begin{proof}
	Recall that inverse limit commutes with taking $(1 \times \Gamma_R \times \Gamma_R)\textrm{-invariants}$.
	Moreover, $\AR$ is $\mu\adically$ complete, $N$ is a finite $\AR\module$, $\AR(2)$ is $p_1(\mu)\adically$ complete and $r_i \colon \AR \rightarrow \AR(2)$ is faithfully flat, for $i = 1, 2, 3$ (see Lemma \ref{lem:pi_rj_ff}).
	Therefore, we have that
	\begin{align*}
		(\AR(2) \otimes_{r_3, \AR} N)^{1 \times \Gamma_R \times \Gamma_R} &= (\lim_n \AR(2)/(p_1(\mu))^n \otimes_{r_3, \AR} N)^{1 \times \Gamma_R \times \Gamma_R}\\
		&= \lim_n (\AR(2)/(p_1(\mu))^n \otimes_{r_3, \AR} N)^{1 \times \Gamma_R \times \Gamma_R}\\
		&\xrightarrow{\hspace{1mm} \Delta_N \hspace{1mm}} \lim_n N/\mu^n N = N.
	\end{align*}
	Since taking limits is left exact, to show the claim, it is enough to show that the following map is injective:
	\begin{equation*}
		(\AR(2)/(p_1(\mu))^n \otimes_{r_3, \AR} N)^{1 \times \Gamma_R \times \Gamma_R} \longrightarrow N/\mu^n N,
	\end{equation*}
	for each $n \in \NN_{\geqslant 1}$.

	We will show this by induction on $n$.
	For $n=1$, from Theorem \ref{thm:integral_comp_relative} (see Remark \ref{rem:p1p2_modmu_comp_relative}), recall that we have a $(\varphi, \Gamma_R^2)\equivariant$ isomorphism (see \eqref{eq:p1p2_modmu_comp_relative}),
	\begin{equation*}
		\AR(1)/(p_1(\mu)) \otimes_{p_1, \AR} N \isomorphic \AR(1)/(p_1(\mu)) \otimes_{p_2, \AR} N,
	\end{equation*}
	where $\Gamma_R^2$ acts on $N$ in the left hand term via projection onto the first coordinate and on $N$ in the right hand term via projection onto the second coordinate.
	Moreover, the composition $\AR \xrightarrow{\hspace{1mm} p_1 \hspace{1mm}} \AR(1) \xrightarrow{\hspace{1mm} p_{13} \hspace{1mm}} \AR(2)$ coincides with the composition $\AR  \xrightarrow{\hspace{1mm} r_1 \hspace{1mm}} \AR(2)$.
	Similarly, the composition $\AR \xrightarrow{\hspace{1mm} p_2 \hspace{1mm}} \AR(1) \xrightarrow{\hspace{1mm} p_{13} \hspace{1mm}} \AR(2)$ coincides with the composition $\AR  \xrightarrow{\hspace{1mm} r_3 \hspace{1mm}} \AR(2)$.
	So, by base changing the top horizontal isomorphism in \eqref{eq:p1p2_modmu_comp_relative} along $\AR(1) \xrightarrow{\hspace{1mm} p_{13} \hspace{1mm}} \AR(2)$, we obtain a $(\varphi, \Gamma_R^3)\equivariant$ isomorphism
	\begin{equation}\label{eq:r1r3_modmu_comp_relative}
		\AR(2)/p_1(\mu) \otimes_{r_1, \AR} N \isomorphic \AR(2)/p_1(\mu) \otimes_{r_3, \AR} N,
	\end{equation}
	where $\Gamma_R^3$ acts on $N$ in the source via projection onto the first coordinate and on $N$ in the target via projection onto the third coordinate.

	Now, consider the following diagram:
	\begin{center}
		\begin{tikzcd}
			(\AR(2)/p_1(\mu) \otimes_{r_1, \AR} N)^{1 \times \Gamma_R \times \Gamma_R} \arrow[r, "\sim"] & (\AR(2)/p_1(\mu) \otimes_{r_3, \AR} N)^{1 \times \Gamma_R \times \Gamma_R} \arrow[d, "\Delta_N"]\\
			N/\mu N \arrow[r, "\sim", "id"'] \arrow[u, "r_1"', "\wr"] & N/\mu N,
		\end{tikzcd}
	\end{center}
	where the top horizontal arrow is obtained as $(1 \times \Gamma_R \times \Gamma_R)\textrm{-invariant}$ of \eqref{eq:r1r3_modmu_comp_relative} and the left vertical arrow is the natural isomorphism from Lemma \ref{lem:r1modmun_relative} below.
	The commutativity of the diagram follows from \eqref{eq:p1p2_modmu_comp_relative} and the observation that the composition $\AR \xrightarrow{\hspace{1mm} r_1 \hspace{1mm}} \AR(2) \xrightarrow{\hspace{1mm} \Delta \hspace{1mm}} \AR$ is the identity.
	Therefore, we see that the right vertical arrow is bijective as well, i.e.\ we obtain a $(\varphi, \Gamma_R)\equivariant$ isomorphism
	\begin{equation*}
		(\AR(2)/p_1(\mu) \otimes_{r_3, \AR} N)^{1 \times \Gamma_R \times \Gamma_R} \isomorphic N/\mu N.
	\end{equation*}

	To prove our claim, let us now describe the induction step.
	From the discussion above, we see that the claim is true for $n=1$, so let $N(2) \coloneq \AR(2) \otimes_{r_3, \AR} N$ and assume that $\Delta_N \textmod p_1(\mu)^n$ is injective for some $n \in \NN_{\geqslant 1}$.
	Now, consider the following diagram with exact rows:
	\begin{center}
		\begin{tikzcd}
			0 \arrow[r] & (N(2)/p_1(\mu))^{1 \times \Gamma_R \times \Gamma_R} \arrow[r, "p_1(\mu)^n"] \arrow[d, "\wr"', "\Delta_N"] & (N(2)/p_1(\mu)^{n+1})^{1 \times \Gamma_R \times \Gamma_R} \arrow[r] \arrow[d, "\Delta_N"] & Q \arrow[r] \arrow[d, "\Delta_N"] & 0\\
			0 \arrow[r] & N/\mu N \arrow[r] & N/\mu^{n+1} N \arrow[r] & N/\mu^n N \arrow[r] & 0,
		\end{tikzcd}
	\end{center}
	where $Q$ is the cokernel of the top left horizontal arrow.
	It is easy to see that we have an injective map $Q \rightarrow (N(2)/p_1(\mu)^n)^{1 \times \Gamma_R \times \Gamma_R} \xrightarrow{\hspace{1mm} \Delta_N \hspace{1mm}} N/\mu^n N$, where the injectivity of the first map is obtained by considering the long exact sequence for the $(1 \times \Gamma_R \times \Gamma_R)\textrm{-cohomology}$ of the short exact sequence $0 \rightarrow N(2)/p_1(\mu) \xrightarrow{p_1(\mu)^n} N(2)/p_1(\mu)^{n+1} \rightarrow N(2)/p_1(\mu)^n$, and the second map is injective by the induction assumption.
	Therefore, it follows that the middle vertical arrow in the diagram above is injective as well.
	This proves the claim.
\end{proof}

The following observation was used above:
\begin{lem}\label{lem:r1modmun_relative}
	Extending scalars along $r_1 \colon \AR \rightarrow \AR(1)$, gives a $(\varphi, \Gamma_R^3)\equivariant$ map $r_1 \colon N \rightarrow \AR(2) \otimes_{r_1, \AR} N$.
	Then, reduction modulo $\mu$, restricts $r_1$ to an $R\linear$ and $\varphi\equivariant$ isomorphism $r_1 \colon N/\mu N \isomorphic (\AR(2)/p_1(\mu) \otimes_{r_1, \AR} N)^{1 \times \Gamma_R \times \Gamma_R}$.
\end{lem}
\begin{proof}
	Note that the map $r_1 \colon R \isomorphic (\AR(2)/p_1(\mu))^{1 \times \Gamma_R \times \Gamma_R}$ is a $\varphi\equivariant$ isomorphism by Remark \ref{rem:a2modp1mun}.
	Moreover, we have that
	\begin{equation*}
		(\AR(2)/p_1(\mu) \otimes_{r_1, \AR} N)^{1 \times \Gamma_R \times \Gamma_R} = (\AR(2)/p_1(\mu))^{1 \times \Gamma_R \times \Gamma_R} \otimes_{r_1, \AR} N.
	\end{equation*}
	Hence, we get the claimed isomorphism $r_1 \colon N/\mu N \isomorphic (\AR(2)/p_1(\mu) \otimes_{r_1, \AR} N)^{1 \times \Gamma_R \times \Gamma_R}$.
\end{proof}

We have all the necessary inputs for proving Proposition \ref{prop:strat_from_wach_relative} (similar to \cite[Proposition 3.18]{morrow-tsuji}).
\begin{proof}[Proof of Proposition \ref{prop:strat_from_wach_relative}]
	For the first claim, note that from the definition of $\varepsilon$ it is clear that its base change along $\Delta \colon \AR(1) \rightarrow \AR$ is the identity.
	Moreover, by using Lemma \ref{lem:strat_gamma_equiv_relative}, we see that for each pair $(i, j) \in \{(1, 2), (2, 3), (1, 3)\}$, the base change $p_{ij}^*(\varepsilon) \colon \AR(2) \otimes_{r_i, \AR} N \rightarrow \AR(2) \otimes_{r_j, \AR} N$ is $\Gamma_R^3\equivariant$.
	Therefore, restrictions of $p_{13}^*(\varepsilon)$ and $p_{23}^*(\varepsilon) \circ p_{12}^*(\varepsilon)$ to $N$ have images in $(\AR(2) \otimes_{r_3, \AR} N)^{1 \times \Gamma_R \times \Gamma_R}$, and their composition with the injective map $\Delta_N$ in Proposition \ref{prop:diagonal_m2_relative} is the identity.
	So, it follows that $p_{13}^*(\varepsilon) = p_{23}^*(\varepsilon) \circ p_{12}^*(\varepsilon)$, since both sides are $\AR(2)\linear$.
	We also get that $\varepsilon$ is an isomorphism because we can write the inverse as its base change along the involution $\AR(1) \isomorphic \AR(1)$ swapping the two factors.
	This proves the first claim.

	To show the second claim, let $g$ in $\Gamma_R$.
	From Lemma \ref{lem:strat_gamma_equiv_relative}, we have the following commutative diagram:
	\begin{center}
		\begin{tikzcd}[column sep=large]
			\AR(1) \otimes_{(\gamma, 1), \AR(1)} (\AR(1) \otimes_{p_1, \AR} N) \arrow[r, "\sim", "1 \otimes \varepsilon"'] \arrow[d, "1 \otimes (g{,} 1)", "\wr"'] & \AR(1) \otimes_{(\gamma, 1), \AR(1)} (\AR(1) \otimes_{p_2, \AR} N) \arrow[d, "1 \otimes (g{,} 1)", "\wr"']\\
			\AR(1) \otimes_{p_1, \AR} N \arrow[r, "\sim", "\varepsilon"'] & \AR(1) \otimes_{p_2, \AR} N.
		\end{tikzcd}
	\end{center}
	From the diagram it is clear that the image of any $x$ in $N$ under the left vertical arrow is $g(x)$ while its image under the right vertical arrow is $x$.
	Base changing the diagram along the $(\Gamma_R \times 1)\equivariant$ map $\Delta \colon \AR(1) \rightarrow \AR$, we obtain the following commutative diagram: \begin{center}
		\begin{tikzcd}[column sep=8em]
			\AR \otimes_{g, \AR} N \arrow[r, "\sim", "\Delta^*((g{,} 1)^*(1 \otimes \varepsilon))"'] \arrow[d, "1 \otimes g", "\wr"'] &  N \arrow[d, "id", "\wr"']\\
			N \arrow[r, "\sim", "id"'] & N,
		\end{tikzcd}
	\end{center}
	where the left vertical arrow is the action of $g$ on $N$.
	This proves the claim and concludes our proof.
\end{proof}

\appendix

\section{Some basic definitions and lemmas}\label{app_sec:basic_defi_lem}

In this section, we collect some standard definitions to enhance the readability of the text.
For more details, please refer to the cited references.
Let $p$ be a fixed prime number.

\subsection{Basic definitions}\label{app_subsec:basic_defi}

Let $R$ be a commutative ring.

\begin{defi}
	Let $a, b \in \ZZ \cup \{\pm\infty\}$ and let $D(R)$ denote the derived $\infty\category$ of $R\modules$.
	\begin{enumarabicup}
	\item An object $M$ in $D(R)$ has \textit{$p\complete$ Tor amplitude} in $[a, b]$ if $M \otimes_R^{\LL} R/pR$ in $D(R/pR)$ has Tor amplitude in $[a, b]$.
		
	\item An object $M$ in $D(R)$ is \textit{$p\completely$ (faithfully) flat} if $M \otimes_R^{\LL} R/pR$ in $D(R/pR)$ is concentrated in degree 0 and it is a (faithfully) flat $R/pR\module$.
	\end{enumarabicup}
\end{defi}

\begin{defi}[Perfectoid rings, {\cite[Definition 3.5]{bhatt-morrow-scholze-1}}]\label{defi:perfectoid_rings}
	A ring $R$ is called \textit{perfectoid} if it is $p\textrm{-adically}$ complete and there is some $\pi$ in $R$ such that $\pi^p = pu$, for some unit $u$ in $R^{\times}$, the ring $R/p$ is semiperfect, i.e.\ the absolute Frobenius map on $R/p$ is surjective, and the kernel of the map $\theta \colon \Ainf(R) \rightarrow R$ is principal.
\end{defi}

\begin{defi}
	Let $(X, \pazo)$ denote a ringed topos.
	An $\pazo\module$ $\paze$ is called a \textit{vector bundle} on $(X, \pazo)$ if there exists a cover $\{U_i\}$ of $X$ and finite projective $\pazo(U_i)\modules$ $\pazp_i$ such that $\paze|_{U_i} \isomorphic \pazp_i \otimes_{\pazo(U_i)} \pazo_{U_i}$, for each $i$.
	Denote by $\Vect(X, \pazo)$ the category of all vector bundles on $(X, \pazo)$.
\end{defi}

\begin{defi}[Koszul complex, {\cite[Definition 7.1]{bhatt-morrow-scholze-1}}]\label{defi:koszul_complex}
	Let $M$ be an abelian group and for $i = 1, \ldots, d$, let $f_i \colon M \rightarrow M$ denote $d$ commuting endomorphisms of $M$.
	The Koszul complex $K_M(f_1, \ldots, f_d)$ is defined to be the following complex:
	\begin{equation*}
		M \xrightarrow{(f_1, \ldots, f_d)} \textstyle \bigoplus_{1 \leqslant i \leqslant d} M \longrightarrow \bigoplus_{1 \leqslant i_1 < i_2 \leqslant d} M \longrightarrow \cdots \longrightarrow \bigoplus_{1 \leqslant i_1 < \cdots < i_k \leqslant d} M \longrightarrow \cdots,
	\end{equation*}
	where the differential from $M$ at the index $i_1 < \cdots < i_k$ to $M$ at the index $j_1 < \cdots < j_{k+1}$ is nonzero if and only if the set $\{i_1, \ldots, i_k\}$ is contained in the set $\{j_1, \ldots, j_{k+1}\}$, and in that case the differential is given as $(-1)^n f_{j_n}$, where $n$ is the unique integer in $\{1, \ldots, k+1\}$ such that $j_n$ does not belong to the set $\{i_1, \ldots, i_k\}$.
\end{defi}

\begin{lem}[{\cite[Lemma 7.3]{bhatt-morrow-scholze-1}}]\label{lem:cont_coh_disc}
	Let $\Gamma_{\textup{disc}} = \prod_{i=1}^d \ZZ$ be the free abelian group on generators $\gamma_1, \ldots, \gamma_d$, and $\Gamma = \prod_{i=1}^d \ZZ_p$ its $\padic$ completion.
	Let $N$ be a continuous $\Gamma\module$ that can be written as an inverse limit $N = \lim_{i \geqslant 1} N_i$ of continuous discrete $\Gamma\modules$ $N_i$ killed by $p^i$.
	Then, the natural map
	\begin{equation*}
		R\Gamma_{\cont}(\Gamma, N) \rightarrow R\Gamma(\Gamma_{\textup{disc}}, N),
	\end{equation*}
	is a quasi-isomorphism, and thus $R\Gamma_{\cont}(\Gamma, N)$ is computed by $K_N(\gamma-1, \ldots, \gamma_d-1)$.
	In particular, we have that $N^{\Gamma} = N^{\gamma_1=\cdots=\gamma_d=1}$.
\end{lem}

We note the following fact from the proof of \cite[Corollary 12.5]{bhatt-morrow-scholze-1}:
\begin{lem}\label{lem:koszul_complex_qiso}
	For $1 \leqslant i \leqslant d$, if $f_i$ are $d$ commuting endomorphisms of an abelian group $M$, and $h_i$ are $d$ automorphisms of $M$ commuting with each other and $f_i\textrm{'s}$, then we have a natural quasi-isomorphism of complexes
	\begin{equation*}
		K_M(f_1h_1, \ldots, f_dh_d) \isomorphic K_M(f_1, \ldots, f_d).
	\end{equation*}
\end{lem}

\subsection{Modules with \texorpdfstring{$\ZZ_p^{\times}\textrm{-action}$}{-}}\label{app_subsec:module_zpx_action}

In this section, we shall consider objects admitting a continuous action of $\Gamma_F \isomorphic \ZZ_p^{\times}$.
From \eqref{eq:gammaf_es}, recall that $\Gamma_F \isomorphic \ZZ_p^{\times}$, via the $\padic$ cyclotomic character, and it fits into the following exact sequence:
\begin{equation*}
	1 \longrightarrow \Gamma_0 \longrightarrow \Gamma_F \longrightarrow \Gamma_{\textrm{tor}} \longrightarrow 1,
\end{equation*}
where, for $p \geqslant 3$ we have $\Gamma_0 \isomorphic 1 + p\ZZ_p$, and for $p=2$ we have $\Gamma_0 \isomorphic 1 + 4\ZZ_2$.
Moreover, for $p \geqslant 3$, we have that $\omega \colon \Gamma_{\textrm{tor}} \isomorphic \FF_p^{\times}$, and the projection map $\Gamma_F \rightarrow \Gamma_{\textrm{tor}}$ admits a section $\Gamma_{\textrm{tor}} \isomorphic \FF_p^{\times} \rightarrow \ZZ_p^{\times} \lisomorphic \Gamma_F$, where the second map is given as $a \mapsto [a]$, the Teichm\"uller lift of $a$.
Finally, for $p = 2$ we have $\Gamma_{\textrm{tor}} \isomorphic \{\pm 1\}$ as groups.

\subsubsection{The action of \texorpdfstring{$\FF_p^{\times}$}{-}}

We will assume that $p \geqslant 3$ and recall the definition of $\FF_p^{\times}\textrm{-decomposition}$ from \cite[Section 3]{iwasawa} ($\Delta\textrm{-decomposition}$ in loc.\ cit.).
For $0 \leqslant i \leqslant p-2$, let $\epsilon_i = \tfrac{1}{p-1} \sum_{a \in \FF_p} \omega(a)^{-i} a$ denote an element in the group ring $\ZZ_p\llbracket\Gamma_F\rrbracket$.
Then, it is easy to check that we have $\epsilon_i^2 = \epsilon_i$, $\epsilon_i \epsilon_j = 0$ for $i \neq j$ and $\sum_{i=0}^{p-2} \epsilon_i = 1$.
Now, let $M$ be a linearly topologised compact $\ZZ_p\module$ admitting a continuous action of $\Gamma_F$.
Set $M_i \coloneq \epsilon_i(M)$, for $0 \leqslant i \leqslant p-2$.
Then, we have a canonical decomposition of $M$ as follows:
\begin{equation}\label{eq:fpx_decomp}
	M = \textstyle \bigoplus_{i=0}^{p-2} M_i.
\end{equation}
Here, each $M_i$ may also be characterised as the submodule of all $x$ in $M$ such that $ax = \omega(a)^ix$, for all $a$ in $\FF_p^{\times}$.
We will refer to the decomposition of $M$ in \eqref{eq:fpx_decomp} as the $\FF_p^{\times}\textrm{-decomposition}$ of $M$.
Moreover, since the action of $\FF_p^{\times}$ and $\Gamma_0$ commute with each other, therefore, we see that each $M_i$ admits a continuous action of $\Gamma_0$.

\subsubsection{The action of \texorpdfstring{$\{\pm 1\}$}{-}}

We will assume that $p = 2$ and recall the following construction:
let $M$ be a linearly topologised compact $\ZZ_p\module$ admitting a continuous action of $\Gamma_F$, and let $\sigma$ be a generator of $\Gamma_{\textrm{tor}} \isomorphic \{\pm 1\}$.
Then, we set $M_{+} \coloneq \{x \in M \textrm{ such that } \sigma(x) = x\}$ and $M_{-} \coloneq \{x \in M \textrm{ such that } \sigma(x) = -x\}$.
Using these notations, we have a canonical injective map of $\ZZ_p\modules$:
\begin{equation}\label{eq:plus_minus_decomp}
	M_{+} \oplus M_{-} \longrightarrow M,
\end{equation}
sending $(x, y) \mapsto x+y$.
Note that the map \eqref{eq:plus_minus_decomp} need not be surjective.
If $M = A$ is a $\ZZ_p\algebra$, then it is easy to verify that $A_{+}$ is a $\ZZ_p\algebra$ as well, $A_{-}$ is an $A_{+}\module$, and the map in \eqref{eq:plus_minus_decomp} is $A_{+}\linear$.

Let us consider a $\ZZ_p\algebra$ $A$, and let $q$ be a unit in $A$ such that $q-1$ is a nonzerodivisor on $A$.
Assume that $A$ is equipped with a $\ZZ_p\linear$ action of $\Gamma_{\textrm{tor}}$ such that $\sigma(q) = q^{-1}$, and the induced action of $\Gamma_{\textrm{tor}}$ on $A/(q-1)A$ is trivial.
Let $M$ be an $A\module$.
Assume that the element $q-1$ is $M\textrm{-regular}$, and the $A\module$ $M$ is equipped with an $A\textrm{-semilinear}$ action of $\Gamma_{\textrm{tor}}$ such that the induced action of $\Gamma_{\textrm{tor}}$ is trivial on $M/(q-1)M$.
Then, we note that the operator $\nabla_{\sigma} \coloneq \tfrac{\sigma-1}{q-1}$ is well defined on $M$, and we claim the following:
\begin{lem}\label{lem:sigma_action_decomp}
	The $A\module$ $M$ admits an $A^{\nabla_{\sigma}=0}\linear$ decomposition $M = M^{\nabla_{\sigma}=0} \oplus M^{\nabla_{\sigma}=1}$.
\end{lem}
\begin{proof}
	Let us first note that $\nabla_{\sigma}^2 = \nabla_{\sigma}$.
	Indeed,
	\begin{align*}
		\nabla_{\sigma}^2 = \tfrac{\sigma-1}{q-1} \circ \tfrac{\sigma-1}{q-1} = \tfrac{1}{q-1}\big(\tfrac{\sigma(\sigma-1)}{\sigma(q)-1} - \tfrac{\sigma-1}{q-1}\big) &= \tfrac{1}{q-1}\big(\tfrac{1-\sigma}{q^{-1}-1} - \tfrac{\sigma-1}{q-1}\big)\\
			&= \tfrac{1}{q-1}\big(\tfrac{q(\sigma-1)}{q-1} - \tfrac{\sigma-1}{q-1}\big) = \tfrac{\sigma-1}{q-1} = \nabla_{\sigma}.
	\end{align*}
	As $\nabla_{\sigma}$ is an idempotent operator on $M$, it follows that $M$ admits a decomposition $M = \nabla_{\sigma}(M) \oplus (\nabla_{\sigma}-1)M = M^{\nabla_{\sigma}=1} \oplus M^{\nabla_{\sigma}=0}$.
	It is easy to see that the decomposition is $A^{\nabla_{\sigma}=0}\linear$.
\end{proof}

Let $M$ and $N$ be two $A\modules$ admitting actions of $\Gamma_{\textrm{tor}}$ such that the induced action of $\Gamma_{\textrm{tor}}$ is trivial modulo $(q-1)$ on both modules, and let $M \twoheadrightarrow N$ be an $A\linear$ map compatible with the action of $\Gamma_{\textrm{tor}}$.
\begin{lem}\label{lem:sigma_action_surjective}
	The $A^{\Gamma_{\textup{tor}}}\linear$ map $M^{\Gamma_{\textup{tor}}} \rightarrow N^{\Gamma_{\textup{tor}}}$ is surjective.
\end{lem}
\begin{proof}
	Note that we have $M^{\Gamma_{\textrm{tor}}} = M^{\nabla_{\sigma}=0}$ and similarly for $N$.
	We get the claim by using the decomposition in Lemma \ref{lem:sigma_action_decomp}.
\end{proof}

\subsection{Crystalline site and \texorpdfstring{$F\crystals$}{-}}\label{subsec:crys_site_crystal}

Let $A$ be a $p\adically$ complete PD-ring, $B$ a $p\textrm{-adically}$ completed smooth algebra over $A$ and $D \coloneq B[\prod_{i=1}^d t_1^{[k_1]} \cdots t_d^{[k_d]}]_p^{\wedge}$ the $\padic$ completion of a PD-polynomial algebra over $B$ in $d \geqslant 0$ variables.
For $m \geqslant 1$, we set $\Sigma_m = \Spec(A/p^m)$, $X_m = \Spec(B/p^m)$, $Y_m = \Spec(D/p^m)$ and note that $X_m \hookrightarrow Y_m$ is a PD-thickening morphism induced via the surjection $D \twoheadrightarrow B$ defined by sending $t_i^{[k_i]} \mapsto 0$.
Let $\CRYS(X_m/\Sigma_m)$ denote the big crystalline site of $X_m$ over $\Sigma_m$ with the PD-structure given by $p(A/p^m) + JA/p^m$, where $J$ denotes the PD-ideal of $A$, and let $\pazo_{X_m/\Sigma_m}$ denote the structure sheaf of rings.
Let $\CR(X_m/\Sigma_m)$ denote the category of quasi-coherent crystals of $\pazo_{X_m/\Sigma_m}\modules$ over $\CRYS(X_m/\Sigma_m)$.

The homomorphisms $\Sigma_m \rightarrow \Sigma_{m+1}$ and $X_m \rightarrow X_{m+1}$ induce a pullback functor $i^*_{m,m+1} \colon \CR(X_{m+1}/\Sigma_{m+1}) \rightarrow \CR(X_m/\Sigma_m)$.
One may define $\CRYS(X_1/\Sigma_m)$ and $\CR(X_1/\Sigma_m)$ similar to above, and note that the pullback functor induces an equivalence of categories $i_m^* \colon \CR(X_m/\Sigma_m) \isomorphic \CR(X_1/\Sigma_m)$ (see \cite[Chapitre IV, Th\'eor\`eme 1.4.1]{berthelot-cohomologie-cristalline}).
We define a quasi-coherent crystal $\paze$ on $X/\Sigma$ to be a system $(\paze_m)_{m\geqslant 1}$, where $\paze_m$ is an object of $\CR(X_m/\Sigma_m)_{\crys}$ for each $m \geqslant 1$, and we have isomorphisms $i_{m,m+1}^* \paze_{m+1} \isomorphic \paze_m$.
Write $\CR(X/\Sigma)$ and $\CR(X_1/\Sigma)$ for the category of quasi-coherent crystals on $X/\Sigma$ and $X_1/\Sigma$, respectively, and note that the obvious pullback functor $i^* \colon \CR(X/\Sigma) \rightarrow \CR(X_1/\Sigma)$ induces an equivalence of categories.
Let $\MIC_{\converge}(D)$ denote the category of $p\textrm{-adically}$ complete $D\modules$ equipped with an $A\linear$ $p\adically$ quasi-nilpotent flat connection.
Then, by \cite[Chapitre IV, Th\'eor\`eme 1.6.5]{berthelot-cohomologie-cristalline} and \cite[\href{https://stacks.math.columbia.edu/tag/07JH}{Tag 07JH}, \href{https://stacks.math.columbia.edu/tag/07L6}{Tag 07L6}]{stacks-project}, we have a natural equivalence of categories
\begin{equation}\label{eq:crys_crystal_connections}
	\CR(X/\Sigma) \isomorphic \MIC_{\converge}(D),
\end{equation}
obtained by sending $(\paze_m)_{m \geqslant 1}$ to the inverse limit of the evaluation of $\paze_m$ on the object $\Spec(X_m) \hookrightarrow \Spec(Y_m)$ of the site $\CRYS(X_m/\Sigma_m)$, equipped with the natural $A\linear$ $p\adically$ quasi-nilpotent flat connection.

Let us now describe $F\textrm{-crystals}$ on $X/\Sigma$.
Assume that $A$ is equipped with a lift of Frobenius modulo $p$.
Then, we note that the absolute Frobenius on $X_1$ and the natural Frobenius on $\Sigma$ induce Frobenius pullbacks $\varphi^* \colon \CR(X_1/\Sigma_m) \rightarrow \CR(X_1/\Sigma_m)$ and $\varphi^* \colon \CR(X_1/\Sigma) \rightarrow \CR(X_1/\Sigma)$.
A quasi-coherent $F\crystal$ $\paze$ on $X/\Sigma$ is an object of $\CR(X_1/\Sigma)$ equipped with an isomorphism $\varphi_{\paze} \colon (\varphi^* i^* \pazf)_{\QQ} \isomorphic (i^* \pazf)_{\QQ}$ in the isogeny category $\CR(X/\Sigma) \otimes \QQ$.
We shall denote the category of quasi-coherent $F\crystals$ on $X/\Sigma$ as $\CR^{\varphi}(X/\Sigma)$.
Let $\MIC^{\varphi}(D)$ denote the following category: an object $M$ is a $p\textrm{-adically}$ complete $D\module$, equipped with an $A\linear$ $p\adically$ quasi-nilpotent flat connection and a $D\linear$ isomorphism $\varphi_M \colon \varphi^*M[1/p] \isomorphic M[1/p]$ compatible with the connection; morphisms between two objects are $D\linear$ maps compatible with the respective Frobenii and connections.
Then, the equivalence of categories in \eqref{eq:crys_crystal_connections} refines to a natural equivalence of categories
\begin{equation}\label{eq:crys_fcrystal_connections}
	\CR^{\varphi}(X/\Sigma) \isomorphic \MIC^{\varphi}(D).
\end{equation}
In fact, similar to \cite[Lemma 2.24]{morrow-tsuji}, it may be shown that for a $p\textrm{-adically}$ complete $D\module$ equipped with an $A\linear$ flat connection and a Frobenius structure, the $p\adic$ quasi-nilpotence of the connection is automatic.

\section{\texorpdfstring{$\delta\textrm{-rings}$}{-} and divided power algebras}\label{app_sec:delta_rings_pd_alg}

The content of this section has been adapted from some notes of Takeshi Tsuji on prismatic envelopes; we are thankful to him for sharing his computations.

In this section, we will describe certain prismatic envelopes explicitly.
We begin this section by fixing some terminology.
Let $p$ be a fixed prime and $A$ a commutative ring.
A \textit{$\delta\ring$} is a pair $(A, \delta)$ where $A$ is a commutative ring and $\delta \colon A \rightarrow A$ is a map of sets with $\delta(0) = \delta(1) = 0$ and satisfying:
\begin{equation}\label{app_eq:delta_sum_prod}
	\begin{aligned}
		\delta(x+y) &= \delta(x) + \delta(y) + \frac{x^p + y^p - (x+y)^p}{p}, \\
		\delta(xy) &= x^p \delta(y) + y^p \delta(x) + p\delta(x)\delta(y).
	\end{aligned}
\end{equation}
Given a $\delta\ring$ $(A, \delta)$ define an endomorphism $\varphi \colon A \rightarrow A$ by the formula $\varphi(x) = x^p + p\delta(x)$, for any $x$ in $A$.
This determines a \textit{lifting of (the absolute) Frobenius} on $A/pA$.
Then the product formula in \eqref{app_eq:delta_sum_prod} may be rewritten as
\begin{equation}\label{app_eq:delta_prod_frob}
	\delta(xy) = \varphi(x) \delta(y) + \delta(x) y^p.
\end{equation}
If $A$ is $p\textrm{-torsion-free}$ then any lift $\varphi \colon A \rightarrow A$ of the absolute Frobenius on $A/pA$ determines a unique $\delta\textrm{-structure}$ on $A$ given as $\delta(x) = (\varphi(x)-x)/p$, for all $x$ in $A$.

A \textit{$\delta\textrm{-homomorphism}$} $f \colon A \rightarrow B$ between $\delta\textrm{-rings}$ is a homomorphism of the underlying rings compatible with the respective $\delta\textrm{-structures}$, i.e.\ $f \circ \delta(x) = \delta \circ f(x)$, for all $x$ in $A$.
For a $\delta\ring$ $A$, a \textit{$\delta\textrm{-algebra}$} over A is a $\delta\textrm{-ring}$ $B$ equipped with a $\delta\textrm{-homomorphism}$ $A \rightarrow B$.
A \textit{$\delta\textrm{-ideal}$} of a $\delta\ring$ $A$ is an ideal $I$ of the underlying ring, stable under $\delta$, i.e.\ $\delta(I) \subset I$.
If $I$ is a $\delta\textrm{-ideal}$ of a $\delta\textrm{-ring}$ $A$, then the quotient ring $A/I$ is equipped with a unique $\delta\textrm{-structure}$ over $A$.
Now, let $S$ be a subset of a $\delta\ring$ $A$.
Then, we see that the ideal of $A$ generated by $\{\delta^n(S)\}_{n \geqslant 1}$ is stable under $\delta$ and it is the smallest $\delta\textrm{-ideal}$ of $A$ containing $S$.
(Note that for any ideal $J$ of $A$, the map $J \rightarrow A/J$, given by $x \mapsto \delta(x) \textrm{ mod } J$, is a $\varphi\textrm{-semilinear}$ homomorphism of $A\modules$ by \eqref{app_eq:delta_sum_prod} and \eqref{app_eq:delta_prod_frob}.)
We write this ideal as $(S)_{\delta}$ and call it the \textit{$\delta\textrm{-ideal generated by } S$}.
If $I$ is an ideal of $A$, then we write $I_{\delta}$ for $(I)_{\delta}$.
If the ideal $I$ is generated by a subset $S$ of $I$, then we have $I_{\delta} = (S)_{\delta}$ because the latter contains $I$.
This implies that if $B$ is a $\delta\textrm{-algebra}$ over a $\delta\textrm{-ring}$ $A$ and $I$ a $\delta\textrm{-ideal}$ of $A$, then the ideal $IB$ of $B$ generated by $I$ is a $\delta\textrm{-ideal}$ of the $\delta\textrm{-ring}$ $B$.

A \textit{$\delta\textrm{-subring}$} of a $\delta\ring$ $A$ is a subring $A' \subset A$, stable under the $\delta\textrm{-structure}$ on $A$, i.e.\ $\delta(A') \subset A'$.
Similarly, a \textit{$\delta\subalgebra$} of a $\delta\textrm{-algebra}$ $B$ over a $\delta\ring$ $A$ is an $A\subalgebra$ $B' \subset B$, stable under the $\delta\textrm{-structure}$ on $B$, i.e.\ $\delta(B') \subset B'$.
Let $S$ be a subset of a $\delta\algebra$ $B$ over a $\delta\ring$ $A$, and note that $\delta(A[\delta^k(S), 1 \leqslant k \leqslant n]) \subset A[\delta^k(S), 1 \leqslant k \leqslant n+1]$, for each $n \in \NN_{\geqslant 1}$.
(Note that for any $A\subalgebra$ $B'$ of $B$, the subset $\{x \textrm{ in } B' \textrm{ such that } \delta(x) \textrm{ in } B\}$ of $B'$ is an $A\subalgebra$ by \eqref{app_eq:delta_sum_prod} and \eqref{app_eq:delta_prod_frob}).
This implies that the $A\subalgebra$ of $B$ generated by $\{\delta^n(S)\}_{n \geqslant 1}$ is stable under $\delta$ and it is the smallest $\delta\subalgebra$ of $B$ over $A$ containing $S$; we denote this $\delta\subalgebra$ by $A[S]_{\delta}$ and call it the \textit{$\delta\subalgebra$ of $B$ over $A$ generated by $S$}.

\begin{assum}\label{app_assum:base_a}
	Let $A$ be a $\delta\ring$ equipped with an element $q$ such that $\varphi(q) = q^p$.
	Set $\mu \coloneq q-1$ and assume that $\varphi^n(\mu)$ is a nonzerodivisor on $A$, for each $n \in \NN$.
	Moreover, assume that $A$ and $A/\mu A$ are $p\torsionfree$.
\end{assum}
In the rest of this section we will consider a $\delta\ring$ $A$ satisfying Assumption \ref{app_assum:base_a}.
Set $[p]_q = \tfrac{q^p-1}{q-1} = 1 + q + \cdots + q^{p-1}$.
Then, by Assumption \ref{app_assum:base_a}, we see that for each $n \in \NN$, the element $\varphi^n([p]_q)$ is a nonzerodivisor on $A$.
Before proceeding further, we note a simple lemma on regular sequences.

\begin{lem}\label{app_lem:reg_seq}
	Let $R$ be a ring.
	\begin{enumerate}
		\item[(1)] If a sequence $\{x, y\}$ is regular on $R$, then $x$ is regular on $R/yR$.
	
		\item[(1)] Let $x, y_1, \ldots, y_d$ be elements of $R$.
			If the sequence $\{x, y_1, \ldots, y_d\}$ is regular on $R$, then $x$ is regular on $R/\sum_{i=1}^d (y_i R)$.
	\end{enumerate}
\end{lem}
\begin{proof}
	The claim in (1) is obtained by applying the snake lemma to the map induced by multiplication by $y$ on the exact sequence $0 \rightarrow R \xrightarrow{\hspace{1mm} x \hspace{1mm}} R \rightarrow R/xR \rightarrow 0$.
	We will prove the claim in (2) by induction on $d$, where the case $d=1$ is clear from (1).
	So, let $d \geqslant 2$, and assume that $x$ is regular on $R' = R/\sum_{i=1}^{d-1} (y_i R)$.
	Additionally, we know that $y_d$ is regular on $R'/xR'$, in particular, $\{x, y_d\}$ is a regular sequence on $R'$.
	Therefore, the claim in (1) implies that $x$ is regular on $R'/y_dR' = R/\sum_{i=1}^d (y_i R)$.
\end{proof}

\begin{rem}\label{app_rem:base_a_torsionfree}
	By using Lemma \ref{app_lem:reg_seq} (1), we see that for the ring $A$ as in Assumption \ref{app_assum:base_a}, the quotient $A/pA$ is $\mu\torsionfree$.
	Moreover, since $\varphi^n(\mu) = \mu^{p^n} \textrm{ mod } pA$, for each $n \in \NN$, and $A$ is $p\torsionfree$, therefore, by using Lemma \ref{app_lem:reg_seq} (1), we get that $A/\varphi^n(\mu)A$ is $p\torsionfree$.
\end{rem}

Let $\Abar$ denote the $p\torsionfree$ algebra $A/\mu A$, and note that the image of $\varphi^n([p]_q)$ in $\Abar$ is $p$, for every $n \in \NN$.
Moreover, the lifting of Frobenius on $A$ induces a lifting of Frobenius on $\Abar$.
Furthermore, as $A/\mu A$ is $p\torsionfree$, we see that the image of $\delta([p]_q)$ in $A$ is $\delta(p)$ in $A/\mu A$, which may be computed as $(\varphi(p)-p^p)/p = 1 - p^{p-1}$.
Hence, it follows that $\delta([p]_q)$ is a unit modulo any power of the ideal $(p, \mu) \subset A$.

\begin{assum}\label{app_assum:alg_b}
	Let $B$ be a $\delta\algebra$ over $A$.
	Assume that $B$ and $B/\mu B$ are $p\torsionfree$ and $\varphi^n(\mu)$ is a nonzerodivisor on $B$, for each $n \in \NN$.
	Let $Y_0, Y_1, \ldots, Y_d$ be elements of $B$ such that the sequence $\{Y_1, \ldots, Y_d\}$ is regular on $B$ and the sequence $\{Y_0, Y_1, \ldots, Y_d\}$ is regular on $B/[p]_q B$ and $B/(p, \mu)B$.
\end{assum}

\begin{rem}\label{app_rem:alg_b_torsionfree}
	Similar to the case of $A$ as in Remark \ref{app_rem:base_a_torsionfree}, by using Lemma \ref{app_lem:reg_seq} (1), we see that $B/pB$ is $\mu\torsionfree$ and $B/\varphi^n(\mu)B$ is $p\torsionfree$, for each $n \in \NN$.
\end{rem}

Let $I$ denote the set of natural numbers $\{0, 1, \ldots, d\}$.
Let $C_0 \coloneq B[Y_{0, 0}, \ldots, Y_{d, 0}]$ denote a polynomial ring over $B$ in $d+1$ variables $Y_{0, 0}, \ldots Y_{d, 0}$, and let $D_0 \coloneq C_0/(\{[p]_qY_{i, 0} - Y_i\}_{i \in I})$.
Moreover, set $E_0$ to be the $B\subalgebra$ of $B[1/[p]_q]$ generated by $y_i \coloneq Y_i/[p]_q$, for $i \in I$.
Then the surjective homomorphism of $B\algebras$ $C_0 \rightarrow E_0$ via $Y_{i, 0} \mapsto y_i$ for $i \in I$, induces a surjective homomorphism of $B\algebras$
\begin{equation}\label{app_eq:d0_to_e0}
	D_0 \longrightarrow E_0.
\end{equation}
Set $\Bbar \coloneq B/\mu B$, $\Czerobar \coloneq C_0/\mu C_0$, $\Dzerobar \coloneq D_0/\mu D_0$ and $\Ezerobar \coloneq E_0/\mu E_0$, and write $\Yizerobar$ (resp.\ $\yibar$) for the image of $Y_{i, 0}$ (resp.\ $y_i$) in $\Czerobar$ (resp.\ $\Ezerobar$).

\begin{lem}\label{app_lem:e0_desc}
	With notations as above, we have the following:
	\begin{enumerate}
		\item[(1)] The homomorphism \eqref{app_eq:d0_to_e0} is an isomorphism.
	
		\item[(2)] The algebra $\Ezerobar$ is $p\torsionfree$.
	\end{enumerate}
\end{lem}
\begin{proof}
	Note that we have $D_0[1/[p]_q] \isomorphic B[1/[p]_q][\{Y_{i,0}\}_{i \in I}]/(\{Y_{i,0}-y_i\}_{i \in I}) \isomorphic B[1/[p]_q]$.
	Therefore, to show (1), it suffices to show that $D_0$ is $[p]_q\torsionfree$.
	Now, note that by Assumption \ref{app_assum:alg_b}, the element $[p]_q$ is regular on $C_0$ and the sequence $\{[p]_qY_{0,0}-Y_0, \ldots, [p]_qY_{d,0}-Y_d\}$ is regular on $C_0/[p]_qC_0$.
	Hence, from Lemma \ref{app_lem:reg_seq} (2), it follows that $[p]_q$ is regular on $D_0$.
	For (2), note that by using claim (1), we are reduced to showing that $\Dzerobar$ is $p\torsionfree$.
	Moreover, as $\Czerobar$ is $p\torsionfree$, by using Lemma \ref{app_lem:reg_seq} (2), we see that it is enough to show that the sequence $\{[p]_qY_{0,0}-Y_0, \ldots, [p]_qY_{d,0}-Y_d\}$ is regular on $\Czerobar/p\Czerobar$.
	Now note that $\Czerobar/p\Czerobar = C_0/(p,\mu)C_0$, and since $[p]_q$ is in the ideal $(p, \mu) \subset A$, therefore, by Assumption \ref{app_assum:alg_b}, we get that the sequence $\{[p]_qY_{0,0}-Y_0, \ldots, [p]_qY_{d,0}-Y_d\}$ is regular on $\Czerobar/p\Czerobar$.
	This concludes our proof.
\end{proof}

The lifting of Frobenius on $B$ naturally extends to $B[1/p, \{1/\varphi^n([p]_q)\}_{n \in \NN}]$ and note that its $E_0\subalgebra$ $E_0[1/p, \{1/\varphi^n([p]_q)\}_{n \in \NN}]$ is stable under $\varphi$.
Let $E$ be the $\delta\subalgebra$ of $E_0[1/p, \{1/\varphi^n([p]_q)\}_{n \in \NN}]$ generated over $B$ by $y_i$, for all $i \in I$.
Note that we have $E_0 \subset E$.
Now, recall that $\varphi^n([p]_q) = p \textrm{ mod } \mu B$, for each $n \in \NN$, therefore, we get that $E_0[1/p, \{1/\varphi^n([p]_q)\}_{n \in \NN}]/\mu \isomorphic \Ezerobar[1/p]$.
Via the preceding isomorphism, the lifting of Frobenius on $E_0[1/p, \{1/\varphi^n([p]_q)\}_{n \in \NN}]$ induces a lifting of Frobenius on $\Ezerobar[1/p]$.
Define $\Ebar$ to be the $\delta\subalgebra$ of $\Ezerobar[1/p]$ generated over $\Bbar$ by $\yibar$, for all $i \in I$.
Note that we have $\Ezerobar \subset \Ebar$ by Lemma \ref{app_lem:e0_desc} (2) and the natural ring homomorphism $E_0[1/p, \{1/\varphi^n([p]_q)\}_{n \in \NN}] \rightarrow \Ezerobar[1/p]$ induces the following surjective ring homomorphism compatible with $\varphi$:
\begin{equation}\label{app_eq:e_to_ebar}
	E \longrightarrow \Ebar.
\end{equation}

We will study the $\delta\textrm{-rings}$ $E$ and $\Ebar$ by comparing them to the $\delta\textrm{-rings}$ obtained by universally adjoining $Y_i/[p]_q$, for all $i \in I$, to the $\delta\textrm{-rings}$ $B$ and $\Bbar$, respectively.
Let $C$ denote a polynomial ring over $B$ in variables $Y_{i,n}$, for $i \in I$ and $n \in \NN$.
Equip $C$ with a lifting of Frobenius $\varphi$ compatible with that on $A$ and defined on the variables as $\varphi(Y_{i,n}) = Y_{i,n}^p + pY_{i,n+1}$, for each $i \in I$ and $n \in \NN$.
Since, $C$ is $p\torsionfree$, therefore, for the $\delta\textrm{-structure}$ associated to $\varphi$, we have $\delta(Y_{i,n}) = Y_{i,n+1}$, for each $i \in I$ and $n \in \NN$.
Set $\Cbar \coloneq C/\mu C$ equipped with a lifting of Frobenius induced by that of $C$.
Denote by $\Yibar$ (resp.\ $\Yinbar$) the image of $Y_i$ (resp.\ $Y_{i,n}$) in $\Cbar$.
We define $D$ (resp.\ $\Dbar$) to be the quotient of $C$ (resp.\ $\Cbar$) by the $\delta\textrm{-ideal}$ generated by $[p]_q Y_{i,0} - Y_i$ (resp.\ $[p]_q \Yizerobar - \Yibar$), for all $i \in I$.
Then we have a $\varphi\textrm{-compatible}$ (in particular, $\delta\textrm{-compatible}$) surjective homomorphism of $B\algebras$ $C \rightarrow E$ (resp.\ surjective homomorphism of $\Bbar\algebras$ $\Cbar \rightarrow \Ebar$) defined by sending $Y_{i, n}$ to $\delta^n(y_i)$ (resp.\ by sending $\Yinbar$ to $\delta^n(\yibar)$), for each $i \in I$ and $n \in \NN$.
The preceding ring homomorphisms induce surjective horizontal arrows in the following diagram:
\begin{equation}\label{eq:d_to_e_bar}
	\begin{tikzcd}
		D = C/([p]_q Y_{i,0}-Y_i)_{\delta} \arrow[d] \arrow[r] & E \arrow[d, "\eqref{app_eq:e_to_ebar}"],\\
		\Dbar = \Cbar/(p\Yizerobar-\Yibar)_{\delta} \arrow[r] & \Ebar,
	\end{tikzcd}
\end{equation}
where the left vertical arrow is induced by the isomorphism $D/\mu D \isomorphic \Dbar$, and the diagram commutes by definition.
In particular, all arrows in the diagram \eqref{eq:d_to_e_bar} are surjective.

For each $i \in I$ and $n \in \NN$, let $C_n^{(i)}$ be the $B\subalgebra$ of $C$ generated by $Y_{i,m}$ for $0 \leqslant m \leqslant n$.
Set $C_{-1}^{(i)} \coloneq B$ and note that we have $\delta(C_{n-1}^{(i)}) \subset C_n^{(i)}$, for each $n \in \NN$.
\begin{lem}\label{app_lem:deltan_pqti}
	For each $i \in I$ and $n \in \NN_{\geqslant 1}$, the element $\delta^n([p]_q Y_{i,0}-Y_i)$ in $C_n^{(i)}$ may be written as
	\begin{equation*}
		\delta^n([p]_q Y_{i,0}-Y_i) = \varphi^n([p]_q)Y_{i,n} + c_0 Y_{i,n-1}^p + \textstyle \sum_{k=0}^{p-1} c_k Y_{i,n-1}^k,
	\end{equation*}
	where $c_0 = \varphi^{n-1}(\delta([p]_q)) \sum_{m=0}^{n-1} p^{m(p-1)}$ is in $A$ and $c_k$ is in $C_{n-2}^{(i)}$, for each $0 \leqslant k \leqslant p-1$.
\end{lem}
\begin{proof}
	We shall prove the claim by induction on $n$.
	For $n=1$, note that we have
	\begin{align*}
		\delta([p]_q Y_{i,0} - Y_i) &= \delta([p]_q Y_{i,0}) + \delta(-Y_i) - \textstyle\sum_{k=1}^{p-1} \tfrac{1}{p}\binom{p}{k} ([p]_q Y_{i,0})^k (-Y_i)^{p-k}\\
		&= \varphi([p]_q) Y_{i, 1} + \delta([p]_q)Y_{i,0}^p + \delta(-Y_i) - \textstyle\sum_{k=1}^{p-1} \tfrac{1}{p}\binom{p}{k} [p]_q^k (-Y_i)^{p-k} Y_{i,0}^k,
	\end{align*}
	which is of the required form because the elements $\delta(-Y_i)$, $[p]_q$ and $Y_i$ belong to $C_{-1}^{(i)} = B$.
	Now, let $n \geqslant 1$ and assume that the claim holds for $\delta^n([p]_q Y_{i,0}-Y_i)$, with $c_0$ in $A$ and $c_k$ in $C_{n-2}^{(i)}$.
	Set $b = c_0 Y_{i,n-1}^p + \textstyle \sum_{k=0}^{p-1} c_k Y_{i,n-1}^k$ in $C_{n-1}^{(i)}$ and note that we have
	\begin{align*}
		\delta^{n+1}([p]_q Y_{i,0} - Y_i) &= \delta(\varphi^n([p]_q) Y_{i,n} + b)\\
		&= \delta(\varphi^n([p]_q) Y_{i,n}) + \delta(b) - \textstyle\sum_{k=1}^{p-1} \tfrac{1}{p} \binom{p}{k} (\varphi^n([p]_q) Y_{i,n})^k b^{p-k}\\
		&= \varphi^{n+1}([p]_q) Y_{i,n+1} + \delta(\varphi^n([p]_q)) Y_{i,n}^P + \delta(b) - \textstyle\sum_{k=1}^{p-1} \tfrac{1}{p} \binom{p}{k} \varphi^n([p]_q)^k b^{p-k} Y_{i,n}^k.
	\end{align*}
	As the elements $\varphi^n([p]_q)$ and $b$ are in $C_{n-1}^{(i)}$, therefore, it is enough to show that $\delta(b)$ is of the form $p^{p-1}c_0 + \sum_{k=0}^{p-1} d_k Y_{i,n}^k$, for some $d_k$ in $C_{n-1}^{(i)}$.
	From \eqref{app_eq:delta_sum_prod}, note that for any $x$ and $y$ in $C_{n-1}^{(i)}$, we have that $\delta(x+y) = \delta(x) + \delta(y) \textrm{ mod } C_{n-1}^{(i)}$ and from \eqref{app_eq:delta_prod_frob}, for $x$ in $C_{n-2}^{(i)}$ and $y$ in $C_{n-1}^{(i)}$, we have that $\delta(xy) = \varphi(x)\delta(y) \textrm{ mod } C_{n-1}^{(i)}$.
	Therefore, we may write
	\begin{equation*}
		\delta(b) = \varphi(c_0) \delta(Y_{i,n-1}^p) + \textstyle \sum_{k=0}^{p-1} \varphi(c_k) \delta(Y_{i,n-1}^k) \textrm{ mod } C_{n-1}^{(i)}.
	\end{equation*}
	Now, note that $\delta(1)=0$ on $A$ because $A$ is $p\torsionfree$, and by Lemma \ref{lem:delta_xn}, for each $1 \leqslant k \leqslant p$, we have
	\begin{equation*}
		\delta(Y_{i,n-1}^k) = \textstyle \sum_{j=1}^k \binom{k}{j} p^{j-1} Y_{i,n-1}^{p(k-j)} Y_{i,n}^j.
	\end{equation*}
	Note that $Y_{i,n-1}$ is in $C_{n-1}^{(i)}$ and $Y_{i,n}^p$ appears only when $k=p$, in which case the coefficient of $Y_{i,n}^p$ is $p^{p-1}$.
	This allows us to conclude.
\end{proof}

The following fact was used above:
\begin{lem}\label{lem:delta_xn}
	Let $R$ be a $\delta\ring$ and $x$ any element of $R$.
	Then, for each $n \in \NN_{\geqslant 1}$, we have that
	\begin{equation*}
		\delta(x^n) = \textstyle \sum_{j=1}^n \binom{n}{j} p^{j-1} x^{p(n-j)} \delta(x)^j.
	\end{equation*}
\end{lem}
\begin{proof}
	We shall prove the claim by induction on $n$.
	The case $n=1$ is obvious, so assume that the claim holds for some $n \geqslant 1$.
	Then, we have that
	\begin{align*}
		\delta(x^{n+1}) &= \delta(x^n) x^p + x^{np} \delta(x) + p\delta(x)\delta(x^n)\\
		&= \textstyle \sum_{j=1}^n \binom{n}{j} p^{j-1} x^{p(n+1-j)} \delta(x)^j + x^{np} \delta(x) + \textstyle \sum_{j=1}^n \binom{n}{j} p^j x^{p(n-j)} \delta(x)^{j+1}.
	\end{align*}
	In the final term of the second line of the displayed equation, by replacing $j$ with $k-1$, for $2 \leqslant k \leqslant n+1$, and expanding, it easily follows that the expression thus obtained is the sum $\sum_{j=1}^{n+1} \binom{n+1}{j} p^{j-1} x^{p(n+1-j)} \delta(x)^j$.
\end{proof}

Let $\Cnbar$ be the $\Bbar\subalgebra$ of $\Cbar$ generated by $\Yimbar$, for all $i \in I$ and $0 \leqslant m \leqslant n$.
Set $\overline{C}_{-1} \coloneq \Bbar$ and note that we have $\delta(\overline{C}_{n-1}) \subset \Cnbar$, for each $n \in \NN$.
Define $\Dnbar$ to be the quotient of $\Cnbar$ by the ideal generated by $\delta^m(p\Yizerobar-\Yibar)$, for all $i \in I$ and $0 \leqslant m \leqslant n$.
Note that the $\Bbar\algebras$ $\Czerobar$ and $\Dzerobar$ coincide with those defined before Lemma \ref{app_lem:e0_desc}.
For each $i \in I$ and $n \in \NN_{\geqslant 1}$, let us set 
\begin{equation}\label{app_eq:xni}
	X_n^{(i)} \coloneq \delta^n(p\Yizerobar-\Yibar) - p\Yinbar \in \Cbar,
\end{equation}
and by Lemma \ref{app_lem:deltan_pqti} note that it is contained in $\overline{u}_n \overline{Y}_{i,n-1}^p + \sum_{k=0}^{p-1} \overline{Y}_{i,n-1}^k \overline{C}_{n-2} \subset \overline{C}_{n-1}$, where we have that $\overline{u}_n = \varphi^{n-1}(\delta(p)) \sum_{m=0}^{n-1} p^{m(p-1)}$ is in $\Abar$.
In particular, we see that for each $n \in \NN$, we have
\begin{equation}\label{app_eq:dn+1bar}
	\overline{D}_{n+1} = \Dnbar[\{\overline{Y}_{i,n+1}\}_{i \in I}]/(\{p\overline{Y}_{i,n+1} + X_{n+1}^{(i)}\}_{i \in I}).
\end{equation}
The inclusion maps $\Cnbar \hookrightarrow \overline{C}_m \hookrightarrow \Cbar$, for $0 \leqslant n < m$, induce maps $\Dnbar \rightarrow \overline{D}_m \rightarrow \Dbar$ and an isomorphism
\begin{equation}\label{app_eq:colim_dnbar}
	\colim_n \Dnbar \isomorphic \Dbar.
\end{equation}

\begin{lem}\label{app_lem:dnbar}
	With notations as above, we have the following:
	\begin{enumerate}
		\item[(1)] The homomorphism $\Dnbar \rightarrow \overline{D}_{n+1}$ induces an isomorphism $\Dnbar[1/p] \isomorphic \overline{D}_{n+1}[1/p]$, for each $n \in \NN$.
	
		\item[(2)] Let $S_n$ denote the image of $\Dnbar/p\Dnbar$ in $\overline{D}_{n+1}/p\overline{D}_{n+1}$, for each $n \in \NN$.
			Then, we have that $S_n = (\Dnbar/p\Dnbar)/(\{X_{n+1}^{(i)}\}_{i \in I})$ and $\overline{D}_{n+1}/p\overline{D}_{n+1} = S_n[\{\overline{Y}_{i,n+1}\}_{i \in I}]$, for each $n \in \NN$.
	
		\item[(3)] For each $n \in \NN$, the algebra $\Dnbar$ is $p\torsionfree$.
	\end{enumerate}
\end{lem}
\begin{proof}
	The claim in (1) follows immediately from \eqref{app_eq:dn+1bar} and the fact that $X_{n+1}^{(i)}$ is in $\Cnbar$.
	For (2), note that by taking the reduction modulo $p$ of \eqref{app_eq:colim_dnbar}, we get that
	\begin{equation*}
		\overline{D}_{n+1}/p\overline{D}_{n+1} = (\Dnbar/p\Dnbar)/(\{X_{n+1}^{(i)}\}_{i \in I})[\{\overline{Y}_{i,n+1}\}_{i \in I}].
	\end{equation*}
	To prove (3), let us first recall that $\Dzerobar$ is $p\torsionfree$ by Lemma \ref{app_lem:e0_desc}.
	So, for each $n \in \NN$, by Lemma \ref{app_lem:reg_seq} (2) and \eqref{app_eq:dn+1bar}, it suffices to show that the sequence $\{X_{n+1}^{(i)}\}_{i \in I}$, where $X_{n+1}^{(i)}$ is in $\Cnbar$, is regular on $(\Dnbar/p\Dnbar)[\{\overline{Y}_{i,n+1}\}_{i \in I}]$, i.e.\ it is regular on $\Dnbar/p\Dnbar$.
	For $n \geqslant 1$, the claim follows because we have $\Dnbar/p\Dnbar = S_{n-1}[\{\Yinbar\}_{i \in I}]$ and $\overline{u}_{n+1}$ modulo $p$ is a unit in $\Abar/p\Abar$ (see the discussions after \eqref{app_eq:xni} and before Assumption \ref{app_assum:alg_b}).
	For $n = 0$, it suffices to show that the sequence $\{\{[p]_q Y_{i,0}-Y_i\}_{i \in I}, \{X_1^{(i)}\}_{i \in I}\}$ is regular on $\Czerobar/p\Czerobar$.
	As $[p]_q$ is contained in the ideal $(p, \mu) \subset B$, we are reduced to showing that the sequence $\{\{-Y_i\}_{i \in I}, \{X_1^{(i)}\}_{i \in I}\}$ is regular on $\Czerobar/p\Czerobar = B/(p, \mu)[\{\Yizerobar\}_{i \in I}]$.
	This is obvious because $\overline{u}_1$ modulo $p$ is a unit in $\Abar/p\Abar$ (see the discussions after \eqref{app_eq:xni} and before Assumption \ref{app_assum:alg_b}).
	Hence, the lemma is proved.
\end{proof}

Now, note that reducing the top arrow in \eqref{eq:d_to_e_bar} modulo $\mu$, we obtain the following commutative diagram with surjective arrows:
\begin{equation}\label{eq:d_to_e_bar_2}
	\begin{tikzcd}
		D/\mu D \arrow[d, "\wr"] \arrow[r] & E/\mu E \arrow[d, "\eqref{app_eq:e_to_ebar}"],\\
		\Dbar \arrow[r] & \Ebar.
	\end{tikzcd}
\end{equation}
\begin{prop}\label{app_prop:dbar_ebar_iso}
	All arrows in the diagram \eqref{eq:d_to_e_bar_2} are isomorphisms.
\end{prop}
\begin{proof}
	By definition, the left vertical arrow in \eqref{eq:d_to_e_bar_2} is an isomorphism and all arrows are surjective.
	So, it suffices to show that the lower horizontal arrow is injective.
	Now, we have the following commutative diagram:
	\begin{center}
		\begin{tikzcd}
			\Dzerobar[1/p] \arrow[r] \arrow[d, "\eqref{app_eq:d0_to_e0}", "\wr"'] & \Dbar[1/p] \arrow[d]\\
			\Ezerobar[1/p] \arrow[r] & \Ebar[1/p],
		\end{tikzcd}
	\end{center}
	where the top horizontal arrow is a bijection by \eqref{app_eq:colim_dnbar} and Lemma \ref{app_lem:dnbar} (1), the left vertical arrow is a bijection by Lemma \ref{app_lem:e0_desc}, and the bottom horizontal arrow is a bijection by the definition of $\Ebar$.
	Therefore, it follows that the right vertical arrow is a bijection as well, i.e.\ $\Dbar[1/p] \isomorphic \Ebar[1/p]$.
	Finally, from \eqref{app_eq:colim_dnbar} and Lemma \ref{app_lem:dnbar} (3), note that $\Dbar$ is $p\torsionfree$.
	Hence, we get that the map $\Dbar \rightarrow \Ebar$ in \eqref{eq:d_to_e_bar_2} is injective.
	This completes our proof.
\end{proof}


\setstretch{0.95}
\phantomsection
\printbibliography[heading=bibintoc, title={References}]

@article{abhinandan-relative-wach-i,
    AUTHOR = {Abhinandan},
     TITLE = {Crystalline representations and {W}ach modules in the relative
              case},
   JOURNAL = {Ann. Inst. Fourier (Grenoble)},
  FJOURNAL = {Universit\'e{} de Grenoble. Annales de l'Institut Fourier},
    VOLUME = {75},
      YEAR = {2025},
    NUMBER = {1},
     PAGES = {379--474},
      ISSN = {0373-0956,1777-5310},
   MRCLASS = {11F80 (11S25 14F30)},
  MRNUMBER = {4862349},
       DOI = {10.5802/aif.3670},
       URL = {https://doi.org/10.5802/aif.3670},
}

@article{abhinandan-syntomic,
    AUTHOR = {Abhinandan},
     TITLE = {Syntomic complex and {$p$}-adic nearby cycles},
   JOURNAL = {Algebra Number Theory},
  FJOURNAL = {Algebra \& Number Theory},
    VOLUME = {20},
      YEAR = {2026},
    NUMBER = {1},
     PAGES = {17--108},
      ISSN = {1937-0652,1944-7833},
   MRCLASS = {11S25 (14F20 14F30 14F40)},
  MRNUMBER = {5000344},
       DOI = {10.2140/ant.2026.20.17},
       URL = {https://doi.org/10.2140/ant.2026.20.17},
}

@article{abhinandan-imperfect-wach,
    AUTHOR = {Abhinandan},
     TITLE = {Crystalline representations and {W}ach modules in the
              imperfect residue field case},
   JOURNAL = {Doc. Math.},
  FJOURNAL = {Documenta Mathematica},
    VOLUME = {30},
      YEAR = {2025},
    NUMBER = {6},
     PAGES = {1461--1524},
      ISSN = {1431-0635,1431-0643},
   MRCLASS = {11S23 (14F20 14F30 14F40)},
  MRNUMBER = {4959787},
       DOI = {10.4171/dm/1015},
       URL = {https://doi.org/10.4171/dm/1015},
}

@article{abhinandan-relative-wach-ii,
       author = {{Abhinandan}},
        title = "{Crystalline representations and Wach modules in the relative case II}",
      journal = {arXiv e-prints},
         year = 2023,
        month = sep,
          eid = {arXiv:2309.16446},
	 note = {To appear in Math.\ Ann.},
}

@article{abhinandan-crystalline-galcoh,
       author = {{Abhinandan}},
        title = "{Crystalline part of the Galois cohomology of crystalline representations}",
      journal = {arXiv e-prints},
         year = 2024,
        month = may,
          eid = {arXiv:2405.17012},
	 note = {To appear in Bull.\ Soc.\ Math.\ France},
}

@article {andreatta-phigamma,
    AUTHOR = {Andreatta, Fabrizio},
     TITLE = {Generalized ring of norms and generalized
              {$(\phi,\Gamma)$}-modules},
   JOURNAL = {Ann. Sci. \'{E}cole Norm. Sup. (4)},
  FJOURNAL = {Annales Scientifiques de l'\'{E}cole Normale Sup\'{e}rieure. Quatri\`eme
              S\'{e}rie},
    VOLUME = {39},
      YEAR = {2006},
    NUMBER = {4},
     PAGES = {599--647},
      ISSN = {0012-9593},
   MRCLASS = {12J10 (13F30)},
  MRNUMBER = {2290139},
MRREVIEWER = {Alan Koch},
       DOI = {10.1016/j.ansens.2006.07.003},
       URL = {https://doi.org/10.1016/j.ansens.2006.07.003},
}

@article {anschutz-lebras,
    AUTHOR = {Ansch\"{u}tz, Johannes and Le Bras, Arthur-C\'{e}sar},
     TITLE = {Prismatic {D}ieudonn\'{e} theory},
   JOURNAL = {Forum Math. Pi},
  FJOURNAL = {Forum of Mathematics. Pi},
    VOLUME = {11},
      YEAR = {2023},
     PAGES = {Paper No. e2, 92},
      ISSN = {2050-5086},
   MRCLASS = {14F30 (14L05)},
  MRNUMBER = {4530092},
MRREVIEWER = {Rui\ Miguel\ Saramago},
       DOI = {10.1017/fmp.2022.22},
       URL = {https://doi.org/10.1017/fmp.2022.22},
}

@article {antieau-mathew-morrow-nikolaus,
    AUTHOR = {Antieau, Benjamin and Mathew, Akhil and Morrow, Matthew and
              Nikolaus, Thomas},
     TITLE = {On the {B}eilinson fiber square},
   JOURNAL = {Duke Math. J.},
  FJOURNAL = {Duke Mathematical Journal},
    VOLUME = {171},
      YEAR = {2022},
    NUMBER = {18},
     PAGES = {3707--3806},
      ISSN = {0012-7094,1547-7398},
   MRCLASS = {14F30 (14F40 19D55 19E15)},
  MRNUMBER = {4516307},
MRREVIEWER = {Guillermo\ Corti\~{n}as},
       DOI = {10.1215/00127094-2022-0037},
       URL = {https://doi.org/10.1215/00127094-2022-0037},
}

@article {berger-limites,
    AUTHOR = {Berger, Laurent},
     TITLE = {Limites de repr\'{e}sentations cristallines},
   JOURNAL = {Compos. Math.},
  FJOURNAL = {Compositio Mathematica},
    VOLUME = {140},
      YEAR = {2004},
    NUMBER = {6},
     PAGES = {1473--1498},
      ISSN = {0010-437X},
   MRCLASS = {11S25 (11F80 11R23 13K05 14F30)},
  MRNUMBER = {2098398},
       DOI = {10.1112/S0010437X04000879},
       URL = {https://doi.org/10.1112/S0010437X04000879},
}

@book {berthelot-cohomologie-cristalline,
    AUTHOR = {Berthelot, Pierre},
     TITLE = {Cohomologie cristalline des sch\'{e}mas de caract\'{e}ristique
              {$p>0$}.},
    SERIES = {},
 PUBLISHER = {Springer-Verlag, Berlin-New York,, },
      YEAR = {1974},
     PAGES = {604},
   MRCLASS = {14F30},
  MRNUMBER = {384804},
MRREVIEWER = {L.\ Illusie},
}

@article{bhatt-lurie-apc,
       author = {Bhatt, Bhargav and Lurie, Jacob},
        title = "{Absolute prismatic cohomology}",
      journal = {arXiv e-prints},
         year = 2022,
        month = jan,
          eid = {arXiv:2201.06120},
}

@article {bhatt-mathew,
    AUTHOR = {Bhatt, Bhargav and Mathew, Akhil},
     TITLE = {Syntomic complexes and {$p$}-adic \'{e}tale {T}ate twists},
   JOURNAL = {Forum Math. Pi},
  FJOURNAL = {Forum of Mathematics. Pi},
    VOLUME = {11},
      YEAR = {2023},
     PAGES = {Paper No. e1, 26},
      ISSN = {2050-5086},
   MRCLASS = {14F30 (14F42)},
  MRNUMBER = {4530091},
       DOI = {10.1017/fmp.2022.21},
       URL = {https://doi.org/10.1017/fmp.2022.21},
}

@article {bhatt-scholze-prisms,
    AUTHOR = {Bhatt, Bhargav and Scholze, Peter},
     TITLE = {Prisms and prismatic cohomology},
   JOURNAL = {Ann. of Math. (2)},
  FJOURNAL = {Annals of Mathematics. Second Series},
    VOLUME = {196},
      YEAR = {2022},
    NUMBER = {3},
     PAGES = {1135--1275},
      ISSN = {0003-486X,1939-8980},
   MRCLASS = {14F30 (14F20 14F40)},
  MRNUMBER = {4502597},
       DOI = {10.4007/annals.2022.196.3.5},
       URL = {https://doi.org/10.4007/annals.2022.196.3.5},
}

@article{bhatt-scholze-crystals,
    AUTHOR = {Bhatt, Bhargav and Scholze, Peter},
     TITLE = {Prismatic {$F$}-crystals and crystalline {G}alois
              representations},
   JOURNAL = {Camb. J. Math.},
  FJOURNAL = {Cambridge Journal of Mathematics},
    VOLUME = {11},
      YEAR = {2023},
    NUMBER = {2},
     PAGES = {507--562},
      ISSN = {2168-0930,2168-0949},
   MRCLASS = {14 (11F80 12 13 16)},
  MRNUMBER = {4600546},
}

@article {bhatt-morrow-scholze-1,
    AUTHOR = {Bhatt, Bhargav and Morrow, Matthew and Scholze, Peter},
     TITLE = {Integral {$p$}-adic {H}odge theory},
   JOURNAL = {Publ. Math. Inst. Hautes \'{E}tudes Sci.},
  FJOURNAL = {Publications Math\'{e}matiques. Institut de Hautes \'{E}tudes
              Scientifiques},
    VOLUME = {128},
      YEAR = {2018},
     PAGES = {219--397},
      ISSN = {0073-8301},
   MRCLASS = {14F30},
  MRNUMBER = {3905467},
MRREVIEWER = {Daniel Robert Gulotta},
       DOI = {10.1007/s10240-019-00102-z},
       URL = {https://doi.org/10.1007/s10240-019-00102-z},
}

@article {bhatt-morrow-scholze-2,
    AUTHOR = {Bhatt, Bhargav and Morrow, Matthew and Scholze, Peter},
     TITLE = {Topological {H}ochschild homology and integral {$p$}-adic
              {H}odge theory},
   JOURNAL = {Publ. Math. Inst. Hautes \'{E}tudes Sci.},
  FJOURNAL = {Publications Math\'{e}matiques. Institut de Hautes \'{E}tudes
              Scientifiques},
    VOLUME = {129},
      YEAR = {2019},
     PAGES = {199--310},
      ISSN = {0073-8301,1618-1913},
   MRCLASS = {14F30 (13A35)},
  MRNUMBER = {3949030},
MRREVIEWER = {Lance\ Edward\ Miller},
       DOI = {10.1007/s10240-019-00106-9},
       URL = {https://doi.org/10.1007/s10240-019-00106-9},
}

@article{brinon-relatif,
    AUTHOR = {Brinon, Olivier},
     TITLE = {Repr\'{e}sentations {$p$}-adiques cristallines et de de {R}ham dans le cas relatif},
   JOURNAL = {M\'{e}m. Soc. Math. Fr. (N.S.)},
  FJOURNAL = {M\'{e}moires de la Soci\'{e}t\'{e} Math\'{e}matique de France. Nouvelle S\'{e}rie},
    NUMBER = {112},
      YEAR = {2008},
     PAGES = {vi+159},
      ISSN = {0249-633X},
      ISBN = {978-2-85629-250-1},
   MRCLASS = {14F30 (11S25 11S80)},
  MRNUMBER = {2484979},
MRREVIEWER = {Laurent N. Berger},
       DOI = {10.24033/msmf.424},
       URL = {https://doi.org/10.24033/msmf.424},
}

@article {colmez-hauteur,
    AUTHOR = {Colmez, Pierre},
     TITLE = {Repr\'{e}sentations cristallines et repr\'{e}sentations de hauteur
              finie},
   JOURNAL = {J. Reine Angew. Math.},
  FJOURNAL = {Journal f\"{u}r die Reine und Angewandte Mathematik. [Crelle's
              Journal]},
    VOLUME = {514},
      YEAR = {1999},
     PAGES = {119--143},
      ISSN = {0075-4102},
   MRCLASS = {11S20 (11S25 14F30)},
  MRNUMBER = {1711279},
MRREVIEWER = {Abdellah Mokrane},
       DOI = {10.1515/crll.1999.068},
       URL = {https://doi.org/10.1515/crll.1999.068},
}

@article{colmez-niziol,
    AUTHOR = {Colmez, Pierre and Nizio{\l} , Wies{\l}awa},
     TITLE = {Syntomic complexes and {$p$}-adic nearby cycles},
   JOURNAL = {Invent. Math.},
  FJOURNAL = {Inventiones Mathematicae},
    VOLUME = {208},
      YEAR = {2017},
    NUMBER = {1},
     PAGES = {1--108},
      ISSN = {0020-9910},
   MRCLASS = {14F30 (11S25 14F20 14F40 14G20 14G22)},
  MRNUMBER = {3621832},
MRREVIEWER = {Adolfo Quir\'{o}s},
       DOI = {10.1007/s00222-016-0683-3},
       URL = {https://doi.org/10.1007/s00222-016-0683-3},
}

@ARTICLE{du-liu,
       author = {{Du}, Heng and {Liu}, Tong},
        title = "{A prismatic approach to $(\varphi, \hat G)$-modules and $F$-crystals}",
      journal = {arXiv e-prints},
         year = 2021,
        month = jul,
          eid = {arXiv:2107.12240},
}

@article{du-liu-moon-shimizu,
    AUTHOR = {Du, Heng and Liu, Tong and Moon, Yong Suk and Shimizu, Koji},
     TITLE = {Completed prismatic {$F$}-crystals and crystalline
              {$Z_p$}-local systems},
   JOURNAL = {Compos. Math.},
  FJOURNAL = {Compositio Mathematica},
    VOLUME = {160},
      YEAR = {2024},
    NUMBER = {5},
     PAGES = {1101--1166},
      ISSN = {0010-437X,1570-5846},
   MRCLASS = {14F30 (11F80 14G45)},
  MRNUMBER = {4733770},
       DOI = {10.1112/S0010437X24007097},
       URL = {https://doi.org/10.1112/S0010437X24007097},
}

@book {fontaine-pdivisibles,
    AUTHOR = {Fontaine, Jean-Marc},
     TITLE = {Groupes {$p$}-divisibles sur les corps locaux.},
    SERIES = {},
 PUBLISHER = {Soci\'{e}t\'{e} Math\'{e}matique de France, Paris,, },
      YEAR = {1977},
     PAGES = {i+262},
   MRCLASS = {14L05},
  MRNUMBER = {498610},
MRREVIEWER = {Loren\ D.\ Olson},
}

@incollection {fontaine-phigamma,
    AUTHOR = {Fontaine, Jean-Marc},
     TITLE = {Repr\'{e}sentations {$p$}-adiques des corps locaux. {I}},
 BOOKTITLE = {The {G}rothendieck {F}estschrift, {V}ol. {II}},
    SERIES = {Progr. Math.},
    VOLUME = {87},
     PAGES = {249--309},
 PUBLISHER = {Birkh\"{a}user Boston, Boston, MA},
      YEAR = {1990},
   MRCLASS = {11S23 (14F30 14L05)},
  MRNUMBER = {1106901},
MRREVIEWER = {Rutger Noot},
}

@article{fontaine-corps-periodes,
	AUTHOR = {Fontaine, Jean-Marc},
	TITLE = {Le corps des p\'{e}riodes {$p$}-adiques},
	NOTE = {With an appendix by Pierre Colmez, P\'{e}riodes $p$-adiques (Bures-sur-Yvette, 1988)},
	JOURNAL = {Ast\'{e}risque},
	FJOURNAL = {Ast\'{e}risque},
	NUMBER = {223},
	YEAR = {1994},
	PAGES = {59--111},
	ISSN = {0303-1179},
	MRCLASS = {11G25 (14F20 14F30 14F40)},
	MRNUMBER = {1293971},
	MRREVIEWER = {Adolfo Quir\'{o}s},
}

@article {gros-lestum-quiros-1,
    AUTHOR = {Gros, Michel and Le Stum, Bernard and Quir\'{o}s, Adolfo},
     TITLE = {Twisted differential operators of negative level and prismatic
              crystals},
   JOURNAL = {Tunis. J. Math.},
  FJOURNAL = {Tunisian Journal of Mathematics},
    VOLUME = {4},
      YEAR = {2022},
    NUMBER = {1},
     PAGES = {19--53},
      ISSN = {2576-7658,2576-7666},
   MRCLASS = {14F30 (05A30 13N10 14F40)},
  MRNUMBER = {4401787},
MRREVIEWER = {Alexander\ B.\ Levin},
       DOI = {10.2140/tunis.2022.4.19},
       URL = {https://doi.org/10.2140/tunis.2022.4.19},
}

@ARTICLE{gros-lestum-quiros-2,
       author = {{Gros}, Michel and {Le Stum}, Bernard and {Quir{\'o}s}, Adolfo},
        title = "{Absolute calculus and prismatic crystals on cyclotomic rings}",
      journal = {arXiv e-prints},
         year = 2023,
        month = oct,
          eid = {arXiv:2310.13790},
}

@article {guo-reinecke,
    AUTHOR = {Guo, Haoyang and Reinecke, Emanuel},
     TITLE = {A prismatic approach to crystalline local systems},
   JOURNAL = {Invent. Math.},
  FJOURNAL = {Inventiones Mathematicae},
    VOLUME = {236},
      YEAR = {2024},
    NUMBER = {1},
     PAGES = {17--164},
      ISSN = {0020-9910,1432-1297},
   MRCLASS = {14F30},
  MRNUMBER = {4712864},
       DOI = {10.1007/s00222-024-01238-4},
       URL = {https://doi.org/10.1007/s00222-024-01238-4},
}

@article {iwasawa,
    AUTHOR = {Iwasawa, Kenkichi},
     TITLE = {On the theory of cyclotomic fields},
   JOURNAL = {Ann. of Math. (2)},
  FJOURNAL = {Annals of Mathematics. Second Series},
    VOLUME = {70},
      YEAR = {1959},
     PAGES = {530--561},
      ISSN = {0003-486X},
   MRCLASS = {10.66},
  MRNUMBER = {124320},
MRREVIEWER = {M.\ Krasner},
       DOI = {10.2307/1970328},
       URL = {https://doi.org/10.2307/1970328},
}

@ARTICLE{min-wang,
    AUTHOR = {Min, Yu and Wang, Yupeng},
     TITLE = {Relative {$(\varphi,\Gamma)$}-modules and prismatic
              {$F$}-crystals},
   JOURNAL = {Math. Z.},
  FJOURNAL = {Mathematische Zeitschrift},
    VOLUME = {310},
      YEAR = {2025},
    NUMBER = {1},
     PAGES = {Paper No. 16, 25},
      ISSN = {0025-5874,1432-1823},
   MRCLASS = {14G22 (14F30 14G45)},
  MRNUMBER = {4884577},
       DOI = {10.1007/s00209-025-03696-9},
       URL = {https://doi.org/10.1007/s00209-025-03696-9},
}

@ARTICLE{morrow-tsuji,
       author = {{Morrow}, Matthew and {Tsuji}, Takeshi},
        title = "{Generalised representations as q-connections in integral $p$-adic Hodge theory}",
      journal = {arXiv e-prints},
         year = 2020,
        month = oct,
          eid = {arXiv:2010.04059},
}

@misc{stacks-project,
  author       = {The {Stacks project authors}},
  title        = {The Stacks project},
  howpublished = {\url{https://stacks.math.columbia.edu}},
  year         = {2023},
}

@article {tian,
    AUTHOR = {Tian, Yichao},
     TITLE = {Finiteness and duality for the cohomology of prismatic
              crystals},
   JOURNAL = {J. Reine Angew. Math.},
  FJOURNAL = {Journal f\"{u}r die Reine und Angewandte Mathematik. [Crelle's
              Journal]},
    VOLUME = {800},
      YEAR = {2023},
     PAGES = {217--257},
      ISSN = {0075-4102,1435-5345},
   MRCLASS = {14F30 (14G45)},
  MRNUMBER = {4609829},
       DOI = {10.1515/crelle-2023-0032},
       URL = {https://doi.org/10.1515/crelle-2023-0032},
}

@ARTICLE{tsuji-qhiggs-field,
       author = {{Tsuji}, Takeshi},
        title = "{Prismatic crystals and $q$-Higgs fields}",
      journal = {arXiv e-prints},
         year = 2024,
        month = mar,
          eid = {arXiv:2403.11676},
}

@article {wach-free,
    AUTHOR = {Wach, Nathalie},
     TITLE = {Repr\'{e}sentations {$p$}-adiques potentiellement cristallines},
   JOURNAL = {Bull. Soc. Math. France},
  FJOURNAL = {Bulletin de la Soci\'{e}t\'{e} Math\'{e}matique de France},
    VOLUME = {124},
      YEAR = {1996},
    NUMBER = {3},
     PAGES = {375--400},
      ISSN = {0037-9484},
   MRCLASS = {11S23 (14F30 14L05)},
  MRNUMBER = {1415732},
MRREVIEWER = {Abdellah Mokrane},
       URL = {http://www.numdam.org/item?id=BSMF_1996__124_3_375_0},
}

@article {wach-torsion,
    AUTHOR = {Wach, Nathalie},
     TITLE = {Repr\'{e}sentations cristallines de torsion},
   JOURNAL = {Compositio Math.},
  FJOURNAL = {Compositio Mathematica},
    VOLUME = {108},
      YEAR = {1997},
    NUMBER = {2},
     PAGES = {185--240},
      ISSN = {0010-437X},
   MRCLASS = {11S23 (14F30 14L05)},
  MRNUMBER = {1468834},
MRREVIEWER = {Abdellah Mokrane},
       DOI = {10.1023/A:1000108818774},
       URL = {https://doi.org/10.1023/A:1000108818774},
}

@article {wu,
    AUTHOR = {Wu, Zhiyou},
     TITLE = {Galois representations, {$(\varphi,\Gamma)$}-modules and
              prismatic {F}-crystals},
   JOURNAL = {Doc. Math.},
  FJOURNAL = {Documenta Mathematica},
    VOLUME = {26},
      YEAR = {2021},
     PAGES = {1771--1798},
      ISSN = {1431-0635,1431-0643},
   MRCLASS = {11F80},
  MRNUMBER = {4353339},
MRREVIEWER = {Shiang\ Tang},
}

\Addresses

\end{document}